\documentclass[11pt,reqno]{amsart}
\usepackage{geometry} 

\usepackage{amsmath,amssymb,amsthm,mathrsfs,appendix}
\usepackage{xcolor}
\usepackage{graphicx}
\usepackage{tikz}
\usepackage{tikz-3dplot}
\usetikzlibrary{positioning}
\usepackage{pgfplots}
\usepackage{subcaption}
\usepackage{adjustbox}

\usetikzlibrary{arrows.meta}

\usepackage{marvosym}
\usepackage{cancel}

\usepackage{latexsym}
\usepackage[nobysame]{amsrefs}[11pts]
\usepackage{bm}
\usepackage{enumerate}
\usepackage{indentfirst}
\usepackage[colorlinks,
            linkcolor=black,
            anchorcolor=black,
            citecolor=black
            ]{hyperref}

\newcommand{\into}{\hookrightarrow}

\newcommand{\ds}{\mathrm{d}s}
\newcommand{\dt}{\mathrm{d}t}

\newcommand\norm[1]{\left\|#1\right\|}

\DeclareMathOperator{\dive}{div}

\newcommand{\II}{\mathbb{I}}

\newcommand{\NN}{\mathbb{N}}

\newcommand{\RR}{\mathbb{R}}
\renewcommand{\SS}{\mathbb{S}}
\newcommand{\TT}{\mathbb{T}}

\newcommand{\ZZ}{\mathbb{Z}}

\newcommand{\cB}{\mathcal{B}}

\newcommand{\cD}{\mathcal{D}}
\newcommand{\cE}{\mathcal{E}}

\newcommand{\cG}{\mathcal{G}}
\newcommand{\cH}{\mathcal{H}}

\newcommand{\cJ}{\mathcal{J}}

\newcommand{\cT}{\mathcal{T}}

\newcommand{\cV}{\mathcal{V}}

\newcommand{\cY}{\mathcal{Y}}
\newcommand{\cZ}{\mathcal{Z}}

\newcommand{\wv}{\widehat{U}}

\newcommand{\weta}{\widehat{\eta}}
\newcommand{\sD}{\mathscr{D}}

\newtheorem{Definition}{Definition}
\newtheorem{Theorem}{Theorem}[section]
\newtheorem{Lemma}{Lemma}[section]
\newtheorem{Proposition}{Proposition}
\newtheorem{Remark}{Remark}

\numberwithin{Remark}{section}
\numberwithin{Proposition}{section}
\numberwithin{Definition}{section}
\numberwithin{Lemma}{section}
\numberwithin{equation}{section}
\numberwithin{Theorem}{section}

\allowdisplaybreaks[4]
\pgfplotsset{compat=1.18}

\begin{document}
\title[Viscous Gaseous Stars]{
On a Local Existence Theorem for  the Evolution Equation of Viscous Gaseous Stars in a Physical Vacuum}
\date{\today}

\author{Demin Wang}
\address[Demin Wang]{School of Mathematical Sciences, Shanghai Jiao Tong University, Shanghai 200240, P. R. China} \email{\tt de-min-wang@sjtu.edu.cn}

\author{Jiawen Zhang}
\address[Jiawen Zhang]{School of Mathematical Sciences, Shanghai Jiao Tong University, Shanghai 200240, P. R. China} \email{\tt zhangjiawen317@sjtu.edu.cn}

\author{Shengguo Zhu }
\address[Shengguo Zhu]{School of Mathematical Sciences, CMA--Shanghai, and MOE--LSC, Shanghai Jiao Tong University, Shanghai 200240, P. R. China.} \email{\tt  zhushengguo@sjtu.edu.cn}

\begin{abstract}
This paper focuses on the free boundary problem of the three-dimensional compressible Navier--Stokes--Poisson equations with degenerate viscosities for self-gravitating viscous gaseous stars. For spherically symmetric barotropic motion, we establish the local-in-time existence  of the unique  classical solution. The solutions obtained here are classical up to the moving vacuum boundary for $0<t\le T_*$, where $T_*>0$ is the local existence time. The density exhibits the physical vacuum boundary behavior of the Lane--Emden star configuration for all adiabatic exponents $\gamma>\frac{4}{3}$.

\end{abstract}

\date{\today}

\maketitle

\tableofcontents

\section{Introduction}\label{section1}
The motion of a self-gravitating viscous gaseous star in the universe can be described
 by the following 
 vacuum free boundary problem (\textbf{VFBP}) of the three-dimensional (3-D)  compressible Navier--Stokes--Poisson equations (\textbf{CNSP}):
\begin{equation}\label{eq:1.1-vfbp}
\begin{cases}
\rho_t+\dive(\rho \boldsymbol{u})=0 &\text{in }\Omega(t),\\[2pt]
(\rho \boldsymbol{u})_t+\dive(\rho \boldsymbol{u}\otimes \boldsymbol{u})+\nabla P+\rho\nabla\Psi =\dive \TT&\text{in }\Omega(t),\\[2pt]
\rho>0&\text{in }\Omega(t),\\[2pt]
\rho=0&\text{on }\partial\Omega(t),\\[2pt]
\cV(\partial\Omega(t))=\boldsymbol{u}\cdot\boldsymbol{n}(t)&\text{on }\partial\Omega(t),\\[2pt]
(\rho,\boldsymbol{u})|_{t=0}=(\rho_0,\boldsymbol{u}_0) &\text{in }\Omega:=\Omega(0).
\end{cases}
\end{equation}
Here, $t\geq 0$ is the time, $\boldsymbol{x}=(x_1,x_2,x_3)^{\top}\in \mathbb{R}^3$  is the Eulerian spatial coordinate, the open and bounded subset  $\Omega(t)\subset \mathbb{R}^3$ denotes the changing volume occupied by the fluid and $
\Omega(0)=B_1=\{\boldsymbol{x}: \,|\boldsymbol{x}| <1\}$,
$\partial \Omega(t)$ denotes the moving vacuum boundary,  $\cV(\partial \Omega(t))$ denotes the normal velocity of $\partial \Omega(t)$, and $ \boldsymbol{n}(t) $ denotes the exterior unit normal vector to $\partial \Omega(t)$. 
Moreover, $\rho\geq 0$ denotes the mass density of the fluid, $\boldsymbol{u}=(u_1,u_2, u_3)^\top$ $\in \mathbb{R}^3$   the Eulerian velocity field, and $P$  the pressure. 
For polytropic fluids, the constitutive relation is given by  
\begin{equation}\label{pressureform}
P=A\rho^{\gamma},
\end{equation}
where $A>0$ is  the entropy  constant and  $\gamma>1$ is the adiabatic exponent.
$\Psi$ denotes the gravitational potential as
\begin{equation}\label{gravitational potential explicit expression}
    \Psi(t,\boldsymbol{x})=
\displaystyle -G\int_{\Omega(t)}\frac{\rho(t,\boldsymbol{\hat{x}})}{|\boldsymbol{x}-\boldsymbol{\hat{x}}|}\,\mathrm{d}\boldsymbol{\hat{x}}
\end{equation}
with the gravitational constant $G>0$. Since
$\rho> 0$ in $\Omega(t)$ and 
$\rho\equiv 0$ in $\mathbb{R}^3 \setminus \Omega(t)$, $\Psi$ satisfies 
\begin{equation}\label{gravitational potential}
\Delta \Psi=4\pi G\rho \quad \text{in $\mathbb{R}^3$},\qquad\quad \lim_{|\boldsymbol{x}|\rightarrow \infty} \Psi(t,\boldsymbol{x})=0.
\end{equation}
In particular, $\Psi(t,\boldsymbol{x})$ is defined in the whole space $\mathbb{R}^3$ for each $t>0$.
$\mathbb{T}$ denotes the viscous stress tensor as: 
\begin{equation}\label{eq:1.1t}
\mathbb{T}=2\mu(\rho)D(\boldsymbol{u})+\lambda(\rho)\dive\boldsymbol{u}\,\mathbb{I}_3,
\end{equation} 
where $D(\boldsymbol{u})=\frac{1}{2}(\nabla \boldsymbol{u}+(\nabla \boldsymbol{u})^\top)$ is the deformation tensor, $\mathbb{I}_3$ is the $3\times 3$ identity matrix,
\begin{equation}\label{fandan}
\mu(\rho)=a_1 \rho^\delta,\qquad \lambda(\rho)=a_2\rho^\delta,
\end{equation}
for some  constant $\delta> 0$, $\mu(\rho)$ is the shear viscosity coefficient, $\lambda(\rho)+\frac{2}{3}\mu(\rho)$ is the bulk viscosity coefficient,  $a_1$ and $a_2$ are both constants satisfying
\begin{equation}\label{kelaoxiusi}
a_1>0,\qquad\quad 2a_1+3 a_2\geq 0.
\end{equation} 
Equation $\eqref{eq:1.1-vfbp}_3$ asserts that there is no vacuum inside the fluid, $\eqref{eq:1.1-vfbp}_4$ is the vacuum boundary condition stating that $\rho$ vanishes along the moving  boundary $\partial \Omega(t)$, $\eqref{eq:1.1-vfbp}_5$ is the kinematic boundary condition requiring that the moving boundary travels with the normal velocity of the fluid, and $\eqref{eq:1.1-vfbp}_6$ provides the initial conditions for the density, velocity, and domain. We refer the reader to \cites{Chan, cox} for more details of the related background on \textbf{CNSP}.

With the local sound speed given by $c_s=\sqrt{P'(\rho)}$ and $c_{s,0}=c_s|_{t=0}$,  satisfaction of the condition 
\begin{equation}\label{PVcondition}
-\infty<\frac{\partial c^2_{s,0}}{\partial\boldsymbol{n}}<0 \qquad \text{on $\partial \Omega$}
\end{equation}
defines a physical vacuum boundary (see \cites{LiuTP,taiping2,  coutand3,Jang-M2}), which  ensures a finite, nonzero pressure-induced acceleration in the outward normal direction on the moving boundary.
 In fact,  this notion of physical vacuum can be realized by the stationary solutions of the compressible Euler--Poisson equations or \textbf{CNSP} for gaseous stars, \textit{i.e.}, the Lane--Emden solution (see \cites{Chan, Lin,Makino}).
The aim of this  paper is to establish the local-in-time existence  of the unique classical solution with  physical vacuum  to  \textbf{VFBP} \eqref{eq:1.1-vfbp} of the 3-D \textbf{CNSP} with degenerate  viscosities in the spherically symmetric  motion.

It is worth noting that, for barotropic flows, since $\rho_0>0$ in $\Omega$, the physical vacuum boundary condition \eqref{PVcondition} implies that
 $\rho_0^{\gamma-1}$  tends to zero  like the distance function $\text{dist} (x,\partial \Omega)$ near $\partial \Omega$, which leads to a  strong  degeneracy in the time evolution of hydrodynamic equations. Then  it  is intricate to provide an effective propagation  mechanism for the  regularity of $\boldsymbol{u}$ near the vacuum,
which makes the study on the existence  of classical solutions to the corresponding free boundary problem very difficult  for both the inviscid and viscous  flows.

For the {\rm\textbf{VFBP}} of the isentropic compressible Euler equations (\eqref{eq:1.1-vfbp} with $a_1=a_2=G=0$), some significant advances on the well-posedness of smooth solutions satisfying \eqref{PVcondition} have been obtained. The local existence theory was developed by Coutand--Shkoller \cites{coutand1,coutand3} and Jang--Masmoudi \cites{Jang-M1,Jang-M2}, respectively. Recently,  Hadžić--Jang \cite{Jang-Hadzic} constructed unique global solutions when $\gamma \in (1,\frac{5}{3}]$, and the initial data lie sufficiently close to the expanding compactly supported affine motions constructed by Sideris \cite{sideris}, and they satisfy \eqref{PVcondition}. Later,  by a different approach, Shkoller--Sideris \cite{sideris2}  proved the stability of affine flows and global existence of classical solutions for all $\gamma>1$. On the other hand, for the \textbf{VFBP} of 3-D isentropic  compressible Euler--Poisson equations, the local existence  of classical solutions  was established by Gu--Lei \cite{GL1}, and the global existence  of classical solutions was obtained by Had\v{z}i\'{c}--Jang \cite{Jang-Hadzic2}, which stay close to Sideris affine solutions of the Euler equations. Some other important developments can be found in \cites{
coutand2, LinZeng,Oli1, Rei,LiuTP,taiping2,Yan,Strauss} and the references therein.

For the corresponding \textbf{VFBP} of viscous compressible flows in a physical vacuum (as defined  in \eqref{PVcondition}), under the assumption that the viscosity coefficients obey a power law of $\rho$ ({\it i.e.}, $\rho^\delta$ with some exponent $\delta>0$ in \eqref{fandan}), a series of important advances on the well-posedness of strong or classical solutions has been achieved. 
By taking the effect of gravity force into account,   the global existence of the \text{1-D} strong solution of the barotropic compressible Navier--Stokes equations with small data was proved by Ou--Zeng \cite{OZ}. Later, under a proper smallness assumption,  Luo--Xin--Zeng \cite{LXZ3}  established the global existence of strong solutions satisfying \eqref{PVcondition}  of the \text{3-D} spherically symmetric  \textbf{CNSP}.  
It is worth pointing out that the solutions obtained in \cites{OZ,LXZ3} are just strong ones, which are not smooth all the way up to the moving boundary. To make sure that the solution is smooth all the way up to the moving boundary, by assuming that $\rho_0\in H^5(I)$ and $u_0$ belongs to  a weighted $H^6(I)$ space, Li--Wang--Xin \cites{LWX} established the local well-posedness of classical solutions satisfying \eqref{PVcondition} to one-dimensional (\text{1-D}) viscous Saint--Venant system, and then extended this theory to the two-dimensional (2-D) problem in \cites{LWX2} under the assumption that $\rho_0\in H^7$ and $u_0$ belongs to a weighted $H^8$ space.
Based on some new observations on the crossing derivatives estimates,  when  $\rho_0^{\beta}\in H^{3}$  ($\beta\in (\frac{1}{3},1]$) and vanishes on the moving boundary as the distance function,  Xin--Zhang--Zhu \cite{ZJW}  established the global existence of classical solutions with large data to the \text{1-D} viscous Saint--Venant system. Recently, in Chen--Zhang--Zhu \cite{CZZ2}, the global well-posedness of classical solutions in a physical vacuum to  {\rm\textbf{VFBP}} of the viscous Saint--Venant system with large data has been established in the spherically symmetric motion.
On the other hand, if the  viscosity coefficients are all constants, \textit{i.e.}, $\delta=0$ in \eqref{fandan}, Jang \cite{Jang} established the local well-posedness of strong 
solutions to  \textbf{CNSP} in the spherically symmetric and barotropic motion.
Some other important developments can be found in \cites{Makino,Makino2} and the references therein.

It is worth noting that the corresponding problems will become considerably easier when $\rho_0$ possesses a higher decay rate near the vacuum, {\it i.e.}, $\rho_0^{\gamma-1}\sim (\mathrm{dist}(y,\partial\Omega))^k$ with $k=2,3,\cdots$. For example, the initial sound-speed profile $c_{\mathrm{s},0}=\sqrt{A\gamma}\rho_0^{(\gamma-1)/2}\in H^3(\mathbb{R}^3)$, together with suitable regularity assumptions on the initial velocity and density, allows the use of the symmetrization framework introduced by Makino--Ukai--Kawashima \cite{tms1} to study local well-posedness for the Cauchy problem of the isentropic compressible Euler equations. Related results for the CNS under additional assumptions on the initial data and viscosity coefficients were obtained by  Geng--Li--Zhu \cite{GLZ}. However, if $\rho^{\gamma-1}$ decays to zero like  $ (\mathrm{dist}(y,\partial\Omega))^k$ with $k=2,3,\cdots$,  the pressure-induced acceleration vanishes at the boundary.

\subsection{Lagrangian Reformulation in Spherical Coordinates}\label{sec-1.2} 
In this paper,  we will establish the local existence of the  unique  spherically symmetric (classical) solution of \textbf{VFBP} \eqref{eq:1.1-vfbp}, taking the following form 
\begin{equation}\label{ss-ass}
(\rho,\boldsymbol{u})(t,\boldsymbol{x}) = (\rho(t,|\boldsymbol{x}|), u(t,|\boldsymbol{x}|)\frac{\boldsymbol{x}}{|\boldsymbol{x}|}),
\end{equation}
with the initial data:
\begin{equation}\label{eq:IC}
(\rho,\boldsymbol{u})(0,\boldsymbol{x}) =(\rho_0,\boldsymbol{u}_0)(\boldsymbol{x})= (\rho_0(|\boldsymbol{x}|), u_0(|\boldsymbol{x}|)\frac{\boldsymbol{x}}{|\boldsymbol{x}|}).
\end{equation}
Our results hold for $(\delta,\gamma)$ satisfying
\begin{equation}\label{gammadelta}
\delta\in (0,1),\qquad \gamma\in (\frac{4}{3},\infty),
\end{equation}
without any restriction on the size of the initial data.  The  initial density  $\rho_0$ we consider satisfies the following condition:
\begin{equation}\label{distanceeuler} 
\rho_0^{\gamma-1}(\boldsymbol{x})\in H^3(\Omega), \quad  K_1(1-|\boldsymbol{x}|) \leq \rho_0^{\gamma-1}\leq K_2(1-|\boldsymbol{x}|) \qquad \text{for all $\boldsymbol{x}\in \overline\Omega$},
\end{equation}
for  some  constants $K_2>K_1>0$. It is worth noting that  \eqref{distanceeuler} implies that $\rho_0$  satisfies the  physical vacuum boundary condition for spherically symmetric flows, {\it i.e.},
\begin{equation}\label{PVconditionr}
\rho^{\gamma-1}_0\sim 1-|\boldsymbol{x}| \qquad \text{as $|\boldsymbol{x}|$ near the vacuum boundary $|\boldsymbol{x}|=1$}.
\end{equation}

Since we focus on the  spherically symmetric flow, we first reformulate problem \eqref{eq:1.1-vfbp} into the following form in $I(t)=[0,R(t))$ as the radial projection of the moving domain $\Omega(t)$ 
with $R(0)=1$ and $I=[0,1)$:
\begin{equation}\label{shallow-SSR-euler}
\begin{cases}
\displaystyle 
\rho_t+u\rho_x+ \rho \big(u_x+\frac{2 u}{x}\big)=0&\text{in } I(t),\\[1pt]
\displaystyle
\rho u_t+\rho u u_x+P_x+\rho \Psi_x &\\
\quad\quad=(2a_1+a_2)\Big(\rho^\delta\big(u_x+\frac{2 u}{x}\big)\Big)_x-4 a_1 \frac{(\rho^\delta)_x u}{x}&\text{in } I(t),\\[2pt]
\rho>0&\text{in } I(t),\\[2pt]
\rho=0&\text{on } \{x=R(t)\},\\[2pt]
R'(t)=u(t,R(t))&\text{on } \{x=R(t)\},\\[2pt]
(\rho,u)|_{t=0}=(\rho_0,u_0)&\text{in } I(0):=I, 
\end{cases}
\end{equation}
where $x=|\boldsymbol{x}|$. The radial derivative of the
gravitational potential $\Psi_x$ is given by 
\begin{equation}\label{psi_x express int}
 \Psi_x(t,x)=\displaystyle\frac{4\pi G }{x^2}\int_0^x\rho(t,\hat{x})\hat{x}^2\,\mathrm{d}\hat{x} \qquad \text{for } x\in \bar{I}(t).
% \begin{cases}
% \displaystyle\frac{4\pi G }{x^2}\int_0^x\rho(t,\hat{x})\hat{x}^2\,\mathrm{d}\hat{x} \qquad &\text{for } x\in \bar{I}(t).\\[8pt]
% % \displaystyle
% % \frac{4\pi G}{x^2}\int_0^{R(t)}\rho(t,\hat{x})\hat{x}^2\,\mathrm{d}\hat{x} \qquad &\text{for } x\in [0,\infty)\backslash\bar{I}(t).
% \end{cases}
\end{equation}
Moreover, $u|_{x=0}=0$, which can be derived from the continuity of $\boldsymbol{x}\mapsto \boldsymbol{u}(t,\boldsymbol{x})$ at $\boldsymbol{x}=\boldsymbol{0}$.

Second, problem \eqref{shallow-SSR-euler}, formulated in Eulerian coordinates on the moving interval $I(t)$, can be transformed to a problem on the fixed interval $I$ by introducing the Lagrangian coordinates. To this end, let $x=\eta(t,r)$ denote the radial position at time $t\geq 0$ of a fluid particle whose initial radial position is $r\in \bar I$. The flow map $\eta$ satisfies
\begin{equation}\label{flowmap-r-la}
\eta_t(t,r)=u(t,\eta(t,r)),\qquad \eta(0,r)=r,
\end{equation}
and $(t,r)$ are the Lagrangian coordinates. Moreover,  $\eta(t,0)=0$ for $t\in [0,T]$ due to  $u(t,0)=0$.
Then, by introducing the Lagrangian density, velocity, and gravitational potential,
\begin{equation}\label{varrho-U}
\varrho(t,r)= \rho(t, \eta(t,r)),\qquad U(t,r)= u(t, \eta(t,r)),\qquad \Phi(t,r)=\Psi(t,\eta(t,r)),
\end{equation}
problem \eqref{shallow-SSR-euler} can be written in the following initial-boundary value problem ({\rm\bf IBVP}) in the fixed domain $I$ in Lagrangian coordinates $(t,r)$:
\begin{equation}\label{eq:VFBP-La}
\begin{cases}
\displaystyle \varrho_t + \varrho\big(\frac{U_r}{\eta_r}+\frac{2U}{\eta}\big)=0 & \text{in } (0, T] \times I,\\[2pt]
\displaystyle \eta_r \varrho U_t +A  (\varrho^{\gamma})_r+\varrho\Phi_r=(2a_1+a_2)\Big(\varrho^\delta\big(\frac{U_r}{\eta_r}+ \frac{2U}{\eta}\big)\Big)_r - 4a_1 \frac{(\varrho^\delta)_r U}{\eta}& \text{in }(0, T] \times I,\\[2pt]
\eta_t = U & \text{in } (0, T] \times I,\\[2pt]
\varrho>0 & \text{in } (0, T] \times I,\\[2pt]
\varrho|_{r=1}=0 & \text{on } (0, T],\\[2pt] 
(\varrho, U, \eta)(0,r)= (\rho_0(r), u_0(r),r) & \text{for $r\in I$},
\end{cases} 
\end{equation}
where $\delta\in (0,1)$, $\gamma\in (\frac{4}{3},\infty)$, and $\Phi_r$ is given by
\begin{equation}\label{phi_r express int 1}
\Phi_r=\displaystyle\frac{4\pi G \eta_r}{\eta^2}\int_0^r\varrho(t,\hat{r})\eta^2\eta_r\,\mathrm{d}\hat{r}\qquad \text{for } r\in \bar{I}.
% \begin{cases}
% \displaystyle\frac{4\pi G \eta_r}{\eta^2}\int_0^r\varrho(t,\hat{r})\eta^2\eta_r\,\mathrm{d}\hat{r} &\text{for } r\in \bar{I},\\[6pt]
% \displaystyle
% \frac{4\pi G \eta_r}{\eta^2}\int_0^1\varrho(t,\hat{r})\eta^2\eta_r\,\mathrm{d}\hat{r} &\text{for } r\in [0,\infty)\backslash\bar{I}.
% \end{cases}
\end{equation}
It follows from the facts that  $u|_{x=0}=0$ and $\eta|_{r=0}=0$ that $U(t,0)= u(t, \eta(t,0))=0$.

Moreover, it follows from Lemma \ref{lemma-initial} in Appendix \ref{appb} that condition \eqref{distanceeuler}, which is initially satisfied by $\rho_0(\boldsymbol{x})$ in 3-D Eulerian coordinates, can be rewritten in spherical coordinates as a condition satisfied by $\rho_0(r)$: for some  constants $K_2>K_1>0$,
\begin{equation}\label{distance-la}
\begin{aligned}
&r\big(\rho_0^{\gamma-1},(\rho_0^{\gamma-1})_r,(\rho_0^{\gamma-1})_{rr},\frac{(\rho_0^{\gamma-1})_r}{r},(\rho_0^{\gamma-1})_{rrr},(\frac{(\rho_0^{\gamma-1})_r}{r})_r\big)\in L^2(I),\\
&K_1(1-r) \leq \rho_0^{\gamma-1}(r)\leq K_2(1-r) \qquad \text{for all $r\in \bar I$}.
\end{aligned}
\end{equation}
The corresponding study of \eqref{eq:VFBP-La} is very difficult, since the structure of the momentum equation $\eqref{eq:VFBP-La}_2$ is full of nonlinearity, degeneracy, and singularity.  Thus, new ideas are required to establish the local existence and uniqueness of classical solutions in the presence of these difficulties.

\subsection{Notations and Conventions} Before stating the main theorem, we first list the notations and conventions used throughout this paper.
\subsubsection{Notations on coordinates and operators}
\begin{itemize}
\item  $\boldsymbol{y}\in \Omega:=\{\boldsymbol{y}:\,|\boldsymbol{y}|<1\}$ denotes the 3-D Lagrangian spatial coordinates (see Appendix B);
 $I:=[0,1)$; and $r=|\boldsymbol{y}|\in I$ denotes the radial coordinate.
\smallskip
\item For any function $f(t,r)$,
\begin{align*}
\partial_t\partial_r^l f=f_{t\underbrace{\text{\tiny$r\cdots r$}}_{\text{$l$-times}}},\quad
D_\eta f:=\frac{f_r}{\eta_r},\quad D_\eta^k f:=D_\eta(D_\eta^{k-1} f) \quad\text{for $k\in \NN^*$ and $k\geq 2$.}
\end{align*}
Moreover, to simplify the notations, we define the operator $\mathscr{D}_\eta$ as{\rm:}
\qquad \begin{align*}
\mathscr{D}_\eta f&:=(D_\eta f,\frac{f}{\eta}),\qquad \mathscr{D}_\eta^2f:=(D_\eta^2 f,D_\eta(\frac{f}{\eta})),\\
\mathscr{D}_\eta^3f&:=(D_\eta^3 f,\frac{D_\eta^2 f}{\eta},D_\eta^2(\frac{f}{\eta}),\frac{1}{\eta}D_\eta(\frac{f}{\eta})),\\
\mathscr{D}_\eta^4f&:=(D_\eta^4 f,D_\eta(\frac{D_\eta^2 f}{\eta}),D_\eta^3(\frac{f}{\eta}),D_\eta(\frac{1}{\eta}D_\eta(\frac{f}{\eta}))).
\end{align*}
$\mathscr{D}_r^k$ ($k=1,\cdots,4$) is defined in the same way as $\mathscr{D}_\eta^k$ ($k=1,\cdots,4$), except with $r$ in place of $\eta(r)$.
Notice from Lemma \ref{lemma-initial}  in Appendix \ref{appb} that the operator $\mathscr{D}_\eta$ can be formally seen as a representative of the gradient $\nabla$ in the radial Lagrangian coordinate. Besides, the following useful properties will be used in later analysis:
\begin{equation*}
\begin{aligned}
&D_\eta(\mathscr{D}_\eta f)=\mathscr{D}_\eta^2 f,\qquad D_\eta(\mathscr{D}_\eta^3 f) =\mathscr{D}_\eta^4 f,\\
&|D_\eta f|^2+\Big|\frac{f}{\eta}\Big|^2=|\mathscr{D}_\eta f|^2,\qquad |D_\eta^2 f|^2+\Big|D_\eta(\frac{f}{\eta})\Big|^2=|\mathscr{D}_\eta^2 f|^2,\\
&|D_\eta(\mathscr{D}_\eta^2 f)|^2+\Big|\frac{\mathscr{D}_\eta^2 f}{\eta}\Big|^2=|\mathscr{D}_\eta^3 f|^2,\qquad |D_\eta^2(\mathscr{D}_\eta^2 f)|^2+\Big|D_\eta\big(\frac{\mathscr{D}_\eta^2 f}{\eta}\big)\Big|^2=|\mathscr{D}_\eta^4 f|^2.
\end{aligned}
\end{equation*} 
\end{itemize}

\subsubsection{Notations on function spaces}
\begin{itemize}

\item For any function space $X(I)$ appearing in this paper, unless otherwise specified, the following notations are used{\rm :}
\begin{align*}
&X=X(I),\quad X^*\text{ --- the dual space of $X$},\\ 
&X^*([0,T];Y^*)\text{ --- the dual space of $X([0,T];Y)$},\\[4pt]
&|f|_p=\|f\|_{L^p},\quad \|f\|_{k,p}=\|f\|_{W^{k,p}},\quad \|f\|_k=\|f\|_{H^k},\\[4pt]
&L^p_{\mathrm{loc}}:=\big\{f:\, f\in L^p(K) \,\, \text{for any open interval $K$ such that $\bar K\subset I\backslash \{0\}$}\big\},\\[4pt]
&H^k_{\mathrm{loc}}:=\big\{f: \, \partial_r^jf\in L^2_{\mathrm{loc}} \,\, \text{for any $0\leq j\leq k$}\big\},\qquad \|f\|_{X_t(Y)}=\|f\|_{X([0,T];Y(I))}.
\end{align*}
\item Unless otherwise specified, the following definitions of weighted function spaces are used: let $J\subset I$ and let $0\leq \mathrm{w}=\mathrm{w}(r)$ be some function on $J$,
\begin{equation*}
\begin{aligned}
&H^{k}_\mathrm{w}(J):=\big\{f:\, \sqrt{\mathrm{w}}\partial_r^j f\in L^2(J) \,\,\text{for $0\leq j\leq k$}\big\}, \qquad H^{-k}_{\mathrm{w}}(J):=(H^{k}_{\mathrm{w}}(J))^*,\\[-2pt]
&L^2_\mathrm{w}(J)=H^{0}_\mathrm{w}(J),\quad \|f\|_{L^2_\mathrm{w}(J)}=\|\sqrt{\mathrm{w}} f\|_{L^2(J)},\quad \|f\|_{H^k_\mathrm{w}(J)}=\sum_{j=0}^k \|\partial_r^j f\|_{L^2_\mathrm{w}(J)}.
\end{aligned}
\end{equation*}
In particular, if $J=I$, then 
\begin{equation*}
\begin{aligned}
&\qquad \ \ \ \ H^{k}_\mathrm{w}=H^{k}_\mathrm{w}(I), \quad H^{-k}_\mathrm{w}=H^{-k}_\mathrm{w}(I), \quad L^2_\mathrm{w}=L^2_\mathrm{w}(I),\quad
|f|_{2,\mathrm{w}}=\|f\|_{L^2_\mathrm{w}},\quad \|f\|_{k,\mathrm{w}}=\|f\|_{H^k_\mathrm{w}}.
\end{aligned}
\end{equation*}

\smallskip
\item 
We denote by $\cH^1_{\mathrm{w}}(J)$ the space of all functions $f$ satisfying $(f,f_r,\frac{f}{r})\in L^2_\mathrm{w}(J)$ for some interval $J=(0,a)$ with $a\in (0,1)$:
\begin{equation*}
\cH^1_{\mathrm{w}}(J):=\big\{f:\,\|f\|_{\cH^1_{\mathrm{w}}(J)}<\infty\big\}\,\,\,\,
\text{with} \,\,\,\, \|f\|_{\cH^1_{\mathrm{w}}(J)}^2:=\int_J \mathrm{w}\Big(f^2+f_r^2+\frac{f^2}{r^2}\Big)\,\mathrm{d}r,
\end{equation*}
where $\mathrm{w}=\mathrm{w}(r)\ge 0$ is a weight function on $J$, and we let $\cH^{-1}_{\mathrm{w}}(J):=(\cH^1_{\mathrm{w}}(J))^*$.  In particular, we write $\cH^1_{\mathrm{w}}$=$\cH^1_{\mathrm{w}}(I)$.

\smallskip
% \item For any function space $X$ and functions $(\varphi,g_1,\cdots\!,g_k)$,
% \begin{equation*}
% \|\varphi(g_1,\cdots\!,g_k)\|_{X}:=\sum_{i=1}^k\|\varphi g_i\|_X,\qquad|\varphi(g_1,\cdots\!,g_k)|:=\sum_{i=1}^k|\varphi g_i|.
% \end{equation*}
\item For any function space $X$ and functions $\varphi,\,g_1,\ldots,g_k$,
\begin{equation*}
\|\varphi(g_1,\cdots\!,g_k)\|_{X}:=\sum_{i=1}^k\|\varphi g_i\|_X,\qquad|\varphi(g_1,\cdots\!,g_k)|:=\sum_{i=1}^k|\varphi g_i|.
\end{equation*}
The above summation conventions do not apply to
$\mathscr D_\eta^k f$, $\mathscr D_r^k f$
$(k=1,\ldots,4)$, or to their scalar multiples
and componentwise derivatives.
For these vector-valued quantities $V$, $|\cdot|$
denotes the Euclidean norm, and
\begin{equation*}
\|\varphi V\|_{L^p}:=\bigl\||\varphi|\,|V|\bigr\|_{L^p},
  \qquad 1\le p\le\infty.
\end{equation*}
\item Denote by $\langle\cdot,\cdot\rangle_{X^*\times X}$ the duality pairing between the space  $X$ and its dual space $X^*$:
\begin{align*}
\left<F,f\right>_{X^*\times X}:=F(f) \ \ \text{for } F\in X^*, \ \ f\in X.
\end{align*}
If $F \in X^*$ is identified with an $L^1_{\mathrm{loc}}$-function (still denoted by $F$) in the canonical way and, for a given $f \in X$, the product $Ff\in L^1$, then the duality pairing reduces to
\begin{equation*}
\left<F,f\right>_{X^*\times X}=\langle F,f\rangle:=\int_0^1 Ff\,\mathrm{d}r.
\end{equation*}
$\left< F,f \right>_{X_t^*(Y^*)\times X_t(Y)}$ denotes the pairing between $X([0,T];Y)$ and $X^*([0,T];Y^*)$. 
\end{itemize}

\subsubsection{Other notations}\label{othernotation}
\begin{itemize}
\item $B_a:=\{\boldsymbol{y}: \,|\boldsymbol{y}|<a\}$ denotes the ball centered at the origin with radius $a$.
\item $\delta_{ij}$ denotes the Kronecker symbol with indices $(i,j)$: $\delta_{ij}= 1$ if $i=j$, $\delta_{ij}=0$ if $i\neq j$.
\item For any $n\times n$ real matrix $\mathcal{M}$, $\mathcal{M}_{ij}$ denotes its $(i,j)$-th entry. Moreover, $\mathrm{SO}(n)$  denotes the set of all $n\times n$ real orthogonal matrices $\mathcal{O}$ such that $\det \mathcal{O}=1$, where $\det \mathcal{O}$ is the determinant of $\mathcal{O}$.
\item $\zeta_{a}=\zeta_{a}(r)\in C^\infty[0,1]$ ($a\in (0,1)$) denotes a cut-off function satisfying
\begin{equation*}
\zeta_{a}\in [0,1],\qquad (\zeta_{a})_r\leq 0, \qquad  \zeta_{a}=1 \ \ \text{on $[0,a]$},\qquad \zeta_{a}=0 \ \ \text{on $\big[\frac{1+3a}{4},1\big]$},
\end{equation*}
and $\zeta_{a}^\sharp=\zeta_{a}^\sharp(r):=1-\zeta_{a}(r)$. In particular, if $a=\frac{1}{2}$, define \begin{equation*}
\zeta=\zeta(r):=\zeta_{\frac{1}{2}}(r),\qquad \zeta^\sharp=\zeta^\sharp(r):=1-\zeta(r).
\end{equation*}
\item $\chi_{a}=\chi_{a}(r)$ denotes the characteristic function on $[0,a]$ $(a\in (0,1))$, {\it i.e.}, $\chi_{a}=1$ on $[0,a]$ and $\chi_{a}=0$ on $(a,1]$, and $\chi_{a}^\sharp=1-\chi_{a}$. 
In particular, if $a=\frac{1}{2}$, define 
\begin{equation*}
\chi=\chi(r):=\chi_{\frac{1}{2}}(r),\qquad \chi^\sharp=\chi^\sharp(r):=1-\chi(r).
\end{equation*}

\end{itemize}

\subsection{Main Results}
This section is devoted to stating our main results. First, $\eqref{eq:VFBP-La}_1$ and $\eqref{eq:VFBP-La}_3$ imply that
\begin{equation}\label{eq:eta}
\varrho(t,r) = \frac{r^2\rho_0(r)}{\eta^2\eta_r}.
\end{equation}
Then  \eqref{eq:VFBP-La}, along with \eqref{eq:eta}, can be written as the following  \textbf{IBVP} for $(U, \eta)$ in $[0,T]\times I$:
\begin{equation}\label{eq:VFBP-La-eta}
\begin{cases}
\displaystyle \varrho U_t +A D_\eta(\varrho^{\gamma})+ 4\pi G\frac{\varrho}{\eta^2}\int_0^r\hat r^2\rho_0\,\mathrm{d}\hat{r}\\[1pt]
\displaystyle\quad =(2a_1+a_2) D_\eta\Big(\varrho^{\delta}\big(D_\eta U+ \frac{2U}{\eta}\big)\Big) - 4a_1 D_\eta(\varrho^{\delta}) \frac{U}{\eta},\\[1pt]
\eta_t = U,\\[1pt]
(U, \eta)(0,r)= (u_0(r), r)\qquad  \text{for $r\in I$}.
\end{cases}
\end{equation}

Second, we define classical solutions of problem \eqref{eq:VFBP-La-eta} as follows:
\begin{Definition}\label{definition-lag}
Let $T>0$. A vector function $(U,\eta)(t,r)$ is called 
a classical solution of {\rm\bf IBVP} \eqref{eq:VFBP-La-eta} 
in $[0,T]\times \bar I$ if the following properties hold{\rm:}
\begin{enumerate}
\item[{\rm (i)}] $(U,\eta)(t,r)$ satisfies equations 
$\eqref{eq:VFBP-La-eta}_1${\rm--}$\eqref{eq:VFBP-La-eta}_2$ pointwise in $(0,T]\times \bar I$, and  takes the initial data $\eqref{eq:VFBP-La-eta}_3$  continuously{\rm ;}
\item[{\rm (ii)}] $\eta_r(t,r)$ and $\frac{\eta}{r}(t,r)$ are strictly positive in $[0,T]\times \bar I${\rm:}
\begin{equation*}
\inf_{[0,T]\times \bar I} \ \eta_r >0, \qquad \inf_{[0,T]\times \bar I} \ \frac{\eta}{r}>0;
\end{equation*}
\item[{\rm (iii)}] $(U,\eta)(t,r)$ satisfies the following regularity properties{\rm:}
\begin{equation*}
\begin{aligned}
&\big(U,U_r,\frac{U}{r}\big)\in C([0,T];C(\bar I)),\qquad \big(U_{rr},(\frac{U}{r})_r,U_t\big)\in C((0,T];C(\bar I)),\\
&\big(\eta,\eta_r,\frac{\eta}{r}\big)\in C^1([0,T];C(\bar I)),\qquad \,\big(\eta_{rr},(\frac{\eta}{r})_r\big)\in
C^1((0,T];C(\bar I)).
\end{aligned}
\end{equation*}
\end{enumerate}
At $r=1$, the product
\begin{equation*}
D_\eta(\varrho^\delta)
\big((2a_1+a_2)D_\eta U+2a_2\frac{U}{\eta}\big)
\end{equation*}
is understood through its continuous extension to
$(0,T]\times[\frac12,1]$, whose existence is required.
\end{Definition}

Next, to clearly state our main results, we need to define the following nonlinear weighted energy functional and related parameters:
\begin{itemize}
\item  Throughout the paper, we fix a universal parameter $\varepsilon_0>0$ satisfying 
\begin{equation}\label{varepsilon0}
0<\varepsilon_0 < \min\Big\{\frac{3-\frac{1}{\gamma-1}}{2},\,\frac{1-\delta}{\gamma-1},\,\frac{1}{100}\Big\}.
\end{equation}

\smallskip
\item The total energy:
\begin{equation}\label{E-1}
\cE(t,f)=\cE_{\mathrm{in}}(t,f)+\cE_{\mathrm{ex}}(t,f),
\end{equation}
where 
\begin{equation}\label{E-2}
\begin{aligned}
\qquad \, \mathcal{E}_{\mathrm{in}}(t,f)&:=\|\zeta r(f,\sD_\eta f,f_t,\sD_\eta f_t,\sD_\eta^2 f,\sD_\eta^3 f)(t)\|_{L^2(I)}^2,\\
\qquad  \, \mathcal{E}_{\mathrm{ex}}(t,f)&:=\big\|\rho_0^\frac{1}{2}(f,f_t)(t)\big\|_{L^2(\frac{1}{2},1)}^2+\big\|\rho_0^\frac{\delta}{2}(D_\eta f,D_\eta f_t)(t)\big\|_{L^2(\frac{1}{2},1)}^2\\
&\quad +\big\|\rho_0^{(\frac{3}{2}-\varepsilon_0)(\gamma-1)}(D_\eta^2 f,D_\eta^3 f)(t)\big\|_{L^2(\frac{1}{2},1)}^2,
\end{aligned}
\end{equation}
and $\zeta=\zeta(r)\in C^\infty[0,1]$ denotes a decreasing cut-off function satisfying
\begin{equation}\label{zeta}
\zeta\in [0,1],\qquad \zeta(r)=1 \ \ \text{for  $r\in \big[0,\frac{1}{2}\big]$},\qquad \zeta(r)=0 \ \ \text{for $r\in \big[\frac{5}{8},1\big]$}.
\end{equation}

\smallskip
\item The total dissipation:
\begin{equation}\label{D-1}
\cD(t,f)=\cD_{\mathrm{in}}(t,f)+\cD_{\mathrm{ex}}(t,f),
\end{equation}
where 
\begin{equation}\label{D-2}
\begin{aligned}
\qquad\cD_{\mathrm{in}}(t,f)&:=\|\zeta r(f_{tt},\sD_\eta^2 f_{t},\sD_\eta^4f)(t)\|_{L^2(I)}^2,\\
\qquad\cD_{\mathrm{ex}}(t,f)&:=\big\|\rho_0^\frac{1}{2}f_{tt}(t)\big\|_{L^2(\frac{1}{2},1)}^2+\big\|\rho_0^{(\frac{3}{2}-\varepsilon_0)(\gamma-1)}(D_\eta^2 f_t,D_\eta^4 f)(t)\big\|_{L^2(\frac{1}{2},1)}^2.
\end{aligned}
\end{equation}
\end{itemize}

Based on the choice of $(\cE,\cD)(t,U)$ in the above, we are now ready to state the main result.
\begin{Theorem}\label{local-Theorem1.1} 
Let $A>0$,  $G>0$, $a_1>0$, $2a_1+3a_2>0$, and \eqref{gammadelta} hold.
If  $\rho_0$ satisfies \eqref{distance-la} and $u_0(r)$ satisfies
\begin{equation}\label{a2-lo}
\cE(0,U)<\infty,
\end{equation}
then  there exists $T_*>0$, which depends only on $(a_1,a_2,A,\delta,\gamma,\varepsilon_0,\rho_0,u_0,K_1,K_2,G)$, such that {\rm\textbf{IBVP}} \eqref{eq:VFBP-La-eta} admits a unique classical solution $(U,\eta)(t,r)$ in $[0,T_*]\times \bar I $ satisfying  
\begin{align}
&\cE(t,U)+t \cD(t,U)\in L^\infty(0,T_*),\qquad  \cD (t,U)\in L^1(0,T_*),\label{b1-lo}\\
& (\eta_r,\frac{\eta}{r})(t,r)\in \big[\frac{1}{2},\frac{3}{2}\big] \qquad \quad \ \ \,\text{for $(t,r)\in [0,T_*]\times \bar I$}.\label{b1-lo2}
\end{align}
Moreover, such a classical solution admits the following boundary conditions{\rm:}
\begin{equation}\label{N111}
U|_{r=0}=\big(D_\eta U+\frac{2a_2}{2a_1+a_2}\frac{U}{\eta}\big)\Big|_{r=1}=0\qquad \text{on $(0,T_*]$}.
\end{equation}
\end{Theorem}

In fact, {\rm Theorem \ref{local-Theorem1.1}} can be extended to a more general case.
For the continuation of solutions, we establish local existence with a general initial flow map. Throughout this extension, we keep the original particle labels $r$ and the fixed reference density $\rho_0(r)$, and replace only the initial state of $(U,\eta)$. In particular, \eqref{eq:eta} and the evolution equations in \eqref{eq:VFBP-La-eta} remain unchanged. Let $(U,\eta)|_{t=0}=(u_0,\eta_0)$ and 
\begin{equation}
\cJ_0=\big(\frac{\eta_0}{r}\big)^2(\eta_0)_r,\qquad \varrho_{*}=\frac{\rho_0}{\cJ_0},\qquad m_0(r)=\int_0^r \hat r^2\rho_0(\hat r)\,\mathrm{d}\hat r.
\end{equation}
The formula $\varrho(0,r)=\varrho_{*}(r)$ is the value of \eqref{eq:eta} at the initial time, which is not an additional assumption. Here $\rho_0$ is the fixed reference density, while the density at the initial time is $\varrho_{*}(r)$. Moreover, $\cE(0,U)$ is computed with $\eta(0,r)=\eta_0(r)$ and $U(0,r)=u_0(r)$, where $U_t(0,r)$ is determined by the momentum equation in \eqref{eq:VFBP-La-eta} at the initial time.

% \begin{Theorem}\label{remk31}
% If we assume that $\eta(0,r)=\eta_0(r)$ with general initial map $\eta_0$ satisfying 
% \begin{align*}
% &((\eta_0)_r,\frac{\eta_0}{r})(r)\in [\delta_*,\delta^*] \ \text{for $r\in \bar I$} ,\\
% &\mathcal{F}(\eta_0):=(\zeta r  \sD_r^k \eta_0,\chi^\sharp\rho_0^\frac{1}{2} \eta_0,\chi^\sharp\rho_0^\frac{\delta}{2}(\eta_0)_r,\chi^\sharp\rho_0^{(\gamma-1)(\frac{3}{2}-\varepsilon_0)} \partial_{r}^l \eta_0) \in L^2,
% \end{align*}
% for integers $k\in [0,4]$, $l\in [2,4]$,
% and some constants  $\delta^*>\delta_*>0$, then {\rm Theorem \ref{local-Theorem1.1}} remains valid. In this case, \eqref{b1-lo2} is replaced by
% \begin{equation*}
% (\eta_r,\frac{\eta}{r})(t,r)\in \big[\frac{\delta_*}{2},\frac{3\delta^*}{2}\big] \qquad \quad \ \ \,\text{for $(t,r)\in [0,T_*]\times \bar I$},
% \end{equation*}
% and $T_*>0$ depends only on $(\delta_*,\delta^*,A,a_1,a_2,\delta,\gamma,\varepsilon_0,\rho_0,u_0,|\mathcal{F}(\eta_0)|_2,K_1,K_2,G)$. 
% \end{Theorem}

\begin{Theorem}\label{remk31}
Let $A>0$, $G>0$, $a_1>0$, $2a_1+3a_2>0$, and
\eqref{gammadelta} hold.
Assume that $\rho_0$ satisfies \eqref{distance-la} and
$u_0(r)$ satisfies \eqref{a2-lo}.
For general $\eta(0,r)=\eta_0(r)$, if
\begin{align*}
&((\eta_0)_r,\frac{\eta_0}{r})(r)\in [\delta_*,\delta^*]
\ \text{for $r\in \bar I$},\\
&\mathcal{F}(\eta_0):=
(\zeta r\sD_r^k\eta_0,
\chi^\sharp\rho_0^\frac{1}{2}\eta_0,
\chi^\sharp\rho_0^\frac{\delta}{2}(\eta_0)_r,
\chi^\sharp\rho_0^{(\gamma-1)(\frac{3}{2}-\varepsilon_0)}
\partial_r^l\eta_0)\in L^2,
\end{align*}
for integers $k\in [0,4]$, $l\in [2,4]$,
and some constants $\delta^*>\delta_*>0$,
then there exists $T_*>0$, which depends only on
$(\delta_*,\delta^*,A,a_1,a_2,\delta,\gamma,\varepsilon_0,
\rho_0,u_0,|\mathcal{F}(\eta_0)|_2,K_1,K_2,G)$,
such that {\rm\textbf{IBVP}} \eqref{eq:VFBP-La-eta},
with $(U,\eta)|_{t=0}=(u_0,\eta_0)$,
admits a unique classical solution $(U,\eta)(t,r)$
in $[0,T_*]\times \bar I$ satisfying
\begin{align*}
&\cE(t,U)+t\cD(t,U)\in L^\infty(0,T_*),
\qquad \cD(t,U)\in L^1(0,T_*),\\
&(\eta_r,\frac{\eta}{r})(t,r)
\in \big[\frac{\delta_*}{2},\frac{3\delta^*}{2}\big]
\qquad \text{for $(t,r)\in [0,T_*]\times \bar I$}.
\end{align*}
Moreover, such a classical solution admits the boundary
conditions \eqref{N111} on $(0,T_*]$.
\end{Theorem}
% This theorem can be proved by following the methodology developed in this paper with minor modifications.
The proof is given at the end of \S \ref{subsection9.3}. The existence time can be chosen uniformly positive when the initial-data quantities bounded by $c_0$ in \eqref{38}, together with $\delta_*^{-1}$, $\delta^*$, and $|\mathcal F(\eta_0)|_2$ are uniformly bounded, with $(A,a_1,a_2,\delta,\gamma,\varepsilon_0,G)$ fixed. Next, we make some remarks on Theorem \ref{local-Theorem1.1}.
\begin{Remark}\label{remark-energy function}
We show some explanations of the definitions of $(\cE,\cD)$ given in \eqref{E-1}{\rm--}\eqref{D-2}.
Due to the coordinate singularity at the origin and the strong degeneracy on the moving boundary, we need to select the energy functionals separately in their respective neighborhoods and then combine them appropriately.

First, the interior energy functionals $(\cE_{\mathrm{in}},\cD_{\mathrm{in}})$ are chosen based on the  3-D Lagrangian spatial coordinates $(t,\boldsymbol{y})$. If we let  $\boldsymbol{U}(t,\boldsymbol{y}):=U(t,r)\frac{\boldsymbol{y}}{r}$, then $\cE_{\mathrm{in}}(t,U)+t\cD_{\mathrm{in}}(t,U)\in L^\infty(0,T)$ can be read as the following $H^k$-norms of $\boldsymbol{U}$ due to  {\rm Lemma \ref{lemma-initial}} in {\rm Appendix \ref{appb}: }
\begin{equation}\label{regularity-La-M}
\sum_{l=0}^1 \|\zeta\partial_t^{1-l}\boldsymbol{U} (t)\|_{H^{2l+1}(\Omega)} +\sqrt{t}\sum_{l=0}^2 \|\zeta\partial_t^{2-l}\boldsymbol{U}(t)\|_{H^{2l}(\Omega)}  \in L^\infty(0,T).
\end{equation}

Second, to ensure that the solution is classical in the region  $\{r:\,\frac{1}{2}\leq r< 1\}$ and that the velocity has spatial \(C^2\)-regularity up to the vacuum boundary for $t>0$, we establish some weighted $H^4(\frac{1}{2},1)$-estimates on $U$ inspired by the following embedding, due to the Hardy and Sobolev inequalities {\rm(}see {\rm Lemmas \ref{sobolev-embedding}--\ref{hardy-inequality}} in {\rm Appendix \ref{appendix A}}{\rm)}{\rm:}
\begin{equation*}
H^4_{\rho_0^{2\alpha}}\big(\frac{1}{2},1\big)\hookrightarrow W^{3,1}\big(\frac{1}{2},1\big)\hookrightarrow C^2\big[\frac{1}{2},1\big] \quad\text{for some $\alpha<\frac{3}{2}(\gamma-1)$ and $\big|\alpha-\frac{3}{2}(\gamma-1)\big|\ll 1$},
\end{equation*}
where  the weighted Sobolev space $H^4_{\rho_0^{2\alpha}}(\frac{1}{2},1)$ is defined by 
\begin{equation*}
H^4_{\rho_0^{2\alpha}}\big(\frac{1}{2},1\big)
:=\big\{f\in L^1_{\mathrm{loc}}\big(\frac{1}{2},1\big): \,\rho_0^{\alpha}\partial_r^j f\in L^2\big(\frac{1}{2},1\big) 
\ \text{for $0\leq j\leq 4$}\}.
\end{equation*}
Then the natural exterior energy and dissipation take the form{\rm:}
\begin{equation*}
\begin{aligned}
\cE^*_{\mathrm{ex}}(t,U)&=\big|\chi^\sharp\rho_0^\frac{1}{2}(U,U_t)(t)\big|_{2}^2+\big|\chi^\sharp\rho_0^\frac{\delta}{2}(D_\eta U,D_\eta U_t)(t)\big|_{2}^2+
\big|\chi^\sharp\rho_0^{\alpha}(D_\eta^2 U,D_\eta^3 U)(t)\big|_{2}^2,\\
\cD^*_{\mathrm{ex}}(t,U)&=\big|\chi^\sharp\rho_0^\frac{1}{2}U_{tt}(t)\big|_{2}^2+\big|\chi^\sharp\rho_0^{\alpha}(D_\eta^2 U_t,D_\eta^4 U)(t)\big|_{2}^2,
\end{aligned}
\end{equation*}
where $\alpha$ can be determined in the elliptic estimates near the vacuum boundary. In fact, to derive the highest-order elliptic estimates, we reformulate $\eqref{eq:VFBP-La-eta}_1$ by multiplying $\varrho^{\gamma-\delta-1}${\rm:}
\begin{equation}\label{reform}
\begin{aligned}
    (2a_1+a_2)\varrho^{\gamma-1} D_\eta^2 U &=\varrho^{\gamma-\delta}U_t-\frac{(2a_1+a_2)\delta}{\gamma-1} D_\eta(\varrho^{\gamma-1})\big(D_\eta U+\frac{2a_2}{2a_1+a_2}\frac{U}{\eta}\big)\\
    &\quad -2(2a_1+a_2) \varrho^{\gamma-1} D_\eta\big(\frac{U}{\eta}\big)+\frac{A\gamma}{\gamma-1}\varrho^{\gamma-\delta}D_\eta(\varrho^{\gamma-1})+\varrho^{\gamma-\delta}D_\eta\Phi.
\end{aligned}
\end{equation}
Then we formally obtain from the above that
\begin{equation*}
\begin{aligned}
\varrho^\alpha D_\eta^2 U_{t} &=\underline{\frac{1}{2a_1+a_2}\varrho^{\alpha+1-\delta} U_{tt}}_{:=\clubsuit_1}+\mathrm{R}_{(\alpha)}\\
\varrho^\alpha D_\eta^4 U &=\underline{\frac{A\gamma(\gamma-\delta)(1-\delta)}{(2a_1+a_2)(\gamma-1)^3}\varrho^{\gamma-\delta+\alpha-3(\gamma-1)}(D_\eta(\varrho^{\gamma-1}))^3}_{:=\clubsuit_2}+ \underline{\frac{1}{2a_1+a_2}\varrho^{\alpha+1-\delta} D_\eta^2 U_{t}}_{:=\clubsuit_3}+\widetilde{\mathrm{R}}_{(\alpha)},
\end{aligned}
\end{equation*}
where $(\mathrm{R}_{(\alpha)},\widetilde{\mathrm{R}}_{(\alpha)})$ denote some harmless terms which possess either higher-order $\varrho$-weights or lower-order derivatives of $(\varrho,U)$.  As can be checked, $\clubsuit_1$   contains the highest-order tangential derivative and constitutes the main obstacle in controlling the $L^2(\frac{1}{2},1)$-norm of $\varrho^\alpha D_\eta^4 U$,  $\clubsuit_2$ is the most singular part of the pressure's derivatives, and $\clubsuit_3$ is the highest order elliptic derivatives involving the temporal derivatives.  
Thus, it follows from \eqref{distance-la} that{\rm:} 
\begin{itemize}
\item 
if $\alpha \geq \delta- \frac{1}{2}$, then $\gamma>\frac{2\delta+2}{3}$, and $\clubsuit_1\in L^2\big(\frac{1}{2},1\big)$;\\
\item 
if $\alpha  > \frac{5}{2}(\gamma-1)-\gamma+\delta$, then $\delta<1$, and $\clubsuit_2\in L^2\big(\frac{1}{2},1\big)$;\\
\item 
if $\alpha  \geq 2\delta- \frac{3}{2}$, then $\gamma>\frac{4}{3}\delta$, and $\clubsuit_3\in L^2\big(\frac{1}{2},1\big)$.
\end{itemize}
Finally, by setting 
$$\alpha:=(\frac{3}{2}-\varepsilon_0)(\gamma-1)<\frac{3}{2}(\gamma-1)$$ with some suitable small $\varepsilon_0$ defined in \eqref{varepsilon0}, we recover the desired exterior energy and dissipation $(\cE_{\mathrm{ex}},\cD_{\mathrm{ex}})(t,U)$ in \eqref{E-2} and \eqref{D-2}.
\end{Remark}

\begin{Remark}\label{initialexample}
We give an example of the initial data required in {\rm Theorem \ref{local-Theorem1.1}}. More precisely, we find that {\rm Theorem \ref{local-Theorem1.1}} can be established if  the initial data
belong to the following class{\rm :}
\begin{equation*} 
\rho_0^{\gamma-1}(r)=1-r^2, \qquad  u_0(r)\in C_\mathrm{c}^\infty(0,1).
\end{equation*}
Clearly, we can check that $\rho_0$ satisfies assumption \eqref{distance-la}. For $u_0$, it suffices to show that the initial values $(U_t,U_{tr})(0,r)$ satisfy \eqref{a2-lo}. Indeed, we can obtain from $\eqref{eq:VFBP-La-eta}_1$ and a direct calculation that $(U_t,U_{tr})(0,r)$ satisfy the following compatibility conditions{\rm:}
\begin{equation}\label{116}
\begin{aligned}
U_t(0,r)&=\frac{2a_1+a_2}{\rho_0}\Big(\rho_0^{\delta}\big((u_0)_r+\frac{2a_2}{2a_1+a_2}\frac{u_0}{r}\big)\Big)_r+4a_1\rho_0^{\delta-1} (\frac{u_0}{r})_r\\
&\quad - \frac{A \gamma}{\gamma-1} (\rho_0^{\gamma-1})_r-\frac{4\pi G}{r^2}\int_0^r \hat r^2\rho_0\,\mathrm{d}\hat r,\\
U_{tr}(0,r)&=(2a_1+a_2)\Big(\frac{1}{\rho_0} \big(\rho_0^{\delta} (u_0)_r\big)_r+\frac{2a_2}{2a_1+a_2}\frac{1}{\rho_0}\big(\rho_0^{\delta}\frac{u_0}{r}\big)_r \Big)_r+4a_1\Big(\rho_0^{\delta-1} (\frac{u_0}{r})_r\Big)_r\\
&\quad - \frac{A \gamma}{\gamma-1} (\rho_0^{\gamma-1})_{rr}- 4\pi G \rho_0+ \frac{8\pi G}{r^3}\int_0^r \hat r^2\rho_0\,\mathrm{d}\hat r. 
\end{aligned}
\end{equation}
It is then straightforward to verify that 
\begin{equation*}
\zeta r U_t(0,r), \ \zeta r \sD_r U_{t}(0,r)\in L^2(I),\qquad  \rho_0^\frac{1}{2}U_t(0,r), \ \rho_0^\frac{\delta}{2}U_{tr}(0,r)\in L^2(\frac{1}{2},1),
\end{equation*}    
due to the regularities of $\rho_0$ and the fact that $u_0$ is compactly supported in $(0,1)$.
\end{Remark}

\begin{Remark}\label{rmk1.4}
We briefly explain how to derive the boundary condition in \eqref{N111}. First, it follows from \eqref{flowmap-r-la}, \eqref{distance-la}, $\eqref{b1-lo2}$, and {\rm Definition \ref{definition-lag}} that 
\begin{equation}\label{eq117}
(\varrho^{\gamma-1}, \, D_\eta\varrho^{\gamma-1},\, U,\ D_{\eta}U, \, D_{\eta}^2U, \, U_t)\in  C((0,T]\times [\frac{1}{2},1]).
\end{equation}
Next, we reformulate $\eqref{eq:VFBP-La-eta}_1$ by multiplying $\varrho^{\gamma-\delta-1}${\rm:}
\begin{equation}\label{reform boundary condition}
\begin{aligned}
&\varrho^{\gamma-\delta}U_t+\frac{A\gamma}{\gamma-1}\varrho^{\gamma-\delta}D_\eta(\varrho^{\gamma-1})+\varrho^{\gamma-\delta}D_\eta\Phi\\
&=(2a_1+a_2)\varrho^{\gamma-1}\big(D_\eta^2U+2D_\eta(\frac{U}{\eta})\big)+ \frac{\delta}{\gamma-1} D_\eta(\varrho^{\gamma-1})\big((2a_1+a_2) D_\eta  U +2a_2\frac{U}{\eta}\big).
\end{aligned}
\end{equation}
Then, taking the limit $r\to 1$ in \eqref{reform boundary condition} and using \eqref{distance-la}, \eqref{eq117}, and the lower bounds of  $(\eta,\eta_r)$ near the boundary, we obtain 
\begin{equation*}
D_\eta(\varrho^{\gamma-1})\big(D_\eta U+\frac{2a_2}{2a_1+a_2}\frac{U}{\eta}\big)\big|_{r=1}=0.
\end{equation*}
Since $\rho_0^{\gamma-1} \sim 1-r$ and, by \eqref{eq:eta}, $D_\eta(\varrho^{\gamma-1})|_{r=1}\neq 0$, the boundary condition
follows directly. 
\end{Remark}

\smallskip
\begin{Remark}
For  {\rm\textbf{VFBP}} \eqref{eq:1.1-vfbp}, the usual
stress-free boundary condition holds automatically{\rm:}
\begin{equation*}
(\mathbb{T}-P\II_3)\cdot \boldsymbol{n} = \big(2a_1\rho^\delta D(\boldsymbol{u})+a_2\rho^\delta \dive\boldsymbol{u}\,\mathbb{I}_3-A\rho^\gamma\II_3\big)\cdot \boldsymbol{n}=\boldsymbol{0}.
\end{equation*}
\end{Remark}

The rest of the paper is organized as follows: 
\S \ref{subsection9.1} is devoted to  constructing  global approximation solutions for the corresponding  linearized problem. In \S \ref{subsection9.2},  we establish the uniform  energy  estimates for the solutions of the linearized problem introduced in \S \ref{subsection9.1}.
Finally, in \S \ref{subsection9.3}, based on the uniform estimates obtained above, we will give the proof for the local-in-time existence  of the unique classical solution to the nonlinear problem \eqref{eq:VFBP-La-eta} by the classical Picard iteration. Several auxiliary lemmas and useful coordinate transformations for spherically symmetric functions are collected in Appendices \ref{appendix A}--\ref{subsection2.2}.

\section{Linearization and Global Approximation Solutions}\label{subsection9.1}

This section is  devoted to  constructing  global-in-time  approximation solutions for the corresponding  linearized problem. In \S\ref{subsection9.1}, $C\in (1,\infty)$ denotes a generic constant depending only on $(a_1,a_2,A,\delta,\gamma,\varepsilon_0,\rho_0,u_0,K_1,K_2,G)$, and $C(l_1,\cdots\!,l_k)\in (1,\infty)$ a generic  constant depending on $C$
and additional  parameters $(l_1,\cdots\!,l_k)$, which may be different at each occurrence.

\subsection{Linearization} In order to solve the nonlinear problem  \eqref{eq:VFBP-La-eta}, we first need to study the  following linearized problem in $[0, T] \times I$: 
\begin{equation}\label{lp}
\begin{cases}
\displaystyle r^2\rho_0 U_t +A\big(\bar\eta^2\bar\varrho^{\gamma}\big)_r- 2A\bar\eta\bar\eta_r\bar\varrho^{\gamma} +4\pi G\frac{r^2\rho_0}{\bar\eta^2}\int_0^r\hat r^2\rho_0\,\mathrm{d}\hat r\\[10pt]
\displaystyle\quad = \Big(\bar\eta^2\bar\varrho^{\delta} \big((2a_1+a_2)D_{\bar\eta}U+2a_2\frac{U}{\bar\eta}\big)\Big)_r -\bar\eta\bar\eta_r\bar\varrho^{\delta} \Big(2a_2  D_{\bar\eta}U+4(a_1+a_2)\frac{U}{\bar\eta}\Big),\\[8pt]
U|_{t=0}= u_0 \qquad \text{in }I,
\end{cases}
\end{equation}
where 
\begin{equation}
\bar\varrho:=\frac{r^2\rho_0}{\bar\eta^2\bar\eta_r}, 
\end{equation}
$\bar \eta $ stands for the flow map corresponding to $\bar U$:
\begin{equation}\label{given-flow}
\bar \eta (t,r)=r+\int_0^t \bar U(s,r)\,\ds,\quad \bar \eta(0,r)=r, \quad \bar \eta(t,0)=0,
\end{equation}
and $\bar U$ is a given function satisfying $\bar U(0,r)=u_0(r)$ for $r\in I$ 
and, for any $T>0$,
\begin{equation}\label{given}
\begin{aligned}
&\bar\cE(t,\bar U)+t\bar\cD(t,\bar U)\in L^\infty(0,T),\qquad \bar\cD (t,\bar U)\in L^1(0,T),\\
&(\bar U,\sD_r\bar U)\in C([0,T];C(\bar I)),\qquad\ \, (\sD_r^2\bar U,\bar U_t)\in C((0,T];C(\bar I)), \quad \bar{U}|_{r=0}=0,
\end{aligned} 
\end{equation}
where $(\bar\cE,\bar\cD)(t,f)$ are defined in the same way as $(\cE,\cD)(t,f)$ in \eqref{E-1} and \eqref{D-1}, except with $\bar\eta$ in place of $\eta$. Besides, 
we define $(\bar\cE_{\mathrm{in}},\bar\cE_{\mathrm{ex}},\bar\cD_{\mathrm{in}},\bar\cD_{\mathrm{ex}})(t,f)$ in a similar manner to $(\bar\cE,\bar\cD)(t,f)$.

Clearly, \eqref{given-flow}--\eqref{given}  also provide  the regularity of $\bar\eta$, that is,
\begin{equation}\label{given-bareta}
\begin{aligned}
&\zeta r (\bar \eta,\sD_r\bar \eta,\sD_r^2\bar \eta,\sD_r^3\bar \eta,\sD_r^4\bar \eta )\in C([0,T];L^2),\\
&\chi^\sharp (\rho_0^\frac{1}{2}\bar \eta,\rho_0^\frac{\delta}{2}\bar \eta_r)\in C([0,T];L^2),\qquad \chi^\sharp\rho_0^{(\gamma-1)(\frac{3}{2}-\varepsilon_0)}(\bar \eta_{rr},\bar \eta_{rrr},\bar \eta_{rrrr})\in C([0,T];L^2),\\[4pt]
&(\bar \eta, \sD_r\bar\eta)\in C^1([0,T];C(\bar I)),\qquad \sD_r^2\bar\eta\in C^1((0,T];C(\bar I)).
\end{aligned} 
\end{equation}
Moreover, we assume here that 
\begin{equation}\label{jibenjiashe}
(\bar\eta_r,\frac{\bar\eta}{r})(t,r)\in \big[\frac{1}{2},\frac{3}{2}\big] \qquad \text{for all }(t,r)\in [0,T]\times \bar I.
\end{equation}
This requirement will be fulfilled in \S \ref{subsection9.2} for the corresponding linearization procedure. 

Now we define the classical solution of the linearized problem \eqref{lp}, which slightly differs from Definition \ref{definition-lag}.
\begin{Definition}\label{fed-cl}
We say that $U(t,r)$ is a classical solution of the linearized problem \eqref{lp} in $[0,T]\times \bar I$ if $U(t,r)$ satisfies equation $\eqref{lp}_1$ pointwise in $(0,T]\times \bar I$, and  takes the initial data $\eqref{lp}_2$  continuously, and
\begin{equation}\label{regu-class}
(U,\sD_rU)\in C([0,T];C(\bar I)),\qquad (\sD_r^2 U,U_t)\in C((0,T];C(\bar I)).
\end{equation}
\end{Definition}

Then the main conclusion in \S\ref{subsection9.1} can be stated in the following lemma:
\begin{Lemma}\label{existence-linearize}
Let $A>0$,  $G>0$, $a_1>0$, $2a_1+3a_2>0$, and \eqref{gammadelta} hold. Assume that $(\rho_0,u_0)(r)$ satisfy \eqref{distance-la} and $\cE(0,U)<\infty$. Then, for any $T>0$, problem \eqref{lp} admits a unique classical solution $U$ in $[0,T]\times \bar I$ satisfying 
\begin{equation}\label{n-loc}
\begin{aligned}
&\bar\cE(t,U)+t\,\bar\cD(t,U)\in L^\infty(0,T),\qquad \bar\cD (t,U)\in L^1(0,T),\\[4pt]
&U|_{r=0}=\big(D_{\bar\eta}U+\frac{2a_2}{2a_1+a_2}\frac{U}{\bar\eta}\big)\Big|_{r=1}=0\qquad\qquad  \text{for $t\in(0,T]$}.
\end{aligned}
\end{equation}
\end{Lemma}

\subsection{The Modified Galerkin Scheme for Some General Problems}\label{revise section 2.2}

To prove Lemma \ref{existence-linearize}, we first study the following general initial-boundary value problem:
\begin{equation}\label{galerkin-w}
\begin{cases}
\displaystyle  r^2\rho_0 w_t-2(r^2\rho_0^{\delta})^\frac{1}{2}\frac{g}{\bar\eta}+\Big((r^2\rho_0^{\delta})^\frac{1}{2}\frac{h}{\bar\eta_r}\Big)_r\\[6pt] \displaystyle=\Big(\bar\eta^2\bar\varrho^{\delta} \big((2a_1+a_2)D_{\bar\eta}w+2a_2\frac{w}{\bar\eta}\big)\Big)_r -\bar\eta\bar\eta_r\bar\varrho^{\delta}\Big(2a_2  D_{\bar\eta}w+4(a_1+a_2)\frac{w}{\bar\eta}\Big),\\[8pt]
w|_{t=0}=w_0\qquad \text{on } I,
\end{cases}
\end{equation}
where $(g,h,w_0)$ are given functions satisfying 
\begin{equation*}
(g,h)\in L^2([0,T];L^2),\qquad w_0\in L^2_{r^2\rho_0}.
\end{equation*}

First, we study the weak solutions of problem \eqref{galerkin-w}. 
\begin{Definition}\label{def3.1}
We say that a function $w(t,r)$ is a weak solution in $[0,T]\times I$ of 
problem \eqref{galerkin-w} if the following three properties hold{\rm:}
\begin{enumerate}
\item[{\rm (i)}] $w\in C([0,T];L^2_{r^2\rho_0})\cap L^2([0,T];\cH^1_{r^2\rho_0^{\delta}})$ and $r^2\rho_0 w_t\in L^2([0,T];\cH^{-1}_{r^2\rho_0^{\delta}});$
\item[{\rm (ii)}]
for all $\varphi$ satisfying $\varphi \in \cH^1_{r^2\rho_0^{\delta}}$ and {\it a.e.} time $t\in(0, T)$,
\begin{equation}\label{weak.F.}
\begin{aligned}
&\big<r^2\rho_0 w_t, \varphi\big>_{\cH^{-1}_{r^2\rho_0^{\delta}}\times \cH^1_{r^2\rho_0^{\delta}}} +(2a_1+a_2)\big< \bar\eta^2\bar\eta_r\bar\varrho^{\delta} D_{\bar\eta} w,D_{\bar\eta}\varphi\big>+4(a_1+a_2)\Big<\bar\eta^2\bar\eta_r\bar\varrho^{\delta}\frac{w}{\bar\eta},\frac{\varphi}{\bar\eta}\Big>\\
&\quad + 2a_2\Big<\bar\eta^2\bar\eta_r\bar\varrho^{\delta}\frac{w}{\bar\eta},D_{\bar\eta}\varphi\Big>+ 2a_2\Big<\bar\eta^2\bar\eta_r\bar\varrho^{\delta}D_{\bar\eta} w,\frac{\varphi}{\bar\eta}\Big>\\
&=\Big<g,(r^2\rho_0^{\delta})^\frac{1}{2} \frac{2\varphi}{\bar\eta}\Big>+\big< h,(r^2\rho_0^{\delta})^\frac{1}{2}  D_{\bar\eta}\varphi\big>;
\end{aligned}
\end{equation}
\item[{\rm (iii)}] $w(0,r)=w_0(r)$ for {\it a.e.} $ r\in I$.
\end{enumerate}
\end{Definition}

Now we establish the following existence of weak solutions and their related estimates. 
\begin{Proposition}\label{prop1}
For all $T>0$,  problem \eqref{galerkin-w} admits a unique weak solution $w$ in $[0,T]\times I$, satisfying
\begin{equation*}
\begin{aligned}
&\sup_{t\in[0,T]}|w|_{2,r^2\rho_0}^2 + \int_0^T  \big(\|w\|_{\cH^{1}_{r^2\rho_0^{\delta}}}^2 + \big\|r^2\rho_0 w_t\big\|_{\cH^{-1}_{r^2\rho_0^{\delta}}}^2\big)\,\dt\\
&\leq C(T)\Big(|w_0|_{2,r^2\rho_0}^2 +  \int_0^T |(g,h)|_2^2\,\dt\Big).
\end{aligned}
\end{equation*}
\end{Proposition}

\begin{proof}
We divide the proof into four steps.

\smallskip
\textbf{Step 1. Introduction of the Galerkin scheme.}  First, it follows from Lemma \ref{W-space} that, for given $w_0\in L^2_{r^2\rho_0}$, there exists a smooth sequence $\{w_0^\vartheta\}_{\vartheta>0}\subset C^\infty(\bar I)$ satisfying
\begin{equation}\label{wdelta-w}
\lim_{\vartheta\to0} |w^\vartheta_0- w_0|_{2,r^2\rho_0}=0.
\end{equation}

Next, we construct a basis required in the Galerkin scheme. To match   the spherical symmetry structure of $\eqref{galerkin-w}_1$, we consider the following Sturm--Liouville problem:
\begin{equation}\label{SL}
-(r^2\xi_r)_r+2\xi=\nu\, r^2 \xi \quad \text{on $I$}, \qquad  \ \
\xi|_{r=0}=\xi_r|_{r=1}=0.
\end{equation}

Based on the Sturm--Liouville theorem, we can construct a Hilbert basis $\{\xi_j\}_{j\in \NN^*}$ of $\cH^1_{r^2}$, which is orthonormal in $L^2_{r^2}$ and orthogonal in $\cH^1_{r^2}$:
\begin{Lemma}[\cite{zettl}]\label{hilbert}
Consider the Sturm--Liouville problem \eqref{SL}. Then all eigenvalues $\nu$ of problem \eqref{SL} are nonzero, real, and have multiplicity one. 
Moreover, there are infinite but countable eigenvalues 
$\{\nu_j\}_{j\in\NN^*}$, which are bounded below, strictly increasing and $\nu_j\to \infty$ as $j\to\infty$. Besides, the following approximation results hold{\rm:}
\begin{enumerate}
\item[{\rm(i)}]  There exists a sequence of eigenfunctions $\{\xi_j\}_{j\in\NN^*}\subset \cH^1_{r^2}$, corresponding to the eigenvalues $\{\nu_j\}_{j\in\NN^*}$. Such a sequence of eigenfunctions $\{\xi_j\}_{j\in\NN^*}$ is orthonormal and  complete in $L^2_{r^2}$, namely, $\langle r^2  \xi_j ,  \xi_k \rangle=\delta_{kj}$ for $k,j\in\NN^*$, and
\begin{equation*}
\lim_{N\to\infty}|f_N-f|_{2,r^2}=0 \qquad \text{for any $f\in L^2_{r^2}$}, \qquad f_N:=\sum_{j=1}^N \langle r^2 f, \xi_j\rangle\xi_j.
\end{equation*}
\item[{\rm(ii)}] Such a sequence of eigenfunctions $\{\xi_j\}_{j\in\NN^*}$ is also orthogonal and complete in $\cH^1_{r^2}$, namely, $\langle r^2 (\xi_j)_r, (\xi_k)_r\rangle+2\big<\xi_j, \xi_k\big>=\nu_j\delta_{jk}$ for $k,j\in\NN^*$, and
\begin{equation*}
\lim_{N\to\infty}\|f_N-f\|_{\cH^1_{r^2}}=0 \qquad \text{for any $f\in \cH^1_{r^2}$}.
\end{equation*}
\end{enumerate}
\end{Lemma}
\begin{proof}
We note that the Sturm--Liouville problem  \eqref{SL} is equivalent to the classical  eigenvalue problem of the first-order spherical Bessel equation:
\begin{equation}\label{SL-Bessel}
(r^2\xi_r)_r+( \nu\, r^2 -   2)\xi=0 \quad \text{on $I$}, \qquad  \ \
\xi|_{r=0}=\xi_r|_{r=1}=0.
\end{equation}
Then the existence of the desired eigenvalues $\{\nu_j\}_{j\in\NN^*}$ and the eigenfunctions $\{\xi_j\}_{j\in\NN^*}$ satisfying the property (i) follows from the classical theories on the Sturm--Liouville problem or the Bessel equation, which can be found in \cites{zettl,watson}. 

Thus, we only need to prove the desired property in (ii). The orthogonality 
\begin{equation*}
\langle r^2 (\xi_j)_r, (\xi_k)_r\rangle+2\big<\xi_j, \xi_k\big>=\nu_j\delta_{jk}
\end{equation*}
follows from the facts that $\{\xi_j\}_{j\in\NN^*}$ solves the eigenvalue problem \eqref{SL} and is orthonormal in $L^2_{r^2}$. To show that $\{\xi_j\}_{j\in\NN^*}$ is complete in $\cH^1_{r^2}$, we define the bilinear form $B[\cdot,\cdot]$ by
\begin{equation*}
B[f,g]:=\langle r^2 f_r, g_r\rangle+ 2\langle f, g\rangle.
\end{equation*}
Then, it follows from Lemma \ref{hardy-inequality} and the H\"older inequality that
\begin{equation*}
\begin{aligned}
&B[f,g]\leq |f_r|_{2,r^2}|g_r|_{2,r^2}+2|f|_{2} |g|_{2}\leq C\|f\|_{\cH^1_{r^2}}\|g\|_{\cH^1_{r^2}},\\
&B[f,f]=|f_r|_{2,r^2}^2+  2|f|_{2}^2\geq C^{-1}\|f\|_{\cH^1_{r^2}}^2,
\end{aligned}   
\end{equation*}
which implies that $B[f,g]$ is  bounded and coercive in $\cH^1_{r^2}$. 

Finally, suppose 
\begin{equation}\label{comet}
B[\xi_j,f]=0 \qquad \text{for all $j\in\NN^*$}.  
\end{equation}
Since $\xi_j$ is the eigenfunction that corresponds to the eigenvalue $\nu_j$, we have
\begin{equation*}
B[\xi_j,f]=\nu_j\langle r^2 \xi_j,f\rangle \qquad \ 
\text{for all $f\in \cH^1_{r^2}$},
\end{equation*}
which, along with \eqref{comet}, yields 
\begin{equation*} 
\langle r^2 \xi_j,f\rangle=0 \qquad \text{for all $j\in\NN^*$},  
\end{equation*}
Thanks to the completeness of $\{\xi_j\}_{j\in\NN^*}$ in $L^2_{r^2}$, we obtain from the above that $f=0$. Hence, the sequence $\{\xi_j\}_{j\in \NN^*}$ is complete in $\cH^1_{r^2}$.
\end{proof}

Consequently, based on the above  basis $\{\xi_j\}_{j\in \NN^*}$, define 
\begin{equation}\label{U^n}
w^{N,\vartheta}(t,r):=\sum_{k=1}^N \mu_k^{N,\vartheta}(t) \xi_k(r) \qquad \text{for }\vartheta\in (0,1)\text{ and }N\in \NN^*.
\end{equation}
Here, $\mu_k^{N,\vartheta}(t)$ are chosen by solving the following  ODE problem in $[0,T]$:
\begin{equation}\label{galerkin-n}
\begin{cases}
\displaystyle  \big<r^2\rho_0 w_t^{N,\vartheta}, \xi_j\big> +(2a_1+a_2)\big<\bar\eta^2\bar\eta_r\bar\varrho^{\delta} D_{\bar\eta} w^{N,\vartheta},D_{\bar\eta}\xi_j\big>+4(a_1+a_2)\Big<\bar\eta^2\bar\eta_r\bar\varrho^{\delta}\frac{w^{N,\vartheta}}{\bar\eta},\frac{\xi_j}{\bar\eta}\Big>\\
\displaystyle\quad + 2 a_2\Big<\bar\eta^2\bar\eta_r\bar\varrho^{\delta}\frac{w^{N,\vartheta}}{\bar\eta},D_{\bar\eta}\xi_j\Big>+ 2a_2\Big<\bar\eta^2\bar\eta_r\bar\varrho^{\delta}D_{\bar\eta} w^{N,\vartheta},\frac{\xi_j}{\bar\eta}\Big>\\
\displaystyle=\Big<g,(r^2\rho_0^{\delta})^\frac{1}{2} \frac{2\xi_j}{\bar\eta}\Big>+\big< h,(r^2\rho_0^{\delta})^\frac{1}{2}  D_{\bar\eta}\xi_j\big>,\\[8pt]
\mu_j^{N,\vartheta}(0)=\langle r^2 w_0^\vartheta,\xi_j\rangle, \ \ j=1,2,\cdots\!,N.
\end{cases}
\end{equation}
The above system can also  be rewritten as 
\begin{equation}\label{mu^n}
\begin{cases}
\displaystyle \mathfrak{A}\cdot\frac{\mathrm{d}}{\dt}\mu^{N,\vartheta}(t)+\mathfrak{B}(t)\cdot\mu^{N,\vartheta}(t)=\mathfrak{c}(t) \qquad
\text{in $(0,T]$},\\[10pt]
\mu_j^{N,\vartheta}(0)=\langle r^2w_0^\vartheta,\xi_j\rangle \
\qquad \text{for $j=1,2,\cdots\!,N$},
\end{cases}
\end{equation}
where $\mu^{N,\vartheta}(t):=(\mu_1^{N,\vartheta},\cdots\!,\mu_N^{N,\vartheta})^\top(t)$,
\begin{equation*} 
\mathfrak{A} =(\mathfrak{A}_{kj})_{1\leq k,j\leq N}, \qquad \mathfrak{B}(t)=(\mathfrak{B}_{kj}(t))_{1\leq k,j\leq N}, \qquad \mathfrak{c}(t)=(\mathfrak{c}_{1},\cdots\!,\mathfrak{c}_{N})^\top(t),
\end{equation*}
and
\begin{equation}\label{ABC}
\begin{aligned}
\mathfrak{A}_{kj}&:=\int_0^1 r^2\rho_0 \xi_k \xi_j\,\mathrm{d}r,\qquad \mathfrak{c}_j(t):= \int_0^1(r^2\rho_0^{\delta})^\frac{1}{2}\Big(g \frac{2\xi_j}{\bar\eta} +h D_{\bar\eta}\xi_j\Big)\,\mathrm{d}r.\\
\mathfrak{B}_{kj}(t)&:= \int_0^1 \bar\eta^2\bar\eta_r\bar\varrho^{\delta}\Big((2a_1+a_2) D_{\bar\eta}\xi_k D_{\bar\eta}\xi_j+4(a_1+a_2) \frac{\xi_k \xi_j}{\bar\eta^2} \Big)\,\mathrm{d}r\\
&\quad \ + 2a_2\int_0^1 \bar\eta^2\bar\eta_r\bar\varrho^{\delta}\Big( \frac{\xi_k}{\bar\eta} D_{\bar\eta}\xi_j+ D_{\bar\eta}\xi_k\frac{\xi_j}{\bar\eta} \Big)\,\mathrm{d}r.
\end{aligned}
\end{equation}

To solve  \eqref{mu^n}, thanks to the facts that $\rho_0^{\gamma-1}\sim 1-r$ and the linear independence of $\{\xi_j\}_{j\in \NN^*}$, we see that $\{(r^2\rho_0)^\frac{1}{2} \xi_j\}_{j\in \NN^*}$ are linearly independent, and hence its Gram matrix $\mathfrak{A}$ is non-singular. Next, it follows from \eqref{given-flow} and Lemma \ref{hilbert} that  $\mathfrak{B}(t)\in C^1[0,T]$ and $\mathfrak{c}(t)\in L^2(0,T)$. Therefore, by the classical  ODE theory, the following existence result for \eqref{mu^n} holds:
\begin{Lemma}[\cite{cod}]
Problem \eqref{mu^n} admits a unique solution $\mu_j^{N,\vartheta}\in AC[0,T]$, for each $j=1,2,\cdots\!, N$, $N\in \NN^*$, and $\vartheta\in (0,1)$, where $AC$ denotes the space of absolutely continuous functions. As a consequence, $w^{N,\vartheta} \in AC([0,T];\cH^1_{r^2})$ and $w^{N,\vartheta}$ is differentiable \textit{a.e.} in $t$, for each $N\in \NN^*$ and $\vartheta\in (0,1)$.
\end{Lemma}

\smallskip
\textbf{Step 2. Uniform estimates of $w^{N,\vartheta}$.} First, multiplying \eqref{galerkin-n} by $\mu^{N,\vartheta}_j(t)$ and  summing the resulting equality with respect to $j$ from $1$ to $N$ implies 
\begin{equation}\label{L-0}
\frac{1}{2}\frac{\mathrm{d}}{\dt}\int_0^1 r^2\rho_0|w^{N,\vartheta}|^2\,\mathrm{d}r+L_1 =L_2,
\end{equation}
where
\begin{equation}\label{def-L1L2}
\begin{aligned}
L_1&:=\int_0^1 \bar\eta^2\bar\eta_r\bar\varrho^{\delta} \big((2a_1+a_2)X^2+4a_2 XY+4(a_1+a_2)Y^2\big)\,\mathrm{d}r,\\
L_2&:=\int_0^1(r^2\rho_0^{\delta})^\frac{1}{2} \big(2g Y + hX\big)\,\mathrm{d}r,\qquad X := D_{\bar\eta}w^{N,\vartheta},\quad   Y:= \frac{w^{N,\vartheta}}{\bar\eta}.
\end{aligned}
\end{equation}

For $L_1$, thanks to \eqref{jibenjiashe}, $a_1>0$, and $2a_1+3a_2>0$, we can find a constant $c_*>0$, depending only on $(a_1,a_2)$, such that  
\begin{equation}\label{L-1}
L_1\geq c_*|\sD_r w^{N,\vartheta}|_{2,r^2\rho_0^{\delta}}^2.
\end{equation}
For $L_2$, it follows from the Young inequality that
\begin{equation}\label{L-2}
L_2\leq C|(g,h)|_2^2+\frac{c_*}{10}|\sD_r w^{N,\vartheta}|_{2,r^2\rho_0^{\delta}}^2.
\end{equation}
Besides, by Lemmas \ref{sobolev-embedding}--\ref{hardy-inequality}, we have
\begin{equation}\label{kongzhi}
\begin{aligned}
|w^{N,\vartheta}|_{2,r^2\rho_0^{\delta}}&\leq C\big(|\chi w^{N,\vartheta}|_{2,r^2}+|\chi^\sharp w^{N,\vartheta}|_{2, \rho_0^{\delta}}\big)\\
&\leq C\big(\big|\chi \frac{w^{N,\vartheta}}{r}\big|_{2,r^2}+|\chi^\sharp w^{N,\vartheta}|_{2,\rho_0}+|\chi^\sharp w_r^{N,\vartheta}|_{2,\rho_0^{\delta}}\big)\\
&\leq C\big(| w^{N,\vartheta}|_{2,r^2\rho_0}+|\sD_rw^{N,\vartheta}|_{2, r^2\rho_0^{\delta}}\big).   
\end{aligned}
\end{equation}
Collecting \eqref{L-0} and \eqref{L-1}--\eqref{L-2} and integrating over $[0,t]$, combined with \eqref{kongzhi}, implies
\begin{equation}\label{33125}
|w^{N,\vartheta}(t)|_{2,r^2\rho_0}^2+\int_0^{t} \|w^{N,\vartheta}\|_{\cH^1_{r^2\rho_0^{\delta}}}^2\!\dt\leq C(T)\Big(|w^{N,\vartheta}(0)|_{2,r^2\rho_0}^2+ \int_0^{t} |(g,h)|_2^2\,\dt\Big),
\end{equation}
where
$w^{N,\vartheta}(0,r)=\sum_{j=1}^N \mu_j^{N,\vartheta}(0)\xi_j=\sum_{j=1}^N \langle r^2 w_0^\vartheta,\xi_j\rangle \xi_j$.

Next, to derive the $L^2_{r^2\rho_0}$-boundedness of $w^{N,\vartheta}(0,r)$, we first see from \eqref{wdelta-w} that, for any $\varepsilon>0$, there exists $\vartheta_0=\vartheta_0(\varepsilon)>0$ satisfying
\begin{equation}\label{zer1}
|w^\vartheta_0-w_0|_{2,r^2\rho_0}<\frac{\varepsilon}{2} \qquad \text{ for any $0<\vartheta\leq\vartheta_0$}.
\end{equation}
Then, for such $\varepsilon,\vartheta_0>0$ and fixed $\vartheta\in (0,\vartheta_0]$, by $w_0^\vartheta\in L^2\subset L^2_{r^2\rho_0}$ and Lemma \ref{hilbert}, there exists a large $N_0=N_0(\varepsilon,\vartheta)\in \NN^*$ so that
\begin{equation*}
|w^{N,\vartheta}(0)-w^\vartheta_0|_{2,r^2\rho_0}=\Big|\sum_{j=1}^N \langle r^2 w_0^\vartheta,\xi_j\rangle \xi_j-w_0^\vartheta\Big|_{2,r^2\rho_0}<\frac{\varepsilon}{2} \qquad \text{for any $N\geq N_0$},
\end{equation*}
which, combined with \eqref{zer1}, yields 
\begin{equation*}
|w^{N,\vartheta}(0)-w_0|_{2,r^2\rho_0}\leq|w^{N,\vartheta}(0)-w^\vartheta_0|_{2,r^2\rho_0}+|w^\vartheta_0-w_0|_{2,r^2\rho_0}<\varepsilon.
\end{equation*}
Hence, if $w_0\ne 0$, setting $\varepsilon:=\min\{|w_0|_{2,r^2\rho_0},\frac{1}{2}\}$, we obtain a constant $\vartheta_0=\vartheta_0(\varepsilon)>0$, such that for any $\vartheta\in (0,\vartheta_0]$, there exists $N_0=N_0(\varepsilon,\vartheta)\in \NN^*$ such that
\begin{equation}\label{initial-converge}
|w^{N,\vartheta}(0)|_{2,r^2\rho_0}\leq 2|w_0|_{2,r^2\rho_0}\quad \text{for any $N\geq N_0$}.
\end{equation}
If $w_0=0$, we choose $w_0^\vartheta\equiv0$. Then $w^{N,\vartheta}(0)=0$ for every $N\in\mathbb N^*$ and $\vartheta\in(0,1)$, so that \eqref{initial-converge} also holds in this case.

Finally, substituting  \eqref{initial-converge} into \eqref{33125} yields
\begin{equation}\label{3...18}
\sup_{t\in[0,T]} |w^{N,\vartheta}|_{2,r^2\rho_0}^2+\int_0^{T} \|w^{N,\vartheta}\|_{\cH^1_{r^2\rho_0^{\delta}}}^2\dt
\leq C(T)\Big(|w_0|_{2,r^2\rho_0}^2+\int_0^{T}|(g,h)|_2^2\,\dt\Big).
\end{equation}

\textbf{Step 3. Taking  the limit as $N,\vartheta^{-1}\to \infty$.}
Since $(L^2_{r^2\rho_0},L^2_{r^2\rho_0^{\delta}})$ are separable reflexive Banach spaces due to  Lemma \ref{W-space}, it follows from \eqref{3...18} and the weak convergence arguments that there exists a  subsequence (still denoted by) $w^{N,\vartheta}$ and some limits $(w,X_1,X_2)$ such that
\begin{equation}\label{3..124}
\begin{aligned}
w^{N,\vartheta} \rightharpoonup w\qquad  &\text{weakly* in }L^\infty([0,T];L^2_{r^2\rho_0}),\\
\big(w^{N,\vartheta},w^{N,\vartheta}_r,\frac{w^{N,\vartheta}}{r}\big) \rightharpoonup (w,X_1,X_2)\qquad &\text{weakly \, in }L^2([0,T];L^2_{r^2\rho_0^{\delta}}).
\end{aligned}
\end{equation}
Then the definition of weak derivatives implies that $(X_1,X_2)=(w_r,\frac{w}{r})$.  
Clearly, $w$ also satisfies \eqref{3...18} due to the lower semi-continuity of weak convergence \eqref{3..124}.

Now we pass to the limit $N,\vartheta^{-1}\to \infty$ in \eqref{galerkin-n}. Let $\varphi^M(t,r)=\sum_{j=1}^M \varphi_j(t)\xi_j(r)$ with $\varphi_j \in C_{\mathrm{c}}^\infty(0,T)$ for $j=1,\cdots\!,M$ and $M\in \NN^*$. 
Then by \eqref{galerkin-n}, for any $N\geq M$,
\begin{equation}\label{equ3320}
\begin{aligned}
&\int_0^T \big<r^2\rho_0 w^{N,\vartheta}_t  , \varphi^M\big>\,\dt\\
&+  (2a_1+a_2)\int_0^T \big<\bar\eta^2\bar\eta_r\bar\varrho^{\delta} D_{\bar\eta} w^{N,\vartheta} ,D_{\bar\eta} \varphi^M \big>\,\dt  +4(a_1+a_2) \int_0^T \Big<\bar\eta^2\bar\eta_r\bar\varrho^{\delta}\frac{w^{N,\vartheta}}{\bar\eta} ,\frac{\varphi^M}{\bar\eta}\Big>\,\dt\\
&+ 2a_2\int_0^T \Big<\bar\eta^2\bar\eta_r\bar\varrho^{\delta}  \frac{w^{N,\vartheta}}{\bar\eta} ,D_{\bar\eta} \varphi^M \Big>\,\dt  +  2a_2\int_0^T \Big<\bar\eta^2\bar\eta_r\bar\varrho^{\delta} D_{\bar\eta}w^{N,\vartheta} ,\frac{\varphi^M}{\bar\eta}\Big>\,\dt\\
&=\int_0^T\Big<g,(r^2\rho_0^{\delta})^\frac{1}{2}  \frac{2\varphi^M}{\bar\eta}\Big>\,\dt+\int_0^T\big< h,(r^2\rho_0^{\delta})^\frac{1}{2} D_{\bar\eta} \varphi^M\big>\,\dt.
\end{aligned}
\end{equation}
Since $(w^{N,\vartheta},\varphi^M)$ are differentiable with respect to $t$, we can transfer $\partial_t$ in \eqref{equ3320} from $w^{N,\vartheta}$ to $\varphi^M$ via integration by parts. Then, letting $N,\vartheta^{-1}\to \infty$, we see from \eqref{3..124} that
\begin{equation}\label{3,,20}
\begin{aligned}
&-\int_0^T \langle r^2\rho_0 w, \varphi^M_t\rangle\,\dt\\
&\quad +  (2a_1+a_2)\int_0^T \big<\bar\eta^2\bar\eta_r\bar\varrho^{\delta} D_{\bar\eta} w,D_{\bar\eta} \varphi^M \big>\,\dt
+4(a_1+a_2) \int_0^T \Big<\bar\eta^2\bar\eta_r\bar\varrho^{\delta}\frac{w}{\bar\eta} ,\frac{\varphi^M}{\bar\eta}\Big>\,\dt\\
&\quad + 2a_2\int_0^T \Big<\bar\eta^2\bar\eta_r\bar\varrho^{\delta}  \frac{w}{\bar\eta} ,D_{\bar\eta} \varphi^M \Big>\,\dt+2a_2\int_0^T \Big<\bar\eta^2\bar\eta_r\bar\varrho^{\delta}D_{\bar\eta}w,\frac{\varphi^M}{\bar\eta}\Big>\,\dt\\
&=\int_0^T\Big<g,(r^2\rho_0^{\delta})^\frac{1}{2}  \frac{2\varphi^M}{\bar\eta}\Big>\,\dt+\int_0^T\big< h,(r^2\rho_0^{\delta})^\frac{1}{2} D_{\bar\eta} \varphi^M\big>\,\dt.
\end{aligned}
\end{equation}

Next, since $r^2\rho_0 w_t\in H^{-1}([0,T];L^2)$ due to $w\in L^2([0,T];L^2_{r^2\rho_0})$, by \eqref{jibenjiashe}, \eqref{3,,20}, and the definition of  distributional derivatives, we obtain
\begin{equation*}
\begin{aligned}
&\,\Big|\langle r^2\rho_0 w_t, \varphi^M\rangle_{H_t^{-1}((H^1)^*)\times H^1_{0,t}(H^1)}\Big|
=\Big|\int_0^T\langle r^2\rho_0 w, \varphi^M_t\rangle\,\dt\Big| \\
&\leq C \big(\|w\|_{L^2_t(\cH^1_{r^2\rho_0^{\delta}})}+ \|(g,h)\|_{L^2_t(L^2)}\big)\|\varphi^M\|_{L^2_t(\cH^1_{r^2\rho_0^{\delta}})}. 
\end{aligned}    
\end{equation*}
We now claim that
\begin{equation}\label{dens}
\left\{\varphi^M\right\}_{M\in\NN^*}\text{ is dense in }L^2([0,T];\cH^1_{r^2\rho_0^{\delta}}).
\end{equation}
Indeed, it suffices to prove the density of  $\mathrm{span}\{\xi_j\}_{j\in\NN^*}$ in $\cH^1_{r^2\rho_0^{\delta}}$. Thanks to Lemma \ref{prop-bijin}, for any given $f\in \cH^1_{r^2\rho_0^{\delta}}$, there exists a sequence $\{f^\varepsilon\}_{\varepsilon>0}\subset C^\infty(\bar I)\cap \cH^1_{r^2}$ such that
\begin{equation*}
f^\varepsilon\to f \ \ \text{in $\cH^1_{r^2\rho_0^{\delta}}$} \qquad \text{as $\varepsilon\to 0$}.
\end{equation*}
Next, for any $f\in C^\infty(\bar I)\cap \cH^1_{r^2}\subset \cH^1_{r^2}$, we obtain from Lemma \ref{hilbert} and $\cH^1_{r^2}\subset \cH^1_{r^2\rho_0^{\delta}}$ that
\begin{equation*}
\sum_{j=1}^M \langle r^2f,\xi_j\rangle \xi_j\to f \, \,\,\, \text{in $\cH^1_{r^2\rho_0^{\delta}}$} \qquad\,\, \text{as $M\to\infty$}.
\end{equation*}
Consequently, if setting $f^{M,\varepsilon}:=\sum_{j=1}^M \langle r^2 f^\varepsilon,\xi_j\rangle \xi_j$, then $\{f^{M,\varepsilon}\}\subset\mathrm{span}\{\xi_j\}_{j\in\NN^*}$ will be a sequence that converges to $f$ in $\cH^1_{r^2\rho_0^{\delta}}$ as $\varepsilon^{-1}, M\to\infty$, which shows the claim.

Let $\mathscr{T}$ denote the space of finite sums $\sum_{j} \varphi_j(t)\xi_j(r)$, with $\varphi_j\in C^\infty_{\rm c}(0,T)$. By \eqref{3,,20}, the restriction of the distribution $r^2\rho_0w_t$ to $\mathscr{T}$ is bounded in $L^2(0,T;H^1_{r^2\rho_0^\delta})$-norm. Since $\mathscr{T}$ is dense in this space by \eqref{dens}, it admits a unique continuous extension. Identifying this extension with an element of $L^2(0,T;H^{-1}_{r^2\rho_0^\delta})$, and comparing on smooth test functions, yields  
\begin{equation}\label{H-1}
\|r^2\rho_0 w_t\|_{L^2_t(\cH^{-1}_{r^2\rho_0^{\delta}})}\leq C\big(\|w\|_{L^2_t(\cH^1_{r^2\rho_0^{\delta}})}+ \|(g,h)\|_{L^2_t(L^2)}\big) <\infty.
\end{equation}

Finally, for any $\hat\varphi\in L^2([0,T];\cH^1_{r^2\rho_0^{\delta}})$,  
there exists a sequence $\{\varphi^M\}_{M\in\NN^*}$ 
with each $\varphi^M\in L^2([0,T];\cH^1_{r^2\rho_0^{\delta}})$ such that $\varphi^M\to \hat\varphi$ in $L^2([0,T];\cH^1_{r^2\rho_0^{\delta}})$ as $M\to\infty$. Passing to the limit $M\to \infty$ in \eqref{3,,20}, along with  \eqref{H-1}, implies that, for all $\hat\varphi\in L^2([0,T];\cH^1_{r^2\rho_0^{\delta}})$,
\begin{equation}\label{3..133}
\begin{aligned}
&\int_0^T \langle r^2\rho_0 w_t, \hat\varphi\rangle_{\cH^{-1}_{r^2\rho_0^{\delta}} \times \cH^1_{r^2\rho_0^{\delta}}}\,\dt\\
&\quad  +  (2a_1+a_2)\int_0^T \big<\bar\eta^2\bar\eta_r\bar\varrho^{\delta} D_{\bar\eta} w,D_{\bar\eta} \hat\varphi  \big>\,\dt  +4(a_1+a_2) \int_0^T \Big<\bar\eta^2\bar\eta_r\bar\varrho^{\delta}\frac{w}{\bar\eta} ,\frac{\hat\varphi}{\bar\eta}\Big>\,\dt\\
&\quad +2 a_2  \int_0^T \Big<\bar\eta^2\bar\eta_r\bar\varrho^{\delta}  \frac{w }{\bar\eta} ,D_{\bar\eta} \hat\varphi \Big>\,\dt  +  2a_2\int_0^T \Big<\bar\eta^2\bar\eta_r\bar\varrho^{\delta} D_{\bar\eta}w ,\frac{\hat\varphi}{\bar\eta}\Big>\,\dt\\
&=\int_0^T\Big<g, (r^2\rho_0^{\delta})^\frac{1}{2}\frac{2\hat\varphi}{\bar\eta}\Big>\,\dt+\int_0^T\langle h, (r^2\rho_0^{\delta})^\frac{1}{2}  D_{\bar\eta} \hat\varphi\rangle\,\dt.
\end{aligned}
\end{equation}
Hence, for any fixed $\varphi\in\mathcal H^1_{r^2\rho_0^\delta}$ and $\tau\in[0,T]$, choosing $\widehat\varphi(s,r)=\mathbf1_{(0,\tau)}(s)\varphi(r)$ in \eqref{3..133} gives the corresponding identity with all time integrals taken over $[0,\tau]$. Since the integrands belong to $L^1(0,T)$, differentiating this identity with respect to $\tau$ yields the weak formulation (2.10) for a.e. $\tau\in(0,T)$.

\medskip
\textbf{Step 4. Uniqueness and time continuity.}
First, since $0<\delta<1$ and $\rho_0$ is bounded, we have
\begin{equation*}
|f|_{L^2_{r^2\rho_0}}^2
\le |\rho_0|_\infty^{1-\delta}|f|_{L^2_{r^2\rho_0^\delta}}^2
\le C|f|_{\cH^1_{r^2\rho_0^\delta}}^2.
\end{equation*}
Thus, $\cH^1_{r^2\rho_0^\delta}\hookrightarrow L^2_{r^2\rho_0}$ continuously.
It follows from Lemma \ref{Aubin} with $s=1$ and $\cH=\cH^1_{r^2\rho_0^\delta}$,
$w\in L^2([0,T];\cH^1_{r^2\rho_0^\delta})$, and
$r^2\rho_0w_t\in L^2([0,T];\cH^{-1}_{r^2\rho_0^\delta})$ that
$w\in C([0,T];L^2_{r^2\rho_0})$.

Next, we  show that $w(0,r)=w_0(r)$ for \textit{a.e.} $r\in I$. To this end, thanks to \eqref{weak.F.} and $w\in C([0,T];L^2_{r^2\rho_0})$, for any $\hat\varphi\in C_{\mathrm{c}}^1([0,T);\cH^1_{r^2\rho_0^\delta})$, 
\begin{equation}\label{3...125}
\begin{aligned}
&-\int_0^T \big<r^2\rho_0 w, \hat\varphi_t\big>\,\dt\\
&\quad  +  (2a_1+a_2)\int_0^T \big<\bar\eta^2\bar\eta_r\bar\varrho^{\delta} D_{\bar\eta} w,D_{\bar\eta} \hat\varphi  \big>\,\dt  +4(a_1+a_2) \int_0^T \Big<\bar\eta^2\bar\eta_r\bar\varrho^{\delta}\frac{w}{\bar\eta} ,\frac{\hat\varphi}{\bar\eta}\Big>\,\dt\\
&\quad +2a_2\int_0^T \Big<\bar\eta^2\bar\eta_r\bar\varrho^{\delta}  \frac{w }{\bar\eta} ,D_{\bar\eta} \hat\varphi \Big>\,\dt  +  2a_2\int_0^T \Big<\bar\eta^2\bar\eta_r\bar\varrho^{\delta} D_{\bar\eta}w ,\frac{\hat\varphi}{\bar\eta}\Big>\,\dt\\
&=\underline{\big< r^2\rho_0 w(0), \hat\varphi(0)\big>}_{:=\mathring{L}}+\int_0^T\Big<g, (r^2\rho_0^\delta)^\frac{1}{2}\frac{2\hat\varphi}{\bar\eta}\Big>\,\dt+\int_0^T\langle h, (r^2\rho_0^\delta)^\frac{1}{2}  D_{\bar\eta} \hat\varphi\rangle\,\dt.
\end{aligned}
\end{equation}
On the other hand, let $\varphi^M(t,r)=\sum_{j=1}^M \varphi_j(t) \xi_j$ in \eqref{equ3320} with $\varphi_j(t)\in C^\infty_{\mathrm{c}}([0,T))$ ($j=1,\cdots\!,M$) satisfy $\varphi^M\to \hat\varphi$ in $C^1_{\mathrm{c}}([0,T);\cH^1_{r^2\rho_0^\delta})$ as $M\to \infty$. Then by \eqref{equ3320}, for all $M\leq N$,
\begin{equation*}
\begin{aligned}
&-\int_0^T \big<r^2\rho_0 w^{N,\vartheta},\varphi^M_t\big>\,\dt+  (2a_1+a_2)\int_0^T \big<\bar\eta^2\bar\eta_r\bar\varrho^{\delta}D_{\bar\eta} w^{N,\vartheta} ,D_{\bar\eta} \varphi^M \big>\,\dt \\
&\quad   +4(a_1+a_2) \int_0^T \Big<\bar\eta^2\bar\eta_r\bar\varrho^{\delta}\frac{w^{N,\vartheta}}{\bar\eta} ,\frac{\varphi^M}{\bar\eta}\Big>\,\dt\\
&\quad + 2a_2\int_0^T \Big<\bar\eta^2\bar\eta_r\bar\varrho^{\delta}\frac{w^{N,\vartheta}}{\bar\eta} ,D_{\bar\eta} \varphi^M \Big>\,\dt  +  2a_2\int_0^T \Big<\bar\eta^2\bar\eta_r\bar\varrho^{\delta} D_{\bar\eta}w^{N,\vartheta} ,\frac{\varphi^M}{\bar\eta}\Big>\,\dt\\
&=\big<r^2\rho_0 w^{N,\vartheta}(0),\varphi^M(0)\big>+\int_0^T\Big<g,(r^2\rho_0^\delta)^\frac{1}{2}  \frac{2\varphi^M}{\bar\eta}\Big>\,\dt+\int_0^T\big< h,(r^2\rho_0^\delta)^\frac{1}{2} D_{\bar\eta} \varphi^M\big>\,\dt.
\end{aligned}
\end{equation*}
Passing to the limit along the subsequence in \eqref{3..124}, with $N\to\infty$ and $\vartheta\to0$, and then letting $M\to\infty$ in the above implies that, for all $\hat\varphi\in C^1_{\mathrm{c}}([0,T);\cH^1_{r^2\rho_0^\delta})$, \eqref{3...125} still holds except for replacing term $\mathring{L}$ with $\langle r^2\rho_0 w_0,\hat\varphi (0)\rangle$. Hence, this yields that $w(0,r)=w_0$ for \textit{a.e.} $r\in I$. Finally, setting $w_0=0$, $\varphi=w$, and $g=h=0$ in \eqref{weak.F.} implies $w=0$, which leads to the uniqueness.  
\end{proof}

Next, if $w_0\in \cH^1_{r^2\rho_0^\delta}$, we can improve the regularity of weak solutions in Proposition \ref{prop1}.
\begin{Proposition}\label{prop-strong}
Let $w$ be the weak solution of \eqref{galerkin-w} obtained in {\rm Proposition \ref{prop1}}. Assume that $w_0\in \cH^1_{r^2\rho_0^\delta}$, $(g,h)\in L^\infty([0,T];L^2)$, and $(g_t,h_t)\in L^2([0,T];L^2)$. Then 
\begin{equation}\label{reg-H}
\begin{aligned}
&w\in L^\infty([0,T];\cH^1_{r^2\rho_0^\delta}),\quad w_t\in L^2([0,T];L^2_{r^2\rho_0}),\\
&\sqrt{t} w_t\in L^\infty([0,T];L^2_{r^2\rho_0})\cap L^2([0,T];\cH^1_{r^2\rho_0^\delta}),\quad \sqrt{t}(r^2\rho_0 w_{tt})\in L^2([0,T];\cH^{-1}_{r^2\rho_0^\delta}).
\end{aligned}
\end{equation}
Moreover, $w$ satisfies the following weak formulation for any \textit{a.e.} $t>0$ and $\varphi\in  \cH^1_{r^2\rho_0^\delta}${\rm:}
\begin{align}
&\begin{aligned}
&\ \big<r^2\rho_0 w_{tt} , \varphi\big>_{\cH^{-1}_{r^2\rho_0^\delta}\times \cH^{1}_{r^2\rho_0^\delta}}  +\Big< \bar\eta^2\bar\eta_r\bar\varrho^{\delta}, (2a_1+a_2)D_{\bar\eta} w_t D_{\bar\eta}\varphi +4(a_1+a_2) \frac{w_t}{\bar\eta} \frac{\varphi}{\bar\eta}\Big>\notag\\
&\quad + 2a_2\Big<\bar\eta^2\bar\eta_r\bar\varrho^{\delta},\frac{w_t}{\bar\eta} D_{\bar\eta}\varphi+ D_{\bar\eta} w_t \frac{\varphi}{\bar\eta}\Big>
\end{aligned}\\
&\begin{aligned}
&\quad +(2a_1+a_2)\Big< \bar\eta^2\bar\eta_r\bar\varrho^{\delta}\Big(-(\delta+1)D_{\bar\eta}\bar U+(2-2\delta)\frac{\bar U}{\bar\eta}\Big) D_{\bar\eta} w,D_{\bar\eta}\varphi\Big>\\
&\quad +4(a_1+a_2)\Big<\bar\eta^2\bar\eta_r\bar\varrho^{\delta}\Big((1-\delta)D_{\bar\eta}\bar U- 2\delta \frac{\bar U}{\bar\eta} \Big)\frac{w}{\bar\eta},\frac{\varphi}{\bar\eta}\Big>
\end{aligned}\label{weak-F-H}\\
&\begin{aligned}
&\quad + 2a_2\Big<\bar\eta^2\bar\eta_r\bar\varrho^{\delta}\Big(- \delta D_{\bar\eta}\bar U+(1-2\delta)\frac{\bar U}{\bar\eta}\Big),\frac{w}{\bar\eta} D_{\bar\eta}\varphi+D_{\bar\eta} w \frac{\varphi}{\bar\eta}\Big>\notag\\
&=\Big<(r^2\rho_0^{\delta})^\frac{1}{2} \big(g_t- g \frac{\bar U}{\bar\eta}\big), \frac{2\varphi}{\bar\eta}\Big>+\big<(r^2\rho_0^{\delta})^\frac{1}{2} \big(h_t- hD_{\bar\eta}\bar U\big), D_{\bar\eta}\varphi\big>, 
\end{aligned}
\end{align}
and satisfies the following energy equality for \textit{a.e.} $t>0${\rm:}
\begin{align}
&\frac{1}{2}\frac{\mathrm{d}}{\dt}\int_0^1 r^2\rho_0|w_t|^2\,\mathrm{d}r +\int_0^1 \bar\eta^2\bar\eta_r\bar\varrho^{\delta}\Big((2a_1+a_2)|D_{\bar\eta}w_t|^2+4a_2 D_{\bar\eta}w_t\frac{w_t}{\bar\eta}+4(a_1+a_2) \frac{|w_t|^2}{\bar\eta^2}\Big)\,\mathrm{d}r\notag\\
&=-(2a_1+a_2)\int_0^1\bar\eta^2\bar\eta_r\bar\varrho^{\delta}\Big(-(\delta+1)D_{\bar\eta}\bar U+(2-2\delta)\frac{\bar U}{\bar\eta}\Big)D_{\bar\eta}w D_{\bar\eta}w_t\,\mathrm{d}r\notag\\
&\quad -4(a_1+a_2)\int_0^1\bar\eta^2\bar\eta_r\bar\varrho^{\delta}\Big((1-\delta)D_{\bar\eta}\bar U- 2\delta \frac{\bar U}{\bar\eta} \Big) \frac{w}{\bar\eta}\frac{w_t}{\bar\eta}\,\mathrm{d}r\label{Energy-Id}\\
&\quad - 2a_2\int_0^1 \bar\eta^2\bar\eta_r\bar\varrho^{\delta}\Big(- \delta D_{\bar\eta}\bar U+(1-2\delta)\frac{\bar U}{\bar\eta}\Big)\Big(\frac{w}{\bar\eta} D_{\bar\eta}w_t+D_{\bar\eta}w\frac{w_t}{\bar\eta}\Big)\,\mathrm{d}r\notag\\
&\quad +\int_0^1(r^2\rho_0^{\delta})^\frac{1}{2} \Big(\big(g_t- g \frac{\bar U}{\bar\eta}\big) \frac{2w_t}{\bar\eta} + \big(h_t- h D_{\bar\eta}\bar U\big)D_{\bar\eta}w_t\Big)\,\mathrm{d}r.\notag
\end{align}
\end{Proposition}
\begin{proof}
We divide the proof into four steps. Define
\begin{equation}\label{MT}
M_T=1+\sup_{t\in[0,T]}\big(\bar \cE(t,\bar U)+t \bar D(t,\bar U)\big)+\int_0^T \bar D(t,\bar U)\,\mathrm{d}t.
\end{equation}
By \eqref{given}, we obtain that $1\leq M_T<\infty$.

\smallskip
\textbf{Step 1.} First, by Lemma \ref{prop-bijin}, for given $w_0\in \cH^1_{r^2\rho_0^{\delta}}$, we see that there exists a smooth sequence $\{w_0^\vartheta\}_{\vartheta>0}\subset C^\infty(\bar I)\cap \cH^1_{r^2\rho_0^{\delta}}$ satisfying
\begin{equation}\label{wdelta-w2}
w^\vartheta_0\to w_0 \qquad \text{in $\cH^1_{r^2\rho_0^{\delta}}$ \ as $\vartheta\to 0$}.
\end{equation}

\smallskip
\textbf{Step 2.} Based on this smooth initial data $w_0^\vartheta$ and the Galerkin scheme presented in Step 1 of the proof of Proposition \ref{prop1}, we can construct a Galerkin approximate sequence $w^{N,\vartheta}\in AC([0,T];\cH^1_{r^2\rho_0^{\delta}})$ and a system of  ODEs, which are the same as those given in \eqref{U^n}--\eqref{mu^n}. 

Now, multiplying $\eqref{galerkin-n}_1$ by $(\mu_j^{N,\vartheta})_t(t)$ and summing the
resulting equality with respect to $j$ from $1$ to $N$ implies
\begin{equation}\label{I0*}
\frac{1}{2}\frac{\mathrm{d}}{\mathrm{d}t}L_1+\int_0^1 r^2\rho_0|w^{N,\vartheta}_t|^2\,\mathrm{d}r=L_3+L_4  +\frac{\mathrm{d}}{\mathrm{d}t}L_5,
\end{equation}
where $(L_1,X,Y)$ are defined as in \eqref{def-L1L2} and
\begin{align*}
&\begin{aligned}
L_3&:=\frac{1}{2}\int_0^1 \bar\eta^2\bar\eta_r\bar\varrho^{\delta}\Big((2a_1+a_2)\big(-(\delta+1)D_{\bar\eta}\bar U+(2-2\delta)\frac{\bar U}{\bar\eta}\big) X^2 \\
&\qquad\qquad + 4(a_1+a_2)\big((1-\delta)D_{\bar\eta}\bar U- 2\delta \frac{\bar U}{\bar\eta} \big) Y^2 + 4a_2\big(- \delta D_{\bar\eta}\bar U+(1-2\delta)\frac{\bar U }{\bar\eta}\big)XY\Big)\,\mathrm{d}r,
\end{aligned}\\
&\begin{aligned}
L_4&:= \int_0^1 (r^2\rho_0^{\delta})^\frac{1}{2}\Big(2\big(g\frac{\bar U}{\bar\eta}-g_t\big)Y+(h D_{\bar\eta}\bar U-h_t) X \Big)\,\mathrm{d}r,\\
L_5&:=\int_0^1 (r^2\rho_0^{\delta})^\frac{1}{2}\big(2gY+h X\big)\,\mathrm{d}r.
\end{aligned}
\end{align*}

To estimate $L_3$--$L_4$, we obtain from \eqref{jibenjiashe} and the H\"older and Young inequalities that
\begin{equation*} 
\begin{aligned}
L_3&\leq C|\sD_r\bar U|_\infty|\sD_r w^{N,\vartheta}|_{2,r^2\rho_0^{\delta}}^2\leq C(T,M_T)|\sD_r w^{N,\vartheta}|_{2,r^2\rho_0^{\delta}}^2,\\
L_4&\leq C\big(|\sD_r\bar U|_\infty|(g,h)|_2+|(g_t,h_t)|_2\big)|\sD_r w^{N,\vartheta}|_{2,r^2\rho_0^{\delta}}\\
&\leq C(T,M_T)\big(|\sD_r w^{N,\vartheta}|_{2,r^2\rho_0^{\delta}}^2 + |(g,h,g_t,h_t)|_2^2\big).
\end{aligned}
\end{equation*}
While for $L_5$, we have
\begin{equation*} 
|L_5|\leq C\varepsilon^{-1}|(g,h)|_2^2+\varepsilon|\sD_r w^{N,\vartheta}|_{2,r^2\rho_0^{\delta}}^2 \qquad\text{for all $\varepsilon\in(0,1)$}.
\end{equation*}
Thus, by choosing $\varepsilon$ small enough, we obtain from \eqref{I0*}, the estimates of $L_3$--$L_5$, the estimate of $L_1$ (see \eqref{L-1}), and the Gr\"onwall inequality that 
\begin{equation*}
\begin{aligned}
&\|w^{N,\vartheta} \|_{L^\infty_t(\cH^1_{r^2\rho_0^{\delta}})}+ \|w^{N,\vartheta}_t\|_{L^2_t(L^2_{r^2\rho_0})} \\
&\leq C(T,M_T)\Big(\|w^{N,\vartheta}(0)\|_{\cH^1_{r^2\rho_0^{\delta}}}+ \|(g,h)\|_{L^\infty_t(L^2)}+\|(g_t,h_t)\|_{L^2_t(L^2)} \Big).
\end{aligned}
\end{equation*}
For the initial data, it follows from (ii) of Lemma \ref{hilbert}, \eqref{wdelta-w2}, and the argument similar to \eqref{zer1}--\eqref{initial-converge} that there exists $\vartheta_0>0$, such that for any $\vartheta\in (0,\vartheta_0]$, we can select $N_0=N_0(\vartheta)\in \NN^*$ such that 
\begin{equation*} 
\|w^{N,\vartheta}(0)\|_{\cH^1_{r^2\rho_0^{\delta}}}\leq 2\|w_0\|_{\cH^1_{r^2\rho_0^{\delta}}}\qquad \text{for any $N\geq N_0$},
\end{equation*}
which thus leads to the uniform estimate:
\begin{equation}\label{uni-1}
\begin{aligned}
&\|w^{N,\vartheta} \|_{L^\infty_t(\cH^1_{r^2\rho_0^{\delta}})}+ \|w^{N,\vartheta}_t\|_{L^2_t(L^2_{r^2\rho_0})} \\
&\leq C(T,M_T)\Big(\|w_0\|_{\cH^1_{r^2\rho_0^{\delta}}}+ \|(g,h)\|_{L^\infty_t(L^2)}+\|(g_t,h_t)\|_{L^2_t(L^2)} \Big).
\end{aligned}
\end{equation}

\smallskip
\textbf{Step 3.} Note that, thanks to the regularity assumptions of $(g,h)$ and $(g_t,h_t)$ in Proposition \ref{prop-strong}, we have $\mathfrak{B}(t)\in C^1[0,T]$ and $\mathfrak{c}(t)\in H^1(0,T)$, where $(\mathfrak{B} ,\mathfrak{c})(t)$ are defined in \eqref{ABC}. Then the classical theory of ODEs (\textit{cf.} \cite{cod}) implies that $(\mu^{N,\vartheta})_t(t)\in AC([0,T])$, so that $w^{N,\vartheta}_{t}$ is differentiable \textit{a.e.} in $t$. Hence, applying $\partial_t$ to $\eqref{galerkin-n}_1$ gives
\begin{align*}
&\ \big<r^2\rho_0 w_{tt}^{N,\vartheta}, \xi_j\big>  +\Big< \bar\eta^2\bar\eta_r\bar\varrho^{\delta}, (2a_1+a_2)D_{\bar\eta} w_t^{N,\vartheta}D_{\bar\eta}\xi_j +4(a_1+a_2) \frac{w_t^{N,\vartheta}}{\bar\eta} \frac{\xi_j}{\bar\eta}\Big>\\
&\quad + 2a_2\Big<\bar\eta^2\bar\eta_r\bar\varrho^{\delta},\frac{w_t^{N,\vartheta}}{\bar\eta} D_{\bar\eta}\xi_j+ D_{\bar\eta} w_t^{N,\vartheta} \frac{\xi_j}{\bar\eta}\Big>\\
&\quad +(2a_1+a_2)\Big< \bar\eta^2\bar\eta_r\bar\varrho^{\delta}\Big(-(\delta+1)D_{\bar\eta}\bar U+(2-2\delta)\frac{\bar U}{\bar\eta}\Big) D_{\bar\eta} w^{N,\vartheta},D_{\bar\eta}\xi_j\Big>\\
&\quad +4(a_1+a_2)\Big<\bar\eta^2\bar\eta_r\bar\varrho^{\delta}\Big((1-\delta)D_{\bar\eta}\bar U- 2\delta \frac{\bar U}{\bar\eta} \Big)\frac{w^{N,\vartheta}}{\bar\eta},\frac{\xi_j}{\bar\eta}\Big>\\
&\quad + 2a_2\Big<\bar\eta^2\bar\eta_r\bar\varrho^{\delta}\Big(- \delta D_{\bar\eta}\bar U+(1-2\delta)\frac{\bar U}{\bar\eta}\Big),\frac{w^{N,\vartheta}}{\bar\eta} D_{\bar\eta}\xi_j+D_{\bar\eta} w^{N,\vartheta} \frac{\xi_j}{\bar\eta}\Big>\\
&=\Big<(r^2\rho_0^{\delta})^\frac{1}{2} \big(g_t- g \frac{\bar U}{\bar\eta}\big), \frac{2\xi_j}{\bar\eta}\Big>+\big<(r^2\rho_0^{\delta})^\frac{1}{2} \big(h_t- hD_{\bar\eta}\bar U\big), D_{\bar\eta}\xi_j\big>.
\end{align*}

Next, multiplying the above by $(\mu_j^{N,\vartheta})_{t} $ and summing $j$ from $1$ to $N$,  we have
\begin{equation*}
\frac{1}{2}\frac{\mathrm{d}}{\dt}\int_0^1 r^2\rho_0|w^{N,\vartheta}_t|^2\,\mathrm{d}r+L_6=-L_7+L_8,
\end{equation*}
where $(X,Y)$ are defined as in \eqref{def-L1L2} and
\begin{align*}
L_6&:=\int_0^1 \bar\eta^2\bar\eta_r\bar\varrho^{\delta}\big((2a_1+a_2)\widetilde X^2+4a_2 \widetilde X \widetilde Y+4(a_1+a_2) \widetilde Y^2\big)\,\mathrm{d}r,\\
L_7&:=(2a_1+a_2)\int_0^1\bar\eta^2\bar\eta_r\bar\varrho^{\delta}\Big(-(\delta+1)D_{\bar\eta}\bar U+(2-2\delta)\frac{\bar U}{\bar\eta}\Big)X \widetilde X\,\mathrm{d}r\\
&\quad +4(a_1+a_2)\int_0^1\bar\eta^2\bar\eta_r\bar\varrho^{\delta}\Big((1-\delta)D_{\bar\eta}\bar U- 2\delta \frac{\bar U}{\bar\eta} \Big) Y \widetilde Y\,\mathrm{d}r\\
&\quad + 2a_2\int_0^1 \bar\eta^2\bar\eta_r\bar\varrho^{\delta}\Big(- \delta D_{\bar\eta}\bar U+(1-2\delta)\frac{\bar U}{\bar\eta}\Big)\big(Y \widetilde X+X\widetilde Y\big)\,\mathrm{d}r,\\
L_8&:= \int_0^1(r^2\rho_0^{\delta})^\frac{1}{2} \Big(\big(g_t- g \frac{\bar U}{\bar\eta}\big) 2\widetilde Y + \big(h_t- h D_{\bar\eta}\bar U\big) \widetilde X\Big)\,\mathrm{d}r,\quad (\widetilde X,\widetilde Y) :=(D_{\bar\eta}w_t^{N,\vartheta},\frac{w_t^{N,\vartheta}}{\bar\eta}).
\end{align*}

For $L_6$--$L_8$, it follows from \eqref{jibenjiashe}, $a_1>0$, and $2a_1+3a_2>0$ that  
\begin{equation*}
\begin{aligned}
L_6&\geq c_*|\sD_r w_t^{N,\vartheta}|_{2,r^2\rho_0^{\delta}}^2,\qquad |L_7|\leq C |\sD_r\bar U|_\infty |\sD_r w^{N,\vartheta}|_{2,r^2\rho_0^{\delta}}|\sD_r w_t^{N,\vartheta}|_{2,r^2\rho_0^{\delta}}, \\
L_8&\leq C\big(|\sD_r\bar U|_\infty|(g,h)|_2+|(g_t,h_t)|_2\big)|\sD_r w_t^{N,\vartheta}|_{2,r^2\rho_0^{\delta}}.
\end{aligned}
\end{equation*}

Thus, it follows from the above estimates, an estimate similar to \eqref{kongzhi}:
\begin{equation*} 
|w_t^{N,\vartheta}|_{2,r^2\rho_0^{\delta}} \leq C\big(| w_t^{N,\vartheta}|_{2,r^2\rho_0}+|\sD_rw_t^{N,\vartheta}|_{2, r^2\rho_0^{\delta}}\big)   
\end{equation*}  
and the Young inequality that
\begin{equation*}
\frac{\mathrm{d}}{\mathrm{d}t}|w^{N,\vartheta}_t|_{2,r^2\rho_0}^2+ c_*\|w^{N,\vartheta}_t\|_{\cH^1_{r^2\rho_0^{\delta}}}^2 \leq C(T,M_T)\big(|\sD_r w^{N,\vartheta}|_{2,r^2\rho_0^{\delta}}^2+|(g,h,g_t,h_t)|_2^2+|w^{N,\vartheta}_t|_{2,r^2\rho_0}^2\big).
\end{equation*}
Multiplying the above by $t$ and integrating over $[\tau,t]$, together with \eqref{uni-1}, we have
\begin{equation}\label{uni-2pre}
\begin{aligned}
&\|\sqrt{t}w^{N,\vartheta}_t \|_{L^\infty_t(L^2_{r^2\rho_0})}^2  +\int_\tau^ts\|w^{N,\vartheta}_t\|_{\cH^1_{r^2\rho_0^{\delta}}}^2\,\mathrm{d}s\\
&\leq\tau|w^{N,\vartheta}_t(\tau)|_{2,r^2\rho_0}^2+C(T,M_T)\Big(\|w_0\|_{\cH^1_{r^2\rho_0^{\delta}}}^2+ \|(g,h)\|_{L^\infty_t(L^2)}^2+\|(g_t,h_t)\|_{L^2_t(L^2)}^2 \Big).
\end{aligned}
\end{equation}
Finally, due to Lemma \ref{bjr} and the fact that $(r^2\rho_0)^\frac{1}{2}w^{N,\vartheta}_t\in L^2([0,T];L^2 )$, there exists a sequence $\{\tau_k\}_{k=1}^\infty\subset[0,T]$ such that
\begin{equation*}
\tau_k\to 0, \quad \tau_k|w^{N,\vartheta}_t(\tau_k)|_{2,r^2\rho_0}^2\to 0\qquad \ \ \text{as $k\to \infty$}.
\end{equation*}
Hence, letting $\tau=\tau_k\to 0$ in \eqref{uni-2pre}, we arrive at
\begin{equation*}
\begin{aligned}
&\|\sqrt{t}w^{N,\vartheta}_t\|_{L^\infty_t(L^2_{r^2\rho_0})}^2  +\int_0^t s\| w^{N,\vartheta}_t\|_{\cH^1_{r^2\rho_0^{\delta}}}^2\,\mathrm{d}s \\
&\leq C(T,M_T)\Big(\|w_0\|_{\cH^1_{r^2\rho_0^{\delta}}}^2+ \|(g,h)\|_{L^\infty_t(L^2)}^2+\|(g_t,h_t)\|_{L^2_t(L^2)}^2 \Big).
\end{aligned}
\end{equation*}

\smallskip
\textbf{Step 4. Weak formulation and energy equality.} First, the regularity property \eqref{reg-H} can be derived via a weak convergence argument similar to Step 3 of the proof of Proposition \ref{prop1}. We omit the details here for brevity.

Next, we prove that $r^2\rho_0 w_{tt}\in L^2([\tau,T];\cH^{-1}_{r^2\rho_0^{\delta}})$ with $\tau\in(0,T)$. Since $w_{t}\in L^2([0,T];L^2_{r^2\rho_0})$, we see from \eqref{weak.F.} that, for any $\varphi\in \cH^1_{r^2\rho_0^{\delta}}$,
\begin{equation*}
\begin{aligned}
&\big<r^2\rho_0 w_t, \varphi\big> +(2a_1+a_2)\big< \bar\eta^2\bar\eta_r\bar\varrho^{\delta} D_{\bar\eta} w,D_{\bar\eta}\varphi\big>+4(a_1+a_2)\Big<\bar\eta^2\bar\eta_r\bar\varrho^{\delta}\frac{w}{\bar\eta},\frac{\varphi}{\bar\eta}\Big>\\
&\quad + 2a_2\Big<\bar\eta^2\bar\eta_r\bar\varrho^{\delta}\frac{w}{\bar\eta},D_{\bar\eta}\varphi\Big>+ 2a_2\Big<\bar\eta^2\bar\eta_r\bar\varrho^{\delta}D_{\bar\eta} w,\frac{\varphi}{\bar\eta}\Big>\\
&=\Big<g,(r^2\rho_0^{\delta})^\frac{1}{2} \frac{2\varphi}{\bar\eta}\Big>+\big< h,(r^2\rho_0^{\delta})^\frac{1}{2}  D_{\bar\eta}\varphi\big>.
\end{aligned}
\end{equation*}
Differentiating the above with respect to $t$ gives that, for \textit{a.e.} $t>0$,
\begin{align}
\frac{\mathrm{d}}{\mathrm{d}t}\big<(r^2\rho_0 w_t), \varphi\big>&=-\Big< \bar\eta^2\bar\eta_r\bar\varrho^{\delta}, (2a_1+a_2)D_{\bar\eta} w_t D_{\bar\eta}\varphi +4(a_1+a_2) \frac{w_t}{\bar\eta} \frac{\varphi}{\bar\eta}\Big>\notag\\
&\quad - 2a_2\Big<\bar\eta^2\bar\eta_r\bar\varrho^{\delta},\frac{w_t}{\bar\eta} D_{\bar\eta}\varphi+ D_{\bar\eta} w_t \frac{\varphi}{\bar\eta}\Big>\notag\\
&\begin{aligned}
&\quad -(2a_1+a_2)\Big< \bar\eta^2\bar\eta_r\bar\varrho^{\delta}\Big(-(\delta+1)D_{\bar\eta}\bar U+(2-2\delta)\frac{\bar U}{\bar\eta}\Big) D_{\bar\eta} w,D_{\bar\eta}\varphi\Big>\\
&\quad -4(a_1+a_2)\Big<\bar\eta^2\bar\eta_r\bar\varrho^{\delta}\Big((1-\delta)D_{\bar\eta}\bar U- 2\delta \frac{\bar U}{\bar\eta} \Big)\frac{w}{\bar\eta},\frac{\varphi}{\bar\eta}\Big>
\end{aligned}\label{weak.F.tt} \\
&\begin{aligned}
&\quad - 2a_2\Big<\bar\eta^2\bar\eta_r\bar\varrho^{\delta}\Big(- \delta D_{\bar\eta}\bar U+(1-2\delta)\frac{\bar U}{\bar\eta}\Big),\frac{w}{\bar\eta} D_{\bar\eta}\varphi+D_{\bar\eta} w \frac{\varphi}{\bar\eta}\Big>\notag\\
&\quad +\Big<(r^2\rho_0^{\delta})^\frac{1}{2} \big(g_t- g \frac{\bar U}{\bar\eta}\big), \frac{2\varphi}{\bar\eta}\Big>+\big<(r^2\rho_0^{\delta})^\frac{1}{2} \big(h_t- hD_{\bar\eta}\bar U\big), D_{\bar\eta}\varphi\big>, 
\end{aligned}
\end{align}
which yields
\begin{equation*}
\begin{aligned}
\Big|\frac{\mathrm{d}}{\mathrm{d}t}\big<r^2\rho_0 w_t, \varphi\big>\Big| &\leq C\big(|(g_t,h_t)|_2+ |\sD_r\bar U|_\infty|(g,h)|_2\big)\|\varphi\|_{\cH^{1}_{r^2\rho_0^{\delta}}}\\
&\quad + C\big(\|w_t\|_{\cH^{1}_{r^2\rho_0^{\delta}}}+|\sD_r\bar U|_\infty\|w\|_{\cH^{1}_{r^2\rho_0^{\delta}}}\big)\|\varphi\|_{\cH^{1}_{r^2\rho_0^{\delta}}} \leq Q(t)\|\varphi\|_{\cH^{1}_{r^2\rho_0^{\delta}}},
\end{aligned}
\end{equation*}
for some $Q(t)\in L^2(\tau,T)$ with $\tau \in (0,T)$. This, along with Lemma 1.1 on \cite{temam}*{page 250}, leads to $r^2\rho_0 w_{tt}\in L^2([\tau,T];\cH^{-1}_{r^2\rho_0^{\delta}})$. Since $w_t\in L^2([\tau,T];\cH^1_{r^2\rho_0^\delta})$ for every $\tau\in(0,T)$,
we can apply Lemma \ref{Aubin} on $[\tau,T]$, with $f=w_t$, $s=1$, and
$\cH=\cH^1_{r^2\rho_0^\delta}$, to obtain
\begin{equation*}
\langle r^2\rho_0w_{tt},w_t\rangle_{\cH^{-1}_{r^2\rho_0^\delta}\times\cH^1_{r^2\rho_0^\delta}}
=\frac{1}{2}\frac{\mathrm d}{\mathrm dt}|w_t|_{L^2_{r^2\rho_0}}^2
\quad\text{for a.e. }t\in(\tau,T).
\end{equation*}
Finally, it follows from the above and \eqref{weak.F.tt}  that \eqref{weak-F-H}--\eqref{Energy-Id} hold. 

This completes the proof of Proposition \ref{prop-strong}.
\end{proof}

\subsection{Global Existence and Uniqueness for the Linearized Problem}\label{subsection3.3}
In this subsection, we present the proof of Lemma \ref{existence-linearize} based on Propositions \ref{prop1}--\ref{prop-strong}. We divide the proof into nine steps. Moreover, we adapt the constant $M_T$ defined in \eqref{MT}, namely,
\begin{equation*}
M_T=1+\sup_{t\in[0,T]}\big(\bar \cE(t,\bar U)+t \bar D(t,\bar U)\big)+\int_0^T \bar D(t,\bar U)\,\mathrm{d}t.
\end{equation*}

\smallskip
\textbf{Step 1. Tangential estimate $U\in C([0,T];L^2_{r^2\rho_0})$.} First, let 
\begin{equation*}
\begin{aligned}
&w^{(0)}|_{t=0}=w_0^{(0)}:=u_0,\qquad h^{(0)}:=A\Big(\frac{r^2}{\bar\eta^2\bar\eta_r}\Big)^{\gamma-1} r\rho_0^{\gamma-\frac{\delta}{2}},\\
&g^{(0)}:=h^{(0)}- 2\pi G\frac{r\rho_0^{1-\frac{\delta}{2}}}{\bar\eta}\int_0^r\hat r^2\rho_0\,\mathrm{d}\hat r
\end{aligned}
\end{equation*}
in \eqref{lp}. Then $(g^{(0)},h^{(0)})\in L^2([0,T];L^2)$ and $w_0^{(0)}\in L^2_{r^2\rho_0}$, and hence we can deduce from Proposition \ref{prop1} that problem \eqref{lp} admits a unique weak solution $w^{(0)}=U$ such that
\begin{equation}\label{3..142}
U\in C([0,T];L^2_{r^2\rho_0})\cap L^2([0,T];\cH^1_{r^2\rho_0^{\delta}}),\qquad r^2\rho_0 U_t\in L^2([0,T];\cH^{-1}_{r^2\rho_0^{\delta}}).
\end{equation}

\smallskip
\textbf{Step 2. Tangential estimates $U\in C([0,T];\cH^1_{r^2\rho_0^{\delta}})$ and $U_t\in C([0,T];L^2_{r^2\rho_0})$.}
First,  formally applying $\partial_t$ to the equation in \eqref{lp}, we arrive at
\begin{equation}\label{33151}
\begin{aligned}
&r^2\rho_0 U_{tt}-2(r^2\rho_0^{\delta})^\frac{1}{2}\frac{g^{(1)}}{\bar\eta}+\Big((r^2\rho_0^{\delta})^\frac{1}{2}\frac{h^{(1)}}{\bar\eta_r}\Big)_r \\
&=\Big(\bar\eta^2\bar\varrho^{\delta}\big((2a_1+a_2)D_{\bar\eta}U_t+2a_2\frac{U_t}{\bar\eta}\big)\Big)_r -\bar\eta \bar\eta_r\bar\varrho^{\delta}\Big(2a_2  D_{\bar\eta}U_t+4(a_1+a_2)\frac{U_t}{\bar\eta}\Big),
\end{aligned}
\end{equation}
where $g^{(1)}=g^{(1)}_*+g^{(1)}_{**}$, $h^{(1)}=h^{(1)}_*+h^{(1)}_{**}$, and
\begin{align}
&\begin{aligned}
g^{(1)}_*&:= a_2\Big(\frac{r^2}{\bar\eta^2\bar\eta_r}\Big)^{\delta-1}r\rho_0^\frac{\delta}{2}\Big(\delta D_{\bar\eta} \bar U-(1-2\delta)\frac{\bar U}{\bar\eta}\Big)D_{\bar\eta} U \\
&\quad \ +2(a_1+a_2)\Big(\frac{r^2}{\bar\eta^2\bar\eta_r}\Big)^{\delta-1}r\rho_0^\frac{\delta}{2}\Big((\delta-1)D_{\bar\eta} \bar U+2\delta \frac{\bar U}{\bar\eta}\Big)\frac{U}{\bar\eta},\notag
\end{aligned}\\
&\begin{aligned}
h^{(1)}_*&:=(2a_1+a_2)\Big(\frac{r^2}{\bar\eta^2\bar\eta_r}\Big)^{\delta-1}r\rho_0^\frac{\delta}{2}\Big((\delta+1)D_{\bar\eta} \bar U-(2-2\delta)\frac{\bar U}{\bar\eta}\Big) D_{\bar\eta}U\\
&\quad \ +2a_2\Big(\frac{r^2}{\bar\eta^2\bar\eta_r}\Big)^{\delta-1}r\rho_0^\frac{\delta}{2}\Big(\delta D_{\bar\eta} \bar U-(1-2\delta)\frac{\bar U}{\bar\eta}\Big)\frac{U}{\bar\eta}, 
\end{aligned}\label{q1-q2-1}\\
&\begin{aligned}
g^{(1)}_{**}&:=- A\Big(\frac{r^2}{\bar\eta^2\bar\eta_r}\Big)^{\gamma-1} r\rho_0^{\gamma-\frac{\delta}{2}}\Big(\frac{(2\gamma-1)\bar U}{\bar\eta}+(\gamma-1) D_{\bar\eta}\bar U\Big)+4\pi G\frac{r\bar U}{\bar\eta^2}\rho_0^{1-\frac{\delta}{2}}\int_0^r\hat r^2\rho_0\,\mathrm{d}\hat r,\\
h^{(1)}_{**}&:= -A\Big(\frac{r^2}{\bar\eta^2\bar\eta_r}\Big)^{\gamma-1} r\rho_0^{\gamma-\frac{\delta}{2}} \Big(\frac{(2\gamma-2)\bar U}{\bar\eta}+\gamma D_{\bar\eta}\bar U\Big). \notag
\end{aligned}
\end{align}
Clearly, a direct calculation yields the following relations:
\begin{equation}\label{the-relate}
\big(\frac{g^{(0)}}{\bar\eta}\big)_t=\frac{g^{(1)}_{**}}{\bar\eta},\qquad\big(\frac{h^{(0)}}{\bar\eta_r}\big)_t=\frac{h^{(1)}_{**}}{\bar\eta_r}.
\end{equation}

Next, we regard \eqref{33151} as the equation of $w^{(1)}:=U_t$ and study the problem: 
\begin{equation}\label{33152}
\begin{cases}
\displaystyle r^2\rho_0 w^{(1)}_t-2(r^2\rho_0^{\delta})^\frac{1}{2}\frac{g^{(1)}}{\bar\eta}+\Big((r^2\rho_0^{\delta})^\frac{1}{2}\frac{h^{(1)}}{\bar\eta_r}\Big)_r \\
\displaystyle  =\Big(\bar\eta^2 \bar\varrho^{\delta}\big((2a_1+a_2)D_{\bar\eta}w^{(1)}+2a_2\frac{w^{(1)}}{\bar\eta}\big)\Big)_r-\bar\eta\bar\eta_r\bar\varrho^{\delta}\Big(2a_2  D_{\bar\eta}w^{(1)}+4(a_1+a_2)\frac{w^{(1)}}{\bar\eta}\Big),\\[10pt]
w^{(1)}|_{t=0}=w^{(1)}_0=U_t\big|_{t=0}\qquad\text{on } I.
\end{cases}
\end{equation}
We can check that $w^{(1)}_0\in L^2_{r^2\rho_0}$ and $(g^{(1)},h^{(1)})\in L^2([0,T];L^2)$ due to \eqref{given}--\eqref{jibenjiashe} and \eqref{3..142}. Hence, it follows from Proposition \ref{prop1} that  \eqref{33152} admits a unique weak solution $w^{(1)}$ satisfying
\begin{equation}\label{3.142'}
w^{(1)}\in C([0,T];L^2_{r^2\rho_0})\cap L^2([0,T];\cH^1_{r^2\rho_0^{\delta}}),\qquad r^2\rho_0 w^{(1)}_t\in L^2([0,T];\cH^{-1}_{r^2\rho_0^{\delta}}).
\end{equation}
Now, we claim that  $w^{(1)}=U_t$ for {\it a.e.} $(t,r)\in (0,T)\times I$. Set 
\begin{equation}\label{rela}
\tilde{U}(t,r):=\int_0^t w^{(1)}(\tau,r)\,\mathrm{d}\tau+ u_0(r), \qquad \ \varpi:=\tilde{U}-U.
\end{equation}
We only need to show that $\varpi=0$ for {\it a.e.} $(t,r)\in (0,T)\times I$.

Indeed, since $U$ and $w^{(1)}$ are weak solutions of \eqref{lp} and \eqref{33152}, respectively, we obtain from \eqref{weak.F.} that, for all $\varphi \in \cH^1_{r^2\rho_0^{\delta}}$ and {\it a.e.} $t\in(0,T)$,
\begin{align}
&\big<r^2\rho_0 \mathcal{Z}^{(j)}_t, \varphi\big>_{\cH^{-1}_{r^2\rho_0^{\delta}}\times \cH^1_{r^2\rho_0^{\delta}}}+(2a_1+a_2)\big< \bar\eta^2\bar\eta_r\bar\varrho^{\delta} D_{\bar\eta} \mathcal{Z}^{(j)},D_{\bar\eta}\varphi\big>+4(a_1+a_2)\Big<\bar\eta^2\bar\eta_r\bar\varrho^{\delta}\frac{\mathcal{Z}^{(j)}}{\bar\eta},\frac{\varphi}{\bar\eta}\Big>\notag\\
&\quad + 2a_2\Big<\bar\eta^2\bar\eta_r\bar\varrho^{\delta}\frac{\mathcal{Z}^{(j)}}{\bar\eta},D_{\bar\eta}\varphi\Big>+ 2a_2\Big<\bar\eta^2\bar\eta_r\bar\varrho^{\delta}D_{\bar\eta} \mathcal{Z}^{(j)},\frac{\varphi}{\bar\eta}\Big>\notag\\
&=\Big<g^{(j)},(r^2\rho_0^{\delta})^\frac{1}{2} \frac{2\varphi}{\bar\eta}\Big>+\big< h^{(j)},(r^2\rho_0^{\delta})^\frac{1}{2}  D_{\bar\eta}\varphi\big>,\label{equ1}
\end{align}
where $j=0,1$ and $(\mathcal{Z}^{(0)},\mathcal{Z}^{(1)}):=(U,w^{(1)})$.

Next, in view of \eqref{rela}, we replace $\mathcal{Z}^{(1)}=w^{(1)}$ in \eqref{equ1} by $\tilde{U}_t$ and integrate over $[0,t]$ for $t\in (0,T]$. Then based on the compatibility condition $\eqref{116}_1$: %in Appendix \ref{AppB}:
\begin{equation*}
\begin{aligned}
r^2 \rho_0 U_t(0,r)&=(2a_1+a_2)\Big(r^2\rho_0^{\delta}\big((u_0)_r+\frac{2a_2}{2a_1+a_2}\frac{u_0}{r}\big)\Big)_r -2a_2 r\rho_0^{\delta}(u_0)_r-4(a_1+a_2) \rho_0^{\delta} u_0 \\
&\quad +2A  r \rho_0^{\gamma} -A (r^2 \rho_0^{\gamma})_r-4\pi G \rho_0\int_0^r \hat r^2\rho_0\,\mathrm{d}\hat r,
\end{aligned}
\end{equation*}
and the relations \eqref{the-relate}, we arrive at the following equality:
\begin{equation*}
\begin{aligned}
&\,\big<r^2\rho_0 \tilde{U}_t,\varphi\big>_{\cH^{-1}_{r^2\rho_0^{\delta}}\times \cH^1_{r^2\rho_0^{\delta}}}+(2a_1+a_2)\big<\bar\eta^2\bar\eta_r \bar\varrho^{\delta} D_{\bar\eta} \tilde{U},D_{\bar\eta}\varphi\big>+4(a_1+a_2)\Big<\bar\eta^2\bar\eta_r \bar\varrho^{\delta}\frac{\tilde{U}}{\bar\eta},\frac{\varphi}{\bar\eta}\Big>\\
&\quad + 2a_2\Big<\bar\eta^2\bar\eta_r \bar\varrho^{\delta}\frac{\tilde{U}}{\bar\eta},D_{\bar\eta}\varphi\Big>+ 2a_2\Big<\bar\eta^2\bar\eta_r \bar\varrho^{\delta}D_{\bar\eta}\tilde{U},\frac{\varphi}{\bar\eta}\Big>\\
&=\Big<g^{(0)},(r^2\rho_0^{\delta})^\frac{1}{2} \frac{2\varphi}{\bar\eta}\Big>+\big< h^{(0)},(r^2\rho_0^{\delta})^\frac{1}{2}  D_{\bar\eta}\varphi\big>+L_9,
\end{aligned}
\end{equation*}
where
\begin{equation*}
\begin{aligned}
L_9:=&\int_0^t(2a_1+a_2)\Big< \bar\eta^2\bar\eta_r \bar\varrho^{\delta}\Big(-(\delta+1)D_{\bar\eta}\bar U+(2-2\delta)\frac{\bar U}{\bar\eta}\Big) D_{\bar\eta} \varpi,D_{\bar\eta}\varphi\Big>\\
&\quad +4(a_1+a_2)\Big<\bar\eta^2\bar\eta_r \bar\varrho^{\delta}\Big((1-\delta)D_{\bar\eta}\bar U- 2\delta \frac{\bar U}{\bar\eta} \Big)\frac{\varpi}{\bar\eta},\frac{\varphi}{\bar\eta}\Big>\\
&\quad +2a_2\Big<\bar\eta^2\bar\eta_r \bar\varrho^{\delta}\Big(- \delta D_{\bar\eta}\bar U+(1-2\delta)\frac{\bar U}{\bar\eta}\Big),\frac{\varpi}{\bar\eta} D_{\bar\eta}\varphi+D_{\bar\eta} \varpi \frac{\varphi}{\bar\eta}\Big>\,\mathrm{d}s.
\end{aligned}
\end{equation*}

Subtracting \eqref{equ1} from the above equality leads to
\begin{equation}\label{equ3}
\begin{aligned}
&\,\big<r^2\rho_0 \varpi_t,\varphi\big>_{\cH^{-1}_{r^2\rho_0^{\delta}}\times \cH^{1}_{r^2\rho_0^{\delta}}} +(2a_1+a_2)\Big< \bar\eta^2\bar\eta_r \bar\varrho^{\delta} D_{\bar\eta} \varpi,D_{\bar\eta}\varphi\Big>+4(a_1+a_2)\Big<\bar\eta^2\bar\eta_r \bar\varrho^{\delta}\frac{\varpi}{\bar\eta},\frac{\varphi}{\bar\eta}\Big>\\
&\quad + 2a_2\Big<\bar\eta^2\bar\eta_r \bar\varrho^{\delta}\frac{\varpi}{\bar\eta},D_{\bar\eta}\varphi\Big>+ 2a_2\Big<\bar\eta^2\bar\eta_r \bar\varrho^{\delta}D_{\bar\eta}\varpi,\frac{\varphi}{\bar\eta}\Big>=L_9. 
\end{aligned}
\end{equation}
Then, using the fact that $\varpi\in L^2([0,T];\cH^1_{r^2\rho_0^{\delta}})$, we can let $\varphi=\varpi$ in \eqref{equ3} and obtain from Lemma \ref{hardy-inequality}, the Young inequality, and the similar calculations in Step 2 of the proof of Proposition \ref{prop1} that
\begin{equation*}
\begin{aligned}
\frac{\mathrm{d}}{\dt}|\varpi|_{2,r^2\rho_0}^2+c_*\|\varpi\|_{\cH^1_{r^2\rho_0^{\delta}}}^2&\leq Ct\big(\sup_{0\leq s \leq t}|\sD_r \bar U|_{\infty}^2+1\big)\int_0^t \|\varpi\|_{\cH^1_{r^2\rho_0^{\delta}}}^2\,\ds\\
&\leq C(T,M_T)\int_0^t \|\varpi\|_{\cH^1_{r^2\rho_0^{\delta}}}^2\,\ds.
\end{aligned}
\end{equation*}
Here $c_*$ is obtained as in \eqref{L-1}. Then integrating the above over $[0,t]$, together  with the strong continuity of $\varpi$ at $t=0$ and $\varpi|_{t=0}=0$, yields 
\begin{equation*}
\int_0^t \|\varpi\|_{\cH^1_{r^2\rho_0^{\delta}}}^2\,\ds\leq C(T,M_T)\int_0^t \Big(\int_0^\tau\|\varpi\|_{\cH^1_{r^2\rho_0^{\delta}}}^2\,\ds\Big)\mathrm{d}\tau.
\end{equation*}
Then  the Gr\"onwall inequality implies that $\|\varpi\|_{L^2_t(\cH^1_{r^2\rho_0^\delta})}= 0$, and hence $\varpi=0$.

Thus, \eqref{3.142'} holds for $U_t$:
\begin{equation}\label{33153}
U_t\in C([0,T];L^2_{r^2\rho_0})\cap L^2([0,T];\cH^1_{r^2\rho_0^{\delta}}),\qquad r^2\rho_0 U_{tt}\in L^2([0,T];\cH^{-1}_{r^2\rho_0^{\delta}}),
\end{equation}
which, along with \eqref{3..142} and Lemma \ref{sobolev-embedding}, gives
\begin{equation}\label{equ33.38}
U\in C([0,T];\cH^1_{r^2\rho_0^{\delta}}).
\end{equation}

\smallskip
\textbf{Step 3. Boundary condition of $U$.} Thanks to  \eqref{lp} and  \eqref{weak.F.}, we can show that $\eqref{lp}_1$ holds for \textit{a.e.} $(t,r)\in (0,T)\times (0,1)$, which is crucial to our further analysis.
\begin{Lemma}\label{Lemma-point}
For any $a\in (0,1)$ and \textit{a.e.} $t\in (0,T)$,
\begin{equation}\label{div-Uxx}
r\sD_{\bar\eta}^2 U\in L^2(0,a),\qquad \Big(\bar\eta^2 \bar\varrho^{\delta}\big((2a_1+a_2)D_{\bar\eta} U + 2a_2 \frac{U}{\bar\eta}\big)\Big)_r\in L^2(a,1).
\end{equation}
Furthermore, equation $\eqref{lp}_1$ holds for \textit{a.e.} $(t,r)\in (0,T)\times (0,1)$, and $U$ satisfies 
\begin{equation}\label{equ33.41}
U\in H^3_{\mathrm{loc}},\qquad \ \rho_0^{\delta}\Big(D_{\bar\eta} U+ \frac{2a_2}{2a_1+a_2}\frac{U}{\bar\eta}\Big)\Big|_{r=1}=0 \quad \text{for {\it a.e.} $t\in (0,T)$}.
\end{equation}
\end{Lemma}
\begin{proof}
We divide the proof into four steps.

\smallskip
\textbf{Step 1.} Note that, once $\eqref{div-Uxx}_2$ is obtained, based on the regularity properties of $\bar\eta$ in \eqref{given-bareta} and those of $U$ in \eqref{3..142} and \eqref{33153}--\eqref{equ33.38}, we can directly deduce that each term in $\eqref{lp}_1$ belongs to $L^1_{\mathrm{loc}}$, and hence $\eqref{lp}_1$ holds for \textit{a.e.} $(t,r)\in (0,T)\times (0,1)$.

To derive $\eqref{div-Uxx}_2$, fix $a\in(0,1)$, we see  from \eqref{equ1}, \eqref{33153}  that, for all $\varphi\in C_{\mathrm{c}}^\infty((a,1))$,
\begin{equation*}
\begin{aligned}
&\Big|\Big< \bar\eta^2 \bar\varrho^{\delta}\Big((2a_1+a_2)D_{\bar\eta} U + 2a_2 \frac{U}{\bar\eta}\Big), \varphi_r\Big>\Big| \\
&\leq C\big|\big<r^2\rho_0 U_t, \varphi\big>\big|+C\Big|\Big<(r^2\rho_0^{\delta})^\frac{1}{2}\frac{2g^{(0)}}{\bar\eta}-\big((r^2\rho_0^{\delta})^\frac{1}{2}\frac{h^{(0)}}{\bar\eta_r}\big)_r,\varphi\Big>\Big|\\
&\quad+C\Big|\Big<\bar\eta^2\bar\eta_r \bar\varrho^{\delta}\frac{U}{\bar\eta},\frac{\varphi}{\bar\eta}\Big>\Big|+ C\Big|\Big<\bar\eta^2\bar\eta_r \bar\varrho^{\delta}D_{\bar\eta} U,\frac{\varphi}{\bar\eta}\Big>\Big| \leq C(a,T,M_T) |\varphi|_2,
\end{aligned}
\end{equation*}
where $C(a,T,M_T)>0$ is independent of $\varphi$. This, along with the definition of the weak derivative, yields $\eqref{div-Uxx}_2$.

\smallskip
\textbf{Step 2.} We prove $\eqref{equ33.41}_1$. Since equation $\eqref{lp}_1$ takes the form:
\begin{equation*}
(\mathfrak{A}_1 U_{r}+\mathfrak{A}_2 U)_r+\mathfrak{A}_3 U_r+\mathfrak{A}_4 U=\mathfrak{A}_5,
\end{equation*}
with $\mathfrak{A}_i$ ($1\leq i\leq 5$) naturally defined,
we can check that $\mathfrak{A}_i\in H^3_{\mathrm{loc}}$ ($1\leq i\leq 4$) and $\mathfrak{A}_5\in H^1_{\mathrm{loc}}$  for \textit{a.e.} $t\in (0,T)$. Hence, it follows from the classical regularity theory of elliptic equations (\textit{cf.} \cites{evans}) that $U\in H^3_{\mathrm{loc}}$ for \textit{a.e.} $t\in (0,T)$.

\smallskip
\textbf{Step 3.} To obtain $\eqref{div-Uxx}_1$, since $U\in H^3_{\mathrm{loc}}$ and $U_t\in H^1_{\mathrm{loc}}$ for \textit{a.e.} $t\in (0,T)$, and $\eqref{lp}_1$ holds for \textit{a.e.} $(t,r)\in (0,T)\times (0,1)$, we rewrite $\eqref{lp}_1$ as
\begin{equation}\label{new-lp}
\begin{aligned}
D_{\bar\eta}\Big(D_{\bar\eta} U+ \frac{2U}{\bar\eta} \Big)&=-\frac{\delta}{\gamma-1} (D_{\bar\eta}\log(\bar\varrho^{\gamma-1})) \Big(D_{\bar\eta} U+ \frac{2a_2}{2a_1+a_2}  \frac{U}{\bar\eta}\Big)+\frac{1}{2a_1+ a_2}\bar\varrho^{1-\delta} U_t\\
&\quad +\frac{A}{2a_1+a_2}\frac{\gamma}{\gamma-\delta}D_{\bar\eta}\bar\varrho^{\gamma-\delta}+ \frac{4\pi G}{2a_1+ a_2}\frac{\bar\varrho^{1-\delta}}{\bar\eta^2}\int_0^r\hat r^2\rho_0\,\mathrm{d}\hat{r}.
\end{aligned}
\end{equation}

Then we can directly check that, for any $a\in (0,1)$, 
\begin{equation*}
\begin{aligned}
\Big|\zeta_a D_{\bar\eta}\Big(D_{\bar\eta} U+ \frac{2U}{\bar\eta}\Big)\Big|_{2,r^2}&\leq C(a,T,M_T)\big(|\zeta_a (U_t,\sD_r U)|_{2,r^2}+ |\zeta_a(\bar\varrho^{\gamma-\delta})_r|_{2,r^2}+|\zeta_a \bar\varrho^{1-\delta}|_{2,r^2}\big)\\
&\leq C(a,T,M_T),
\end{aligned}
\end{equation*}
due to \eqref{33153}--\eqref{equ33.38}, the regularity of $\bar\eta$ in \eqref{given-bareta}, and the fact that $\bar\varrho^{\gamma-1}\sim 1-r$. Hence, by Lemma \ref{im-1}, we derive $\eqref{div-Uxx}_1$.

\smallskip
\textbf{Step 4.} Finally, we give the proof for  $\eqref{equ33.41}_2$. First, recall that \eqref{equ1} holds for $\cZ^{(0)}=U$ and for all $\varphi\in C_{\mathrm{c}}^\infty((0,1])$. Next, by \eqref{equ33.38}--\eqref{div-Uxx}, for any $a\in(0,1)$ and a.e. $t\in[0,T]$, we have
\begin{equation*}
\bar\eta^2\bar\varrho^\delta
\Big((2a_1+a_2)D_{\bar\eta}U+2a_2\frac{U}{\bar\eta}\Big)
\in H^1(a,1).
\end{equation*}
Hence, by Lemma \ref{sobolev-embedding}, $\bar\eta^2\bar\varrho^\delta
\Big((2a_1+a_2)D_{\bar\eta}U+2a_2\frac{U}{\bar\eta}\Big)$  has a continuous representative on $[a,1]$ and a well-defined trace at $r=1$. Since $A\bar\eta^2\bar\varrho^\gamma$ vanishes at $r=1$, multiplying $\eqref{lp}_1$ by such $\varphi$ and integrating over $I$ yields 
\begin{equation*} 
\begin{aligned}
&\big<r^2\rho_0 U_t, \varphi\big> +(2a_1+a_2)\big<\bar\eta^2\bar\eta_r\bar\varrho^{\delta} D_{\bar\eta} U,D_{\bar\eta}\varphi\big>+4(a_1+a_2)\Big<\bar\eta^2\bar\eta_r\bar\varrho^{\delta}\frac{U}{\bar\eta},\frac{\varphi}{\bar\eta}\Big>\\
&\quad + 2a_2\Big<\bar\eta^2\bar\eta_r\bar\varrho^{\delta}\frac{U}{\bar\eta},D_{\bar\eta}\varphi\Big>+ 2a_2\Big<\bar\eta^2\bar\eta_r\bar\varrho^{\delta}D_{\bar\eta} U,\frac{\varphi}{\bar\eta}\Big>\\
&=(2a_1+a_2) \bar\eta^2 \bar\varrho^{\delta}\Big(D_{\bar\eta} U + \frac{2a_2}{2a_1+a_2} \frac{U}{\bar\eta}\Big)\varphi\Big|_{r=1} +\Big<g^{(0)},(r^2\rho_0^{\delta})^\frac{1}{2} \frac{2\varphi}{\bar\eta}\Big>+\big< h^{(0)},(r^2\rho_0^{\delta})^\frac{1}{2}  D_{\bar\eta}\varphi\big>.
\end{aligned}
\end{equation*}

Hence, this, compared with \eqref{equ1} and together with \eqref{jibenjiashe}, gives
\begin{equation*}
\rho_0^{\delta}\Big(D_{\bar\eta} U + \frac{2a_2}{2a_1+a_2} \frac{U}{\bar\eta}\Big)\varphi\big|_{r=1}=0 \qquad\text{for all $\varphi\in C_{\mathrm{c}}^\infty((0,1])$},
\end{equation*}
which implies $\eqref{equ33.41}_2$.  This completes the proof of Lemma \ref{Lemma-point}.
\end{proof}

\textbf{Step 4. $L^\infty([0,T];L^2)$-estimate of $\chi^\sharp\rho_0^{(\gamma-1)(\frac{3}{2}-\varepsilon_0)}D_{\bar\eta}^2U$.} First, we claim 
\begin{equation}\label{L2Linfty}
\rho_0^\frac{\delta-\gamma+1}{2}\Big(D_{\bar\eta} U+\frac{2a_2}{2a_1+a_2}\frac{U}{\bar\eta}\Big)\in L^\infty([0,T];L^\infty(a,1)) \qquad\text{for $a\in(0,1)$}.
\end{equation}
Indeed, by \eqref{equ33.38}--\eqref{div-Uxx} and $\eqref{equ33.41}_2$, we can integrate $\eqref{lp}_1\times \bar\eta^{-2}$ over $[r,1]$ ($r\in [a,1)$) to obtain
\begin{equation}\label{old3...136} 
\begin{aligned}
\bar\varrho^{\delta}\Big(D_{\bar\eta} U+\frac{2a_2}{2a_1+a_2}\frac{U}{\bar\eta}\Big) &=\frac{A}{2a_1+a_2}\bar\varrho^{\gamma}-\frac{4\pi G}{2a_1+a_2}\int_r^1 \frac{\hat r^2\rho_0}{\bar\eta^4}\int_0^{\hat r}s^2\rho_0\,\mathrm{d}s\mathrm{d}\hat r\\
&\quad +\frac{4a_1}{2a_1+a_2}\int_r^1 \bar\eta_r\bar \varrho^{\delta} \Big(\frac{D_{\bar\eta} U}{\bar\eta}-\frac{U}{\bar\eta^2}\Big)\mathrm{d}\hat r-\frac{1}{2a_1+a_2}\int_r^1\frac{\hat r^2}{\bar\eta^2}\rho_0 U_t\,\mathrm{d}\hat r.
\end{aligned}
\end{equation}
Then it follows from the fact that $1-r\sim \rho_0^{\gamma-1}\in H^3(a,1)$, \eqref{jibenjiashe}, \eqref{33153}--\eqref{equ33.38}, Lemma \ref{hardy-inequality}, and the H\"older inequality that, for all $r\in [a,1)$ and $a\in (0,1)$,
\begin{equation}\label{3...136}
\begin{aligned}
\rho_0^{\delta}\Big|D_{\bar\eta} U+\frac{2a_2}{2a_1+a_2}\frac{U}{\bar\eta}\Big|&\leq C(a)\Big(\rho_0^{\gamma}\big(1+|\overline{D_\eta\Phi}|_\infty\big)
+\int_r^1\big(\rho_0^{\delta} |(U,D_{\bar\eta} U)|+\rho_0|U_t|\big)\mathrm{d}\hat r\Big)\\
&\leq C(a)\Big(\rho_0^{\gamma} +\Big(\int_r^1 \rho_0^{\delta}\,\mathrm{d}\hat r\Big)^\frac{1}{2}\big|\chi_a^\sharp\big(\rho_0^\frac{\delta}{2}|(U,D_{\bar\eta} U)|,\rho_0^{1-\frac{\delta}{2}}U_t\big)\big|_2\Big)\\
&\leq C(a,T,M_T)\big(\rho_0^{\gamma}+\rho_0^\frac{\gamma+\delta-1}{2}\big), 
\end{aligned}
\end{equation}
where 
\begin{equation*}
\overline{D_\eta\Phi}:=\frac{4\pi G}{\bar\eta^2}\int_0^r\hat r^2\rho_0\,\mathrm{d}\hat{r}, \qquad |\overline{D_\eta\Phi}|_\infty\leq C.
\end{equation*}
Multiplying \eqref{3...136} by $\rho_0^{-\frac{\gamma+\delta-1}{2}}$ implies claim \eqref{L2Linfty}. 

Next, it follows from \eqref{new-lp} that
\begin{equation}\label{DDv}
D_{\bar\eta}^2 U =L_{10},    
\end{equation}
where $\chi^\sharp L_{10}$ has control of the following form due to \eqref{jibenjiashe}:
\begin{equation}\label{l10}
\begin{aligned}
|\chi^\sharp L_{10}|&\leq C\chi^\sharp\frac{|(\rho_0^{\gamma-1})_r|}{\rho_0^{\gamma-1}}|U|+C\chi^\sharp\Big(\Big|\Big(\log \frac{\bar\eta^2\bar\eta_r}{r^2}\Big)_r\Big|+\frac{|(\rho_0^{\gamma-1})_r|}{\rho_0^{\gamma-1}}\Big)\Big|\Big(D_{\bar\eta} U+\frac{2a_2}{2a_1+a_2}\frac{U}{\bar\eta}\Big)\Big|\\
&\quad +C\chi^\sharp\rho_0^{1-\delta} |U_t|+C\chi^\sharp\rho_0^{1-\delta}\big|(1,D_{\bar\eta}(\bar\varrho^{\gamma-1}))\big|.
\end{aligned}
\end{equation}

Recall from $\bar\eta_t=\bar U$ and $\bar\eta(0,r)=r$ that
\begin{equation}\label{jacobi}
|\chi^\sharp D_{\bar\eta}(\bar\varrho^{\gamma-1})|_\infty+\Big|\chi^\sharp\Big(\log \frac{\bar\eta^2\bar\eta_r}{r^2}\Big)_r\Big|_\infty\leq C+C\int_0^t \big|\chi^\sharp (\bar U, D_{\bar\eta} \bar U,D_{\bar\eta}^2 \bar U)\big|_\infty\,\mathrm{d}s\leq C(T,M_T),  
\end{equation}
where we have used Lemma \ref{hardy-inequality} and 
\begin{equation*}
|\chi^\sharp  D_{\bar\eta}^2 \bar U|_\infty\leq C\big|\chi^\sharp (D_{\bar\eta}^2 \bar U,D_{\bar\eta}^3 \bar U)\big|_1\leq  C\big|\chi^\sharp  \rho_0^{(\gamma-1)(\frac{3}{2}-\varepsilon_0)}(D_{\bar\eta}^3 \bar U,D_{\bar\eta}^4 \bar U)\big|_2\leq C(M_T)+\frac{C(M_T)}{\sqrt{t}}.
\end{equation*}

Hence, it follows from the facts that $\rho_0^{\gamma-1}\sim 1-r$ and $(\gamma-1)(\frac{3}{2}-\varepsilon_0)>\frac{1}{2}>\frac{\delta}{2}$ and \eqref{33153}--\eqref{equ33.38}, \eqref{L2Linfty}, and \eqref{DDv}--\eqref{jacobi} that
\begin{equation}\label{DDv-L2}
\!\!\! \begin{aligned}
&\big|\chi^\sharp\rho_0^{(\gamma-1)(\frac{3}{2}-\varepsilon_0)}D_{\bar\eta}^2 U\big|_2\leq C(T,M_T)\Big(1+\Big|\chi^\sharp\rho_0^{(\gamma-1)(\frac{1}{2}-\varepsilon_0)}\Big(D_{\bar\eta} U+\frac{2a_2}{2a_1+a_2}\frac{ U}{\bar\eta}\Big)\Big|_2\Big)\\
&\leq C(T,M_T)\Big(1+\big|\chi^\sharp\rho_0^{(\gamma-1)(1-\varepsilon_0)-\frac{\delta}{2}}\big|_2\Big|\chi^\sharp\rho_0^\frac{\delta-(\gamma-1)}{2}\Big(D_{\bar\eta} U\!+\frac{2a_2}{2a_1\!+\!a_2}\frac{ U}{\bar\eta}\Big)\Big|_\infty\Big)\leq C(T,M_T).
\end{aligned}
\end{equation}

\textbf{Step 5. Tangential estimates $U_t\in L^\infty([0,T];\cH^1_{r^2\rho_0^{\delta}})$ and time-weighted estimates.} We continue to improve the regularity of $U$. First, by direct calculation, $w^{(1)}(0,r)\in \cH^1_{r^2\rho_0^{\delta}}$ and $(g^{(1)},h^{(1)})\in L^\infty([0,T];L^2)$ due to \eqref{given}--\eqref{jibenjiashe}, and \eqref{3..142}, it still remains to show that
\begin{equation}\label{q1t-q2t}
(g^{(1)}_t,h^{(1)}_t)\in L^2([0,T];L^2).
\end{equation}
To obtain this, we see that $g^{(1)}_t$ and $h^{(1)}_t$ have controls of the following form due to \eqref{jibenjiashe}:
\begin{equation}\label{q1-q2-1-t}
\begin{aligned}
\big|(g^{(1)}_t,h^{(1)}_t)\big|&\leq Cr\rho_0^\frac{\delta}{2}\big(|\sD_{\bar\eta}\bar U|^2 |\sD_{\bar\eta} U|+|\sD_{\bar\eta}\bar U_t||\sD_{\bar\eta} U| +|\sD_{\bar\eta}\bar U||\sD_{\bar\eta}  U_t|\big)\\
&\quad +C r\rho_0^{1-\frac{\delta}{2}}\big(|\sD_{\bar\eta}\bar U|^2+|\sD_{\bar\eta}\bar U_t| 
\big).
\end{aligned}
\end{equation}
Here, with the help of \eqref{L2Linfty}, we can obtain \eqref{q1t-q2t} by using the regularities of $(\bar\eta,\bar U)$, the facts that $\rho_0^{\gamma-1}\sim 1-r$, \eqref{33153}--\eqref{equ33.38}, \eqref{DDv-L2}, and Lemmas \ref{Lemma-point} and \ref{sobolev-embedding}--\ref{hardy-inequality}, for example,
\begin{align*}
&\,\big|r\rho_0^\frac{\delta}{2} \sD_{\bar\eta}\bar U_t\sD_{\bar\eta}U \big|_2 \leq C\big(\big|\chi r \sD_{\bar\eta}\bar U_t\sD_{\bar\eta}U\big|_2+\big|\chi^\sharp \rho_0^\frac{\delta}{2} \sD_{\bar\eta}\bar U_t\sD_{\bar\eta}U \big|_2\big)\\
&\leq C \big|\chi r^\frac{1}{2}\sD_{\bar\eta}\bar U_t\big|_4\big|\chi r^\frac{1}{2} \sD_{\bar\eta}U|_4\\
&\quad +C\Big(\big|\chi^\sharp\rho_0^\frac{\gamma-1}{2}\sD_{\bar\eta}\bar U_t\big|_2\Big|\chi^\sharp\rho_0^\frac{\delta-\gamma+1}{2}\Big(D_{\bar\eta}U+\frac{2a_2}{2a_1+a_2}\frac{U}{\bar\eta}\Big)\Big|_\infty+\big|\chi^\sharp\rho_0^{\frac{\delta}{2}}\sD_{\bar\eta}\bar U_t\big|_2 |\chi^\sharp U|_\infty\Big)\\
&\leq C(T,M_T)\big(\big|\chi r^\frac{5}{4}(\sD_{\bar\eta}U,\sD_{\bar\eta}^2U)\big|_2\big|\chi  r^\frac{5}{4}(\sD_{\bar\eta}\bar U_t,\sD_{\bar\eta}^2\bar U_t)\big|_2+ \big|\chi^\sharp \rho_0^{\frac{3}{2}(\gamma-1)}(\bar U_t,D_{\bar\eta}\bar U_t,D_{\bar\eta}^2\bar U_t)\big|_2\big)\\
&\quad +C(T,M_T)\big|\chi^\sharp \rho_0^{(\gamma-1)(\frac{3}{2}-\varepsilon_0)}( U,D_{\bar\eta} U,D_{\bar\eta}^2 U)\big|_2\leq C(T,M_T)(1+\bar\cD(t,\bar U)^\frac{1}{2}),
\end{align*}
which implies that $r\rho_0^\frac{\delta}{2}\sD_{\bar\eta}U \sD_{\bar\eta}\bar U_t\in L^2([0,T];L^2)$ due to $\bar\cD(t,\bar U)\in L^1(0,T)$. The remaining calculation is straightforward.

Therefore, from Proposition \ref{prop-strong}, it follows that the weak solution $U_t$ of \eqref{33152} satisfies \eqref{weak-F-H}--\eqref{Energy-Id} with $(w,g,h)$ replaced by $(U_t,g^{(1)},h^{(1)})$ and
\begin{equation}\label{linear-D3-qiexiang}
\begin{aligned}
&U_t\in L^\infty([0,T];\cH^1_{r^2\rho_0^{\delta}}),\qquad U_{tt}\in L^2([0,T];L^2_{r^2\rho_0}),\\
&\sqrt{t}U_{tt}\in L^\infty([0,T];L^2_{r^2\rho_0})\cap L^2([0,T];\cH^{1}_{r^2\rho_0^{\delta}}),\quad \sqrt{t}(r^2\rho_0 U_{ttt})\in L^2([0,T];\cH^{-1}_{r^2\rho_0^{\delta}}).
\end{aligned}
\end{equation}

% Moreover, we can deduce from \eqref{weak-F-H} and a similar argument in Lemma \ref{Lemma-point} that \eqref{33151} holds for \textit{a.e.} $(t,r)\in (0,T)\times (0,1)$, and for any $a\in (0,1)$ and \textit{a.e.} $t\in (0,T)$,
% \begin{equation}\label{div-Uxx-H}
% \begin{aligned}
% &r\sD_{\bar\eta}^2U_t \in L^2(0,a),\qquad \Big(\bar\eta^2 \bar\varrho^{\delta}\big((2a_1+a_2)D_{\bar\eta} U_t + 2a_2 \frac{U_t}{\bar\eta}\big)\Big)_r\in L^2(a,1),\\
% &U_t\in H^2_{\mathrm{loc}},\qquad U\in H^4_{\mathrm{loc}}, \qquad \rho_0^{\delta}\Big(D_{\bar\eta} U_t + \frac{2a_2}{2a_1+a_2} \frac{U_t}{\bar\eta}\Big)\Big|_{r=1}=0.
% \end{aligned}
% \end{equation}

Moreover, by \eqref{q1-q2-1}, \eqref{equ1}, and \eqref{linear-D3-qiexiang}, together with an argument
similar to that in Lemma \ref{Lemma-point}, as detailed below, equation \eqref{33151} holds for a.e.\
$(t,r)\in(0,T)\times(0,1)$, and, for any $a\in(0,1)$ and a.e.\
$t\in(0,T)$,
\begin{equation}\label{div-Uxx-H}
\begin{aligned}
&r\sD_{\bar\eta}^{\,2}U_t\in L^2(0,a),
\qquad
\Big(
\bar\eta^2\bar\varrho^\delta
\big((2a_1+a_2)D_{\bar\eta}U_t
+2a_2\frac{U_t}{\bar\eta}\big)
-\frac{r\rho_0^{\frac{\delta}{2}}}{\bar\eta_r}h^{(1)}
\Big)_r\in L^2(a,1),
\\[3pt]
&U_t\in H^2_{\mathrm{loc}},\qquad
U\in H^4_{\mathrm{loc}},\qquad
\rho_0^\delta
\Big(D_{\bar\eta}U_t+
\frac{2a_2}{2a_1+a_2}\frac{U_t}{\bar\eta}\Big)
\Big|_{r=1}=0.
\end{aligned}
\end{equation}
Indeed, the weak formulation \eqref{equ1} with $j=1$ and $Z^{(1)}=U_t$
gives, in the sense of distributions,
\begin{equation*}
\begin{aligned}
&\Big(
\bar\eta^2\bar\varrho^\delta
\big((2a_1+a_2)D_{\bar\eta}U_t
+2a_2\frac{U_t}{\bar\eta}\big)
-\frac{(r^2\rho_0^{\delta})^\frac{1}{2}}{\bar\eta_r}h^{(1)}
\Big)_r
\\
&\quad=r^2\rho_0U_{tt}
-\frac{2r\rho_0^{\delta/2}}{\bar\eta}g^{(1)}
+\bar\eta\bar\eta_r\bar\varrho^\delta
\big(2a_2D_{\bar\eta}U_t
+4(a_1+a_2)\frac{U_t}{\bar\eta}\big).
\end{aligned}
\end{equation*}
The right-hand side belongs to $L^2(a,1)$ for a.e.\ $t\in(0,T)$, by \eqref{jibenjiashe}, \eqref{linear-D3-qiexiang},
and the preceding estimates for $(g^{(1)},h^{(1)})$.
The expression in parentheses also belongs to $L^2(a,1)$, and hence
to $H^1(a,1)$. Integrating the above by parts and comparing with \eqref{equ1},
for any $\varphi\in C_c^\infty((a,1])$, we obtain
\begin{equation*}
\Big(
\bar\eta^2\bar\varrho^\delta
\big((2a_1+a_2)D_{\bar\eta}U_t
+2a_2\frac{U_t}{\bar\eta}\big)
-\frac{(r^2\rho_0^{\delta})^\frac{1}{2}}{\bar\eta_r}h^{(1)}
\Big)\Big|_{r=1}=0.
\end{equation*}
By \eqref{equ33.38}, for any $r\in(r_0,1)$ with $r_0\in(a,1)$, the H\"older inequality gives
\begin{equation*}
\begin{aligned}
|U(r)|&\le |U(r_0)|+\int_{r_0}^r|U_r(\tilde r)|\,\mathrm d \tilde r\\
&\le |U(r_0)|+\Big(\int_{r_0}^r(1-\tilde r)^{\frac{\delta}{\gamma-1}}|U_r(\tilde r)|^2\,\mathrm d\tilde r\Big)^{\frac12}\Big(\int_{r_0}^r(1-\tilde r)^{-\frac{\delta}{\gamma-1}}\,\mathrm d\tilde r\Big)^{\frac12},
\end{aligned}
\end{equation*}
where for each fixed time, $|U(r_0)|$ is finite and the first integral factor is bounded independently of $r$. Multiplying the preceding estimate by $\rho_0^\delta$ and evaluating the integral on its right-hand side, we obtain $\rho_0^\delta U\to0$ as $r\to1^-$. Combining this with \eqref{equ33.41} and the geometric bounds \eqref{jibenjiashe} gives $\rho_0^\delta D_{\bar\eta}U\to0$ as $r\to1^-$. Thus, \eqref{q1-q2-1} and the background regularity imply
\begin{equation*}
\Big|\frac{(r^2\rho_0^{\delta})^\frac{1}{2}}{\bar\eta_r}h^{(1)}\Big|
\le C(t,a,M_T)\rho_0^\delta
\big(|D_{\bar\eta}U|+\big|\frac U{\bar\eta}\big|\big)
+C(t,a,M_T)\rho_0^\gamma\longrightarrow0
\qquad\text{as }r\to1^-.
\end{equation*}
This proves the boundary condition in \eqref{div-Uxx-H}.

On compact subintervals of $(0,1)$, \eqref{q1-q2-1}, the background
regularity, and $U\in H^3_{\mathrm{loc}}$ from Lemma \ref{Lemma-point} give
$h^{(1)}\in H^1_{\mathrm{loc}}$.
Consequently, \eqref{33151} and \eqref{linear-D3-qiexiang} yield $U_t\in H^2_{\mathrm{loc}}$ by local elliptic regularity, then we obtain $U\in H^4_{\mathrm{loc}}$. In particular, \eqref{33151} holds almost everywhere in the interior. Finally, by an argument similar to the proof of Lemma \ref{Lemma-point}, using \eqref{div-Uxx}, \eqref{linear-D3-qiexiang}, the background regularity, and the
second-order estimate in Lemma \ref{im-1}, we obtain $r\sD_{\bar\eta}^2U_t\in L^2(0,a)$. This completes the justification of \eqref{div-Uxx-H}.

In conclusion, collecting \eqref{3..142}, \eqref{33153}--\eqref{equ33.38}, and \eqref{linear-D3-qiexiang}, we arrive at all the tangential estimates for $U$:
\begin{equation}\label{TETE}
\begin{aligned}
&U\in C([0,T];\cH^1_{r^2\rho_0^{\delta}}),\qquad U_t\in C([0,T];L^2_{r^2\rho_0})\cap L^\infty([0,T];\cH^1_{r^2\rho_0^{\delta}}),\\
&U_{tt}\in L^2([0,T];L^2_{r^2\rho_0}),\qquad
\sqrt{t}U_{tt}\in L^\infty([0,T];L^2_{r^2\rho_0})\cap L^2([0,T];\cH^{1}_{r^2\rho_0^{\delta}}),\\
&r^2\rho_0  U_{tt}\in L^2([0,T];\cH^{-1}_{r^2\rho_0^{\delta}}),\qquad \sqrt{t}(r^2\rho_0 U_{ttt})\in L^2([0,T];\cH^{-1}_{r^2\rho_0^{\delta}}).
\end{aligned}
\end{equation}

\smallskip
\textbf{Step 6. Total energy and dissipation estimates for $U$.} With the help of \eqref{TETE}, we now can  show that
\begin{equation}\label{bbbb}
\begin{aligned}
\bar\cE_{\mathrm{in}}(t,U)+t\bar\cD_{\mathrm{in}}(t,U)\in L^\infty(0,T),\qquad \bar\cD_{\mathrm{in}}(t,U)\in L^1(0,T),\\
\bar\cE_{\mathrm{ex}}(t,U)+t\bar\cD_{\mathrm{ex}}(t,U)\in L^\infty(0,T),\qquad \bar\cD_{\mathrm{ex}}(t,U)\in L^1(0,T).
\end{aligned}
\end{equation}
In fact, the regularity properties $U\in H^4_{\mathrm{loc}}$ and $U_t\in H^2_{\mathrm{loc}}$ obtained from \eqref{div-Uxx-H} justify applying $\sD_{\bar\eta}$, $\sD_{\bar\eta}^2$, or $\partial_t$ directly to \eqref{new-lp}. One can then rigorously bootstrap this argument to gain higher spatial regularity for $U$. We omit the somewhat tedious calculation here, and refer the reader to Lemmas \ref{c_1-c_2}--\ref{c_3} in \S \ref{subsub-11.2.2}, where the method carries over in a similar manner.

For simplicity, we only emphasize the major difference here  compared with the calculations in Lemmas \ref{c_1-c_2}--\ref{c_3}, namely, show how to establish
\begin{equation}\label{eequ3.66}
\chi^\sharp\rho_0^{(\gamma-1)(\frac{3}{2}-\varepsilon_0)}D_{\bar\eta}^4 U \in L^2([0,T];L^2).
\end{equation}
Indeed, let $\eqref{bbbb}_1$ hold and $U$ satisfy the following exterior estimates:
\begin{equation}\label{lowlow}
\bar\cE_{\mathrm{ex}}(t,U)\in L^\infty(0,T),\qquad (\rho_0^\frac{1}{2}U_{tt},\rho_0^{(\gamma-1)(\frac{3}{2}-\varepsilon_0)}D_{\bar\eta}^2U_t)\in L^2([0,T];L^2).
\end{equation}
First, we can apply $\bar\varrho^{-(\gamma-1)}\bar\eta_rD_{\bar\eta}^2$ to $\eqref{new-lp}\times \bar\varrho^{\gamma-1}$ and obtain
\begin{equation}\label{xxxx-l}
\bar\cT_{\mathrm{cross}}:=(D_{\bar\eta}^3U)_r+\big(\frac{\delta}{\gamma-1}+2\big)\frac{(\rho_0^{\gamma-1})_r}{\rho_0^{\gamma-1}}D_{\bar\eta}^3U =\sum_{i=15}^{18}J_{i},
\end{equation}
where $J_i$ ($i=15,16,17,18$) is defined in \eqref{rr1-rr3} of \S\ref{subsub-11.2.2}. As can be checked, $\chi^\sharp\rho_0^{(\gamma-1)(\frac{3}{2}-\varepsilon_0)}J_i\in L^2([0,T];L^2)$ for $i=15,16,17,18$ (see \eqref{xxxx-l'}--\eqref{QQQ} in the proof of Lemma \ref{c_3} for details), and hence $\chi^\sharp\rho_0^{(\gamma-1)(\frac{3}{2}-\varepsilon_0)}\bar\cT_{\mathrm{cross}} \in L^2([0,T];L^2)$. Consequently, once Lemma \ref{prop2.1} in Appendix \ref{subsection2.2} can be utilized, we obtain the desired conclusion \eqref{eequ3.66} immediately. 

However, Lemma \ref{prop2.1} can not be applied directly, since we need an \textit{a priori} assumption:
\begin{equation*}
\chi^\sharp \rho_0^{\frac{3(\gamma-1)+\delta}{2}} (D_{\bar\eta}^3 U)_r \in L^2([0,T];L^2).  
\end{equation*}
By \eqref{xxxx-l} and the $L^2([0,T];L^2)$-estimates for $\chi^\sharp \rho_0^{(\gamma-1)(\frac{3}{2}-\varepsilon_0)}J_i$ ($15\leq i\leq 18$), it suffices to prove
\begin{equation}\label{refine-3}
\chi^\sharp \rho_0^{\frac{\gamma+\delta-1}{2}} D_{\bar\eta}^3 U \in L^2([0,T];L^2). 
\end{equation}

To achieve this, we claim
\begin{equation}\label{refine-Ux}
\chi^\sharp\rho_0^{\frac{\delta}{2}-(\gamma-1)}\Big(D_{\bar\eta}U+\frac{2a_2}{2a_1+a_2}\frac{U}{\bar\eta}\Big)\in L^\infty([0,T];L^2).
\end{equation}
Indeed, using an argument similar to \eqref{L2Linfty}--\eqref{3...136}, we can obtain from Lemma \ref{hardy-inequality} that 
\begin{equation}\label{cal-2}
\begin{aligned}
\Big|\chi^\sharp\rho_0^{\frac{\delta}{2}-(\gamma-1)}\Big(D_{\bar\eta}U+&\frac{2a_2}{2a_1+a_2}\frac{U}{\bar\eta}\Big)\Big|_2\leq C\big(1+\big|\chi^\sharp\big(\rho_0^{\frac{\delta}{2}-\varepsilon_0}|(U,D_{\bar\eta} U)|,\rho_0^{1-\frac{\delta}{2}-\varepsilon_0}U_t\big)\big|_2\big)\\
&\quad\qquad\leq C\big(1+\big|\chi^\sharp\rho_0^{\frac{\delta}{2}-\varepsilon_0+2(\gamma-1)}(U,D_{\bar\eta} U,D_{\bar\eta}^2 U,D_{\bar\eta}^3 U)\big|_2\big)\\
&\quad\qquad\quad+C\big|\chi^\sharp\rho_0^{\frac{1}{2}-\varepsilon_0+\gamma-1}(U_t,D_{\bar\eta} U_t)\big|_2 \leq C(1+\bar\cE_{\mathrm{ex}}(t,U)^\frac{1}{2}),
\end{aligned}
\end{equation}
which, along with \eqref{lowlow}, leads to \eqref{refine-Ux}.

We continue to show \eqref{refine-3}. Multiplying \eqref{new-lp} by $\chi^\sharp\rho_0^{\frac{\delta}{2}}$ and applying the $L^2$-norm to the resulting equality, we can deduce from $\rho_0^{\gamma-1}\sim 1-r$, \eqref{jacobi}, \eqref{lowlow}, \eqref{refine-Ux}, and the regularity of $\bar\eta$ in \eqref{given-bareta} that
\begin{equation*} 
\begin{aligned}
\big|\chi^\sharp\rho_0^{\frac{\delta}{2}}D_{\bar\eta}^2 U\big|_2&\leq C\big|\chi^\sharp\rho_0^{\frac{\delta}{2}}(\rho_0^{1-\delta}U_t,U,D_{\bar\eta}U)\big|_2+C|\bar\varrho^{1-\frac{\delta}{2}}|_\infty\Big|\chi^\sharp \frac{1}{\bar\eta^2}\int_0^r\hat r^2\rho_0\,\mathrm{d}\hat{r}\Big|_\infty\\
&\quad+C\Big(\Big|\chi^\sharp\rho_0^{\frac{\delta}{2}-(\gamma-1)} \Big(D_{\bar\eta}U+\frac{2a_2}{2a_1+a_2}\frac{U}{\bar\eta}\Big)\Big|_2+ |\bar\varrho|_\infty^{1-\delta}\Big) |\chi^\sharp D_{\bar\eta}(\bar\varrho^{\gamma-1})|_\infty\leq C(T,M_T),
\end{aligned}
\end{equation*}
that is, $\chi^\sharp\rho_0^{\frac{\delta}{2}}D_{\bar\eta}^2 U\in L^\infty([0,T];L^2)$.  Similarly, based on this, \eqref{L2Linfty}, and \eqref{refine-Ux}, applying $\chi^\sharp\rho_0^{\frac{\delta}{2}+\gamma-1}D_{\bar\eta}$ to \eqref{new-lp}, we can further obtain
\begin{equation}\label{refine-2}
\chi^\sharp\rho_0^{\frac{\delta}{2}+\gamma-1} D_{\bar\eta}^3 U\in L^\infty([0,T];L^2),
\end{equation}
where the following inequality derived in the same way as \eqref{jacobi} is also used:
\begin{equation}\label{jacobiD2L}
\big|\chi^\sharp D_{\bar\eta}^2(\bar\varrho^{\gamma-1})\big|_\infty\leq C(T,M_T).
\end{equation}
Next, multiplying \eqref{new-lp} by $\bar\varrho^{\gamma-1}$ and then applying $\bar\varrho^{-(\gamma-1)}\bar\eta_rD_{\bar\eta}$ to the resulting equality, we arrive at the following type of equation:
\begin{equation*} 
\widetilde{\cT}_{\mathrm{cross}}:=(D_{\bar\eta}^2 U)_r+ \big(\frac{\delta}{\gamma-1}+1\big)\frac{(\rho_0^{\gamma-1})_r}{\rho_0^{\gamma-1}} D^2_{\bar\eta} U=L_{11},
\end{equation*}
where $L_{11}$ has control of the following form in view of \eqref{jibenjiashe}:
\begin{equation*}
\begin{aligned}
|\chi^\sharp L_{11}|&\leq C\chi^\sharp\Big(|\bar\varrho^{1-\delta}| \Big(|D_{\bar\eta}U_t|+\frac{(\bar\varrho^{\gamma-1})_r}{\bar\varrho^{\gamma-1}}U_t\Big)+\Big|\big(1,\frac{(\bar\varrho^{\gamma-1})_r}{\bar\varrho^{\gamma-1}}\big)\Big| \big|(U,D_{\bar\eta}U)\big|+|D_{\bar\eta}^2 U|\Big)\\
&\quad +C\chi^\sharp\Big|\frac{D_{\bar\eta}^2(\bar\varrho^{\gamma-1})}{\bar\varrho^{\gamma-1}} \Big|\Big|D_{\bar\eta}U+\frac{2a_2}{2a_1+a_2}\frac{U}{\bar\eta}\Big|+C\big|\big(\frac{r^2}{\bar\eta^2\bar\eta_r}\big)_r\big||D_{\bar\eta}^2U|\\
&\quad+C\chi^\sharp|\bar\varrho^{1-\delta}|\big(|D_{\bar\eta}^2\bar\varrho^{\gamma-1}|+\frac{1}{\bar\varrho^{\gamma-1}}|((D_{\bar\eta}\bar\varrho^{\gamma-1})^2,1)|+1\big).
\end{aligned}
\end{equation*}
After a direct calculation by using $\rho_0^{\gamma-1}\sim 1-r$, \eqref{L2Linfty}, \eqref{jacobi}, \eqref{lowlow}, \eqref{jacobiD2L}, Lemma \ref{hardy-inequality}, and the regularity of $\bar\eta$ in \eqref{given-bareta}, we can check that
\begin{equation*}
\chi^\sharp\rho_0^\frac{\delta+\gamma-1}{2}L_{11}\in L^2([0,T];L^2)\implies \zeta^\sharp\rho_0^\frac{\delta+\gamma-1}{2}\widetilde{\cT}_{\mathrm{cross}}\in L^2([0,T];L^2),
\end{equation*}
which, along with \eqref{refine-2} and the fact that $ \chi^\sharp\rho_0^{\frac{\delta+\gamma-1}{2}}D_{\bar\eta}^2 U \in L^\infty([0,T];L^2)$ and Lemma \ref{prop2.1}, leads to $\zeta^\sharp\rho_0^\frac{\delta+\gamma-1}{2}D_{\bar\eta}^3U\in L^2([0,T];L^2)$, so that
\begin{equation*}
\begin{aligned}
\big|\chi^\sharp\rho_0^\frac{\delta+\gamma-1}{2}D_{\bar\eta}^3U\big|_2&\leq \big|(\zeta-\chi)\rho_0^\frac{\delta+\gamma-1}{2}D_{\bar\eta}^3U\big|_2+\big|\zeta^\sharp\rho_0^\frac{\delta+\gamma-1}{2}D_{\bar\eta}^3U\big|_2\\
&\leq \bar\cE_{\mathrm{in}}(t,U)^\frac{1}{2}+\bar\cD_{\mathrm{in}}(t,U)^\frac{1}{2}+\big|\zeta^\sharp\rho_0^\frac{\delta+\gamma-1}{2}D_{\bar\eta}^3U\big|_2.
\end{aligned}
\end{equation*}
This, together with $\eqref{bbbb}_1$, yields the desired estimate \eqref{refine-3}, and thus yields $\eqref{eequ3.66}$.

\smallskip
\textbf{Step 7. Regularity of $U$ given in \eqref{regu-class}.} Now, we can show that
\begin{equation}\label{classical-1}
(U,\sD_r U)\in C([0,T];C(\bar I)),\qquad  (\sD_r^2 U,U_t)\in C((0,T];C(\bar I)).
\end{equation}
By \eqref{given-flow}--\eqref{given}, \eqref{jibenjiashe}, and \eqref{bbbb}, we obtain from Lemma \ref{lemma-gaowei} with the generating pair $(\bar\eta,\bar U)$ and $f=U$ that,
\begin{equation}\label{revise C.1  01}
\mathring{\mathcal E}(t,U)+t\mathring{\mathcal D}(t,U)
\in L^\infty(0,T),
\qquad
\mathring{\mathcal D}(t,U)\in L^1(0,T).
\end{equation}
In particular, the same time-integrability properties hold for the interior and exterior components separately.

\smallskip
\textbf{7.1. Regularity of $U$ near the origin.} In this step, we aim to show that
\begin{equation}\label{classical-in}
(U,\sD_r U)\in C([0,T];C(\bar I_\flat)),\qquad  \sD_r^2 U\in C((0,T];C(\bar I_\flat)),
\end{equation}
where $I_\flat:=[0,\frac{1}{2})$, and the regularity of $U_t$ can be derived similarly.

First, define the 3-D representative of $U$ as $\boldsymbol{U}(t,\boldsymbol{y})=U(t,r)\frac{\boldsymbol{y}}{r}$. By the preceding fixed-coordinate estimates \eqref{revise C.1  01} and Lemma \ref{lemma-initial}, we have
\begin{equation*}
\zeta\boldsymbol{U}\in L^\infty([0,T];H^3_0(\Omega))\cap L^2([0,T];H^4_0 (\Omega)),\qquad
\zeta\boldsymbol{U}_{t} \in L^2([0,T];H^2_0 (\Omega)),
\end{equation*}
where $\zeta=\zeta(\boldsymbol{y})=\zeta(r)$ defined on $\Omega$. This implies $\zeta\boldsymbol{U}\in C([0,T];H^3_0(\Omega))$ due to Lemma \ref{triple}. Hence, using Lemma \ref{lemma-initial} again, we obtain $r\sD_r^k U \in C([0,T];L^2(I_\flat))$ for $k=0,1,2,3$, which, along with Lemmas \ref{sobolev-embedding}--\ref{hardy-inequality}, 
implies
\begin{equation*}
(U, \sD_r U, \sD_r^2 U)\in C([0,T];L^2(I_\flat))\implies (U, \sD_r U)\in C([0,T];C(\bar I_\flat)).
\end{equation*}

Next, it follows similarly from Lemma \ref{lemma-initial} and \eqref{revise C.1  01} that
\begin{equation*}
t\zeta\nabla_{\boldsymbol{y}}^2 \boldsymbol{U}\in L^\infty([0,T];H^2_0 (\Omega)),\qquad
(t\zeta\nabla_{\boldsymbol{y}}^2 \boldsymbol{U})_t\in L^2([0,T];L^2 (\Omega)),
\end{equation*}
which, along with Lemma \ref{triple}, leads to $t\zeta\nabla_{\boldsymbol{y}}^2 \boldsymbol{U}\in C([0,T];W^{1,4} (\Omega))$. This, together with Lemma \ref{lemma-initial}, implies
\begin{equation}\label{444}
r^\frac{1}{2}(\sD_r^2 U,\partial_r(\sD_r^2 U))\in C((0,T];L^4(I_\flat)).
\end{equation}
On the other hand, since, for any function $f=f(r)$, $|f|_1 \leq C|r^\frac{1}{2}f|_4$ in $I_\flat$ due to the H\"older inequality, we can derive from  \eqref{444} and Lemma \ref{sobolev-embedding} that
\begin{equation*} 
\sD_r^2 U\in C((0,T];W^{1,1}(I_\flat))\implies \sD_r^2 U\in C((0,T];C(\bar I_\flat)),
\end{equation*}
which completes the proof of \eqref{classical-in}.

\smallskip
\textbf{7.2. Regularity of $U$ away from the origin.} Define $I_\sharp:=(\frac{1}{2},1)$. To obtain \eqref{classical-1}, it still remains to show
\begin{equation}\label{classical-ex}
U\in C([0,T];C^1(\bar I_\sharp)),\qquad (U_{rr},U_t)\in C((0,T];C(\bar I_\sharp)).
\end{equation}
 
First, it follows from \eqref{TETE} that
\begin{equation}\label{tete}
(\rho_0^\frac{1}{2}U,\rho_0^\frac{\delta}{2}U_r,\rho_0^\frac{1}{2}U_t)\in C([0,T];L^2(I_\sharp)).
\end{equation}

Next, we obtain from \eqref{revise C.1  01}, $\frac{1}{2}-\varepsilon_0>-\frac{1}{2}$, and Lemma \ref{hardy-inequality} that $\rho_0^{(\gamma-1)(\frac{3}{2}-\varepsilon_0)}(U_{rr},U_{rrr})\in H^1(I_\sharp)$ for \textit{a.e.} $t\in (0,T)$. Hence, $\rho_0^{(\gamma-1)(\frac{3}{2}-\varepsilon_0)}(U_{rr},U_{rrr}) \in C(\bar I_\sharp)$ due to Lemma \ref{sobolev-embedding}, which, along with \eqref{bbbb} again, leads to
\begin{equation*}
\begin{aligned}
\zeta_\frac{1}{3}^\sharp\rho_0^\frac{(\gamma-1)(3-\varepsilon_0)}{2} (U_{rr},U_{rrr})& \in L^\infty([0,T];L^2)\cap L^2([0,T];H^1_0),\\
\zeta_\frac{1}{3}^\sharp\rho_0^\frac{(\gamma-1)(3-\varepsilon_0)}{2}(U_{trr},U_{trrr})& \in L^2([0,T];H^{-1}).
\end{aligned}
\end{equation*}
Then it follows from the above and Lemma \ref{triple} that
\begin{equation}\label{tete2}
\rho_0^\frac{(\gamma-1)(3-\varepsilon_0)}{2} (U_{rr},U_{rrr}) \in C([0,T];L^2(I_\sharp)).
\end{equation}
\eqref{tete2}, combined with \eqref{tete} and Lemma \ref{hardy-inequality}, yields $U\in C([0,T];W^{2,1}(I_\sharp))$, which, along with Lemma \ref{sobolev-embedding}, leads to $\eqref{classical-ex}_1$.

It remains to prove $\eqref{classical-ex}_2$.  
 To this end, we first obtain from \eqref{TETE} that
\begin{equation*}
\begin{aligned}
tU_{tt}&\in L^2([0,T];\cH^1_{r^2\rho_0^\delta}),\\
r^2\rho_0(tU_{tt})_t
&=r^2\rho_0U_{tt}+t r^2\rho_0U_{ttt}
\in L^2([0,T];\cH^{-1}_{r^2\rho_0^\delta}).
\end{aligned}
\end{equation*}
Hence, Lemma \ref{Aubin} with $s=1$ and $\cH=\cH^1_{r^2\rho_0^\delta}$ gives
$tU_{tt}\in C([0,T];L^2_{r^2\rho_0})$.
On the other hand, since $\rho_0\le C\rho_0^\delta$, \eqref{TETE} also yields
\begin{equation*}
tU_{tr},\quad (tU_{tr})_t=U_{tr}+tU_{ttr}
\in L^2([0,T];L^2_{r^2\rho_0}).
\end{equation*}
Thus, $tU_{tr}\in W^{1,2}(0,T;L^2_{r^2\rho_0})$, and hence
$tU_{tr}\in C([0,T];L^2_{r^2\rho_0})$.
Since $r\ge\frac{1}{2}$ on $I_\sharp$, we obtain

% This, together with \eqref{bbbb} and Lemmas \ref{sobolev-embedding} and \ref{Aubin}, gives
\begin{equation}\label{tete4}
t(U_{tr},U_{tt})\in C([0,T];L^2_{r^2\rho_0})\implies  \rho_0^\frac{1}{2}(U_{tr},U_{tt})\in C((0,T];L^2(I_\sharp)).
\end{equation}

Now, following an argument similar to the proof of $\sqrt{t}\rho_0^{(\gamma-1)(\frac{3}{2}-\varepsilon_0)}D_{\bar\eta}^2U_t \in L^\infty([0,T];L^2(I_\sharp))$ 
(this proof is omitted in Step 6 above; see Steps 2.1 and 3 in the proof of Lemma \ref{c_3} below for details),
with the help of time-continuities \eqref{tete}--\eqref{tete4} and the regularities of $(\bar U,\bar\eta)$ in \eqref{given}--\eqref{given-bareta}, we also have $\rho_0^\frac{(\gamma-1)(3-\varepsilon_0)}{2} D_{\bar\eta}^2U_t \in C((0,T];L^2(I_\sharp))$. Then from the chain rules, we obtain 
\begin{equation}\label{tete5}
\rho_0^\frac{(\gamma-1)(3-\varepsilon_0)}{2} U_{trr} \in C((0,T];L^2(I_\sharp)).
\end{equation}

Finally, recalling \eqref{xxxx-l}, we similarly define
\begin{equation*}
\bar\cT_{\mathrm{cross}}^*=\bar\cT_{\mathrm{cross}}^*(t,r):=U_{rrrr}+(\frac{\delta}{\gamma-1}+2)\frac{(\rho_0^{\gamma-1})_r}{\rho_0^{\gamma-1}}U_{rrr}.
\end{equation*}
Clearly, based on \eqref{tete2}--\eqref{tete5}, the chain rules and the regularity of $\bar\eta$ in \eqref{given-bareta}, we can follow a similar argument in Steps 2.2 and 3 of Lemma \ref{c_3} (this argument is actually contained in Step 6 above, however, we omit it)  and obtain $\rho_0^\frac{(\gamma-1)(3-\varepsilon_0)}{2} \bar\cT^*_{\mathrm{cross}}\in C((0,T];L^2(I_\sharp))$. Then, by Lemma \ref{prop2.1}, for any $0<t,t_0\leq T$,
\begin{equation*}
\begin{aligned}
&\big|\rho_0^\frac{(\gamma-1)(3-\varepsilon_0)}{2} (U_{rrrr}(t)-U_{rrrr}(t_0))\big|_2\\
&\leq C(T)\big|\rho_0^\frac{(\gamma-1)(3-\varepsilon_0)}{2} \big(\bar\cT^*_{\mathrm{cross}}(t)-\bar\cT^*_{\mathrm{cross}}(t_0),U_{rrr}(t)-U_{rrr}(t_0)\big)\big|_2.
\end{aligned}
\end{equation*}
Passing to the limit $t\to t_0$, together with \eqref{tete2}, yields
\begin{equation}\label{tete6}
\rho_0^\frac{(\gamma-1)(3-\varepsilon_0)}{2} U_{rrrr} \in C((0,T];L^2(I_\sharp)). 
\end{equation}

Thus,  it follows from \eqref{tete2}--\eqref{tete6} and Lemma \ref{hardy-inequality} that $(U_{rr},U_t)\in C((0,T];W^{1,1}(I_\sharp))$, which, along with Lemma \ref{sobolev-embedding}, yields $\eqref{classical-ex}_2$. This completes the proof of \eqref{classical-ex}, and hence the proof of \eqref{classical-1}.

\smallskip
\textbf{Step 8. Derivation of the boundary condition.} The boundary condition of $U$ in \eqref{N111} can be proved by basically following the argument used in  Remark \ref{rmk1.4} of  \S \ref{section1}. First, the boundary condition $U|_{r=0}=0$ follows directly from $\frac{U}{r}\in C([0,T];C(\bar I))$. Next, since $\eqref{lp}_1$ holds pointwise in $(0,T]\times (0,1)$, we can divide $\eqref{lp}_1$ by $\bar\eta^2\bar\eta_r$ to obtain
\begin{equation}\label{lp-pre}
\bar\varrho U_t +A D_{\bar\eta}(\bar\varrho^{\gamma})+ 4\pi G\frac{\bar\varrho}{\bar\eta^2}\int_0^r\hat r^2\rho_0\,\mathrm{d}\hat{r} =(2a_1+a_2) D_{\bar\eta}\Big(\bar\varrho^{\delta}\big(D_{\bar\eta} U+ \frac{2U}{\bar\eta}\big)\Big) - 4a_1 D_{\bar\eta}(\bar\varrho^{\delta}) \frac{U}{\bar\eta},
\end{equation}
where $\bar\varrho=\frac{r^2\rho_0}{\bar\eta^2\bar\eta_r}$. Then multiplying the above by $\bar\varrho^{\gamma-\delta-1}$ gives
\begin{equation}\label{lp-pre;}
\begin{aligned}
&\frac{\delta}{\gamma-1} D_{\bar\eta}(\bar\varrho^{\gamma-1})\Big((2a_1+a_2) D_{\bar\eta} U+2a_2\frac{U}{\bar\eta}\Big)\\
&=\bar\varrho^{\gamma-\delta}U_t+\frac{A\gamma}{\gamma-1}\bar\varrho^{\gamma-\delta}D_{\bar\eta}(\bar\varrho^{\gamma-1})+4\pi G\frac{\bar\varrho^{\gamma-\delta}}{\bar\eta^2}\int_0^r \hat r^2\rho_0\,\mathrm{d}\hat{r}\\
&\quad -(2a_1+a_2)\bar\varrho^{\gamma-1}\Big(D_{\bar\eta}^2U+2D_{\bar\eta}\big(\frac{U}{\bar\eta}\big)\Big).
\end{aligned}
\end{equation}

Due to the facts that $1-r\sim\rho_0^{\gamma-1}\in C^1(\bar I)$ and
\begin{equation}\label{the facts}
(\sD_r U,\sD_r^2 U,\sD_r \bar\eta,\sD_r^2 \bar\eta)\in C((0,T];C(\bar I)),\qquad(\bar\eta_r,\frac{\bar\eta}{r})\in \big[\frac{1}{2},\frac{3}{2}\big],
\end{equation}
we see that the right-hand side of \eqref{lp-pre;} belongs to  $C((0,T]\times \bar I)$ and vanishes at the boundary $\{r=1\}$. Consequently, taking the limit as $r\to 1$ in \eqref{lp-pre;}, we obtain from $(\rho_0^{\gamma-1})_r|_{r=1}\neq 0$ that
\begin{equation}\label{nn-lov}
D_{\bar\eta}(\bar\varrho^{\gamma-1})\Big(D_{\bar\eta} U+\frac{2a_2}{2a_1+a_2}\frac{U}{\bar\eta}\Big)\Big|_{r=1} =0 \implies \Big(D_{\bar\eta} U+\frac{2a_2}{2a_1+a_2}\frac{U}{\bar\eta}\Big)\Big|_{r=1} =0.
\end{equation}

\smallskip
\textbf{Step 9. Equation $\eqref{lp}_1$ holds pointwise in $(0, T]\times \bar I$.} By Definition \ref{fed-cl}, it remains to show that $U(t,r)$ satisfies equation $\eqref{lp}_1$ pointwise in $(0, T]\times \bar I$. 
However, to make the argument in this step applicable to the construction of classical solutions to the nonlinear problem \eqref{eq:VFBP-La-eta}, we consider a more general case here.
%However, to make the method in this step applicable to the well-posedness of the nonlinear problem \eqref{eq:VFBP-La-eta}, we consider a more general case here.
More precisely, we show that $\eqref{lp}_1\times (\bar\eta^2\bar\eta_r)^{-1}$, that is, \eqref{lp-pre} holds pointwise in $(0, T]\times \bar I$. Furthermore, based on the structure of \eqref{lp-pre}, it suffices to show that the ``most singular'' term
\begin{equation*} 
\mathbb{S}:=D_{\bar\eta}(\bar\varrho^{\delta})\underline{\Big(D_{\bar\eta} U+\frac{2a_2}{2a_1+a_2}\frac{U}{\bar\eta}\Big)}_{:=\SS_*}\qquad\text{holds pointwise in $(0,T]\times \bar I$}.
\end{equation*}

A direct calculation gives 
\begin{equation}\label{most-singular1'}
\begin{aligned}
\SS&=-\delta \Big(\frac{r^2}{\bar\eta^2\bar\eta_r}\Big)^{\delta-1}\Big(\frac{2r^3}{\bar\eta^3\bar\eta_r^2}\big(\frac{\bar\eta}{r}\big)_r +\frac{r^2\bar\eta_{rr}}{\bar\eta^2\bar\eta_r^3}\Big)\rho_0^{\delta} \SS_*+\frac{\delta}{\gamma-1} \Big(\frac{r^2}{\bar\eta^2\bar\eta_r}\Big)^\delta\frac{(\rho_0^{\gamma-1})_r}{\bar\eta_r}\rho_0^{\delta-(\gamma-1)} \SS_* .
\end{aligned}
\end{equation}
From \eqref{the facts}, it follows that the first term in \eqref{most-singular1'} belongs to  $C((0,T]\times \bar I)$, hence holding pointwise in $(0,T]\times \bar I$.

To prove that the second term in \eqref{most-singular1'} holds pointwise in $(0,T]\times \bar I$, we only need to check
\begin{equation}\label{claim-ss'}
\lim_{r\to1^-}\rho_0^{\delta-\gamma+1}(r)S_*(t,r)=0
\qquad\text{for any }t\in(0,T].
\end{equation}
Indeed, due to \eqref{nn-lov} and $(\sD_r^k U,\sD_r^k \bar\eta) \in C((0,T];C(\bar I))$ ($k=0,1,2$), we have, for any $0<\tau<T$,
\begin{equation*} 
|\SS_*|=\Big|\int_r^1 (\SS_*)_r\,\mathrm{d}\tilde{r}\Big| \leq C(\tau,T, M_T)(1-r) \qquad\text{for any $(t,r)\in [\tau,T]\times\bar I$}.
\end{equation*}
Hence, we obtain that, $|\rho_0^{\delta-\gamma+1}\SS_* |\leq C(\tau,T, M_T) \rho_0^{\delta}$ for all $(t,r)\in [\tau,T]\times I$, which gives \eqref{claim-ss'}. Since the above convergence is uniform for $t\in[\tau,T]$
with any $0<\tau<T$, it follows from \eqref{the facts} and the preceding
expansion that the combined viscous term in \eqref{lp-pre} extends
continuously to $(0,T]\times[\frac12,1]$ and vanishes at $r=1$.

This completes the proof of Lemma \ref{existence-linearize}. 

\section{Uniform Estimates for the Linearized Problem}\label{subsection9.2}

In this section,  based on  Lemma \ref{existence-linearize}, we will  establish the uniform estimates 
for the classical solution $U$ of problem \eqref{lp}.
In \S\ref{subsection9.2}--\S\ref{subsection9.3}, $C\in (1,\infty)$ denotes a generic constant depending only on $(a_1,a_2,A,\delta,\gamma,\varepsilon_0,G)$, and $C(l_1,\cdots\!,l_k)\in (1,\infty)$ a generic  constant depending on $C$
and the additional  parameters $(l_1,\cdots\!,l_k)$, which may be different at each occurrence. Moreover, for simplicity, we denote $\mathfrak{p}(\cdot)$ a generic polynomial function, taking the form: 
\begin{equation*}
\mathfrak{p}(s)=\sum_{j=1}^k s^j \qquad\text{with some $k\in \NN^*$}.
\end{equation*}

 Let $(\rho_0,u_0)$ be given initial data satisfying the hypothesis of Lemma \ref{existence-linearize}.
Assume that 
there exists a constant $c_0>1$ such that 
\begin{equation}\label{38}
1+K_1^{-1}+K_1+K_2+\|\rho_0^{\gamma-1}\|_{H^3(\Omega)}+\cE(0,U)\leq c_0.
\end{equation}
Then fix $T>0$, and assume that there exist some constants $T^*\in (0,T]$ and  $c_1$ such that $1<c_0\leq c_1$ and, for all $t\in [0,T^*]$, 
\begin{equation}\label{39}
\bar\cE(t,\bar U)+t\bar\cD(t,\bar U)+\int_0^{t}\bar\cD(s,\bar U)\,\mathrm{d}s\leq c_1,\qquad
|\bar\eta_r(t)-1|_\infty+\Big|\frac{\bar\eta(t)}{r}-1\Big|_\infty\leq \frac{1}{2}.
\end{equation}
Here, $(c_1,T^*)$ will be determined later, which depend only on $(c_0,\varepsilon_0,a_1,a_2,A,\delta,\gamma,G,T)$.

\subsection{Preliminary}\label{Remark-useful-bounds}

First, we have the following estimates associated with $\bar U$.
\begin{Lemma}\label{lemma-useful1}
For any $t\in[0,T^*]$,
\begin{equation*}
\big|(\bar U,\sD_{\bar\eta} \bar U)\big|_\infty
\leq C\mathfrak{p}(c_1),\qquad
\big|(\sD_{\bar\eta}^2 \bar U,\bar U_t)\big|_\infty
\leq C(\mathfrak{p}(c_1)+\mathfrak{p}(c_0)\bar\cD(t,\bar U)^\frac{1}{2}),
\end{equation*}
and, for any $a\in (0,1)$,
\begin{equation*}
\begin{aligned}
&\big|\chi_a r^\frac{1}{2}\sD_{\bar\eta} \bar U\big|_{4}\leq  C(a)\mathfrak{p}(c_1),\qquad
\big|\chi_a r^\frac{1}{2}\sD_{\bar\eta} \bar U_t\big|_{4} \leq  C(a)(\mathfrak{p}(c_1)+\bar\cD(t,\bar U)^\frac{1}{2}),\\
&\big|\chi^\sharp_a   \rho_0^{\gamma-1} (D_{\bar\eta}^2\bar U,\bar U_t)\big|_{\infty} \leq  C(a) \mathfrak{p}(c_1),\qquad
\big|\chi_a^\sharp \rho_0^{\gamma-1} D_{\bar\eta}\bar U_{t}\big|_{\infty} \leq  C(a)(\mathfrak{p}(c_1)+\mathfrak{p}(c_0)\bar\cD(t,\bar U)^\frac{1}{2}).
\end{aligned}
\end{equation*}
\end{Lemma}
\begin{proof}
The proof can be directly derived by \eqref{39} and Lemmas \ref{sobolev-embedding}--\ref{hardy-inequality}. For example, it follows that, for all $t\in [0,T^*]$ and $a\in (0,1)$,
\begin{align*}
|\sD_{\bar\eta}^2 \bar U|_\infty&\leq C\big|(\sD_{\bar\eta}^2 \bar U,\sD_{\bar\eta}^3 \bar U)\big|_1 \leq  C \sum_{j=2}^4\Big(\big|\chi r \sD_{\bar\eta}^j\bar U\big|_{2} +\big|\chi^\sharp\rho_0^{(\gamma-1)(\frac{3}{2}-\varepsilon_0)}D_{\bar\eta}^j\bar U\big|_{2}\Big) \\
&\leq C(\mathfrak{p}(c_1)+\mathfrak{p}(c_0)\bar\cD(t,\bar U)^\frac{1}{2}),\\
\big|\chi_a r^\frac{1}{2}\sD_{\bar\eta}\bar U_{t}\big|_{4} &\leq C(a) \big|\chi_a r^\frac{5}{4}(\sD_{\bar\eta}\bar U_{t},\sD_{\bar\eta}^2\bar U_{t})\big|_{2} \leq C(a)\big(\mathfrak{p}(c_1)+\bar\cD(t,\bar U)^\frac{1}{2}\big),\\
\big|\chi^\sharp_a\rho_0^\frac{\gamma-1}{2} D_{\bar\eta}\bar U_{t}\big|_{2} &\leq  C(a) \big|\chi^\sharp \rho_0^\frac{3(\gamma-1)}{2}(D_{\bar\eta}\bar U_{t},D_{\bar\eta}^2\bar U_{t})\big|_{2} \leq C(a)\big(\mathfrak{p}(c_1)+\mathfrak{p}(c_0)\bar\cD(t,\bar U)^\frac{1}{2}\big).
\end{align*}

The rest of this lemma can be proved analogously, we omit the details here for brevity.
\end{proof}
 
Next, to further simplify the calculations, we define the quantities: 
\begin{equation}\label{bar-Lambda}
\bar\Lambda:=D_{\bar\eta} \bar\varrho^{\gamma-1}=D_{\bar\eta} (\rho_0^{\gamma-1} \bar{\mathscr{J}}^{1-\gamma}),\qquad \bar{\mathscr{J}}:=\frac{\bar\eta^2\bar\eta_r}{r^2},\qquad \overline{D_\eta\Phi}:=\frac{4\pi G}{\bar\eta^2}\int_0^r\hat r^2\rho_0\,\mathrm{d}\hat{r}.
\end{equation}
Then we obtain some useful estimates for $(\bar\Lambda,\bar{\mathscr{J}},\overline{D_\eta\Phi})$.

\begin{Lemma}\label{lemma-Lambda-guji}
For any $0\leq t\leq T_1=\min\{T^*, \mathfrak{p}(c_1)^{-1}\}$, $a\in (0,1)$, and $\sigma>0$,
\begin{equation*}
\begin{aligned}
\big|\big(\overline{D_\eta\Phi},\sD_{\bar\eta}\overline{D_\eta\Phi},\chi_a^\sharp \rho_0^{\gamma-1} D_{\bar\eta}^2\overline{D_\eta\Phi},\zeta_a r \sD_{\bar\eta}^2\overline{D_\eta\Phi}\big)\big|_\infty+\big|\zeta_a r^\frac{1}{2}\sD_{\bar\eta}\bar\Lambda\big|_4+\big|\zeta_a r(\sD_{\bar\eta}\bar\Lambda,\sD_{\bar\eta}^2\bar\Lambda)\big|_2\leq C(a)\mathfrak{p}(c_0),\\
|\bar\Lambda|_\infty\leq C\mathfrak{p}(c_0),\qquad |\chi^\sharp_a D_{\bar\eta}\bar{\mathscr{J}}|_\infty
+|\chi^\sharp_a D_{\bar\eta}\bar\Lambda|_\infty+\big|\chi^\sharp_a \rho_0^{(\gamma-1)(\frac{1}{2}-\varepsilon_0)} D_{\bar\eta}^2\bar\Lambda\big|_2\leq C(a)\mathfrak{p}(c_0),\\
\big|\chi^\sharp_a\rho_0^{(\gamma-1)(-\frac{1}{2}+\sigma)}\bar\Lambda\big|_2\leq C(a,\sigma)\mathfrak{p}(c_0),\qquad|\zeta_a r^\frac{1}{2}\bar\Lambda_t|_4+|\chi_a^\sharp\bar\Lambda_t|_\infty\leq C(a)\mathfrak{p}(c_1).
\end{aligned}
\end{equation*}
\end{Lemma}
\begin{proof}
First, it follows from \eqref{39} that 
\begin{equation*}
\begin{aligned}
&\big|\,\overline{D_\eta\Phi}\,\big|_\infty=\Big|\frac{4\pi G}{\bar\eta^2}\int_0^r\hat r^2\rho_0\,\mathrm{d}\hat{r}\Big|_\infty
\leq C|\rho_0|_\infty\Big|\frac{1}{r^2}\int_0^r \hat r^2\,\mathrm{d}\hat r\Big|_\infty\leq C\mathfrak{p}(c_0),\\
&|\sD_{\bar\eta}\overline{D_\eta\Phi}|_\infty
\leq C\Big|\frac{1}{r^3}\int_0^r\hat r^2\rho_0\,\mathrm{d}\hat r\Big|_\infty+C\big|\frac{r^2\rho_0}{r^2} \big|_\infty\leq C\mathfrak{p}(c_0),\\
&|\chi^\sharp \rho_0^{\gamma-1} D_{\bar\eta}^2\overline{D_\eta\Phi}|_\infty
\leq C\big|\rho_0^{\gamma-1}(1+\rho_0^{1-(\gamma-1)}D_{\bar\eta}(\rho_0^{\gamma-1}))\big|_\infty\leq C\mathfrak{p}(c_0),\\
&|\zeta r \sD_{\bar\eta}^2\overline{D_\eta\Phi}|_\infty
\leq C\Big|\frac{\zeta r}{r^4}\int_0^r\hat r^2\rho_0\,\mathrm{d}\hat r\Big|_\infty+C\big|\zeta r\frac{r\rho_0+r^2(\rho_0^{\gamma-1})_r}{r^2}\big|_\infty+C\big|\zeta r\frac{\rho_0}{r}\big|_\infty\leq C\mathfrak{p}(c_0).\\
\end{aligned}
\end{equation*}

Next, by \eqref{given-flow} and \eqref{38}--\eqref{39}, we see that, for all $(t,r)\in [0,T^*]\times \bar I$,
\begin{equation}\label{jt}
\bar{\mathscr{J}}_t =\bar{\mathscr{J}}\big(D_{\bar\eta}\bar U+\frac{2\bar U}{\bar\eta}\big) ,\qquad C^{-1}\leq \bar{\mathscr{J}}(t,r)\leq C.
\end{equation}

Then we can obtain from \eqref{jt} and Lemma \ref{lemma-useful1} that
\begin{equation}\label{1189}
\begin{aligned}
|D_{\bar\eta}\bar{\mathscr{J}}|&\leq Ce^{Ct\mathfrak{p}(c_1)}\int_0^t|\sD_{\bar\eta}^2\bar U|\,\mathrm{d}s,\\
|D_{\bar\eta}^2\bar{\mathscr{J}}|&\leq Ce^{2Ct\mathfrak{p}(c_1)}\Big(\int_0^t |\sD_{\bar\eta}^2\bar U|\,\mathrm{d}s\Big)^2+Ce^{Ct\mathfrak{p}(c_1)}\int_0^t|\sD_{\bar\eta}^3\bar U|\,\mathrm{d}s,\\
|D_{\bar\eta}^3\bar{\mathscr{J}}|&\leq Ce^{3Ct\mathfrak{p}(c_1)}\Big(\int_0^t |\sD_{\bar\eta}^2\bar U|\,\mathrm{d}s\Big)^3  +Ce^{2Ct\mathfrak{p}(c_1)}\Big(\int_0^t |\sD_{\bar\eta}^2\bar U|\,\mathrm{d}s\Big)\Big(\int_0^t|\sD_{\bar\eta}^3\bar U|\,\mathrm{d}s\Big)\\
&\quad +Ce^{Ct\mathfrak{p}(c_1)}\int_0^t |\sD_{\bar\eta}^4\bar U|\,\mathrm{d}s.
\end{aligned}
\end{equation}

On the other hand, by the chain rules, we have
\begin{equation}\label{1190}
\begin{aligned}
D_{\bar\eta}^k\bar\Lambda &=\sum_{j=0}^{k+1} C_{k,j} (D_{\bar\eta}^{k+1-j} \rho_0^{\gamma-1}) D_{\bar\eta}^{j}(\bar{\mathscr{J}}^{1-\gamma})\qquad \text{for $k=0,1,2$},\\
\bar\Lambda_t&=(1-\gamma)\bar\varrho^{\gamma-1} D_{\bar\eta}\Big(D_{\bar\eta}\bar U+\frac{2\bar U}{\bar\eta}\Big)-\bar\Lambda\Big(\gamma D_{\bar\eta}\bar U+2(\gamma-1)\frac{\bar U}{\bar\eta}\Big), 
\end{aligned}
\end{equation}
where $C_{k,j}$ are some constants depending only on $(k,j)$.

Therefore, for all $0\leq t\leq T_1:=\min\{T^*,\mathfrak{p}(c_1)^{-1}\}$, combining \eqref{1189}--\eqref{1190}, together with \eqref{38}--\eqref{39}, \eqref{jt}, Lemmas \ref{lemma-useful1} and \ref{sobolev-embedding}--\ref{hardy-inequality}, and the H\"older and Minkowski inequalities, we can recursively obtain the desired estimates of this lemma.
\end{proof}

\subsection{Uniform   Estimates for the Velocity}\label{subsub-11.2.2}
The proof is divided into the following lemmas.

\begin{Lemma}\label{c_0-c_1}
For all $0\leq t\leq T_2=\min\{T_1, \mathfrak{p}(c_1)^{-1}\}$.
\begin{equation*}
\big|r(\rho_0^\frac{1}{2}U,\rho_0^\frac{\delta}{2}\sD_{\bar\eta} U,\rho_0^\frac{1}{2}U_t)(t)\big|_2 +\int_0^t\big|(r^2\rho_0^{\delta})^\frac{1}{2}\sD_{\bar\eta} U_t\big|_2^2\,\mathrm{d}s\leq C\mathfrak{p}(c_0).
\end{equation*}
\end{Lemma}
\begin{proof}
We divide the proof into three steps.

\smallskip
\textbf{Step 1.} Multiplying  $\eqref{lp}_1$ by $U$ and integrating the resulting equation over $I$, along with  \eqref{38}--\eqref{39}, Lemma \ref{lemma-Lambda-guji}, and the H\"older and Young inequalities, yields that
\begin{equation}\label{unif lem3.3 eq} 
\begin{aligned}
\frac{1}{2}\frac{\mathrm{d}}{\dt}\big|(r^2\rho_0)^\frac{1}{2}U\big|_2^2+J_1&=A \int_0^1 \bar\eta^2\bar\eta_r\bar\varrho^\gamma\big(D_{\bar\eta}U+2\frac{U}{\bar\eta}\big)\,\mathrm{d}r-\int_0^1 r^2\rho_0 \overline{D_{\eta}\Phi} U \,\mathrm{d}r\\
&\leq C\mathfrak{p}(c_0)+C\mathfrak{p}(c_0)\big|(r^2\rho_0)^{\frac{1}{2}}U\big|_2^2+\frac{c^*_{a_1,a_2}}{4}\big|(r^2\rho_0^\delta)^\frac{1}{2}\mathscr{D}_{\bar\eta} U\big|_2^2. 
\end{aligned}    
\end{equation}
Here $J_1$ is defined in the same way as $L_1$ in \eqref{def-L1L2}, except with $(X,Y)=\mathscr{D}_{\bar\eta} U$, and similarly to \eqref{L-1}, we can find a constant $c^*_{a_1,a_2}>0$ such that 
\begin{equation}\label{J1}
    J_1\geq c^*_{a_1,a_2}\big|(r^2\rho_0^\delta)^\frac{1}{2}\mathscr{D}_{\bar\eta} U\big|_2^2.
\end{equation}

Then, we obtain from  \eqref{unif lem3.3 eq} and the Gr\"onwall inequality that for $t\in [0,T_2]$,
\begin{equation}\label{1}
\big|(r^2\rho_0)^\frac{1}{2}U(t)\big|_2^2+ \int_0^t\big|(r^2\rho_0^\delta)^\frac{1}{2}\mathscr{D}_{\bar\eta} U\big|_2^2\,\mathrm{d}s
\leq Ce^{C\mathfrak{p}(c_0)t}\big(\bar\cE(0,U)+\mathfrak{p}(c_0)t\big)\leq C\mathfrak{p}(c_0).
\end{equation}

\smallskip
\textbf{Step 2.}
Multiplying  $\eqref{lp}_1$ by $U_t$ and integrating the resulting equation over $I$, we obtain from \eqref{38}--\eqref{39}, Lemma \ref{lemma-useful1}, and the H\"older and Young inequalities that, for any $\varepsilon\in (0,1)$,
\begin{align}
&\,\frac{1}{2}\frac{\mathrm{d}}{\dt}J_1+\big|(r^2\rho_0)^\frac{1}{2}U_t\big|_2^2\notag\\
&=A \int_0^1 \bar\eta^2\bar\eta_r\bar\varrho^\gamma\big(D_{\bar\eta}U_t+\frac{2U_t}{\bar\eta}\big)\,\mathrm{d}r-\int_0^1r^2\rho_0\overline{D_\eta\Phi}U_t\,\mathrm{d}r\notag\\
&\quad+\frac{1-\delta}{2}\int_0^1\bar\eta^2\bar\eta_r\bar\varrho^\delta(D_{\bar\eta}\bar U+2\frac{\bar U}{\bar \eta})   
\Big((2a_1+a_2)D_{\bar\eta}U^2+4a_2 D_{\bar\eta}U\frac{U}{\bar\eta}+4(a_1+a_2)\frac{U^2}{\bar\eta^2}\Big)\,\mathrm{d}r\notag\\
&\quad -\int_0^1 \bar\eta^2\bar\eta_r\bar\varrho^\delta\Big(\big((2a_1+a_2)D_{\bar\eta}U^2+2a_2 D_{\bar\eta}U \frac{U}{\bar\eta}\big)D_{\bar\eta}\bar U+\big(2a_2 D_{\bar\eta}U \frac{U}{\bar\eta}+4(a_1+a_2)\frac{U^2}{\bar\eta^2}\big)\frac{\bar U}{\bar\eta} \Big)\,\mathrm{d}r\notag\\
&\leq  C\mathfrak{p}(c_0) +\varepsilon\big|(r^2\rho_0^\delta)^\frac{1}{2}\sD_{\bar\eta} U_t\big|_2^2+\frac{1}{8}\big|(r^2\rho_0)^\frac{1}{2}U_t\big|_2^2+C\big|\sD_{\bar\eta} \bar U\big|_\infty\big|(r^2\rho_0^\delta)^\frac{1}{2}\sD_{\bar\eta} U\big|_2^2 \notag\\
&\leq  C\mathfrak{p}(c_1)\big(\big|(r^2\rho_0^\delta)^\frac{1}{2}\sD_{\bar\eta} U\big|_2^2+1\big) +\varepsilon\big|(r^2\rho_0^\delta)^\frac{1}{2}\sD_{\bar\eta} U_t\big|_2^2+\frac{1}{8}\big|(r^2\rho_0)^\frac{1}{2}U_t\big|_2^2. \label{DU-t}
\end{align}

\smallskip
\textbf{Step 3.}
Next, multiplying \eqref{33151} by $U_t$ and integrating the resulting equation over $I$ implies that 
\begin{equation}\label{O-2}
\begin{aligned}
\frac{1}{2}\frac{\mathrm{d}}{\dt}\big|(r^2\rho_0)^\frac{1}{2}U_t\big|_2^2+J_2&=\underline{\int_0^1 (r^2\rho_0^\delta)^\frac{1}{2}\big(g^{(1)}\frac{2 U_t}{\bar\eta}+ h^{(1)}D_{\bar\eta}U_t\big)\,\mathrm{d}r}_{:=J_3}. \\
\end{aligned}    
\end{equation}
Here $J_2$ is defined in the same way as $L_1$ in \eqref{def-L1L2}, except with $(X,Y)=\mathscr{D}_{\bar\eta} U_t$, and similarly to \eqref{L-1}, we can find a constant $c^*_{a_1,a_2}>0$ such that 
\begin{equation}\label{J2}
    J_2\geq c^*_{a_1,a_2}\big|(r^2\rho_0^\delta)^\frac{1}{2}\mathscr{D}_{\bar\eta} U_t\big|_2^2.
\end{equation}

For $J_3$, recalling $(g^{(1)},h^{(1)})$ in \eqref{q1-q2-1}, we derive from \eqref{38}--\eqref{39} and Lemmas \ref{lemma-useful1}--\ref{lemma-Lambda-guji} that 
\begin{equation}\label{est-q1-q2-1}
\begin{aligned}
|(g^{(1)},h^{(1)})|_2&\leq C\big|\sD_{\bar\eta} \bar U\big|_\infty\big|(r^2\rho_0^\delta)^\frac{1}{2}\sD_{\bar\eta}  U\big|_2+C\mathfrak{p}(c_0)\big( |\bar\varrho|_\infty^{\gamma-\delta}+|\overline{D_\eta\Phi}|_\infty\big)|(r^2\rho_0^\delta)^\frac{1}{2}\sD_{\bar\eta}\bar U|_2\\
&\leq C\mathfrak{p}(c_1)\big(\big|(r^2\rho_0^\delta)^\frac{1}{2}\sD_{\bar\eta}  U\big|_2+\mathfrak{p}(c_0)\big).
\end{aligned}
\end{equation}
Then it follows from \eqref{39}, \eqref{est-q1-q2-1}, and the H\"older and Young inequalities that
\begin{equation}\label{unif lem3.3 J4}
\begin{aligned}
J_3 &\leq C\mathfrak{p}(c_1)\big(\big|(r^2\rho_0^\delta)^\frac{1}{2}\mathscr{D}_{\bar\eta}U\big|_2^2+1\big) +\frac{c_{a_1,a_2}^*}{8}\big|(r^2\rho_0^\delta)^\frac{1}{2}\mathscr{D}_{\bar\eta} U_t\big|_2^2.
\end{aligned}
\end{equation}
Then, substituting  \eqref{J2}--\eqref{unif  lem3.3 J4} into \eqref{O-2} yields that
\begin{equation}\label{unif lem3.3 step3}
    \begin{aligned}
        \frac{\mathrm{d}}{\dt}\big|(r^2\rho_0)^\frac{1}{2}U_t\big|_2^2+C\big|(r^2\rho_0^\delta)^{\frac{1}{2}}\sD_{\bar\eta}U_t\big|^2_2&\leq C\mathfrak{p}(c_1)\big(\big|(r^2\rho_0^\delta)^\frac{1}{2}\mathscr{D}_{\bar\eta}U\big|_2^2+1\big).
     \end{aligned}
\end{equation}

Finally, combining with \eqref{unif lem3.3 step3} and \eqref{DU-t}, and then taking $\varepsilon$ sufficiently small in \eqref{DU-t}, together with \eqref{39}, \eqref{J1}, and the Gr\"onwall inequality, yields that, for all $0\leq t\leq  T_2$,
\begin{equation}\label{2}
\begin{aligned}
&\,\big|(r^2\rho_0^\delta)^\frac{1}{2}\sD_{\bar\eta}U (t)\big|_2^2+\big|(r^2\rho_0)^\frac{1}{2}U_t(t)\big|_2^2+ \int_0^t\big|(r^2\rho_0^\delta)^\frac{1}{2}\sD_{\bar\eta}U_t\big|_2^2\,\mathrm{d}s
\\
&\qquad\qquad\qquad\qquad\qquad\qquad\quad\qquad
\leq Ce^{C\mathfrak{p}(c_1)t}(\bar\cE(0,U)+ \mathfrak{p}(c_0)) \leq C\mathfrak{p}(c_0).
\end{aligned}
\end{equation}

This completes the proof.
\end{proof}

\begin{Lemma}\label{c_1}
For all $t\in[0,T_2]$,
\begin{equation*}
\big|(r^2\rho_0^\delta)^\frac{1}{2}\big(\sD_{\bar\eta} U_t,\rho_0^\frac{1-\delta}{2}\sqrt{t}U_{tt}\big)(t)\big|_2 +\int_0^t \big|(r^2\rho_0^\delta)^\frac{1}{2}\big(\rho_0^{\frac{1-\delta}{2}}U_{tt},\sqrt{s}\sD_{\bar\eta} U_{tt}\big)\big|_2^2\,\mathrm{d}s\leq C\mathfrak{p}(c_0).
\end{equation*}
\end{Lemma}
\begin{proof}
We divide the proof into three steps.

\smallskip
\textbf{Step 1.} First, it follows from \eqref{new-lp} and \eqref{bar-Lambda} that
\begin{equation}\label{new-lp*}
\begin{aligned}
D_{\bar\eta}\Big(D_{\bar\eta} U+ \frac{2U}{\bar\eta} \Big)&=-\frac{\delta}{\gamma-1}\frac{\bar\Lambda}{\bar\varrho^{\gamma-1}} \Big(D_{\bar\eta} U+ \frac{2a_2}{2a_1+ a_2}  \frac{U}{\bar\eta}\Big)+\frac{1}{2a_1+ a_2}\bar\varrho^{1-\delta} U_t\\
&\quad +\frac{A}{2a_1+ a_2}\frac{\gamma}{\gamma-1}\bar\Lambda \bar\varrho^{1-\delta}+ \frac{1}{2a_1+ a_2}\bar\varrho^{1-\delta}\overline{D_\eta\Phi}.
\end{aligned}
\end{equation}
Then, due to the fact that $\rho_0^{\gamma-1}\sim 1-r$, \eqref{38}--\eqref{39}, and Lemmas \ref{im-1} and \ref{lemma-Lambda-guji}--\ref{c_0-c_1}, we see that, for all $t\in [0,T_2]$, 
\begin{equation}\label{1197}
\begin{aligned}
&\,\Big|\zeta r\big(D_{\bar\eta}^2 U,D_{\bar\eta}(\frac{U}{\bar\eta})\big)\Big|_2\leq C\Big|\zeta r D_{\bar\eta}\big(D_{\bar\eta} U+ \frac{2U}{\bar\eta}\big)\Big|_2\\
&\leq C\big(|\zeta_\frac{5}{8} \bar\varrho^{1-\delta}|_\infty|\zeta rU_t|_2
+|\zeta r\sD_{\bar\eta} U|_2|\bar\Lambda|_\infty+ (|\bar\Lambda|_\infty+|\overline{D_\eta\Phi}|_\infty)|\zeta r\bar\varrho^{1-\delta}|_2\big)\leq C\mathfrak{p}(c_0),   
\end{aligned}
\end{equation}
which, along with Lemma \ref{hardy-inequality}, also leads to
\begin{equation}\label{11123}
\big|\zeta r^\frac{1}{2}\sD_{\bar\eta} U\big|_4\leq\big|r^\frac{5}{4}\big(\zeta \sD_{\bar\eta} U,\zeta_r \sD_{\bar\eta} U,\zeta \sD_{\bar\eta}^2 U\big)\big|_2 \leq C\mathfrak{p}(c_0).   
\end{equation}

Next, using \eqref{1}, \eqref{2}, and the similar argument as \eqref{L2Linfty}--\eqref{3...136} in the proof of Lemma \ref{Lemma-point}, we obtain that, for all $t\in [0,T_2]$, 
\begin{equation}\label{xingxing}
\big|\chi^\sharp\rho_0^\frac{\delta-(\gamma-1)}{2}\big(D_{\bar\eta} U+\frac{2a_2}{2a_1+a_2}\frac{U}{\bar \eta}\big)\big|_\infty \leq C\mathfrak{p}(c_0)\big(1+  \big|\chi^\sharp\rho_0^{\frac{\delta}{2} }(U,D_{\bar\eta} U,\rho_0^{1-\frac{\delta}{2}}U_{t})\big|_2\big)\leq C\mathfrak{p}(c_0).
\end{equation}

Moreover, for any fixed $a\in(0,\frac{1}{2})$, \eqref{1} and \eqref{2}, together with the fact that $\rho_0^{\gamma-1}\sim 1-r$, Lemma \ref{sobolev-embedding}, and \eqref{39}, yield
\begin{equation}\label{revise 3.22 3}
\begin{aligned}
      \|U(t)\|_{L^\infty(a,1-a)}&\leq  C(a)\big(\|(r^2\rho_0)^\frac{1}{2}U\|_{L^2(a,1-a)}+\|(r^2\rho_0^\delta)^\frac{1}{2}D_{\bar\eta}U\|_{L^{2}(a,1-a)}\big)\\
      &\leq C(a)\mathfrak{p}(c_0).
\end{aligned}
\end{equation}
It then follows from the facts that $\frac{\gamma-1-\delta}{2}>-\gamma+1$ and $\rho_0^{\gamma-1}\sim 1-r$, and \eqref{xingxing}, that
\begin{equation}\label{revise 3.22 0}
\big|\chi^\sharp\big(D_{\bar\eta} U+\varsigma\frac{U}{\bar \eta}\big)\big|_1
\leq C\mathfrak{p}(c_0)\int_0^1 \rho_0^{\frac{\gamma-1-\delta}{2}}\,\mathrm{d}\tilde{r}\leq C\mathfrak{p}(c_0), \qquad \text{with } \varsigma=\frac{2a_2}{2a_1+a_2}.
\end{equation}
Note that  
\begin{equation}\label{revise 3.22 1}
    (\bar\eta^\varsigma U)_r= \bar\eta^\varsigma\bar\eta_r \big(D_{\bar\eta} U+\varsigma\frac{U}{\bar \eta}\big).
\end{equation}
Integrating the above over $(\frac{1}{2},r)$ for  $r>\frac{1}{2}$, we obtain from  \eqref{39} and \eqref{revise 3.22 0} that 
\begin{equation}\label{revise 3.22 2}
\begin{aligned}
|\chi^\sharp U(t,r)|&\leq \frac{\bar\eta^\varsigma(t,\frac{1}{2})}{\bar\eta^\varsigma(t,r)}\big|U\big(t,\frac{1}{2}\big)\big|+ \frac{1}{\bar\eta^\varsigma(t,r)}\int_{\frac{1}{2}}^r\big|\bar\eta^\varsigma(t,\tilde{r})\bar\eta_r(t,\tilde{r}) \big(D_{\bar\eta} U+\varsigma\frac{U}{\bar \eta}\big)\big|\,\mathrm{d}\tilde{r}\\
&\leq C\big|U\big(t,\frac{1}{2}\big)\big|+C\big|\chi^\sharp\big(D_{\bar\eta} U+\varsigma\frac{U}{\bar \eta}\big)\big|_1\leq C\mathfrak{p}(c_0)+C\|U\|_{L^{\infty}(\frac{1}{4},\frac{3}{4})} ,
\end{aligned}
\end{equation}
which, along with \eqref{revise 3.22 3}, implies 
\begin{equation}\label{revise 3.22 4}
  |\chi^\sharp U(t)|_{\infty} \leq C\mathfrak{p}(c_0).    
\end{equation}

Of course, based on \eqref{2}, we also derive from \eqref{est-q1-q2-1} that, for all $t\in [0,T_2]$, 
\begin{equation}\label{qq12}
|(g^{(1)},h^{(1)})|_2\leq C\mathfrak{p}(c_1).    
\end{equation}

Finally, recall $((g^{(1)})_t,(h^{(1)})_t)$ in \eqref{q1-q2-1-t}. It then follows from \eqref{38}--\eqref{39}, \eqref{11123}--\eqref{xingxing}, Lemmas \ref{lemma-useful1}, \ref{c_0-c_1}, and \ref{hardy-inequality}, and the H\"older inequality that       
\begin{align}
& \ \begin{aligned}\label{qq12-t-1}
\big|\chi((g^{(1)})_t,(h^{(1)})_t)\big|_2&\leq \mathfrak{p}(c_0)\big( (|\chi r^\frac{1}{2}\sD_{\bar\eta} U|_4+1)|\chi r^\frac{1}{2}\sD_{\bar\eta}\bar U_t|_4 +|\chi r\sD_{\bar\eta}U_t|_2|\sD_{\bar\eta}\bar U|_\infty\big)\\
&\quad +\mathfrak{p}(c_0)(|\chi r\sD_{\bar\eta} U|_2+1)|\sD_{\bar\eta}\bar U|_\infty^2 \\
&\leq  C\mathfrak{p}(c_0)\big(\mathfrak{p}(c_1)+\bar\cD(t,\bar U)^\frac{1}{2}+\mathfrak{p}(c_1)\big|(r^2\rho_0^\delta)^\frac{1}{2}\sD_{\bar\eta}U_t\big|_2\big),
\end{aligned}\\
&\begin{aligned}\label{qq12-t-2}
\big|\chi^\sharp((g^{(1)})_t,(h^{(1)})_t)\big|_2&\leq C(\big|\chi^\sharp\rho_0^\frac{\delta}{2}\sD_{\bar\eta} U\big|_2+1)|\sD_{\bar\eta} \bar U|_\infty^2+ C\big|\chi^\sharp\rho_0^{\frac{\delta}{2}}\sD_{\bar\eta} U_t \big|_2|\sD_{\bar\eta} \bar U|_\infty\\
&\quad +C\big|\chi^\sharp\rho_0^\frac{\gamma-1}{2}\sD_{\bar\eta}\bar U_t\big|_2\Big|\chi^\sharp\rho_0^\frac{\delta-(\gamma-1)}{2}\big(D_{\bar\eta}U+\frac{2a_2}{2a_1+a_2}\frac{U}{\bar \eta}\big)\Big|_\infty\\
&\quad+ C|\chi^\sharp U|_\infty\big|\chi^\sharp \rho_0^\frac{\delta}{2}D_{\bar\eta}\bar U_t\big|_2\\
&\leq C(\big|\chi^\sharp\rho_0^\frac{\delta}{2}\sD_{\bar\eta} U\big|_2+1)|\sD_{\bar\eta} \bar U|_\infty^2+ C\big|\chi^\sharp\rho_0^{\frac{\delta}{2}}\sD_{\bar\eta} U_t \big|_2|\sD_{\bar\eta} \bar U|_\infty+C\mathfrak{p}(c_0)\mathfrak{p}(c_1)\\
&\quad +C\big|\chi^\sharp\rho_0^\frac{3(\gamma-1)}{2}(\bar U_t,D_{\bar\eta}\bar U_t,D_{\bar\eta}^2\bar U_t)\big|_2\Big|\chi^\sharp\rho_0^\frac{\delta-(\gamma-1)}{2}\big(D_{\bar\eta}U+\frac{2a_2}{2a_1+a_2}\frac{U}{\bar \eta}\big)\Big|_\infty\\
% &\quad+C\big|\chi^\sharp\rho_0^{\frac{\delta}{2}}(U,D_{\bar\eta}U)\big|_2\big|\rho_0^{\frac{3}{2}(\gamma-1)}(\bar U_t,D_{\bar\eta}\bar U_t,D_{\bar\eta}^2 \bar U_t)\big|_2\\
&\leq  C\mathfrak{p}(c_0)\big(\mathfrak{p}(c_1)+\bar\cD(t,\bar U)^\frac{1}{2}+\mathfrak{p}(c_1)\big|(r^2\rho_0^\delta)^\frac{1}{2}\sD_{\bar\eta} U_t\big|_2\big).
\end{aligned}
\end{align}

\textbf{Step 2.} Now, multiplying \eqref{33151} by $U_{tt}$ and integrating over $I$, we have
\begin{equation}\label{I0*3}
\begin{aligned}
&\frac{1}{2}\frac{\mathrm{d}}{\mathrm{d}t}J_2+\big|(r^2\rho_0)^\frac{1}{2}U_{tt}\big|_2^2=\frac{\mathrm{d}}{\mathrm{d}t}J_4-  J_5,
\end{aligned}
\end{equation}
with
\begin{equation*}
\begin{aligned}
J_4&:=\int_0^1 (r^2\rho_0^\delta)^\frac{1}{2}\big(g^{(1)}\frac{2 U_t}{\bar\eta}+h^{(1)}D_{\bar\eta} U_t\big)\,\mathrm{d}r,\\
J_5&:=\int_0^1 \bar\eta^2\bar\eta_r\bar\varrho^\delta\Big((2a_1+a_2)D_{\bar\eta}U_t^2D_{\bar\eta}\bar U+2a_2 D_{\bar\eta}U_tD_{\bar\eta}\bar U\frac{U_t}{\bar\eta}\Big)\,\mathrm{d}r\\
&\quad +\int_0^1 \bar\eta^2\bar\eta_r\bar\varrho^\delta\Big(2a_2 D_{\bar\eta}U_t\frac{\bar U}{\bar\eta}\frac{U_t}{\bar\eta}+4(a_1+a_2)\frac{U_t^2}{\bar\eta^2}\frac{\bar U}{\bar\eta} \Big)\,\mathrm{d}r\\
&\quad -\int_0^1 (r^2\rho_0^\delta)^\frac{1}{2}\Big(\big(g^{(1)}\frac{\bar U}{\bar\eta}-(g^{(1)})_t\big)\frac{2 U_t}{\bar\eta}+(h^{(1)}D_{\bar\eta}\bar U-(h^{(1)})_t) D_{\bar\eta} U_t\big)\Big)\,\mathrm{d}r\\
&\quad -\frac{1-\delta}{2}\int_0^1\bar\eta^2\bar\eta_r\bar\varrho^\delta(D_{\bar\eta}\bar U+2\frac{\bar U}{\bar \eta})   
\Big((2a_1+a_2)D_{\bar\eta}U_t^2 +4a_2 D_{\bar\eta}U_t\frac{U_t}{\bar\eta}+4(a_1+a_2)\frac{U_t^2}{\bar\eta^2}\Big)\,\mathrm{d}r. 
\end{aligned}
\end{equation*}
For $J_2$, we have obtained in \eqref{J2} that
\begin{equation*}
        J_2\geq c^*_{a_1,a_2}\big|(r^2\rho_0^\delta)^\frac{1}{2}\mathscr{D}_{\bar\eta} U_t\big|_2^2.
\end{equation*}
For $J_4$, we first obtain from \eqref{39}, \eqref{2}, \eqref{qq12-t-1}--\eqref{qq12-t-2}, and the Young inequality that, for any $t\in [0,T_2]$,
\begin{equation}\label{A2 revise 0}
    \begin{aligned}
        \int_0^t\big|(g^{(1)}_s,h^{(1)}_s)(s)\big|_2^2\,\mathrm{d}s\leq C\mathfrak{p}(c_1).
    \end{aligned}
\end{equation}
Then \eqref{q1-q2-1}, together with the facts 
that $\bar U(0,r)=U(0,r)=u_0(r)$ and $\bar\eta(0,r)=r$, yields
\begin{equation}\label{A2 revise 1}
    \big|(g^{(1)},h^{(1)})(0)\big|_2^2\leq C\mathfrak{p}(c_0).
\end{equation}
Thus, combining with \eqref{A2 revise 0} and \eqref{A2 revise 1}, along with the H\"older inequality, gives 
\begin{equation}\label{A2 revise 2}
\begin{aligned}
        \big|(g^{(1)},h^{(1)})(t)\big|_2^2&\leq 2 \big|(g^{(1)},h^{(1)})(0)\big|_2^2+Ct\int_0^t\big|(g^{(1)}_s,h^{(1)}_s)(s)\big|_2^2\,\mathrm{d}s\\
        &\leq C\mathfrak{p}(c_0)+Ct\mathfrak{p}(c_1).
\end{aligned}
\end{equation}
Hence, by \eqref{A2 revise 2} and the Young inequality, we obtain that 
\begin{equation}\label{I*3}
\begin{aligned}
|J_4|&\leq C\big|(g^{(1)},h^{(1)})(t)\big|_2|(r^2\rho_0^\delta)^\frac{1}{2}\mathscr{D}_{\bar\eta} U_t\big|_2 \\
&\leq C\varepsilon^{-1}(\mathfrak{p}(c_0)+t\mathfrak{p}(c_1))+\varepsilon|(r^2\rho_0^\delta)^\frac{1}{2}\mathscr{D}_{\bar\eta} U_t\big|_2^2\qquad\text{for all $\varepsilon\in(0,1)$}.
\end{aligned}
\end{equation}
For $J_5$, it follows from \eqref{qq12}--\eqref{qq12-t-2}, Lemma \ref{lemma-useful1}, and the H\"older and Young inequalities that
\begin{equation}\label{I1*3}
\begin{aligned}
|J_5|&\leq C\big(\big|\sD_{\bar\eta}\bar U\big|_\infty|(r^2\rho_0^\delta)^\frac{1}{2}\mathscr{D}_{\bar\eta} U_t\big|_2 +\big|\sD_{\bar\eta}\bar U\big|_\infty|(g^{(1)},h^{(1)})|_2\!+|((g^{(1)})_t,(h^{(1)})_t)|_2\big)|(r^2\rho_0^\delta)^\frac{1}{2}\mathscr{D}_{\bar\eta} U_t\big|_2\\
&\leq C\mathfrak{p}(c_1)\big(1+|(r^2\rho_0^\delta)^\frac{1}{2}\mathscr{D}_{\bar\eta} U_t\big|_2^2\big)+C\mathfrak{p}(c_0)\bar\cD(t,\bar U)^\frac{1}{2}|(r^2\rho_0^\delta)^\frac{1}{2}\mathscr{D}_{\bar\eta} U_t\big|_2.
\end{aligned}
\end{equation}
Thus, integrating \eqref{I0*3} over $[0,t]$ and then substituting \eqref{I*3}--\eqref{I1*3} into the resulting equality with $\varepsilon$ sufficiently small, we can deduce from the Gr\"onwall inequality 
that for all $t\in[0,T_2]$,
\begin{equation*}
\begin{aligned}
|(r^2\rho_0^\delta)^\frac{1}{2}\mathscr{D}_{\bar\eta}U_t(t)\big|_2^2
+\int_0^t|(r^2\rho_0)^\frac{1}{2}U_{tt}(s)\big|_2^2\,\mathrm{d}s\leq C\mathfrak{p}(c_0)\Big(1+\int_0^t
\bar{D}(s;\bar{U})^\frac{1}{2}
|(r^2\rho_0^\delta)^\frac{1}{2}\mathscr{D}_{\bar\eta}U_t(s)\big|_2\,\mathrm{d}s\Big).
\end{aligned}
\end{equation*}

To further simplify the above inequality, define
\begin{equation*}
\cY(t)=C\mathfrak{p}(c_0)\Big(1+\int_0^t \bar\cD(s,\bar U )^\frac{1}{2}|(r^2\rho_0^\delta)^\frac{1}{2}\mathscr{D}_{\bar\eta} U_t\big|_2\,\mathrm{d}s\Big).
\end{equation*}
Then $|(r^2\rho_0^\delta)^\frac{1}{2}\mathscr{D}_{\bar\eta} U_t\big|_2^2\leq \cY(t)$ and 
\begin{equation*}
\cY'(t)=C\mathfrak{p}(c_0) \bar\cD(t,\bar U )^\frac{1}{2}|(r^2\rho_0^\delta)^\frac{1}{2}\mathscr{D}_{\bar\eta} U_t\big|_2 \leq C\mathfrak{p}(c_0) \bar\cD(t,\bar U )^\frac{1}{2}\cY(t)^\frac{1}{2}.
\end{equation*}
Clearly, this, together with \eqref{39}, implies that, for all $t\in[0,T_2]$, 
\begin{equation*}
\cY^\frac{1}{2}(t)\leq \cY^\frac{1}{2}(0)+C\mathfrak{p}(c_0)\int_0^t\bar\cD(s,\bar U)^\frac{1}{2}\,\mathrm{d}s\leq C\mathfrak{p}(c_0)^\frac{1}{2}+C\mathfrak{p}(c_0)(c_1t)^\frac{1}{2}\leq C\mathfrak{p}(c_0),
\end{equation*}
which yields that, for all $t\in[0,T_2]$,
\begin{equation}\label{11130}
|(r^2\rho_0^\delta)^\frac{1}{2}\mathscr{D}_{\bar\eta} U_t\big|_2^2+\int_0^t \big|(r^2\rho_0)^\frac{1}{2}U_{tt}\big|_2^2\,\mathrm{d}s\leq C\mathfrak{p}(c_0).
\end{equation}

\smallskip
\textbf{Step 3.} Since we have shown that $U$ satisfies \eqref{Energy-Id} with $(w,g,h)$ replaced by $(U_t,g^{(1)},h^{(1)})$ in Step 5 of \S\ref{subsection3.3}, we have
\begin{equation}\label{O-4}
\begin{aligned}
&\,\frac{1}{2}\frac{\mathrm{d}}{\dt}\big|(r^2\rho_0)^\frac{1}{2}U_{tt}\big|_2^2+J_6=-J_7+J_8,
\end{aligned}    
\end{equation}
where $J_6$ is defined in the same way as $L_1$ in \eqref{def-L1L2}, except with $(X,Y)=\mathscr{D}_{\bar\eta} U_{tt}$, and
\begin{equation}
\begin{aligned}
J_7&:=(2a_1+a_2)\int_0^1\bar\eta^2\bar\eta_r\bar\varrho^{\delta}\Big(-(\delta+1)D_{\bar\eta}\bar U+(2-2\delta)\frac{\bar U}{\bar\eta}\Big)D_{\bar\eta}U_t D_{\bar\eta}U_{tt}\,\mathrm{d}r\\
&\quad +4(a_1+a_2)\int_0^1\bar\eta^2\bar\eta_r\bar\varrho^{\delta}\Big((1-\delta)D_{\bar\eta}\bar U- 2\delta \frac{\bar U}{\bar\eta} \Big) \frac{U_t}{\bar\eta}\frac{U_{tt}}{\bar\eta}\,\mathrm{d}r\\
&\quad + 2a_2\int_0^1 \bar\eta^2\bar\eta_r\bar\varrho^{\delta}\Big(- \delta D_{\bar\eta}\bar U+(1-2\delta)\frac{\bar U}{\bar\eta}\Big)\Big(\frac{U_t}{\bar\eta} D_{\bar\eta}U_{tt}+D_{\bar\eta}U_t\frac{U_{tt}}{\bar\eta}\Big)\,\mathrm{d}r,  \\
J_8&:=\int_0^1(r^2\rho_0^{\delta})^\frac{1}{2} \Big(\big(g^{(1)}_t- g^{(1)} \frac{\bar U}{\bar\eta}\big) \frac{2U_{tt}}{\bar\eta} + \big(h^{(1)}_t- h^{(1)} D_{\bar\eta}\bar U\big)D_{\bar\eta}U_{tt}\Big)\,\mathrm{d}r.
\end{aligned}    
\end{equation}

For $J_6$, similarly to \eqref{L-1}, we can find a constant $c^*_{a_1,a_2}>0$ such that 
\begin{equation}\label{J6}
    J_6\geq c^*_{a_1,a_2}\big|(r^2\rho_0^\delta)^\frac{1}{2}\mathscr{D}_{\bar\eta} U_{tt}\big|_2^2.
\end{equation}

For $J_7$--$J_8$, it follows from \eqref{qq12}--\eqref{qq12-t-2}, \eqref{11130}, Lemma \ref{lemma-useful1}, and the H\"older and Young inequalities that  
\begin{align}
&\begin{aligned}\label{i5}
|J_7|&\leq C\big|\sD_{\bar\eta}\bar U\big|_\infty\big|(r^2\rho_0^\delta)^\frac{1}{2}\sD_{\bar\eta} U_{t}\big|_2\big|(r^2\rho_0^\delta)^\frac{1}{2}\sD_{\bar\eta} U_{tt}\big|_2\leq C\mathfrak{p}(c_1)+\frac{c_{a_1,a_2}^*}{100}\big|(r^2\rho_0^\delta)^\frac{1}{2}\sD_{\bar\eta} U_{tt}\big|_2^2,
\end{aligned}\\
&\begin{aligned}\label{i6}
J_8&\leq C\big(\big|\sD_{\bar\eta}\bar U\big|_\infty|(g^{(1)},h^{(1)})|_2\!+|((g^{(1)})_t,(h^{(1)})_t)|_2\big)\big|(r^2\rho_0^\delta)^\frac{1}{2}\sD_{\bar\eta} U_{tt}\big|_2\\
&\leq C(\mathfrak{p}(c_1)+\mathfrak{p}(c_0)\bar\cD(t,\bar U))+\frac{c_{a_1,a_2}^*}{100}\big|(r^2\rho_0^\delta)^\frac{1}{2}\sD_{\bar\eta} U_{tt}\big|_2^2.
\end{aligned}
\end{align}

Thus, plugging \eqref{J6}--\eqref{i6} into \eqref{O-4}, we obtain from \eqref{39} that
\begin{equation}\label{new utt}
\begin{aligned}
&\frac{\mathrm{d}}{\dt}\big|(r^2\rho_0)^\frac{1}{2}U_{tt}\big|_2^2+ \big|(r^2\rho_0^\delta)^\frac{1}{2}\mathscr{D}_{\bar\eta} U_{tt}\big|_2^2\leq C(\mathfrak{p}(c_1)+\mathfrak{p}(c_0)\bar\cD(t,\bar U)).
\end{aligned}
\end{equation}
Multiplying \eqref{new utt} by $t$ and  integrating over $[\tau,t]$ with $\tau\in (0,t)$, along with \eqref{39} and \eqref{11130}, gives
\begin{equation}\label{tauk}
t\big|(r^2\rho_0)^\frac{1}{2}U_{tt}(t)\big|_2^2+\int_\tau^t s\big|(r^2\rho_0^\delta)^\frac{1}{2}\sD_{\bar\eta} U_{tt}\big|_2^2\,\ds\leq \tau\big|(r^2\rho_0)^\frac{1}{2}U_{tt}(\tau)\big|_2^2+C\mathfrak{p}(c_0). 
\end{equation}
Thanks to \eqref{11130} and Lemma \ref{bjr}, we can find a sequence $\{\tau_k\}_{k=1}^\infty$ such that $\tau_k\to 0$ and $\tau_k|(r^2\rho_0)^\frac{1}{2}U_{tt}(\tau_k)|^2_2\to 0$ as $k\to\infty$. Hence, taking $\tau=\tau_k$ in \eqref{tauk} and then letting $k\to\infty$, we finally obtain the desired estimate.
\end{proof}

\begin{Lemma}\label{c_1-c_2}
For any $t\in [0,T_2]$,
\begin{equation}
\bar\cE(t,U)+\big|(U,\sD_{\bar\eta}U\big)(t)\big|_\infty\leq C\mathfrak{p}(c_0).
\end{equation}
\end{Lemma}
\begin{proof}
We divide the proof into three steps.

\smallskip
\textbf{Step 1. Boundedness of $\bar\cE_{\mathrm{in}}(t,U)$.} It only remains to establish the third-order elliptic estimate for $U$ near the origin. First, it follows from \eqref{new-lp} that
\begin{equation}\label{new-lp*3}
\begin{aligned}
&\,D_{\bar\eta}^2\big(D_{\bar\eta} U+ \frac{2U}{\bar\eta}\big)\\
&=\frac{1}{2a_1+ a_2}\big(\bar\varrho^{1-\delta} D_{\bar\eta}U_t+\frac{1-\delta}{\gamma-1}\Lambda\bar\varrho^{2-\gamma-\delta} U_t\big)-\frac{\delta}{\gamma-1}\frac{\bar\Lambda}{\bar\varrho^{\gamma-1}} \big(D_{\bar\eta}^2 U+ \frac{2a_2}{2a_1+ a_2}  D_{\bar\eta}\big(\frac{U}{\bar\eta}\big)\big)\\
&\quad -\frac{\delta}{\gamma-1}\big(\frac{D_{\bar\eta}\bar\Lambda}{\bar\varrho^{\gamma-1}}-\frac{\bar\Lambda^2}{\bar\varrho^{2(\gamma-1)}}\big) \big(D_{\bar\eta} U+ \frac{2a_2}{2a_1+ a_2}  \frac{U}{\bar\eta}\big)\\
&\quad +\frac{A}{2a_1+ a_2}\frac{\gamma}{\gamma-1}\big(D_{\bar\eta}\bar\Lambda \bar\varrho^{1-\delta}+\frac{1-\delta}{\gamma-1}\bar\Lambda^2\bar\varrho^{2-\delta-\gamma}\big)+ \frac{1}{2a_1+ a_2}D_{\bar\eta}(\bar\varrho^{1-\delta}\overline{D_\eta\Phi}).
\end{aligned}
\end{equation}
Then it follows from the fact that $\rho_0^{\gamma-1}\sim 1-r$, \eqref{38}--\eqref{39}, \eqref{new-lp*}--\eqref{11123}, \eqref{new-lp*3}, and Lemmas \ref{lemma-Lambda-guji}--\ref{c_0-c_1}, \ref{hardy-inequality}, and \ref{im-1} that, for all $t\in [0,T_2]$, 
\begin{equation}\label{11112}
\begin{aligned}
\big|\zeta r\sD_{\bar\eta}^3 U\big|_2&\leq C\Big|\zeta r D_{\bar\eta}^2\big(D_{\bar\eta} U+ \frac{2U}{\bar\eta}\big)\Big|_2+\Big|\zeta r \frac{1}{\bar\eta}D_{\bar\eta}\big(D_{\bar\eta} U+ \frac{2U}{\bar\eta}\big) \Big|_2\\
&\leq C\mathfrak{p}(c_0)\big(\big|\zeta r\sD_{\bar\eta}U_t\big|_2+\big|\zeta_\frac{5}{8} r^\frac{1}{2}\sD_{\bar\eta}\bar\Lambda\big|_4 |\zeta r^\frac{1}{2} \sD_{\bar\eta} U|_4 \big)\\
&\quad +C\mathfrak{p}(c_0)\big(\big|\zeta r(\sD_{\bar\eta}^2 U,\sD_{\bar\eta}\bar\Lambda)\big|_2+|\sD_{\bar\eta}\overline{D_\eta\Phi}|_\infty+1\big)\leq C\mathfrak{p}(c_0).
\end{aligned}
\end{equation} 

\smallskip
\textbf{Step 2. Boundedness of $\bar\cE_{\mathrm{ex}}(t,U)$.} It only remains to establish the second- and third-order elliptic estimates for $U$ away from the origin.

\smallskip
\textbf{2.1.} First, rewrite \eqref{new-lp*} as
\begin{equation}\label{new-lp**}
\begin{aligned}
D_{\bar\eta}^2 U &=-2D_{\bar\eta}(\frac{U}{\bar\eta})-\frac{\delta}{\gamma-1}\frac{\bar\Lambda}{\bar\varrho^{\gamma-1}} \Big(D_{\bar\eta} U+ \frac{2a_2}{2a_1+ a_2}  \frac{U}{\bar\eta}\Big)+\frac{1}{2a_1+ a_2}\bar\varrho^{1-\delta} U_t\\
&\quad +\frac{A}{2a_1+ a_2}\frac{\gamma}{\gamma-1}\bar\Lambda \bar\varrho^{1-\delta}+ \frac{1}{2a_1+ a_2}\bar\varrho^{1-\delta}\overline{D_\eta\Phi}.
\end{aligned}
\end{equation}
Then it follows from the above, \eqref{38}--\eqref{39}, \eqref{xingxing}, and Lemmas \ref{lemma-Lambda-guji}--\ref{c_0-c_1} that, for all $\varepsilon>0$,
\begin{equation}\label{re version delta/2+e s3 D^2 U}
\begin{aligned}
\big|\chi^\sharp \rho_0^{\frac{\delta}{2}+\varepsilon}D_{\bar\eta}^2 U\big|_2&\leq C\big|\chi^\sharp \rho_0^{\frac{\delta}{2}+\varepsilon}(\rho_0^{1-\delta} U_t,U,D_{\bar\eta}U)\big|_2\\
&\quad +C\big|\chi^\sharp \rho_0^{-\frac{\gamma-1}{2}+\varepsilon}\bar\Lambda\big|_2\Big|\chi^\sharp\rho_0^\frac{\delta-(\gamma-1)}{2}\big(D_{\bar\eta} U+\frac{2a_2}{2a_1+a_2}\frac{U}{\bar \eta}\big)\Big|_\infty \\
&\quad +C\big|\chi^\sharp\rho_0^{1-\frac{\delta}{2}+\varepsilon}\big|_2(|\bar\Lambda|_\infty+|\overline{D_\eta\Phi}|_\infty)\leq C(\varepsilon)\mathfrak{p}(c_0).
\end{aligned}
\end{equation}
Clearly, since $(\frac{3}{2}-\varepsilon_0)(\gamma-1)>\frac{\delta}{2}$, the above also leads to
\begin{equation}\label{11139}
\big|\chi^\sharp \rho_0^{(\gamma-1)(\frac{3}{2}-\varepsilon_0)}D_{\bar\eta}^2 U\big|_2 \leq C\mathfrak{p}(c_0).
\end{equation}

\smallskip
\textbf{2.2.} 
Next, indeed, using an argument similar to \eqref{L2Linfty}--\eqref{3...136}, we can obtain that,  for all $\iota<\delta+\frac{\gamma-1}{2}$ and $\sigma>0$, and for all $t\in[0,T_2]$,
\begin{equation}\label{cal-1'}
\Big|\chi^\sharp\rho_0^{\iota-(\gamma-1)+\sigma}\big(D_{\bar\eta} U+\frac{2a_2}{2a_1+a_2}\frac{U}{\bar \eta}\big)\Big|_2\leq C(\sigma,\iota)\mathfrak{p}(c_0)\big(1+\big|\chi^\sharp\rho_0^\iota\big(U,D_{\bar\eta} U,\rho_0^{1-\delta}U_t\big)\big|_2\big).
\end{equation}
Then, due to the facts that
\begin{equation*}
\varepsilon_0<\frac{1}{100}, \qquad \big(\frac{3}{2}-\varepsilon_0\big)(\gamma-1)>\frac{1}{2},
\end{equation*}
we can choose fixed $(\iota,\sigma)$ in \eqref{cal-1'} such that
\begin{equation*}
\iota+\sigma=\big(\frac{1}{2}-\varepsilon_0\big)(\gamma-1),\qquad \iota < \delta+\frac{\gamma-1}{2}, \qquad 0<\sigma<\big(\frac{3}{2}-\varepsilon_0\big)(\gamma-1)-\frac{1}{2}.
\end{equation*}
Hence, it follows from the above, \eqref{re version delta/2+e s3 D^2 U} with $\varepsilon=(\gamma-1)\left(\frac32-\varepsilon_0\right)
-\sigma-\frac\delta2>0$, and Lemmas \ref{c_0-c_1}--\ref{c_1} and \ref{hardy-inequality} that
\begin{equation}\label{4015}
\begin{aligned}
&\Big|\chi^\sharp\rho_0^{-(\frac{1}{2}+\varepsilon_0)(\gamma-1)}\big(D_{\bar\eta} U+\frac{2a_2}{2a_1+a_2}\frac{U}{\bar \eta}\big)\Big|_2\leq C\mathfrak{p}(c_0)\big(1+\big|\chi^\sharp\rho_0^{(\gamma-1)(\frac{1}{2}-\varepsilon_0)-\sigma}(U,D_{\bar\eta} U,\rho_0^{1-\delta}U_t)\big|_2\big)\\
&\leq C\mathfrak{p}(c_0)\big(1+\big|\chi^\sharp\rho_0^{(\gamma-1)(\frac{3}{2}-\varepsilon_0)-\sigma}(U,D_{\bar\eta} U,D_{\bar\eta}^2 U,U_t,D_{\bar\eta} U_{t})\big|_2\big)\leq C\mathfrak{p}(c_0).
\end{aligned}    
\end{equation}

Finally, based on \eqref{new-lp**}, using an argument similar to \eqref{L2Linfty}--\eqref{3...136}, together with the fact $(\gamma-1)(\frac{3}{2}-\varepsilon_0)>\frac{1}{2}$, \eqref{4015}, and Lemmas \ref{lemma-Lambda-guji}--\ref{c_1} and \ref{hardy-inequality}, yields that, for all $0\leq t\leq T_2$,
\begin{equation}\label{4016}
\begin{aligned}
\big|\chi^\sharp\rho_0^{(\gamma-1)(\frac{1}{2}-\varepsilon_0)}D_{\bar\eta}^2 U\big|_2&\leq C\big(\big|\chi^\sharp\rho_0^{(\gamma-1)(\frac{1}{2}-\varepsilon_0)}(U,D_{\bar\eta}U,U_t)\big|_2+
|(\bar\Lambda,\overline{D_\eta\Phi})|_\infty\big)  \\
&\quad+ C\mathfrak{p}(c_0)\Big|\chi^\sharp\rho_0^{-(\frac{1}{2}+\varepsilon_0)(\gamma-1)}\big(D_{\bar\eta} U+\frac{2a_2}{2a_1+a_2}\frac{U}{\bar \eta}\big)\Big|_2|\bar\Lambda|_\infty\\    
&\leq C\mathfrak{p}(c_0)\big(\big|\chi^\sharp\rho_0^{(\gamma-1)(\frac{3}{2}-\varepsilon_0)}(U,D_{\bar\eta}U,D_{\bar\eta}^2 U,U_t,D_{\bar\eta} U_t)\big|_2+1\big)\leq C\mathfrak{p}(c_0).
\end{aligned}
\end{equation}

\smallskip
\textbf{2.3.} Based on \eqref{new-lp*3}, we obtain from \eqref{11139}--\eqref{4016} and Lemmas \ref{lemma-Lambda-guji}--\ref{c_1} and \ref{hardy-inequality}, that, for all $t\in[0,T_2]$,
\begin{equation}\label{3369}
\begin{aligned}
&\,\big|\chi^\sharp\rho_0^{(\gamma-1)(\frac{3}{2}-\varepsilon_0)}D_{\bar\eta}^3 U\big|_2\\
&\leq C\big|\chi^\sharp \rho_0^{(\gamma-1)(\frac{3}{2}-\varepsilon_0)}(U,D_{\bar\eta} U,D_{\bar\eta}^2 U, \rho_0^{1-\delta}U_t,\rho_0^{1-\delta}D_{\bar\eta} U_{t})\big|_2\big(1+|\bar\Lambda|_\infty\big)\\
&\quad+C\big|\chi^\sharp \rho_0^{(\gamma-1)(\frac{1}{2}-\varepsilon_0)}(D_{\bar\eta}^2 U, D_{\bar\eta}U,U)\big|_2|\bar\Lambda|_\infty\\
&\quad +C\big(|\chi^\sharp  D_{\bar\eta}\bar\Lambda|_\infty+|\bar\Lambda|_\infty^2\big) \Big|\chi^\sharp \rho_0^{-(\frac{1}{2}+\varepsilon_0)(\gamma-1)}\big( D_{\bar\eta} U+ \frac{2a_2}{2a_1+ a_2}  \frac{U}{\bar\eta}\big)\Big|_2\\
&\quad  +C\big(|\chi^\sharp D_{\bar\eta}\bar\Lambda|_\infty+| \bar\Lambda|_\infty^2+|\bar\Lambda|_\infty|(\overline{D_\eta\Phi},D_{\bar\eta}(\overline{D_\eta\Phi}))|_\infty\big)\leq C\mathfrak{p}(c_0). 
\end{aligned}
\end{equation}

\smallskip
\textbf{Step 3.} Finally,  it follows from \eqref{1197}, \eqref{11112}, \eqref{11139}, \eqref{3369}, and Lemmas \ref{c_0-c_1}--\ref{c_1} and  \ref{sobolev-embedding}--\ref{hardy-inequality} that, for all $t\in[0,T_2]$,
\begin{equation*}
\begin{aligned}
&\,\big|(U,\sD_{\bar\eta}U)\big|_\infty \leq C\big|(U,\sD_{\bar\eta}U,\sD_{\bar\eta}^2U)\big|_1\\
&\leq C\sum_{j=0}^3 \big(\big|\chi rD_{\bar\eta}^j U\big|_2+\big|\chi^\sharp\rho_0^{(\gamma-1)(\frac{3}{2}-\varepsilon_0)}D_{\bar\eta}^j U\big|_2\big) +C\sum_{j=0}^2\Big|\zeta rD_{\bar\eta}^j(\frac{U}{\bar\eta})\Big|_2\leq C\mathfrak{p}(c_0).
\end{aligned}
\end{equation*}

This completes the proof.
\end{proof}

\begin{Lemma}\label{c_3}
For all $t\in[0,T_2]$,
\begin{equation*}
t\bar\cD(t,U)+\int_0^t\bar\cD(s,U)\,\mathrm{d}s \leq C\mathfrak{p}(c_0).
\end{equation*}
\end{Lemma}
\begin{proof}
We divide the proof into three steps.

\smallskip
\textbf{Step 1. $L^1(0,T_2)$-boundedness of $\bar\cD_{\mathrm{in}}(t,U)$.} 

\smallskip
\textbf{1.1.} 
First, applying $\partial_t$ to \eqref{new-lp*} yields that 
\begin{equation}\label{11140*}
D_{\bar\eta}\big(D_{\bar\eta} U_t+ \frac{2U_t}{\bar\eta}\big)=\sum_{i=9}^{11} J_i,
\end{equation}
where
\begin{align}
J_9 &:=D_{\bar\eta}^2 \bar U D_{\bar\eta} U+2D_{\bar\eta}(\frac{\bar U}{\bar\eta})\frac{ U}{\bar\eta}+2D_{\bar\eta} \bar U D_{\bar\eta}^2 U+2\big(D_{\bar\eta} \bar U+\frac{ \bar U}{\bar\eta}\big)D_{\bar\eta}(\frac{U}{\bar\eta}),\notag\\
J_{10}&:=\frac{1}{2a_1+a_2}\bar\varrho^{1-\delta}\big(U_{tt}-(1-\delta)\big(D_{\bar\eta}\bar U+2\frac{\bar  U}{\bar\eta}\big)U_t\big) - \frac{\delta}{\gamma-1}\frac{\bar\Lambda_t}{\bar\varrho^{\gamma-1}}\big(D_{\bar\eta} U+\frac{2a_2}{2a_1+a_2}\frac{U}{\bar \eta}\big)\notag\\
&\quad  -\delta\frac{\bar\Lambda}{\bar\varrho^{\gamma-1}} \Big(\big(\frac{\gamma-2}{\gamma-1}D_{\bar\eta}\bar U+2\frac{\bar U}{\bar\eta}\big)D_{\bar\eta} U\label{I789}\\
&\quad +\frac{2a_2}{2a_1+a_2}\big(D_{\bar\eta}\bar U+\frac{2\gamma-3}{\gamma-1}\frac{\bar U}{\bar\eta}\big)\frac{U}{\bar\eta}+\frac{1}{\gamma-1} \big(D_{\bar\eta}U_t+\frac{2a_2}{2a_1+a_2}\frac{U_t}{\bar\eta}\big)\Big),\notag\\
J_{11}&:=\frac{\bar\varrho^{1-\delta}}{2a_1+a_2}\Big(\frac{A\gamma}{\gamma-1}\big(\bar\Lambda_t-(1-\delta)\big(D_{\bar\eta}\bar U+2\frac{\bar U}{\bar\eta}\big)\bar\Lambda\big)-\overline{D_\eta\Phi}\big((1-\delta)\big(D_{\bar\eta}\bar U+2\frac{\bar U}{\bar\eta}\big) +\frac{2\bar U}{\bar\eta} \big)\Big).\notag
\end{align}
Then, by the fact that $\rho_0^{\gamma-1}\sim 1-r$, \eqref{38}, \eqref{1197}--\eqref{11123},  Lemmas \ref{lemma-useful1}--\ref{lemma-Lambda-guji}, and the H\"older inequality, we have
\begin{align}
&\begin{aligned}\label{11142}
|\zeta rJ_{9}|_2&\leq C\big|\zeta_\frac{5}{8} r^\frac{1}{2}\sD_{\bar\eta}^2 \bar U\big|_4\big|\zeta r^\frac{1}{2}\sD_{\bar\eta}U\big|_4 +C\big|\sD_{\bar\eta} \bar U\big|_\infty\big|\zeta r\sD_{\bar\eta}^2  U\big|_2\leq C\mathfrak{p}(c_1), 
\end{aligned}\\
&\begin{aligned}
|\zeta rJ_{10}|_2&\leq C(|\zeta rU_{tt}|_2+|\sD_{\bar\eta}U|_\infty|\zeta rU_t|_2)+C|\zeta_{\frac{5}{8}} r^\frac{1}{2}\bar\Lambda_t|_4|\zeta r^\frac{1}{2}\sD_{\bar\eta} U|_4\\
&\quad +C|\bar\Lambda|_\infty\big(\big|\sD_{\bar\eta} \bar U\big|_\infty|\zeta r \sD_{\bar\eta}U|_2+|\zeta r \sD_{\bar\eta} U_t|_2\big)\leq C(|\zeta rU_{tt}|_2+\mathfrak{p}(c_1)),
\end{aligned}\\
&\begin{aligned}\label{11144}
|\zeta rJ_{11}|_2\leq C\big(|\bar\Lambda|_\infty \big|\sD_{\bar\eta} \bar U\big|_\infty + |\zeta r^\frac{1}{2}\bar\Lambda_t|_4+|\overline{D_\eta\Phi}|_\infty\big|\sD_{\bar\eta} \bar U\big|_\infty\big)\leq C\mathfrak{p}(c_1).
\end{aligned}
\end{align}

Hence, collecting \eqref{11140*} and \eqref{11142}--\eqref{11144}, together with Lemmas \ref{c_1} and \ref{im-1}, gives that, for all $t\in [0,T_2]$,
\begin{equation}\label{11145}
\int_0^t\big|\zeta r \sD_{\bar\eta}^2 U_t\big|_2^2\,\mathrm{d}s\leq C\Big(\int_0^t|\zeta rU_{tt}|_2^2\,\mathrm{d}s+\mathfrak{p}(c_1)t\Big)\leq C\mathfrak{p}(c_0).
\end{equation}

\smallskip
\textbf{1.2.} It follows from \eqref{new-lp*3} by applying $D_{\bar\eta}$ that
\begin{equation}\label{11146}
\zeta D_{\bar\eta}^3\big(D_{\bar\eta} U+ \frac{2U}{\bar\eta}\big)=\zeta J_{12},
\end{equation}
where $J_{12}$ has controls of the following form: 
\begin{equation*}
\begin{aligned}
|\zeta J_{12}|&\leq C\zeta \Big((|D_{\bar\eta}^2U_t|+|\bar\Lambda||D_{\bar\eta}U_t|+|D_{\bar\eta}\bar\Lambda||U_t|+|\bar\Lambda|^2|U_t|)\\
&\quad+|\bar\Lambda|\big(|D_{\bar\eta}^3 U|+|D_{\bar\eta}^2(\frac{U}{\bar\eta})|\big)+\big(|D_{\bar\eta}^2\bar\Lambda|+|\bar\Lambda D_{\bar\eta}\bar\Lambda|+|\bar\Lambda|^3\big)\big|\sD_{\bar\eta} U\big|\\
&\quad +(|D_{\bar\eta}\bar\Lambda| +|\bar\Lambda|^2)\big(|D_{\bar\eta}^2 U| +\big|D_{\bar\eta}(\frac{U}{\bar\eta})\big|\big)\\
&\quad +|D_{\bar\eta}^2\bar\Lambda|+|\bar\Lambda|^3+|\bar\Lambda D_{\bar\eta} \bar\Lambda|+(|D_{\bar\eta}\bar\Lambda|+|\bar\Lambda|^2)|\overline{D_\eta \Phi}|+|\bar\Lambda||D_{\bar\eta}\overline{D_\eta \Phi}|+|D^2_{\bar\eta}\overline{D_\eta \Phi}|\Big).
\end{aligned}
\end{equation*}
Then, due to the fact that $\rho_0^{\gamma-1}\sim 1-r$, \eqref{38}, \eqref{11123}, Lemmas \ref{lemma-Lambda-guji}, \ref{c_1-c_2}, and \ref{hardy-inequality}, and the H\"older inequality, we obtain that, for all $t\in [0,T_2]$, 
\begin{equation}\label{11149}
\begin{aligned}
|\zeta rJ_{12}|_2&\leq C|\zeta rD_{\bar\eta}^2 U_t|_2+|\bar\Lambda|_\infty|\zeta rD_{\bar\eta}U_t|_2+|\zeta_{\frac{5}{8}}r^{\frac{1}{2}}D_{\bar\eta}\bar\Lambda|_4|\zeta r^{\frac{1}{2}} U_t|_4+|\bar\Lambda|^2_\infty|\zeta r U_t|_2\\
&\quad+C\big(|\zeta rD_{\bar\eta}^2 \bar\Lambda|_2+|\bar\Lambda|_\infty|\zeta rD_{\bar\eta} \bar\Lambda|_2+|\bar\Lambda|_\infty^3\big)(|\sD_{\bar\eta} U|_\infty+1)\\
&\quad +C\big(\big|\zeta_\frac{5}{8} r^\frac{1}{2}D_{\bar\eta} \bar\Lambda\big|_4+|\bar\Lambda|_\infty+|\bar\Lambda|_\infty^2\big)|\zeta r^\frac{1}{2} \sD_{\bar\eta}^2 U|_4+\big|(\overline{D_\eta\Phi},D_{\bar\eta}\overline{D_\eta\Phi})\big|_\infty\\
&\quad +C |\bar\Lambda|_\infty\Big|\zeta r \big(D_{\bar\eta}^3 U,D_{\bar\eta}^2(\frac{U}{\bar\eta})\big)\Big|_2\leq C(|\zeta rD_{\bar\eta}^2 U_t|_2+\mathfrak{p}(c_0)),
\end{aligned}
\end{equation}
where the following inequality has also been used:
\begin{equation*}
    |\zeta r^{\frac{1}{2}}(\sD_{\bar\eta}^2U,U_t)|_4\leq \big|\zeta r^{\frac{5}{4}}\big( D_{\bar\eta}^2U,D_{\bar\eta}(\frac{U}{\bar\eta}),D_{\bar\eta}^3U,D_{\bar\eta}^2(\frac{U}{\bar\eta}),U_t,D_{\bar\eta}U_t \big)\big|_2\leq C\mathfrak{p}(c_0).
\end{equation*}

Hence, collecting \eqref{11146}--\eqref{11149} gives that, for all $t\in[0,T_2]$,
\begin{equation}\label{11150}
\Big|\zeta rD_{\bar\eta}^3\big(D_{\bar\eta} U+ \frac{2U}{\bar\eta}\big)\Big|_2\leq C(|\zeta rD_{\bar\eta}^2 U_t|_2+\mathfrak{p}(c_0)).
\end{equation}

Next, we multiply \eqref{new-lp*} by $\frac{1}{\bar\eta}$ and apply $D_{\bar\eta}$ to the resulting equality to obtain 
\begin{equation*}
\begin{aligned}
\zeta D_{\bar\eta}\Big(\frac{1}{\bar\eta}D_{\bar\eta}\big(D_{\bar\eta} U+ \frac{2U}{\bar\eta}\big)\Big)&=\zeta J_{13},
\end{aligned}
\end{equation*}
where $J_{13}$ has controls of the following form:
\begin{equation*}
\begin{aligned}
|\zeta J_{13}|&\leq C\zeta \Big(\big|D_{\bar\eta}\big(\frac{U_t}{\bar\eta}\big)\big|+|\bar\Lambda|\big|\frac{U_t}{\bar\eta}\big|+\big(\big|D_{\bar\eta}\big(\frac{\bar\Lambda}{\bar\eta}\big)\big|+|\bar\Lambda|\big|\frac{\bar\Lambda}{\bar\eta}\big|\big)\big(|\sD_{\bar\eta}U|+1\big)\\
&\quad +\big|\frac{1}{\bar\eta}\sD_{\bar\eta}^2U\big||\bar\Lambda|+\Big|\big(\frac{1}{\bar\eta^2}\overline{D_\eta\Phi}, \frac{1}{\bar\eta} D_{\bar\eta}\overline{D_\eta\Phi}\big)\Big|+|\bar\Lambda|\big|\frac{1}{\bar\eta}\overline{D_\eta\Phi}\big|\Big).
\end{aligned}
\end{equation*}
Then, following the calculations \eqref{11146}--\eqref{11150}, we can similarly obtain 
\begin{equation}\label{11150-}
\Big|\zeta rD_{\bar\eta}\Big(\frac{1}{\bar\eta}D_{\bar\eta}\big(D_{\bar\eta} U+ \frac{2U}{\bar\eta}\big)\Big)\Big|_2\leq C\big(\big|\zeta r D_{\bar\eta}(\frac{U_t}{\bar\eta})\big|_2+\mathfrak{p}(c_0)\big).
\end{equation}

Finally, combining \eqref{11150}--\eqref{11150-}, along with Lemma \ref{im-1} and \eqref{11145}, implies that, for all $t\in[0,T_2]$,
\begin{equation}
\begin{aligned}
\int_0^t\big|\zeta r \sD_{\bar\eta}^4 U \big|_2^2\,\mathrm{d}s \leq C\int_0^t\big|\zeta r \sD_{\bar\eta}^2 U_t\big|_2^2\,\mathrm{d}s+Ct \mathfrak{p}(c_0)\leq  C\mathfrak{p}(c_0).
\end{aligned}
\end{equation}

\smallskip
\textbf{Step 2. $L^1(0,T_2)$-boundedness of $\bar\cD_{\mathrm{ex}}(t,U)$.} 

\smallskip
\textbf{2.1.} First, applying $\partial_t$ to \eqref{old3...136} yields
\begin{equation*}
\begin{aligned}
\chi^\sharp D_{\bar\eta} U_t&=\chi^\sharp J_{14},
\end{aligned}
\end{equation*}
where  $J_{14}$ has controls of the following form:
\begin{equation*}
\begin{aligned}
|\chi^\sharp   J_{14}| &\leq C\chi^\sharp\Big(|\sD_{\bar\eta}U||\sD_{\bar\eta}\bar U|+|U_t|+|\rho_0^{\gamma-\delta}\sD_{\bar\eta}\bar U|+\frac{1}{\rho_0^\delta}
\Big|\int_r^1 |\sD_{\bar\eta}\bar U|\frac{\hat r^2\rho_0}{\bar\eta^2}\overline{D_\eta\Phi}\mathrm{d}\hat r\Big|\\ 
&\quad\, +\frac{1}{\rho_0^\delta}\Big|\int_r^1  \rho_0^{\delta}(|\sD_{\bar\eta} U| |\sD_{\bar\eta}\bar U|+|\sD_{\bar\eta}U_t|)\mathrm{d}\hat r\Big|+\frac{1}{\rho_0^\delta}\Big|\int_r^1\rho_0 (|\sD_{\bar\eta}\bar U||U_t|,U_{tt})\,\mathrm{d}\hat r\Big|\Big).
\end{aligned}
\end{equation*}

Then, using an argument similar to \eqref{L2Linfty}--\eqref{3...136}, together with the fact that $(3-2\varepsilon_0)(\gamma-1)-\delta>0$, Lemmas \ref{lemma-useful1}, \ref{c_1-c_2}, and \ref{hardy-inequality}, and the H\"older inequality, we can obtain that, 
\begin{align}
&\,\big|\chi^\sharp
\rho_0^{(\gamma-1)(\frac{1}{2}-\varepsilon_0)}
D_{\bar\eta}U_t\big|_2\notag\\
&\leq C\mathfrak{p}(c_0)
\big(
|\sD_{\bar\eta}U|_\infty
|\sD_{\bar\eta}\bar U|_\infty
+
|\chi^\sharp
\rho_0^{(\gamma-1)(\frac{1}{2}-\varepsilon_0)}
U_t|_2
\big)\notag\\
&\quad+C\mathfrak{p}(c_0)
|\sD_{\bar\eta}\bar U|_\infty
\bigg|
\chi^\sharp
\rho_0^{(\gamma-1)(\frac{1}{2}-\varepsilon_0)-\delta}
\bigg(\int_r^1\rho_0^\delta(\hat r)\,d\hat r\bigg)^{\frac{1}{2}}
\bigg|_2
|\overline{D_\eta\Phi}|_\infty\notag\\
&\quad+C\mathfrak{p}(c_0)
\bigg|
\chi^\sharp
\rho_0^{(\gamma-1)(\frac{1}{2}-\varepsilon_0)-\delta}
\bigg(\int_r^1\rho_0^\delta(\hat r)\,d\hat r\bigg)^{\frac{1}{2}}
\bigg|_2
\big(
|\sD_{\bar\eta}U|_\infty
|\sD_{\bar\eta}\bar U|_\infty
+
|\chi^\sharp\rho_0^{\frac{\delta}{2}}
\sD_{\bar\eta}U_t|_2\notag\\
&\qquad\qquad\qquad\qquad
+
|\chi^\sharp\rho_0^{\frac{1}{2}}
(|\sD_{\bar\eta}\bar U|_\infty U_t,U_{tt})|_2
\big)\notag\\
&\leq C\mathfrak{p}(c_1)
+C\mathfrak{p}(c_0)
\big|\chi^\sharp\rho_0^{\frac{1}{2}}U_{tt}\big|_2.
\label{9997**}
\end{align}
% \begin{align}
% &\,\big|\chi^\sharp\rho_0^{(\gamma-1)(\frac{1}{2}-\varepsilon_0)}D_{\bar\eta} U_t\big|_2\notag\\
% &\leq |\sD_{\bar\eta}U|_\infty|\sD_{\bar\eta}\bar U|_\infty+|\chi^\sharp\rho_0^{(\gamma-1)(\frac{1}{2}-\varepsilon_0)}U_t|_2+|\sD_{\bar\eta}\bar U|_\infty\big|\chi^\sharp\rho_0^{(\gamma-1)(\frac{1}{2}-\varepsilon_0)-\delta}|\chi^\sharp\rho_0^{\frac{\delta}{2}}|_2\big|_2|\overline{D_\eta\Phi}|_\infty\notag\\
% &\quad+\big|\chi^\sharp\rho_0^{(\gamma-1)(\frac{1}{2}-\varepsilon_0)-\delta}|\chi^\sharp\rho_0^{\frac{\delta}{2}}|_2\big|_2\big(|\sD_{\bar\eta}U|_\infty|\sD_{\bar\eta}\bar U|_\infty+|\chi^\sharp\rho_0^{\frac{\delta}{2}}\sD_{\bar\eta}U_t|_2+|\chi^\sharp \rho_0^{\frac{1}{2}}(|\sD_{\bar\eta}\bar U|_\infty U_t,U_{tt})|_2\big)\notag\\
% &\leq C\mathfrak{p}(c_1)+C\mathfrak{p}(c_0)\big|\chi^\sharp\rho_0^\frac{1}{2}U_{tt}\big|_2.\label{9997**} 
% \end{align}

As a consequence, recalling \eqref{11140*}--\eqref{I789}, we obtain from \eqref{varepsilon0}, \eqref{38}, \eqref{4015}--\eqref{4016}, \eqref{9997**}, and Lemmas \ref{lemma-useful1}--\ref{lemma-Lambda-guji} and \ref{hardy-inequality} that, for all $t\in [0, T_2]$,
\begin{align*}
&\begin{aligned} 
\big|\chi^\sharp \rho_0^{(\gamma-1)(\frac{3}{2}-\varepsilon_0)}J_9\big|_2&\leq C\big|\chi^\sharp \rho_0^{\gamma-1} \sD_{\bar\eta}^2 \bar U\big|_\infty \big|\chi^\sharp \rho_0^{(\gamma-1)(\frac{1}{2}-\varepsilon_0)}(U, D_{\bar\eta} U)\big|_2\\
&\quad +C|(\bar U,D_{\bar\eta}\bar U)|_\infty\big |\chi^\sharp\rho_0^{(\gamma-1)(\frac{3}{2}-\varepsilon_0)}(U,D_{\bar\eta} U,D_{\bar\eta}^2 U)\big|_2\leq C\mathfrak{p}(c_1),
\end{aligned}\\
&\begin{aligned}
\big|\chi^\sharp \rho_0^{(\gamma-1)(\frac{3}{2}-\varepsilon_0)}J_{10}\big|_2&\leq C\big|\chi^\sharp  \rho_0^{(\gamma-1)(\frac{3}{2}-\varepsilon_0)}(U_{tt},|\sD_{\bar\eta}\bar U|_\infty U_t)\big|_2 \\
&\quad +C |\chi^\sharp \bar\Lambda_t|_\infty\big|\chi^\sharp  \rho_0^{(\gamma-1)(\frac{1}{2}-\varepsilon_0)}(D_{\bar\eta}U,U)\big|_2 \\
&\quad +C|\bar\Lambda|_\infty\big(|(\bar U,D_{\bar\eta}\bar U)|_\infty\big|\chi^\sharp \rho_0^{(\gamma-1)(\frac{1}{2}-\varepsilon_0)}(D_{\bar\eta}U,U)\big|_2\\
&\quad+\big|\chi^\sharp \rho_0^{(\gamma-1)(\frac{1}{2}-\varepsilon_0)}(D_{\bar\eta}U_t,U_t)\big|_2\big)\leq C\big(\mathfrak{p}(c_1)+\mathfrak{p}(c_0)\big|\chi^\sharp\rho_0^\frac{1}{2}U_{tt}\big|_2\big),
\end{aligned}\\
&\begin{aligned}
\big|\chi^\sharp \rho_0^{(\gamma-1)(\frac{3}{2}-\varepsilon_0)}J_{11}\big|_2&\leq C\big|\chi^\sharp \rho_0^{(\gamma-1)(\frac{3}{2}-\varepsilon_0)}\big|_2 |(1,\bar U,D_{\bar\eta}\bar U)|_\infty|\chi^\sharp(\bar\Lambda,\bar\Lambda_t,\overline{D_\eta\Phi})|_\infty\leq C\mathfrak{p}(c_1), 
\end{aligned}
\end{align*}
which, along with \eqref{11140*} and Lemmas \ref{c_0-c_1}--\ref{c_1}, leads to
\begin{equation}\label{9997***}
\begin{aligned}
\big|\chi^\sharp\rho_0^{(\gamma-1)(\frac{3}{2}-\varepsilon_0)}D_{\bar\eta}^2 U_t\big|_2&\leq C\big|\chi^\sharp\rho_0^{(\gamma-1)(\frac{3}{2}-\varepsilon_0)}(U_t,D_{\bar\eta} U_t)\big|_2+C\big(\mathfrak{p}(c_1)+\mathfrak{p}(c_0)\big|\chi^\sharp\rho_0^\frac{1}{2}U_{tt}\big|_2\big)\\
&\leq C\big(\mathfrak{p}(c_1)+\mathfrak{p}(c_0)\big|\chi^\sharp\rho_0^\frac{1}{2}U_{tt}\big|_2\big).
\end{aligned}
\end{equation}
Finally, integrating the above over $[0,t]$, together with Lemma \ref{c_1}, implies that, for all $t\in[0,T_2]$,
\begin{equation}\label{9997+}
\int_0^t\big|\chi^\sharp\rho_0^{(\gamma-1)(\frac{3}{2}-\varepsilon_0)}D_{\bar\eta}^2 U_t\big|_2^2\,\mathrm{d}s\leq C\mathfrak{p}(c_1)t+C\mathfrak{p}(c_0)\int_0^t\big|\chi^\sharp\rho_0^\frac{1}{2}U_{tt}\big|_2^2\,\mathrm{d}s\leq C\mathfrak{p}(c_0).
\end{equation}

\smallskip
\textbf{2.2.}
First, recall the derivation of \eqref{xxxx-l}:
\begin{equation}\label{xxxx-l'}
\bar\cT_{\mathrm{cross}}:=(D_{\bar\eta}^3U)_r+\big(\frac{\delta}{\gamma-1}+2\big)\frac{(\rho_0^{\gamma-1})_r}{\rho_0^{\gamma-1}}D_{\bar\eta}^3U =\sum_{i=15}^{18}J_i,
\end{equation}
where
\begin{align} 
J_{15}&:=-\big(1+\frac{2\delta}{\gamma-1} \big)\frac{\bar\eta_r}{\bar\varrho^{\gamma-1}}D_{\bar\eta} \bar\Lambda D_{\bar\eta}^2 U-\frac{\delta}{\gamma-1}\frac{\bar\eta_r}{\bar\varrho^{\gamma-1}} D_{\bar\eta}^2\bar\Lambda \big(D_{\bar\eta} U+\frac{2a_2}{2a_1+a_2}\frac{U}{\bar\eta}  \big)\notag\\
&\quad +(\delta+2(\gamma-1))\bar\eta_r\frac{D_{\bar\eta}\bar{\mathscr{J}}}{\bar{\mathscr{J}}}D_{\bar\eta}^3U,\notag\\
J_{16}&:=-2\frac{\bar\eta_r}{\bar\varrho^{\gamma-1}} \big(1+\frac{\delta}{\gamma-1}\frac{2a_2}{2a_1+a_2}\big)D_{\bar\eta} \bar\Lambda D_{\bar\eta}\big(\frac{U}{\bar\eta}\big)\notag\\
&\quad -2\frac{\bar\eta_r}{\bar\varrho^{\gamma-1}}\Big(\big(2+\frac{\delta}{\gamma-1}\frac{a_2}{2a_1+a_2}\big) \bar\Lambda D_{\bar\eta}^2\big(\frac{U}{\bar\eta}\big)+ \bar\varrho^{\gamma-1} D_{\bar\eta}^3\big(\frac{U}{\bar\eta}\big)\Big),\notag\\
J_{17}&:=\frac{1}{2a_1+a_2}\frac{\bar\eta_r}{\bar\varrho^{\gamma+\delta-2}}\Big(\frac{\gamma-\delta}{\gamma-1}\big(D_{\bar\eta} \bar\Lambda\, +\frac{1-\delta}{\gamma-1}\frac{\bar\Lambda^2}{\bar\varrho^{\gamma-1}}\big) U_t+2\frac{\gamma-\delta}{\gamma-1}\bar\Lambda D_{\bar\eta} U_t+ \bar\varrho^{\gamma-1} D_{\bar\eta}^2 U_t\Big),\notag\\
J_{18}&:=\frac{A\gamma}{(2a_1+a_2)(\gamma-1)} \bar\eta_r \bar\varrho^{1-\delta-2(\gamma-1)}\Big( \bar\varrho^{2(\gamma-1)} D_{\bar\eta}^2 \bar\Lambda+3\frac{\gamma-\delta}{\gamma-1}\bar\varrho^{\gamma-1} \bar\Lambda D_{\bar\eta} \bar\Lambda +\frac{\gamma-\delta}{\gamma-1}\frac{1-\delta}{\gamma-1}\bar\Lambda^3\Big)\notag\\
&\quad+\frac{\bar\eta_r\bar\varrho^{-2(\gamma-1)+1-\delta}}{2a_1+a_2}\frac{\gamma-\delta}{\gamma-1}\bar\varrho^{\gamma-1}\big(D_{\bar\eta} \bar\Lambda\, +\frac{1-\delta}{\gamma-1}\frac{\bar\Lambda^2}{\bar\varrho^{\gamma-1}}\big)\overline{D_\eta\Phi}\notag\\
&\quad+\frac{\bar\eta_r\bar\varrho^{-2(\gamma-1)+1-\delta}}{2a_1+a_2}\Big(2\frac{\gamma-\delta}{\gamma-1}\bar\varrho^{\gamma-1}\bar\Lambda D_{\bar\eta}\overline{D_\eta\Phi}+\bar\varrho^{2(\gamma-1)}D_{\bar\eta}^2\overline{D_\eta\Phi}\Big).\label{rr1-rr3}
\end{align}

For the right-hand side of the above, it follows from the facts that
\begin{equation*}
\rho_0^{\gamma-1}\sim 1-r,\quad \big(\frac{1}{2}-\varepsilon_0\big)(\gamma-1)+1-\delta>0,\quad \big(\frac{3}{2}-\varepsilon_0\big)(\gamma-1)-2(\gamma-1)+1-\delta>-\frac{\gamma-1}{2},
\end{equation*} 
\eqref{varepsilon0}, \eqref{38}, \eqref{4016}, and Lemmas \ref{lemma-useful1}--\ref{c_1-c_2} and \ref{hardy-inequality} that, for all $t\in [0,T_2]$,
\begin{align}
&\begin{aligned}
\big|\chi^\sharp \rho_0^{(\gamma-1)(\frac{3}{2}-\varepsilon_0)}J_{15}\big|_2&\leq C|\chi^\sharp D_{\bar\eta}\bar\Lambda|_\infty\big|\chi^\sharp \rho_0^{(\gamma-1)(\frac{1}{2}-\varepsilon_0)} D_{\bar\eta}^2U\big|_2+\big|\chi^\sharp \rho_0^{(\gamma-1)(\frac{3}{2}-\varepsilon_0)}D_{\bar\eta}^3U\big|_2\\
&\quad+C|\chi^\sharp \sD_{\bar\eta} U|_\infty\big|\chi^\sharp \rho_0^{(\gamma-1)(\frac{1}{2}-\varepsilon_0)}D_{\bar\eta}^2 \bar\Lambda\big|_2 \leq C\mathfrak{p}(c_0),\notag
\end{aligned}\\
&\begin{aligned}
\big|\chi^\sharp \rho_0^{(\gamma-1)(\frac{3}{2}-\varepsilon_0)}J_{16}\big|_2&\leq C|\chi^\sharp D_{\bar\eta}\bar\Lambda|_\infty|\chi^\sharp(U,D_{\bar\eta} U)|_\infty\\
&\quad + C(1+|\bar\Lambda|_\infty) \big|\chi^\sharp \rho_0^{(\gamma-1)(\frac{3}{2}-\varepsilon_0)} (U,D_{\bar\eta} U,D_{\bar\eta}^2 U,D_{\bar\eta}^3 U)\big|_2\leq C\mathfrak{p}(c_0),\notag
\end{aligned}\\
&\begin{aligned}[t]\label{QQQ}
\big|\chi^\sharp \rho_0^{(\gamma-1)(\frac{3}{2}-\varepsilon_0)}J_{17}\big|_2&\leq C(|\chi^\sharp D_{\bar\eta} \bar\Lambda|_\infty+|\bar\Lambda|_\infty)\big|\chi^\sharp \rho_0^{(\gamma-1)(\frac{3}{2}-\varepsilon_0)} (U_t,D_{\bar\eta} U_t)\big|_2\\
&\quad +C(1+|\bar\Lambda|_\infty^2)\big|\chi^\sharp \rho_0^{(\gamma-1)(\frac{3}{2}-\varepsilon_0)} (D_{\bar\eta} U_t,D_{\bar\eta}^2 U_t)\big|_2\\
&\leq C\mathfrak{p}(c_0)\big(1+\big|\chi^\sharp \rho_0^{(\gamma-1)(\frac{3}{2}-\varepsilon_0)} D_{\bar\eta}^2 U_t\big|_2\big),
\end{aligned}\\
&\begin{aligned}
\big|\chi^\sharp \rho_0^{(\gamma-1)(\frac{3}{2}-\varepsilon_0)}J_{18 }\big|_2&\leq C\big(\big|\chi^\sharp \rho_0^{(\gamma-1)(\frac{1}{2}-\varepsilon_0)}D_{\bar\eta}^2\bar\Lambda\big|_2+|\chi^\sharp D_{\bar\eta}\bar\Lambda|_\infty|\bar\Lambda|_\infty+|\bar\Lambda|^3_\infty\big)\\
&\quad+\big|(\chi^\sharp D_{\bar\eta}\bar\Lambda,\bar\Lambda^2,\bar\Lambda)\big|_\infty\big|\big(1,\overline{D_\eta\Phi},D_{\bar\eta}\overline{D_\eta\Phi}\big)\big|_\infty+|\chi^\sharp \rho_0^{\gamma-1} D_{\bar\eta}^2\overline{D_\eta\Phi}|_\infty\leq C\mathfrak{p}(c_0).\notag
\end{aligned}
\end{align}

Therefore, substituting \eqref{QQQ} into \eqref{xxxx-l'} gives
\begin{equation*}
\big|\zeta^\sharp\rho_0^{(\gamma-1)(\frac{3}{2}-\varepsilon_0)}\bar\cT_{\mathrm{cross}}\big|_2\leq C\mathfrak{p}(c_0)\big(1+\big|\chi^\sharp \rho_0^{(\gamma-1)(\frac{3}{2}-\varepsilon_0)} D_{\bar\eta}^2 U_t\big|_2\big),
\end{equation*}
which, along with  Lemmas \ref{c_1-c_2} and \ref{prop2.1}, yields
\begin{equation*}
\begin{aligned}
\big|\chi^\sharp\rho_0^{(\gamma-1)(\frac{3}{2}-\varepsilon_0)}D_{\bar\eta}^4U\big|_2&\leq \big|(\zeta-\chi)\rho_0^{(\gamma-1)(\frac{3}{2}-\varepsilon_0)}D_{\bar\eta}^4U\big|_2+\big|\zeta^\sharp\rho_0^{(\gamma-1)(\frac{3}{2}-\varepsilon_0)}D_{\bar\eta}^4U\big|_2\\
&\leq C\bar\cD_{\mathrm{in}}(t,U)^\frac{1}{2}+C\big|\zeta^\sharp\rho_0^{(\gamma-1)(\frac{3}{2}-\varepsilon_0)}\bar\cT_{\mathrm{cross}}\big|_2+C\mathfrak{p}(c_0) \big|\rho_0^{(\gamma-1)(\frac{3}{2}-\varepsilon_0)}D_{\bar\eta}^3U\big|_2\\
&\leq  C\mathfrak{p}(c_0)\big(1+\bar\cD_{\mathrm{in}}(t,U)^\frac{1}{2}+\big|\chi^\sharp \rho_0^{(\gamma-1)(\frac{3}{2}-\varepsilon_0)} D_{\bar\eta}^2 U_t\big|_2\big).
\end{aligned}
\end{equation*}

Finally, it follows from the $L^1(0,T_2)$-estimate of $\bar\cD_{\mathrm{in}}(t,U)$ and \eqref{9997+} that, for all $t\in[0,T_2]$,
\begin{equation*}
\begin{aligned}
&\,\int_0^t\big|\chi^\sharp\rho_0^{(\gamma-1)(\frac{3}{2}-\varepsilon_0)}D_{\bar\eta}^4U\big|_2^2\,\mathrm{d}s\\
&\leq  C\mathfrak{p}(c_0)
\Big(t+\int_0^t\big(\bar\cD_{\mathrm{in}}(s,U)+\big|\chi^\sharp \rho_0^{(\gamma-1)(\frac{3}{2}-\varepsilon_0)} D_{\bar\eta}^2 U_t\big|_2^2\big)\,\mathrm{d}s\Big)\leq C\mathfrak{p}(c_0).
\end{aligned}
\end{equation*}

\smallskip
\textbf{Step 3. $L^\infty(0,T_2)$-boundedness for $t\bar\cD(t,U)$.} Following a similar argument in Steps 1--2 gives
\begin{equation*}
\begin{aligned}
&\sqrt{t}\big|\zeta r \sD_{\bar\eta}^2 U_t\big|_2\leq C\sqrt{t}(|\zeta rU_{tt}|_2+\mathfrak{p}(c_1)),\qquad\sqrt{t}\big|\zeta r \sD_{\bar\eta}^4 U\big|_2\leq C\sqrt{t}\big(\big|\zeta r \sD_{\bar\eta}^2 U_t\big|_2+\mathfrak{p}(c_0)\Big),\\
&\sqrt{t}\big|\chi^\sharp\rho_0^{(\gamma-1)(\frac{3}{2}-\varepsilon_0)}D_{\bar\eta}^2 U_t\big|_2\leq C\sqrt{t}\big(\mathfrak{p}(c_0)\big|\chi^\sharp\rho_0^\frac{1}{2}U_{tt}\big|_2+\mathfrak{p}(c_1)\big),\\
&\sqrt{t}\big|\chi^\sharp\rho_0^{(\gamma-1)(\frac{3}{2}-\varepsilon_0)}D_{\bar\eta}^4U\big|_2\leq C\sqrt{t}\mathfrak{p}(c_0)\big(\big|\chi^\sharp \rho_0^{(\gamma-1)(\frac{3}{2}-\varepsilon_0)} D_{\bar\eta}^2 U_t\big|_2+\bar\cD_{\mathrm{in}}(t,U)^\frac{1}{2}+1\big).
\end{aligned}
\end{equation*}

Therefore, this, combined with Lemma \ref{c_1}, leads to the desired result of this lemma.
\end{proof}

\begin{Lemma}\label{lemma-12.11}
For any $0\leq t\leq T_3:=\min\{T_2,(2C\mathfrak{p}(c_0))^{-1},(\mathfrak{p}(2c_1))^{-2}\}$,
\begin{equation}\label{39''-0}
\cE(t,U)+t\cD(t,U)+\int_0^{t} \cD(s,U)\,\mathrm{d}s\leq C\mathfrak{p}(c_0),\qquad|\eta_r(t)-1|_\infty+\Big|\frac{\eta(t)}{r}-1\Big|_\infty\leq \frac{1}{2}.
\end{equation}
\end{Lemma}
\begin{proof}
Collecting the  estimates established  in Lemmas \ref{c_0-c_1}--\ref{c_3}, we see that, for all $t\in[0,T_2]$,
\begin{equation}\label{39''}
\bar\cE(t,U)+t\bar\cD(t,U)+\int_0^{t} \bar\cD(s,U)\,\mathrm{d}s\leq C\mathfrak{p}(c_0).
\end{equation}

First, \eqref{given-flow}, together  with \eqref{39} and Lemma \ref{c_1-c_2}, yields that, for any $0\leq t\leq T_3$, 
\begin{equation}\label{eta-etar}
|\eta_r(t)-1|_\infty+\Big|\frac{\eta(t)}{r}-1\Big|_\infty\leq \int_0^t |\sD_r U|_\infty\,\mathrm{d}s \leq C\int_0^t |\sD_{\bar\eta}U|_\infty \,\mathrm{d}s\leq C\mathfrak{p}(c_0)T_3\leq \frac{1}{2}.
\end{equation}

Next, let $(\mathring{\mathcal E},\mathring{\mathcal D})$ be defined by \eqref{E-1a}--\eqref{D-1a}. Then, based on \eqref{39''}--\eqref{eta-etar}, we obtain from \eqref{39} and Lemma \ref{lemma-gaowei} that,
\begin{equation*}
\begin{aligned}
% \mathring\cE(t,U)+t\mathring\cD(t,U)+\int_0^{t} \mathring\cD(s,U)\,\mathrm{d}s\leq C\mathfrak{p}(c_0)\Big(\bar\cE(t,U)+t\bar\cD(t,U)+\int_0^{t} \bar\cD(s,U)\,\mathrm{d}s \Big).
% \cE(t,U)+t\cD(t,U)+\int_0^{t} \cD(s,U)\,\mathrm{d}s\leq C\mathfrak{p}(c_0)\Big(\mathring\cE(t,U)+t\mathring\cD(t,U)+\int_0^{t} \mathring\cD(s,U)\,\mathrm{d}s\Big),
&\,\sup_{0\leq s\leq T_3}
\big(\mathring{\mathcal E}(s,U)
+s\mathring{\mathcal D}(s,U)\big)
+\int_0^{T_3}\mathring{\mathcal D}(s,U)\,\mathrm{d}s\\
&\leq C\Big(
\sup_{0\leq s\leq T_3}
\big(\overline{\mathcal E}(s,U)
+s\overline{\mathcal D}(s,U)\big)
+\int_0^{T_3}\overline{\mathcal D}(s,U)\,\mathrm{d}s
\Big)\leq C\mathfrak p(c_0),
\end{aligned}
\end{equation*}
Here, $\sqrt{T_3}\mathfrak p(2c_1)\leq1$ yields that the comparison constant is independent of $c_1$. The generic polynomial in the definition of $T_3$
is chosen sufficiently large, using $c_0\leq c_1$, such that $\sqrt{T_3}\mathfrak p(C\mathfrak p(c_0))\leq1$ also holds. Then we obtain from Lemma \ref{lemma-gaowei} again that
\begin{equation*}
\begin{aligned}
&\sup_{0\leq s\leq T_3}\Big(\mathcal E(s,U)+s\mathcal D(s,U)\Big)
+\int_0^{T_3}\mathcal D(s,U)\,\mathrm{d}s\\
&\leq C\Big(
\sup_{0\leq s\leq T_3}\big(
\mathring{\mathcal E}(s,U)
+s\mathring{\mathcal D}(s,U)\big)
+\int_0^{T_3}\mathring{\mathcal D}(s,U)\,\mathrm{d}s
\Big)\leq C\mathfrak p(c_0),
\end{aligned}
\end{equation*}
which, combined with \eqref{eta-etar}, leads to $\eqref{39''-0}$.
\end{proof}

Hence, define constants $(c_1,T^*)$ as 
\begin{equation}\label{3.84}
c_1:= C\mathfrak{p}(c_0),\qquad
T^*:=T_3=\min\{T_2,(2C\mathfrak{p}(c_0))^{-1},(\mathfrak{p}(2c_1))^{-2}\},
\end{equation}
then we obtain from Lemma \ref{lemma-12.11} that, for all $t\in[0,T^*]$,
\begin{equation}\label{uniform bounds}
\cE(t,U)+t\cD(t,U)+\int_0^{t} \cD(s,U)\,\mathrm{d}s\leq c_1,\qquad |\eta_r(t)-1|_\infty+\Big|\frac{\eta(t)}{r}-1\Big|_\infty \leq \frac{1}{2}.
\end{equation}

\section{Local Existence  and Uniqueness for   the  Nonlinear Problem}\label{subsection9.3}

In this section, we prove the local existence  of the unique classical solution of \eqref{eq:VFBP-La-eta} stated in Theorem \ref{local-Theorem1.1}, and then establish the extension to a general initial flow map stated in Theorem \ref{remk31}. 
For convenience, in the rest of \S \ref{subsection9.3}, 
$(\mathring{\mathcal E},\mathring{\mathcal D})$ are defined by \eqref{E-1a}--\eqref{D-1a}, while $(\mathcal E_k,\mathcal D_k)$ are defined by  \eqref{E-1} and \eqref{D-1}, respectively, with $\eta$ replaced by $\eta^k$, where $\eta^\mathrm{k}$ denotes the $\mathrm{k}$-th generation of the iterative sequence that will be given later.

\subsection{Introduction of the Picard Iteration}
Let $c_0$ be as in \eqref{38}. We first introduce an auxiliary
background only for starting the iteration:
\begin{equation*}
    \bar U^{00}(t,r)\equiv 0,\quad
    \bar\eta^{00}(t,r)\equiv r, \qquad \text{ with }\  \bar\eta^{00}_r=1,\quad
    \frac{\bar\eta^{00}}{r}=1,\quad \bar\varrho^{00}=\rho_0.
\end{equation*}
This background satisfies the flow-map relation, the regularity requirements, and the non-degeneracy condition in \eqref{given-flow}--\eqref{jibenjiashe}, while the initial condition $\bar U(0,r)=u_0(r)$ imposed on the general background is not required for this auxiliary construction, as explained below.

Consider \eqref{lp} with
\begin{equation*}
    \bar\eta=\bar\eta^{00}=r,\qquad \bar U=\bar U^{00}=0,\qquad
    \bar\varrho=\rho_0,\qquad
    U^0(0,r)=u_0(r),
\end{equation*}

We now explain why the construction and estimates of Sections \S\ref{subsection9.1}--\S\ref{subsection9.2} apply to this particular background. The Galerkin construction in
Propositions \ref{prop1}--\ref{prop-strong} applies with time-independent coefficients,  and \eqref{L-1} remains valid. Since the pressure and gravitational terms are also independent of time, the time-differentiated equation is homogeneous. In particular, the source terms $g^{(1)}$ and $h^{(1)}$ in \eqref{q1-q2-1}, together with their time derivatives,
vanish identically.

Evaluating \eqref{lp} at $t=0$ gives precisely the initial
acceleration in \eqref{116}, since $\bar\eta^{00}(0,r)=r$, $U^0(0,r)=u_0(r)$, and all terms
of \eqref{lp} have been retained. Thus the initial energy assumption \eqref{a2-lo} provides the weighted regularity of $U_t^0(0)$ required in Proposition \ref{prop-strong}. The reconstruction argument in \S\ref{subsection3.3} therefore applies with 
the same initial acceleration, without an additional initial compatibility condition.

For this fixed background, the use of $\bar U(0,r)=u_0(r)$ in \eqref{A2 revise 1} is replaced by $g^{(1)}=h^{(1)}=0$, so that the initial source estimate holds automatically. The time derivative estimates follow from the same coercive energy identities, with the terms arising from time differentiation of the background absent. 

Next, the spatial regularity argument in Step 6 of
\S\ref{subsection3.3} and the estimates in Lemmas \ref{c_1-c_2}--\ref{c_3} apply
with $\bar\eta=r$ and $\bar\varrho=\rho_0$. In particular, the preliminary weighted estimates \eqref{refine-3}--\eqref{jacobiD2L} justify the application of Lemma \ref{prop2.1} in obtaining the fourth-order spatial estimate. Together with the preceding time derivative estimates and the continuity argument in \S\ref{subsection3.3}, this yields a classical solution $U^0$ on a sufficiently small time interval $[0,T']$ and
\begin{equation*}
\sup_{0<t\le T'}\left(\mathring{\mathcal E}(t,U^0)+t\mathring{\mathcal D}(t,U^0)\right)+\int_0^{T'}\mathring{\mathcal D}(s,U^0)\,\mathrm{d}s\le C\mathfrak p(c_0).
\end{equation*}

Define
\begin{equation*}
    \eta^0(t,r)=r+\int_0^t U^0(s,r)\,\mathrm{d}s.
\end{equation*}
Since $U^0(t,0)=0$, we also have $\eta^0(t,0)=0$. Arguing as in the preceding estimates for the flow map, and using the fixed-coordinate estimates for $U^0$, we obtain
\begin{equation*}
|\eta_r^0(t)-1|_\infty+\Big|\frac{\eta^0(t)}r-1\Big|_\infty\le Ct\mathfrak p(c_0).
\end{equation*}
After decreasing $T'>0$, the right side is at most $\frac{1}{2}$ on $[0,T']$.

We now apply Lemma \ref{lemma-gaowei} with the generating pair $(\eta^0,U^0)$ and $f=U^0$, using the fixed-coordinate alternatives in \eqref{C333}--\eqref{revise c.4}. After decreasing $T'$
once more, the preceding bound yields the corresponding
estimate for $(\mathcal E_0,\mathcal D_0)$.
Choose $c_1\ge c_0$ as in \eqref{3.84}, with the generic constant sufficiently large to include the auxiliary estimate. Choose $T^*$ accordingly and decrease $T'$ so that $T'\le T^*$. Then, for all $t\in(0,T']$, we obtain

\begin{equation}\label{395}
\cE_0(t,U^0)+t\cD_0(t,U^0)+\int_0^{t} \cD_0(s,U^0)\,\mathrm{d}s\leq c_1,\quad |\eta_r^0(t)-1|_\infty+\Big|\frac{\eta^0(t)}{r}-1\Big|_\infty \leq \frac{1}{2}.
\end{equation}

Next, take
\begin{equation*}
        \bar\eta=\eta^0,\qquad \bar U=U^0
\end{equation*}
in \eqref{lp}. By construction, $U^0(0,r)=u_0(r)$.
The preceding estimates and regularity properties
therefore show that this background satisfies
the assumptions of \S\ref{subsection9.1}. By Lemma \ref{existence-linearize},
problem \eqref{lp} admits a unique classical
solution $U^1$ with $U^1(0,r)=u_0(r)$. Define
\begin{equation*}
    \eta^1(t,r)
    =r+\int_0^t U^1(s,r)\,\mathrm{d}s.
\end{equation*}
It follows from \eqref{uniform bounds} that $(U^1,\eta^1)$ satisfies
the following uniform estimates on $[0,T']$:
\begin{equation}\label{395'''}
\cE_1(t,U^1)+t\cD_1(t,U^1)+\int_0^{t} \cD_1(s,U^1)\,\mathrm{d}s\leq c_1,\quad |\eta_r^1(t)-1|_\infty+\Big|\frac{\eta^1(t)}{r}-1\Big|_\infty \leq \frac{1}{2}.
\end{equation}

As a consequence, the approximate sequence $(U^{\mathrm{k}+1},\eta^{\mathrm{k}+1})$ ($\mathrm{k}\in \NN^*$) can be constructed iteratively as follows: Given $(U^{\mathrm{k}},\eta^\mathrm{k})$, define $(U^{\mathrm{k}+1},  \eta^{\mathrm{k}+1})$ by solving the following problem in $(0,T']\times I$: 
\begin{equation}\label{396}
\begin{cases}
\displaystyle r^2\rho_0 U_t^{\mathrm{k+1}} +A\big((\eta^{\mathrm{k}})^2(\varrho^{\mathrm{k}})^{\gamma}\big)_r- 2A\eta^{\mathrm{k}}\eta_r^{\mathrm{k}}(\varrho^{\mathrm{k}})^{\gamma} +4\pi G\frac{r^2\rho_0}{(\eta^{\mathrm{k}})^2}\int_0^r\hat r^2\rho_0\,\mathrm{d}\hat r\\[10pt]
\displaystyle\ = \Big((\eta^{\mathrm{k}})^2(\varrho^{\mathrm{k}})^{\delta} \big((2a_1+a_2)\frac{U_r^{\mathrm{k+1}}}{\eta_r^{\mathrm{k}}}+2a_2\frac{U^{\mathrm{k+1}}}{\eta^{\mathrm{k}}}\big)\Big)_r\\[10pt]
\displaystyle\quad\quad-\eta^{\mathrm{k}}\eta_r^{\mathrm{k}}(\varrho^{\mathrm{k}})^{\delta} \Big(2a_2  \frac{U^{\mathrm{k+1}}_r}{\eta_r^{\mathrm{k}}}+4(a_1+a_2)\frac{U^{\mathrm{k+1}}}{\eta^{\mathrm{k}}}\Big),\\[8pt]
\eta_t^{\mathrm{k+1}}=U^{\mathrm{k+1}},\\[8pt]
(U^{\mathrm{k+1}},\eta^{\mathrm{k+1}})|_{t=0}(r)=\big(u_0(r),r\big) \qquad \text{for }r \in I,
\end{cases}
\end{equation}
where $\varrho^\mathrm{k}$ is given by 
\begin{equation*}
\varrho^\mathrm{k}=\frac{r^2\rho_0}{(\eta^\mathrm{k})^2\eta^\mathrm{k}_r}.
\end{equation*}
The pair $(U^1,\eta^1)$ constructed above also satisfies \eqref{396} with $k=0$. By \eqref{uniform bounds}, we obtain an iterative solution sequence $(U^{\mathrm{k}},\eta^\mathrm{k})$ satisfying \eqref{395}, that is, for all $t\in[0,T']$, 
\begin{equation}\label{395k}
\left\{
\begin{aligned}
&\cE_\mathrm{k}(t,U^{\mathrm{k}})+t\cD_\mathrm{k}(t,U^{\mathrm{k}})+\int_0^{t} \cD_\mathrm{k}(s,U^{\mathrm{k}})\,\mathrm{d}s\leq c_1,\quad \\[4pt]
&|\eta_r^\mathrm{k}(t)-1|_\infty+\Big|\frac{\eta^\mathrm{k}(t)}{r}-1\Big|_\infty \leq \frac{1}{2},\qquad \text{for } k\in \NN;  \\[4pt]
&U^{\mathrm{k}}\big|_{r=0}=\big(D_{\eta^{\mathrm{k-1}}}U^{\mathrm{k}}+\frac{2a_2}{2a_1+a_2}\frac{U^{\mathrm{k}}}{\eta^{\mathrm{k-1}}}\big)\Big|_{r=1}=0,\qquad \text{for } k\in \NN^*.
\end{aligned}
\right .
\end{equation}
Of course, by Lemma \ref{lemma-gaowei}, \eqref{395k} also implies
\begin{equation}\label{395k-r}
\mathring\cE(t,U^{\mathrm{k}})+t\mathring\cD(t,U^{\mathrm{k}})+\int_0^{t} \mathring\cD(s,U^{\mathrm{k}})\,\mathrm{d}s\leq C\mathfrak{p}(c_1),\quad |\eta_r^\mathrm{k}(t)-1|_\infty+\Big|\frac{\eta^\mathrm{k}(t)}{r}-1\Big|_\infty \leq \frac{1}{2}.
\end{equation}

\subsection{Strong  Convergence of the Approximation Solutions}
% The difference equation below is used only for $k\ge1$. The initial difference $U^1-U^0$ is controlled by \eqref{395}--\eqref{395'''}. 
Define
\begin{equation*}
\begin{aligned}
\widehat U^{\mathrm{k}+1}&:=U^{\mathrm{k}+1}-U^{\mathrm{k}},\qquad &&\weta^{\mathrm{k}+1}:=\eta^{\mathrm{k}+1}-\eta^\mathrm{k}=\int_0^t \wv^{\mathrm{k}+1}(s,r)\,\ds,\\
\psi^\mathrm{k}&:=(\varrho^\mathrm{k})^{\gamma-1},\qquad &&\widehat \psi^\mathrm{k+1}\!:=\psi^\mathrm{k+1}-\psi^\mathrm{k},
\end{aligned}
\end{equation*}
and introduce the following energy function for $k \geq 1$:
\begin{equation*}
\begin{aligned}
\widehat E_\mathrm{k}(t)&:=\sup_{s\in[0,t]}\big|(r^2\rho_0)^\frac{1}{2} \wv^{\mathrm{k}}\big|_2^2+\int_0^t\big|(r^2\rho_0^\delta)^\frac{1}{2} \sD_r \wv^{\mathrm{k}}\big|_2^2\,\ds.   
\end{aligned}
\end{equation*}
% Then, based on \eqref{396}, the problem of $(\wv^{\mathrm{k}+1},\weta^{\mathrm{k}+1})$ can be written as 
The estimates \eqref{395}--\eqref{395'''} imply
$\widehat E_1(T')\leq C\mathfrak p(c_1)$.
For $k\geq1$, subtracting the equations for
$U^{k+1}$ and $U^k$ yields
\begin{equation}\label{398}
\!\!\begin{cases}
\displaystyle r^2\rho_0\wv^{\mathrm{k}+1}_t-\Big((\eta^{\mathrm{k}})^2(\varrho^{\mathrm{k}})^{\delta} \big((2a_1+a_2)\frac{\wv_r^{\mathrm{k+1}}}{\eta_r^{\mathrm{k}}}+2a_2\frac{\wv^{\mathrm{k+1}}}{\eta^{\mathrm{k}}}\big)\Big)_r\\[10pt] 
\displaystyle\quad+\eta^{\mathrm{k}}\eta_r^{\mathrm{k}}(\varrho^{\mathrm{k}})^{\delta} \Big(2a_2  \frac{\wv^{\mathrm{k+1}}_r}{\eta_r^{\mathrm{k}}}+4(a_1+a_2)\frac{ \wv^{\mathrm{k+1}}}{\eta^{\mathrm{k}}}\Big)
=\big((r^2\rho_0^\delta)^\frac{1}{2}\frac{h^\mathrm{k}}{\eta^\mathrm{k}_r}\big)_r+(r^2\rho_0^\delta)^\frac{1}{2}\frac{g^\mathrm{k}}{\eta^\mathrm{k}},\\[8pt] 
\weta^{\mathrm{k}+1}_t=\wv^{\mathrm{k}+1},\\[4pt]
(\wv^{\mathrm{k}+1},\weta^{\mathrm{k}+1})(r)\big|_{t=0}=(0,0) \qquad\text{for }r\in I,
\end{cases}
\end{equation}
where
\begin{equation*}
\begin{aligned}
h^\mathrm{k}&:=(r^2\rho_0^\delta)^\frac{1}{2} \Big(\big(\frac{(\eta^\mathrm{k})^{2}\eta^\mathrm{k}_r}{r^2}\big)^{(1-\delta)}\Big((2a_1+a_2) \frac{U_{r}^\mathrm{k}}{\eta^\mathrm{k}_r}\big(1-\frac{(\eta^\mathrm{k-1})^{2(1-\delta)}(\eta^\mathrm{k-1}_r)^{-(1+\delta)}}{(\eta^\mathrm{k})^{2(1-\delta)}(\eta_r^\mathrm{k})^{-(1+\delta)}} \big)\\
&\qquad+2a_2\frac{U^\mathrm{k}}{\eta^\mathrm{k}}\big(1-\frac{(\eta^\mathrm{k-1})^{2(1-\delta)-1}(\eta^\mathrm{k-1}_r)^{-\delta}}{(\eta^\mathrm{k})^{2(1-\delta)-1}(\eta_r^\mathrm{k})^{ - \delta}}\big)\Big)-A\rho_0^{1-\delta}\big(\widehat \psi^\mathrm{k}+\psi^\mathrm{k-1} (1 - \frac{\eta^\mathrm{k}_r}{\eta^\mathrm{k-1}_r})\big)\Big),\\
g^\mathrm{k}&:=-(r^2\rho_0^\delta)^\frac{1}{2}\Big(\big(\frac{(\eta^\mathrm{k})^{2}\eta^\mathrm{k}_r}{r^2}\big)^{(1-\delta)}\Big(2a_2  \frac{U^{\mathrm{k}}_r}{\eta^\mathrm{k}_r}\big(1-\frac{(\eta^\mathrm{k-1})^{2(1-\delta)-1}(\eta^\mathrm{k-1}_r)^{-\delta}}{(\eta^{\mathrm{k}})^{2(1-\delta)-1}(\eta^\mathrm{k}_r)^{-\delta}}\big)\\
&\qquad+4(a_1+a_2)\frac{U^{\mathrm{k}}}{\eta^\mathrm{k}}\big(1-\frac{(\eta^\mathrm{k-1})^{-2\delta}(\eta^\mathrm{k-1}_r)^{1-\delta}}{(\eta^\mathrm{k})^{-2\delta}(\eta^\mathrm{k}_r)^{1-\delta}}\big)\Big)-2A\rho_0^{1-\delta} \big( \widehat \psi^\mathrm{k}+ \psi^\mathrm{k-1}(1-\frac{\eta^\mathrm{k}}{\eta^\mathrm{k-1}})\big)\Big)\\
&\qquad-4\pi G\frac{r\rho_0^{1-\frac{\delta}{2}}}{\eta^{\mathrm{k}}}\big(1-\frac{(\eta^{\mathrm{k}})^2}{(\eta^{\mathrm{k-1}})^2}\big)\int_0^r\hat r^2\rho_0\,\mathrm{d}\hat r.
\end{aligned}
\end{equation*}

\textbf{Step 1.} Since
\begin{align*}
&\begin{aligned}
&\ \big(\frac{\eta^\mathrm{k}}{\eta^\mathrm{k-1}}-1\big)_t=\frac{\wv^\mathrm{k}}{\eta^\mathrm{k-1}}-\frac{U^{\mathrm{k}-1}}{\eta^\mathrm{k-1}}\big(\frac{\eta^\mathrm{k}}{\eta^\mathrm{k-1}}-1\big)\\
&\implies\frac{\eta^\mathrm{k}}{\eta^\mathrm{k-1}}=1+\int_0^{t} \exp\Big(-\int_s^t\frac{U^{\mathrm{k}-1}}{\eta^\mathrm{k-1}}\,\mathrm{d}\tau\Big)\frac{\wv^\mathrm{k}}{\eta^\mathrm{k-1}}\,\mathrm{d}s,
\end{aligned}
\end{align*}
it follows from \eqref{395k} that, for all $t\in[0,T']$,
\begin{equation}\label{395kk}
\Big|\frac{\eta^\mathrm{k}}{\eta^\mathrm{k-1}}-1\Big|+\Big|\big(\frac{\eta^\mathrm{k}}{\eta^\mathrm{k-1}}\big)^2-1\Big|\leq Ce^{C\mathfrak{p}(c_1)t}\int_0^{t} \Big|\frac{\wv^\mathrm{k}}{r}\Big|\,\mathrm{d}s.
\end{equation}
Similarly, we can also obtain from the mean value theorem and \eqref{395k} that
\begin{equation}\label{395kk'}
\begin{aligned}
\Big|\frac{\eta_r^\mathrm{k}}{\eta_r^\mathrm{k-1}}-1\Big|+\Big|\big(\frac{\eta_r^\mathrm{k}}{\eta_r^\mathrm{k-1}}\big)^2-1\Big| +
\Big|\frac{(\eta^\mathrm{k-1})^{2(1-\delta)}(\eta^\mathrm{k-1}_r)^{-(1+\delta)}}{(\eta^\mathrm{k})^{2(1-\delta)}(\eta_r^\mathrm{k})^{-(1+\delta)}}-1\Big| &\leq Ce^{C\mathfrak{p}(c_1)t}\int_0^{t} |\sD_r\wv^\mathrm{k}|\,\mathrm{d}s, \\[6pt]
\Big|\frac{(\eta^\mathrm{k-1})^{2(1-\delta)-1}(\eta^\mathrm{k-1}_r)^{-\delta}}{(\eta^\mathrm{k})^{2(1-\delta)-1}(\eta_r^\mathrm{k})^{ - \delta}}-1\Big|+
\Big|\frac{(\eta^\mathrm{k-1})^{-2\delta}(\eta^\mathrm{k-1}_r)^{1-\delta}}{(\eta^\mathrm{k})^{-2\delta}(\eta^\mathrm{k}_r)^{1-\delta}}-1\Big|&\leq Ce^{C\mathfrak{p}(c_1)t}\int_0^{t} |\sD_r\wv^\mathrm{k}|\,\mathrm{d}s.
\end{aligned}
\end{equation}

Moreover, since
\begin{equation*}
\begin{aligned}
\widehat \psi^\mathrm{k}_t&=-(\gamma-1)\widehat \psi^\mathrm{k}\big(\frac{U_r^\mathrm{k}}{\eta_r^\mathrm{k}}+\frac{2U^{\mathrm{k}}}{\eta^\mathrm{k}}\big)-(\gamma-1) \psi^\mathrm{k-1}\Big(\frac{\wv_r^\mathrm{k}}{\eta_r^\mathrm{k}}+ \frac{U_r^\mathrm{k-1}}{\eta_r^\mathrm{k}}\big(1-\frac{\eta_r^\mathrm{k}}{\eta_r^\mathrm{k-1}}\big)\Big)\\
&\quad -2(\gamma-1) \psi^\mathrm{k-1}\Big(\frac{\wv^\mathrm{k}}{\eta^\mathrm{k}}+ \frac{U^{\mathrm{k}-1}}{\eta^\mathrm{k}}\big(1-\frac{\eta^\mathrm{k}}{\eta^\mathrm{k-1}}\big)\Big),
\end{aligned}
\end{equation*}
it follows from \eqref{395k} and \eqref{395kk}--\eqref{395kk'} that, for all $t\in [0,T']$,
\begin{equation}\label{395kkk}
|\widehat \psi^\mathrm{k}| \leq C\mathfrak{p}(c_0)e^{C\mathfrak{p}(c_1)t}\int_0^{t} |\sD_r\wv^\mathrm{k}|\,\mathrm{d}s.
\end{equation}

Therefore, it follows from \eqref{395k}, \eqref{395kk}--\eqref{395kkk}, and the H\"older and Minkowski inequalities that, for all $t\in [0,T']$ and $\mathrm{k}\in \NN^*$, 
\begin{equation}\label{basic}
\begin{aligned}
|(g^\mathrm{k},h^\mathrm{k})(t)|_2^2
\leq C\mathfrak{p}(c_1)te^{C\mathfrak{p}(c_1)t} \int_0^t\big|(r^2\rho_0^\delta)^\frac{1}{2}\sD_r \wv^{\mathrm{k}}\big|_2^2\,\ds.
\end{aligned}
\end{equation}

\textbf{ Step 2.} Now, multiplying $\eqref{398}_1$ by $\wv^{\mathrm{k}+1}$ and integrating the resulting equality over $I$, then using an argument similar to $L_1$ in \eqref{L-1}, we obtain from \eqref{395k}, \eqref{basic}, and the H\"older and Young inequalities that
\begin{equation*}
\begin{aligned}
&\,\frac{\mathrm{d}}{\dt}\big|(r^2\rho_0)^\frac{1}{2} \wv^{\mathrm{k}+1}\big|_2^2+\big|(r^2\rho_0^\delta)^\frac{1}{2}\sD_r\wv^{\mathrm{k+1}}\big|_2^2\leq C|(g^\mathrm{k},h^\mathrm{k})|_2^2\\
&\leq C\mathfrak{p}(c_1)te^{C\mathfrak{p}(c_1)t} \int_0^t\big|(r^2\rho_0^\delta)^\frac{1}{2}\sD_r\wv^{\mathrm{k}}\big|_2^2\ds.
\end{aligned}
\end{equation*}
Integrating the above over $[0,t]$ leads to
\begin{equation*}
\widehat E_{\mathrm{k}+1}(t)\leq C\mathfrak{p}(c_1)t^2e^{C\mathfrak{p}(c_1)t} \widehat E_\mathrm{k}(t) \qquad \text{for all $t\in [0,T^\prime]$ and $\mathrm{k}\in \NN^*$}.
\end{equation*}

Choosing $t=T_*$ in the above inequality such that
\begin{equation*}
C\mathfrak{p}(c_1)T_*^2e^{C\mathfrak{p}(c_1)T_*}\leq \frac{1}{2},\qquad T_*\leq T',
\end{equation*}
leads to 
\begin{equation}\label{total-bound}
\sum_{\mathrm{k}=1}^{\infty} \widehat E_\mathrm{k}(T_*) \leq \Big(\sum_{\mathrm{k}=0}^{\infty}2^{-\mathrm{k}}\Big) \widehat E_1(T_*)\leq C\mathfrak{p}(c_1).
\end{equation}
In fact, the preceding contraction estimate gives
\begin{equation*}
    \widehat{\mathcal E}_k(T_*)
    \le 2^{-(k-1)}
        \widehat{\mathcal E}_1(T_*),
    \qquad k\ge1,
\end{equation*}
and hence
\begin{equation*}
    \sum_{k=1}^{\infty}
        \widehat{\mathcal E}_k(T_*)^{\frac{1}{2}}
    <\infty.
\end{equation*}
Consequently, $\{U^k\}_{k\ge0}$ is Cauchy in the following weighted spaces, and there exists a limit $U$ such that
\begin{equation*} 
\begin{aligned}
(r^2\rho_0)^\frac{1}{2} U^{\mathrm{k}}\to (r^2\rho_0)^\frac{1}{2}U&\qquad  \text{in } C([0,T_*];L^2),\\
(r^2\rho_0^\delta)^\frac{1}{2}\sD_r U^{\mathrm{k}}\to (r^2\rho_0^\delta)^\frac{1}{2}\sD_r U&\qquad  \text{in } L^2([0,T_*];L^2),
\end{aligned}
\end{equation*}
which also leads to 
\begin{equation}\label{5.10}
U^{\mathrm{k}}\to U \qquad  \text{in } L^2([0,T_*];H^1_{\mathrm{loc}}) \qquad \text{as $\mathrm{k}\to \infty$}.
\end{equation}
On the other hand, from \eqref{395k-r} and $\rho_0^{\gamma-1}\sim 1-r$, it follows that, for any $t\in [0,T_*]$ and $a\in (0,\frac{1}{2})$, 
\begin{equation}\label{5.11}
t\|U^{\mathrm{k}}\|_{H^4(a,1-a)}^2\leq C(a) (\mathring\cE(t,U^k)+t\mathring\cD(t,U^k))\leq C(a)\mathfrak{p}(c_1).
\end{equation}
Hence, it follows from \eqref{5.10}--\eqref{5.11} and Lemma \ref{GNinequality'} that
\begin{equation}\label{H333}
U^{\mathrm{k}}\to U\qquad \text{in }L^1([0,T_*];H^3_\mathrm{loc}) \ \ \text{as $\mathrm{k}\to \infty$}.
\end{equation}
 
Now, based on \eqref{H333} and $\eqref{396}_2$, we have
\begin{equation}\label{H333'}
\eta^\mathrm{k}=r+\int_0^tU^{\mathrm{k}}\,\mathrm{d}s\to r+\int_0^tU \,\mathrm{d}s \qquad \text{in } C([0,T_*];H^3_\mathrm{loc}). 
\end{equation}
This also implies that $\{\eta^\mathrm{k}\}_{k\in\NN}$ is a Cauchy sequence in $C([0,T_*];H^3_\mathrm{loc})$ which converges to some limit $\eta$ so that
\begin{equation*}
\eta=r+\int_0^tU \,\mathrm{d}s \qquad \text{for \textit{a.e.} $(t,r)\in (0,T_*)\times (0,1)$}.
\end{equation*}

To recover equation $\eqref{eq:VFBP-La-eta}_1$, we first divide $\eqref{396}_1$ by $(\eta^\mathrm{k})^2\eta_r^\mathrm{k}$ to derive
\begin{equation*}
\begin{aligned}
U_t^{\mathrm{k+1}}&= \frac{1}{\varrho^\mathrm{k}\eta_r^\mathrm{k}}\Big((\varrho^{\mathrm{k}})^{\delta} \big((2a_1+a_2)\frac{U_r^{\mathrm{k+1}}}{\eta_r^{\mathrm{k}}}+2a_2\frac{U^{\mathrm{k+1}}}{\eta^{\mathrm{k}}}\big)-A(\varrho^{\mathrm{k}})^{\gamma}\Big)_r 
-4\pi G\frac{1}{(\eta^{\mathrm{k}})^2}\int_0^r\hat r^2\rho_0\,\mathrm{d}\hat r\\
&\displaystyle\quad+\frac{4a_1}{(\varrho^\mathrm{k})^{1-\delta}\eta^\mathrm{k}}\Big(  \frac{U^{\mathrm{k+1}}_r}{\eta_r^{\mathrm{k}}}-\frac{U^{\mathrm{k+1}}}{\eta^{\mathrm{k}}}\Big).
\end{aligned}
\end{equation*}
Then, due to Lemma \ref{sobolev-embedding}, $\rho_0^{\gamma-1}\sim 1-r$, $\rho^{\gamma-1}_0 \in H^3_{\mathrm{loc}}$, and \eqref{H333}--\eqref{H333'}, we take the limit as $k\to\infty$ in the above and obtain
\begin{equation*}
\begin{aligned}
U_t^{\mathrm{k}}&\to \frac{1}{\varrho}D_\eta\Big((2a_1+a_2)\varrho^{\delta} \big(D_\eta U+\frac{2U}{\eta}\big)-A\varrho^{\gamma}\Big) 
-D_\eta \Phi-4a_1\frac{D_\eta(\varrho^\delta)}{\varrho}\frac{U}{\eta}, 
\end{aligned}
\end{equation*}
which, along with the uniqueness of limits, implies that $\eqref{eq:VFBP-La-eta}_1$ holds for \textit{a.e.} $(t,r)\in (0,T_*)\times (0,1)$.  Moreover, by the lower semi-continuity of weak convergence and \eqref{395k-r}, we have
\begin{equation*}
\mathring\cE(t,U)+t\mathring\cD(t,U)+\int_0^{t} \mathring\cD(s,U)\,\mathrm{d}s\leq C\mathfrak{p}(c_1),\quad |\eta_r(t)-1|_\infty+\Big|\frac{\eta(t)}{r}-1\Big|_\infty \leq \frac{1}{2}.
\end{equation*}
Applying Lemma \ref{lemma-gaowei} again, with $f=U$ and fixed-coordinate estimates, we obtain 
\begin{equation*}
\mathcal E(t,U)+t\mathcal D(t,U)\in L^\infty(0,T_*),
\qquad
\mathcal D(t,U)\in L^1(0,T_*).    
\end{equation*}
This completes the proof of the existence.

\subsection{Uniqueness  and Time Continuity} In this subsection, $c_1$ is enlarged, if necessary, to control the energy bounds of both solutions, and the constants may also
depend on their flow-map nondegeneracy bounds.
Let $(U^\mathsf{a},\eta^\mathsf{a})$ and $(U^\mathsf{b},\eta^\mathsf{b})$ be two solutions of \eqref{eq:VFBP-La-eta} in $(0,T_*)\times (0,1)$. Define 
\begin{equation*}
(\weta,\wv):=(\eta^\mathsf{b}-\eta^\mathsf{a},U^\mathsf{b}-U^\mathsf{a}),\quad \,\,
\widehat E(t):=\sup_{s\in[0,t]}\big|(r^2\rho_0)^\frac{1}{2} \wv\big|_2^2+\int_0^t\big|(r^2\rho_0^\delta)^\frac{1}{2} \sD_r \wv\big|_2^2\,\ds.
\end{equation*}

It follows from  \eqref{eq:VFBP-La-eta} that  $(\wv,\weta)$ solves the following problem in $(0,T_*]\times I$:
\begin{equation}
\!\!\begin{cases}
\displaystyle r^2\rho_0\wv_t-\Big((\eta^{\mathsf{b}})^2(\varrho^{\mathsf{b}})^{\delta} \big((2a_1+a_2)\frac{\wv_r}{\eta_r^{\mathsf{b}}}+2a_2\frac{\wv}{\eta^{\mathsf{b}}}\big)\Big)_r\\[10pt] 
\displaystyle\quad+\eta^{\mathsf{b}}\eta_r^{\mathsf{b}}(\varrho^{\mathsf{b}})^{\delta} \Big(2a_2  \frac{\wv_r}{\eta_r^{\mathsf{b}}}+4(a_1+a_2)\frac{ \wv}{\eta^{\mathsf{b}}}\Big)
=\big((r^2\rho_0^\delta)^\frac{1}{2}\frac{\tilde{h}}{\eta^\mathsf{b}_r}\big)_r+(r^2\rho_0^\delta)^\frac{1}{2}\frac{\tilde{g}}{\eta^\mathsf{b}},\\[8pt] 
\weta_t=\wv,\\[4pt]
(\wv,\weta)(r)\big|_{t=0}=(0,0) \qquad\text{for }r\in I,
\end{cases}
\end{equation}
where
\begin{equation*}
\begin{aligned}
\tilde{h}&:=(r^2\rho_0^\delta)^\frac{1}{2} \bigg(\big(\frac{(\eta^\mathsf{b})^{2}\eta^\mathsf{b}_r}{r^2}\big)^{(1-\delta)}\Big((2a_1+a_2) \frac{U_{r}^\mathsf{a}}{\eta^\mathsf{b}_r}\big(1-\frac{(\eta^\mathsf{a})^{2(1-\delta)}(\eta^\mathsf{a}_r)^{-(1+\delta)}}{(\eta^\mathsf{b})^{2(1-\delta)}(\eta_r^\mathsf{b})^{-(1+\delta)}} \big)\\
&\qquad\qquad\qquad+2a_2\frac{U^\mathsf{a}}{\eta^\mathsf{b}}\big(1-\frac{(\eta^\mathsf{a})^{2(1-\delta)-1}(\eta^\mathsf{a}_r)^{-\delta}}{(\eta^\mathsf{b})^{2(1-\delta)-1}(\eta_r^\mathsf{b})^{ - \delta}}\big)\Big)-A\rho_0^{1-\delta}\big(\widehat \psi+\psi^\mathsf{a} (1 - \frac{\eta^\mathsf{b}_r}{\eta^\mathsf{a}_r})\big)\bigg),\\
\tilde{g}&:=-(r^2\rho_0^\delta)^\frac{1}{2}\bigg(\big(\frac{(\eta^\mathsf{b})^{2}\eta^\mathsf{b}_r}{r^2}\big)^{(1-\delta)}\Big(2a_2  \frac{U^{\mathsf{a}}_r}{\eta^\mathsf{b}_r}\big(1-\frac{(\eta^\mathsf{a})^{2(1-\delta)-1}(\eta^\mathsf{a}_r)^{-\delta}}{(\eta^{\mathsf{b}})^{2(1-\delta)-1}(\eta^\mathsf{b}_r)^{-\delta}}\big)\\
&\qquad\qquad\qquad+4(a_1+a_2)\frac{U^{\mathsf{a}}}{\eta^\mathsf{b}}\big(1-\frac{(\eta^\mathsf{a})^{-2\delta}(\eta^\mathsf{a}_r)^{1-\delta}}{(\eta^\mathsf{b})^{-2\delta}(\eta^\mathsf{b}_r)^{1-\delta}}\big)\Big)-2A\rho_0^{1-\delta} \big( \widehat \psi+\psi^\mathsf{a}(1-\frac{\eta^\mathsf{b}}{\eta^\mathsf{a}})\big)\bigg)\\
&\qquad-4\pi G\frac{r\rho_0^{1-\frac{\delta}{2}}}{\eta^{\mathsf{b}}}\big(1-\frac{(\eta^{\mathsf{b}})^2}{(\eta^{\mathsf{a}})^2}\big)\int_0^r\hat r^2\rho_0\,\mathrm{d}\hat r,\\
&\psi^\mathsf{a}:=\big(\frac{r^2\rho_0}{(\eta^\mathsf{a})^2\eta^\mathsf{a}_r}\big)^{\gamma-1}, \qquad \psi^\mathsf{b}:=\big(\frac{r^2\rho_0}{(\eta^\mathsf{b})^2\eta^\mathsf{b}_r}\big)^{\gamma-1}, \qquad \widehat\psi:=\psi^\mathsf{b}-\psi^\mathsf{a}.
\end{aligned}
\end{equation*}

Similarly, we can show that $(\tilde{h} ,\tilde{g})$ satisfy \eqref{basic} with $\wv^{\mathrm{k}}$ replaced by $\wv$. Hence, by the same arguments as in Step 2, we have 
\begin{equation*}
\frac{\mathrm{d}}{\dt}\big|(r^2\rho_0)^\frac{1}{2} \wv\big|_2^2+\big|(r^2\rho_0^\delta)^\frac{1}{2} \sD_r \wv\big|_2^2\leq C\mathfrak{p}(c_1)te^{C\mathfrak{p}(c_1)t} \int_0^t\big|(r^2\rho_0^\delta)^\frac{1}{2} \sD_r \wv\big|_2^2\,\ds,
\end{equation*}
which, together  with the Gr\"onwall inequality, yields that  $\widehat E(t) \equiv 0$, \textit{i.e.}, $U^\mathsf{a}\equiv U^\mathsf{b}$.

Now, we show that $(U,\eta)$ are classical satisfying $\eqref{eq:VFBP-La-eta}_1$. First, the regularity of $U$ and the boundary condition shown in  \eqref{N111} can be proved by the same argument as in Steps 7--8 in \S \ref{subsection3.3}. Then the regularity of $\eta$ follows directly from the formula: $\eta_t=U$. Finally, following a similar argument in Step 9 in \S \ref{subsection3.3}, we can show that $\eqref{eq:VFBP-La-eta}_1$ holds pointwise in $(0,T_*]\times \bar I$. This completes the proof of Theorem \ref{local-Theorem1.1}.

\medskip

\subsection{Proof of Theorem \ref{remk31}}
Now, we prove Theorem \ref{remk31}.
\begin{proof}
First, enlarge $c_0$ in \eqref{38} to include $\delta_*^{-1}$, $\delta^*$, and $|\mathcal F\big(\eta_0\big)|_{2}$. By \eqref{eq:eta},
\begin{equation}
\varrho(0,r)=\frac{r^2\rho_0(r)}{\eta_0^2(r)(\eta_0)_r(r)}.
\end{equation}
The initial acceleration is determined by \eqref{eq:VFBP-La-eta}, namely,
\begin{equation}
\begin{aligned}
U_t(0,r)
&=\frac{2a_1+a_2}{\varrho(0,r)}
D_{\eta_0}\Big(
\varrho^\delta(0,r)
\big(D_{\eta_0}u_0+\frac{2u_0}{\eta_0}\big)
\Big)
-\frac{4a_1}{\varrho(0,r)}
D_{\eta_0}\big(\varrho^\delta(0,r)\big)
\frac{u_0}{\eta_0}\\
&\quad-\frac{A}{\varrho(0,r)}
D_{\eta_0}\big(\varrho^\gamma(0,r)\big)
-\frac{4\pi G}{\eta_0^2(r)}
\int_0^r\hat r^2\rho_0(\hat r)\,\mathrm{d}\hat r.
\end{aligned}
\end{equation}
This formula, rather than the identity-map formula \eqref{116}, is used in the initial energy. Hence all initial acceleration terms required by the linear reconstruction are controlled by $\cE(0,U)$.

Next, we consider the linearized problem in \S \ref{subsection9.1} with
\begin{equation}
\bar\eta(t,r)
=\eta_0(r)+\int_0^t\bar U(s,r)\,\mathrm{d}s,
\qquad
\bar\varrho=\frac{r^2\rho_0}{\bar\eta^2\bar\eta_r}.
\end{equation}
Assume that the background satisfies the regularity
and energy bounds in \S \ref{subsection9.1}--\ref{subsection9.2}, with $\bar U(0)=u_0$,
and replace \eqref{jibenjiashe} by
\begin{equation}
\frac{\delta_*}{2}\leq\bar\eta_r,\frac{\bar\eta}{r}\leq\frac{3\delta^*}{2}.
\end{equation}
Then $\bar\varrho$ is comparable to $\rho_0$, and the bilinear form in \eqref{weak.F.} satisfies the coercive estimate \eqref{L-1}, with constants depending additionally
on the geometric bounds. The weighted spaces remain unchanged, so the Galerkin construction in \S \ref{revise section 2.2} applies.

To complete the regularity argument and obtain uniform
estimates, we control the additional initial geometric terms.
On each fixed exterior interval, the assumptions on $\eta_0$
and Lemmas \ref{sobolev-embedding}--\ref{hardy-inequality}
give the weighted third-order bounds below. Since $\rho_0^{\gamma-1}\sim1-r$, the H\"older inequality yields $\partial_r^3\eta_0\in L^1$.
The fundamental theorem of calculus, together with a local Sobolev estimate,
then controls $|\partial_r^2\eta_0|_\infty$.
Hence,
\begin{equation}
\begin{aligned}
&|\partial_r^2\eta_0|_\infty
+|\rho_0^{\gamma-1}\partial_r^3\eta_0|_\infty\\
&\quad
+|\rho_0^{(\gamma-1)(\frac{1}{2}-\varepsilon_0)}
\partial_r^3\eta_0|_2
+|\rho_0^{(\gamma-1)(\frac{3}{2}-\varepsilon_0)}
\partial_r^4\eta_0|_2
\leq C\mathfrak{p}(c_0).
\end{aligned}
\end{equation}
The corresponding interior estimates follow similarly.

Repeating the derivation of \eqref{1189} with the initial terms retained, and using \eqref{revise App C Djcobi} under the fixed-coordinate alternative, we control the initial contributions by the preceding estimates and the time-integral terms by the background energy and dissipation. Thus, for $T\leq1$ and $\sqrt{T}\mathfrak{p}(c_0+c_1)\leq1$, the required geometric bounds depend only on the initial-data bounds, after enlarging $\mathfrak{p}$ if necessary. The expansion \eqref{1190} now gives the coefficient estimates of Lemma \ref{lemma-Lambda-guji}. In particular, the initial term involving $\partial_r^4\eta_0$ in the exterior estimate for $D_{\bar\eta}^2\bar\Lambda$ has the weight
\begin{equation}
\rho_0^{(\gamma-1)(\frac{1}{2}-\varepsilon_0)}\rho_0^{\gamma-1}\partial_r^4\eta_0=\rho_0^{(\gamma-1)(\frac{3}{2}-\varepsilon_0)}\partial_r^4\eta_0,
\end{equation}
which is controlled by the initial assumption. The remaining terms follow from the lower-order geometric bounds and \eqref{distance-la}. The gravitational estimates use the unchanged mass and the same geometric bounds. These estimates also control the initial contribution in \eqref{jacobi} and yield \eqref{jacobiD2L}.

We return to the regularity argument in \S \ref{subsection3.3}. The time-differentiated identities, including \eqref{the-relate}, remain valid. Substituting $\bar\eta(0)=\eta_0$ and $U(0)=u_0$ into \eqref{lp} gives the formula for $U_t(0,r)$
stated above, whose required weighted norms are controlled by $\mathcal E(0,U)$. Using this initial value for $w^{(1)}$, the reconstruction argument \eqref{rela}--\eqref{33153} shows that
$u_0+\int_0^t w^{(1)}(s)\,\mathrm{d}s=U$, and hence $w^{(1)}=U_t$. The estimates for $U_t$ and $U_{tt}$ then follow as in \S \ref{subsection3.3}, with the same time weights. The spatial regularity argument uses the coefficient bounds established above. In particular, \eqref{refine-3}--\eqref{jacobiD2L} are obtained before applying Lemma \ref{prop2.1} which still applies with the weight $\rho_0^{\gamma-1}$. The fourth-order estimate and the remaining continuity and boundary arguments in \S \ref{subsection3.3} therefore apply.

For the uniform estimates, \eqref{A2 revise 1} follows from $\bar U(0)=u_0$, the initial coefficient bounds, and $\mathcal E(0,U)$. Together with \eqref{A2 revise 0}, it gives
\begin{equation}
|g^{(1)}(t)|_2^2+|h^{(1)}(t)|_2^2
\leq C\mathfrak{p}(c_0)+Ct\mathfrak{p}(c_1).
\end{equation}
Thus, $J_4(t)$ is estimated as in \eqref{A2 revise 2}--\eqref{I*3}. The remaining time derivative estimates follow under the corresponding small-time conditions. Using the interior coefficient bounds in \eqref{11149} and the exterior weighted bound for $D_{\bar\eta}^2\bar\Lambda$ in \eqref{QQQ}, Lemmas \ref{c_0-c_1}--\ref{c_3} yield
\begin{equation}
\sup_{0<t<T}
\Big(\bar{\mathcal E}(t,U)
+t\bar{\mathcal D}(t,U)\Big)
+\int_0^T\bar{\mathcal D}(t,U)\,\mathrm{d}t
\leq C\mathfrak{p}(c_0).
\end{equation}
By Remark \ref{remk C.1}, the coordinate comparisons hold
for the general initial map, with constants depending
only on the initial data bounds, provided that $T\leq1$
and $\sqrt{T}\mathfrak{p}(c_0+\mathrm{E_1})\leq1$.

We first apply this comparison with the background map
$\bar\eta$ to obtain the fixed-coordinate estimates for $U$.
For $\eta(t)=\eta_0+\int_0^tU(s)\,\mathrm{d}s$,
these estimates imply
\begin{equation}
|\eta_r(t)-(\eta_0)_r|_\infty
+\Big|\frac{\eta(t)-\eta_0}{r}\Big|_\infty
\leq Ct\mathfrak{p}(c_0).
\end{equation}
Reducing the time such that the right-hand side is at most $\frac{\delta_*}{2}$ gives the required geometric bounds. We then apply the fixed-coordinate alternative with the generating pair $(\eta,U)$. 
The generating velocity bound depends only on $c_0$, and hence
\begin{equation}
\sup_{0<t<T}
\Big(\mathcal E(t,U)+t\mathcal D(t,U)\Big)
+\int_0^T\mathcal D(t,U)\,\mathrm{d}t
\leq C\mathfrak{p}(c_0).
\end{equation}

To initialize the iteration, we take the static
auxiliary background
\begin{equation}
\bar U^{00}=0,\qquad
\bar\eta^{00}=\eta_0,\qquad
\bar\varrho^{00}
=\frac{r^2\rho_0}{\eta_0^2(\eta_0)_r},
\end{equation}
and solve the full linearized problem \eqref{lp} with
$U^0(0)=u_0$.
As in Section 4.1, the auxiliary background need not satisfy
$\bar U^{00}(0)=u_0$.
Instead, the Galerkin construction applies directly
with time-independent coefficients.
The sources $g^{(1)}$ and $h^{(1)}$ vanish,
so the initial source estimate holds automatically.
Moreover, \eqref{lp} gives the same initial acceleration
as above.
The reconstruction, spatial regularity, and time weighted
estimates therefore give the required bounds for $U^0$ . Define
\begin{equation}
\eta^0(t)=\eta_0+\int_0^tU^0(s)\,\mathrm{d}s.
\end{equation}
The geometric increment estimate and the fixed-coordinate
comparison show that $(U^0,\eta^0)$ satisfies the geometric
and energy bounds required for the iteration. Choose $c_1\geq c_0$ sufficiently large to dominate both these bounds and the linear output bounds. Then choose a common positive existence time satisfying all the preceding small time conditions.

The subsequent iterates are constructed as in \S \ref{subsection9.3},
with the common initial state $(u_0,\eta_0)$. In particular,
\begin{equation}
\eta^{k+1}(t)-\eta^k(t)
=\int_0^t\big(U^{k+1}(s)-U^k(s)\big)\,\mathrm{d}s.
\end{equation}
The common initial map cancels, and hence no additional initial term occurs in the difference estimates. With the uniform geometric and energy bounds above, the contraction, passage to the limit, and uniqueness arguments apply. The time regularity and boundary conditions follow from the corresponding linear estimates.

This proves Theorem \ref{remk31}.

\end{proof}

\appendix
\section{Some Basic Lemmas}\label{appendix A}

For the convenience of readers, we list some basic facts that have been used frequently in this paper. Throughout the following Appendices \ref{appendix A}--\ref{subsection2.2}, let $\rho_0(\boldsymbol{y})=\rho_0(r)$, with $\rho_0(r)$ satisfying \eqref{distance-la}, that is, for some constants $K_2>K_1>0$ and for all $r\in I$,
\begin{equation*}
r\big(\rho_0,\sD_r^j((\rho^{\gamma-1}_0)_r)\big)\in L^2(I)\quad \text{for $j=0,1,2$},\qquad K_1(1-r) \leq \rho^{\gamma-1}_0(r)\leq K_2(1-r).
\end{equation*}

The first two lemmas concern the separability and density of the weighted Sobolev spaces. 
\begin{Lemma}[\cite{kufner}]\label{W-space}
Let $\vartheta_1,\vartheta_2\in (-1,\infty)$, and let $\tilde{d}=\tilde{d}(r)=r^{\vartheta_1}(1-r)^{\vartheta_2}$ be a function defined on $I$. Then, for $k\in \ZZ$, $H^k_{\tilde{d}}$ is a reflexive separable Banach space. Moreover, if $k\in \mathbb{N}$, $C^\infty(\bar I)$ is dense in $H^k_{\tilde{d}}$ with respect to the norm $\norm{\cdot}_{k,\tilde{d}}$. 
\end{Lemma}

\begin{Lemma}\label{prop-bijin}
$\cH^1_{r^2\rho_0^{\upsilon(\gamma-1)}}$  with $\upsilon\geq 0$ is a reflexive separable Banach space. 
Moreover, for any $f\in \cH^1_{r^2\rho_0^{\upsilon(\gamma-1)}}$, there exists 
a sequence $\{f^\varepsilon\}_{\varepsilon>0}\subset C^\infty(\bar I)\cap\cH^1_{r^2}$ such that 
\begin{equation*}
\|f^\varepsilon-f\|_{\cH^1_{r^2\rho_0^{\upsilon(\gamma-1)}}}\to 0 \qquad\text{as $\varepsilon\to 0$}.
\end{equation*}    
\end{Lemma}
\begin{proof}
The first assertion follows from Lemma \ref{W-space}. For the convergence, define $\boldsymbol{f}(\boldsymbol{y}):=f(r)\frac{\boldsymbol{y}}{r}$. Clearly, thanks to   Lemma \ref{lemma-initial} in Appendix \ref{appb} and the fact that $\rho_0^{\gamma-1}\sim 1-r$, we have $\boldsymbol{f}\in H^1(B_a)$ for any $a\in (0,1)$.

Now, we claim that there exists a sequence $\{\boldsymbol{f}_\flat^\varepsilon(\boldsymbol{y})
=f^\varepsilon_\flat(r)\frac{\boldsymbol{y}}{r}\}_{\varepsilon>0}\subset C^\infty(B_\frac{3}{4})$, such that
\begin{equation}\label{cliam121}
\boldsymbol{f}_\flat^\varepsilon\to \boldsymbol{f}\qquad \text{in $H^1(B_\frac{3}{4})$ as $\varepsilon\to 0$}.
\end{equation}
Indeed,  let $\{\omega_\varepsilon(\boldsymbol{y})\}_{\varepsilon>0}$ be the spherically symmetric mollifier defined on $\mathbb{R}^3$. 
Then $\boldsymbol{f}_\flat^\varepsilon(\boldsymbol{y}):=(\boldsymbol{f}*\omega_\varepsilon)(\boldsymbol{y})$ (with $\varepsilon$ sufficiently small) satisfies \eqref{cliam121}. Thus, by Lemma \ref{duichen-dengjia}, in order to verify that $\boldsymbol{f}_\flat^\varepsilon$ is spherically symmetric,    we only need to show that $\boldsymbol{f}_\flat^\varepsilon(\mathcal{O}\boldsymbol{y})=(\mathcal{O}\boldsymbol{f}_\flat^\varepsilon)(\boldsymbol{y})$ for any matrix $\mathcal{O}\in \mathrm{SO}(3)$. In fact, we have
\begin{equation*}
\boldsymbol{f}_\flat^\varepsilon(\mathcal{O}\boldsymbol{y})= \int_{\Omega} \boldsymbol{f}(\mathcal{O}\boldsymbol{y}-\boldsymbol{z})\omega_\varepsilon(\boldsymbol{z})\,\mathrm{d}\boldsymbol{z} \qquad\text{for $0<\varepsilon<\frac{1}{100}$ and $\boldsymbol{y}\in B_\frac{3}{4}$}.
\end{equation*}
Changing the coordinate $\boldsymbol{z}\mapsto \mathcal{O}\boldsymbol{z}$, along with the facts that $|\mathcal{O}\boldsymbol{y}|=|\boldsymbol{y}|$ and $\det \mathcal{O}=1$, gives
\begin{equation*}
\begin{aligned}
\boldsymbol{f}_\flat^\varepsilon(\mathcal{O}\boldsymbol{y})&= \int_{\Omega} \boldsymbol{f}(\mathcal{O}\boldsymbol{y}-\mathcal{O}\boldsymbol{z})\omega_\varepsilon(\mathcal{O}\boldsymbol{z})(\det \mathcal{O})\,\mathrm{d}\boldsymbol{z} = \int_{\Omega} \boldsymbol{f}(\mathcal{O}(\boldsymbol{y}-\boldsymbol{z}))\omega_\varepsilon(\boldsymbol{z}) \,\mathrm{d}\boldsymbol{z}\\
&= \int_{\Omega} (\mathcal{O} \boldsymbol{f})(\boldsymbol{y}-\boldsymbol{z})\omega_\varepsilon(\boldsymbol{z}) \,\mathrm{d}\boldsymbol{z} =\Big(\mathcal{O}\big( \int_{\Omega} \boldsymbol{f}(\cdot-\boldsymbol{z})\omega_\varepsilon(\boldsymbol{z}) \,\mathrm{d}\boldsymbol{z}\big)\Big)(\boldsymbol{y})=(\mathcal{O}\boldsymbol{f}_\flat^\varepsilon)(\boldsymbol{y}).
\end{aligned}
\end{equation*}
This completes the proof of the claim. 

Consequently, it follows from \eqref{cliam121} and Lemma \ref{lemma-initial} that 
\begin{equation}\label{jin} 
f^\varepsilon_\flat\in \cH^1_{r^2}(B_\frac{3}{4}),\qquad\,\, 
\|\zeta f^\varepsilon_\flat-\zeta f\|_{\cH^1_{r^2\rho_0^{\upsilon(\gamma-1)}}}\to 0 \qquad\text{as $\varepsilon\to 0$},
\end{equation}
where the cut-off function $\zeta$ is  defined in \S 1.2.

On the other hand, it follows from Lemma \ref{W-space} that there exists a smooth sequence $\{f_\sharp^\varepsilon\}_{\varepsilon>0}\subset C^\infty[\frac{1}{3},1]$ such that
\begin{equation}\label{yuan}
\|\zeta^\sharp f^\varepsilon_\sharp-\zeta^\sharp f\|_{\cH^1_{r^2\rho_0^{\upsilon(\gamma-1)}}}\leq C\|\zeta^\sharp f^\varepsilon_\sharp-\zeta^\sharp f\|_{1,\rho_0^{\upsilon(\gamma-1)}}\to 0 \qquad\text{as $\varepsilon\to 0$}.
\end{equation}
Therefore, defining $f^\varepsilon:=\zeta f_\flat^\varepsilon+\zeta^\sharp f_\sharp^\varepsilon$, we see that $f^\varepsilon\in C^\infty(\bar I)\cap \cH^1_{r^2}$. Then we can obtain from \eqref{jin}--\eqref{yuan} that
\begin{equation*}
\|f^\varepsilon-f\|_{\cH^1_{r^2\rho_0^{\upsilon(\gamma-1)}}}\leq \|\zeta f^\varepsilon_\flat-\zeta f\|_{\cH^1_{r^2\rho_0^{\upsilon(\gamma-1)}}}  +\|\zeta^\sharp f^\varepsilon_\sharp-\zeta^\sharp f\|_{\cH^1_{r^2\rho_0^{\upsilon(\gamma-1)}}}\to 0 \qquad\text{as $\varepsilon\to 0$}.
\end{equation*}

This completes the proof.
\end{proof}

The third  lemma concerns the well-known interpolation theory for the $H^k$ spaces.
\begin{Lemma}[\cite{leoni}]\label{GNinequality'} 
Let $J\subset \RR$ be some open interval and $F\in H^p(J)\cap H^q(J)$ $(p,q\geq 0)$. Then $F\in H^l(J)$ for $l=p\vartheta+q(1-\vartheta)$ and  $0\leq \vartheta\leq 1$, and the following inequality holds{\rm:}
\begin{equation*}
\norm{F}_{l}\leq C \norm{F}_{p}^\vartheta\norm{F}_{q}^{1-\vartheta},
\end{equation*}
where $C>0$ is a constant depending only on $(p, q,\vartheta,J)$.
\end{Lemma}

The fourth  lemma is on the classical Sobolev embedding theorems.
\begin{Lemma}[\cite{leoni}]\label{sobolev-embedding}
We state the Sobolev embeddings in the one-dimensional and three-dimensional cases separately. 
\begin{enumerate}
\item[{\rm (i)}]
Let $J\subset \RR$ be some bounded open interval and $f=f(r)$ be some function on $J$. Then there exist two positive constants $(s_0,C)$, depending only on the Lebesgue measure of $J$, such that
\begin{equation*}
\begin{aligned}
&\|f\|_{L^\infty(J)} \leq s_0\|f\|_{L^1(J)}+C\|f_r\|_{L^1(J)} &&\quad \text{for all }f\in W^{1,1}(J),\\[4pt]
&\|f\|_{L^\infty(J)} \leq s_0\|f\|_{L^2(J)}+C\|f_r\|_{L^2(J)} &&\quad \text{for all }f\in H^1(J),\\
&\|f\|_{L^\infty(J)} \leq s_0\|f\|_{L^2(J)}+ C\|f\|^\frac{1}{2}_{L^2(J)}\|f_r\|^\frac{1}{2}_{L^2(J)} &&\quad \text{for all }f\in H^1(J).
\end{aligned}
\end{equation*}
In particular, $W^{1,1}(J)\into C(\bar J)$ and $H^1(J)\into C(\bar J)$ continuously{\rm ;} moreover, 
if $f(r_0)=0$ for some $r_0\in \bar J$, then $s_0=0$ can be chosen in the above.

\item[{\rm (ii)}]Let $J \subset \mathbb{R}^{3}$ be some bounded domain with smooth boundary. Assume that $f\in H^1(J)$, then there exists a positive constant $C(J)$ depending only on $J$, such that
\begin{equation*}
\|f\|_{L^{6}(J)} \leq C(J)\, \|f\|_{H^{1}(J)}.
\end{equation*}
In particular, if $J=\RR^3$, then for any $f\in H^{1}(\RR^3)$, there exists a positive universal constant $C_0$ such that
\begin{equation*}
    \|f\|_{L^6(\RR^3)} \leq C_0\, \|\nabla f\|_{L^2(\RR^3)},
\end{equation*}
\end{enumerate}
\end{Lemma}

The next two lemmas are on the Hardy inequality and some weighted interpolation inequality. In Lemmas \ref{hardy-inequality}--\ref{GNinequality}, we let $0\leq a<b<\infty$ and $J=(a,b)$, and let $d=d(r)$ be some function on $J$, taking one of the following two forms{\rm:}
\begin{equation*}
d(r)=r-a,\qquad \text{or} \ \ d(r)=b-r.
\end{equation*}
\begin{Lemma}[\cites{opic1,opic2}]\label{hardy-inequality}
Let $p\in [1,\infty]$ and $\vartheta>-\frac{1}{p}$ $(\vartheta>0\text{ if }p=\infty)$. Then, for any $f$ such that $d^{\vartheta+\frac{1}{2}+\frac{1}{p}} (f,f_r)\in L^2(J)$,
\begin{enumerate}
\item[{\rm (i)}] If $p\in [1,2)$, for any $\varepsilon>0$, there exists a constant $C_1>0$, depending only on $(\varepsilon, p,a,b,\vartheta)$, such that
\begin{equation*}
\|d^{\vartheta+\varepsilon} f\|_{L^p(J)} \leq C_1 \big\|d^{\vartheta+\frac{1}{2}+\frac{1}{p}} (f,f_r)\big\|_{L^2(J)};
\end{equation*}
\item[{\rm (ii)}] If $p\in [2,\infty]$, there exists a constant $C_2>0$, which depends only on $(p,a,b,\vartheta)$ if $p\ne \infty$ and depends only on $(a,b,\vartheta)$ if $p=\infty$, such that
\begin{equation*}
\|d^\vartheta f\|_{L^p(J)} \leq C_2 \big\|d^{\vartheta+\frac{1}{2}+\frac{1}{p}} (f,f_r)\big\|_{L^2(J)}. 
\end{equation*}
\end{enumerate}
\end{Lemma}
 
\begin{Lemma}[\cite{CZZ2}]\label{GNinequality}
Let $p\in [2,\infty]$ and $\vartheta>-\frac{1}{p}$ $(\vartheta>0\text{ if }p=\infty)$. Then, for any $f$ satisfying $d^{\vartheta+\frac{1}{p}} (f,f_r)\in L^2(J)$, there exists a constant $C>0$, which depends only on $(p,a,b,\vartheta)$ if $p\ne \infty$ and depends only on $(a,b,\vartheta)$ if $p=\infty$, such that
\begin{equation}\label{G-N1}
\|d^\vartheta f\|_{L^p(J)} \leq C\big(\big\|d^{\vartheta+\frac{1}{p}} f\big\|_{L^2(J)}+\big\|d^{\vartheta+\frac{1}{p}} f\big\|_{L^2(J)}^\frac{1}{2}\big\|d^{\vartheta+\frac{1}{p}} f_r\big\|_{L^2(J)}^\frac{1}{2}\big).
\end{equation}
\end{Lemma}

The seventh lemma is used to obtain the time-weighted estimates of the velocity.
\begin{Lemma}[\cite{bjr}]\label{bjr}
Let $f\in L^2([0,T]; L^2)$. Then there exists a sequence $\{t_k\}_{k=1}^\infty$ such that
\begin{equation*}
t_k\rightarrow 0, \quad\,\, t_k |f(t_k)|^2_{2}\rightarrow 0 \qquad\,\, \text{as $k\rightarrow\infty$}.
\end{equation*}
\end{Lemma}

The eighth lemma is an equivalent statement for the spherical symmetry of  vector  functions.
\begin{Lemma}[\cite{CZZ2}]\label{duichen-dengjia}
Let $\boldsymbol{f}=\boldsymbol{f}(\boldsymbol{y})$ be a spherically symmetric  continuous vector function on $B_R=\{\boldsymbol{y}: \,|\boldsymbol{y}|<R\}$ for some $R>0$. Then $\boldsymbol{f}$ takes the form{\rm :} $\boldsymbol{f}(\boldsymbol{y})=f(|\boldsymbol{y}|)\frac{\boldsymbol{y}}{|\boldsymbol{y}|}$ if and only if 
\begin{equation}\label{dengjia}
\mathcal{O}\boldsymbol{f}(\boldsymbol{y})=\boldsymbol{f}(\mathcal{O}\boldsymbol{y}) \qquad \text{for all $\boldsymbol{y}\in B_R$ and $\mathcal{O}\in \mathrm{SO}(3)$}.
\end{equation}
In particular, any spherically symmetric  vector  function $\boldsymbol{f}$ satisfies $\boldsymbol{f}(\boldsymbol{0})=\boldsymbol{0}$.
\end{Lemma}

For giving the time continuity of the velocity in our proof, the following two types of evolution triple embedding are required. 
\begin{Lemma}[\cite{evans}]\label{triple}
Let $T>0$ and $J\subset \RR^k$ $(k=1,2,3)$ be some open subset. Assume that $f\in L^2([0,T];H^1_0(J))$ and  $f_t\in L^2([0,T];H^{-1}(J))$. 
Then $f\in C([0,T];L^2(J))$, and the map{\rm :} $t\mapsto \|f(t)\|_{L^2(J)}^2$ is absolutely continuous 
with
\begin{equation*}
\frac{\mathrm{d}}{\mathrm{d}t} \|f(t)\|_{L^2(J)}^2=2\left<f_t, f\right>_{H^{-1}(J)\times H_0^1(J)} \qquad \text{for {\it a.e.} $t\in (0,T)$}.
\end{equation*}
Moreover, if additionally $f\in L^\infty([0,T];H_0^1(J))$, then $f\in C([0,T];L^4(J))$.
\end{Lemma}
\begin{Lemma}[\cite{CZZ2}]\label{Aubin}
Let $T>0$, $s\geq 0$, and 
\begin{equation*}
\cH\into L^2_{r^2\rho_0^s},%\qquad \cH\into L^2  \into \cH^*,
\end{equation*}
where $\cH$ is a Banach space and the embedding is continuous. Assume that $f\in L^2([0,T];\cH)$ and  $r^2\rho_0^sf_t\in L^2([0,T];\cH^*)$. Then $f\in C([0,T];L^2_{r^2\rho_0^s})$, and the map{\rm :} $t\mapsto |f(t)|_{2,r^2\rho_0^s}^2$ is absolutely continuous with 
\begin{equation}\label{timedirivative}
\frac{\mathrm{d}}{\dt} |f(t)|_{2,r^2\rho_0^s}^2=2\left<r^2\rho_0^sf_t, f\right>_{\cH^*\times \cH}.
\end{equation}
\end{Lemma}

\begin{proof}
This lemma can be obtained by basically following the proof of Theorem 3 on page 303 
in  \cite{evans}*{Chapter 5}, and we only sketch it here. By the continuous embedding and the assumptions, we have $r^2\rho_0^s f\in W^{1,2}(0,T;\cH^*)$. Hence, $r^2\rho_0^s f$ has a continuous $\cH^*$-valued representative on $[0,T]$. We first extend $f$ by reflection across $t=0$ and $t=T$ to $(-T,2T)$, that is, $f(t)=f(-t)$ for $-T<t<0$ and $f(t)=f(2T-t)$ for $T<t<2T$, and then extend it by zero outside this larger interval.
The reflected function belongs to $L^2((-T,2T);\cH)$, and the matching $\cH^*$-valued traces of $r^2\rho_0^s f$ at $0$ and $T$ imply that its
distributional time derivative belongs to $L^2((-T,2T);\cH^*)$.
Here and below, the same notation is used for the reflected functions.
For $0<\varepsilon<\frac{T}{2}$, we regularize $f$ with respect to $t$ via the convolution:
\begin{equation*}
f^\varepsilon(t,r):=\int_{-\infty}^\infty f(t-\tau,r)\omega_\varepsilon(\tau)\,\mathrm{d}\tau,
\end{equation*}
where $\omega_\varepsilon$ is the standard mollifier supported in
$[-\varepsilon,\varepsilon]$.
On $[0,T]$, this convolution only uses the reflected extension in $(-T,2T)$,
in particular, $r^2\rho_0^s f_t^\varepsilon=(r^2\rho_0^s f_t)*\omega_\varepsilon$ there.
Thus, for any $0<\varepsilon,\sigma<\frac{T}{2}$, we have
\begin{equation*}
\begin{aligned}
\frac{\mathrm{d}}{\mathrm{d}t}|f^\varepsilon(t)-f^\sigma(t)|_{L^2_{r^2\rho_0^s}}^2
&=2\big\langle r^2\rho_0^s f_t^\varepsilon-r^2\rho_0^s f_t^\sigma,
f^\varepsilon-f^\sigma\big\rangle_{\cH^*\times\cH}.
\end{aligned}
\end{equation*}
Integrating the above from $\tau$ to $t$, applying the Young inequality,
and then averaging over $\tau\in[0,T]$ imply
\begin{equation*}
\begin{aligned}
\sup_{t\in[0,T]}|f^\varepsilon(t)-f^\sigma(t)|_{L^2_{r^2\rho_0^s}}^2
&\leq \frac{1}{T}\int_0^T
|f^\varepsilon-f^\sigma|_{L^2_{r^2\rho_0^s}}^2\,\mathrm{d}t\\
&\quad+\int_0^T\big(|f^\varepsilon-f^\sigma|_{\cH}^2
+|r^2\rho_0^s f_t^\varepsilon-r^2\rho_0^s f_t^\sigma|_{\cH^*}^2\big)\,\mathrm{d}t.
\end{aligned}
\end{equation*}
Since $\cH$ and $\cH^*$ are Banach spaces, applying the approximation properties of time convolution to $f$ and $r^2\rho_0^s f_t$, respectively, we obtain
\begin{equation*}
\lim_{\varepsilon\to0}\int_0^T
\big(|f^\varepsilon-f|_{\cH}^2
+|r^2\rho_0^s f_t^\varepsilon-r^2\rho_0^s f_t|_{\cH^*}^2\big)\,\mathrm{d}t=0.
\end{equation*}
Hence, letting $(\varepsilon,\sigma)\to(0,0)$ and using the continuous
embedding $\cH\hookrightarrow L^2_{r^2\rho_0^s}$, we obtain
\begin{equation*}
\lim_{(\varepsilon,\sigma)\to(0,0)}\sup_{t\in[0,T]}
|f^\varepsilon(t)-f^\sigma(t)|_{L^2_{r^2\rho_0^s}}^2=0.
\end{equation*}
Therefore, $f^\varepsilon$ converges in $C([0,T];L^2_{r^2\rho_0^s})$
to a continuous representative of $f$, which is still denoted by $f$.
Similarly, we have
\begin{equation*}
\begin{aligned}
|f^\varepsilon(t)|_{L^2_{r^2\rho_0^s}}^2
&=|f^\varepsilon(\tau)|_{L^2_{r^2\rho_0^s}}^2
+2\int_\tau^t\big\langle r^2\rho_0^s f_t^\varepsilon,
f^\varepsilon\big\rangle_{\cH^*\times\cH}\,\mathrm{d}t'
\end{aligned}
\end{equation*}
for all $\tau,t\in[0,T]$.
Taking the limit as $\varepsilon\to0$, using the above strong convergences
and H\"older's inequality, implies
\begin{equation}\label{A10}
|f(t)|_{L^2_{r^2\rho_0^s}}^2
=|f(\tau)|_{L^2_{r^2\rho_0^s}}^2
+2\int_\tau^t\big\langle r^2\rho_0^s f_t,f\big\rangle_{\cH^*\times\cH}\,\mathrm{d}t',
\end{equation}
which implies the absolute continuity of $|f(t)|_{L^2_{r^2\rho_0^s}}^2$,
since the integrand belongs to $L^1(0,T)$.
Applying $\partial_t$ to \eqref{A10} yields \eqref{timedirivative}
for almost every $t\in(0,T)$.
\end{proof}

Finally,  we give a variant of the div-curl estimates for spherically symmetric functions. 
\begin{Lemma}\label{im-1}
% Let $\eta$ be the classical solution obtained in {\rm Theorem \ref{local-Theorem1.1}}, 
Let $\eta$ be an increasing $C^1$ radial coordinate map satisfying
\begin{equation*}
\eta(0)=0,\qquad
0<c^{-1}\leq\eta_r\leq c<\infty,\qquad
c^{-1}\leq\frac{\eta}{r}\leq c.
\end{equation*}
Let $f=f(r)$ be the profile of a smooth radial vector field,
regular at the origin in the $\eta$-coordinate. Then the following norms are equivalent{\rm :} for any $a\in (0,1)$,
\begin{equation*}
\begin{aligned}
&|\zeta_a \sD_\eta f|_{2,r^2}\sim\Big|\zeta_a \big(D_\eta f+ \frac{2f}{\eta}\big)\Big|_{2,r^2},\\
&|\zeta_a \sD_\eta^k f|_{2,r^2}\sim\Big|\zeta_a \sD_\eta^{k-2}\Big(D_\eta\big(D_\eta f+ \frac{2f}{\eta}\big)\Big)\Big|_{2,r^2} \quad \text{for $k=2,3,4$}, 
\end{aligned}
\end{equation*}
where $F_1\sim F_2$ denotes that there exists a constant $C\geq 1$ such that $C^{-1}F_1\leq F_2\leq CF_1$. Here the comparison constants depend only on the geometric bound $c$, and are independent of $(a,f)$. Moreover, If $\eta=\eta(t,r)$ satisfies the above hypotheses with uniform geometric bounds throughout the time interval, the estimates hold at each fixed time with comparison constants independent of time.
\end{Lemma}
\begin{proof}
For simplicity, we only give the proof for $a=\frac{1}{2}$ and show that 
\begin{equation*}
\big|\zeta r\mathscr{D}_\eta^4 f\big|_2\leq C \Big|\zeta r 
\mathscr{D}_\eta^2 \Big(D_\eta\big(D_\eta f+ \frac{2f}{\eta}\big)\Big) \Big|_2.
\end{equation*}
Besides, the following fact will be used frequently later:
\begin{equation*}
D_\eta\zeta\leq 0 \qquad \text{for all $r\in \bar I$}. 
\end{equation*}
All boundary terms in the integrations by parts below vanish by regularity at the origin and the support of $\zeta$.

First, a direct calculation gives
\begin{equation}\label{L6}
\begin{aligned}
&\,\Big|\zeta (\eta^2\eta_r)^\frac{1}{2} D_\eta^3\big(D_\eta f+ \frac{2f}{\eta}\big)\Big|_2^2\\
&=\int_0^1 \zeta^2 \eta^2 \eta_r \big(|D_\eta^4 f|^2 +4 \big|D_\eta^3\big(\frac{f}{\eta}\big)\big|^2\big) \,\mathrm{d}r+\underline{4\int_0^1 \zeta^2\eta^2 \eta_r D_\eta^4f D_\eta^3\big(\frac{f}{\eta}\big)\,\mathrm{d}r}_{:={AL}_{1}}.
\end{aligned}
\end{equation}
Then, for ${AL}_{1}$, note that the following identity holds:
\begin{equation*}
D_\eta^4 f=\eta D_\eta^4\big(\frac{f}{\eta}\big)+ 4 D_\eta^3\big(\frac{f}{\eta}\big).
\end{equation*}
Hence, this, together with integration by parts, yields  
\begin{equation*}
\begin{aligned}
{AL}_{1}&=16\int_0^1 \zeta^2 \eta^2 \eta_r \big|D_\eta^3\big(\frac{f}{\eta}\big)\big|^2 \,\mathrm{d}r +4\int_0^1 \zeta^2 \eta^{3} \eta_r D_\eta^4\big(\frac{f}{\eta}\big) D_\eta^3\big(\frac{f}{\eta}\big)\,\mathrm{d}r\\
&=10\int_0^1 \zeta^2 \eta^2 \eta_r \big|D_\eta^3\big(\frac{f}{\eta}\big)\big|^2 \,\mathrm{d}r -4\int_0^1 \zeta D_\eta \zeta \eta^{3} \eta_r  \big|D_\eta^3 \big(\frac{f}{\eta}\big)\big|^2\, \mathrm{d}r\geq 0,    
\end{aligned}
\end{equation*}
which, along with \eqref{L6}, gives
\begin{equation}\label{L6'}
\Big|\zeta (\eta^2\eta_r)^\frac{1}{2} D_\eta^3\big(D_\eta f+ \frac{2f}{\eta}\big)\Big|_2^2\geq  \int_0^1 \zeta^2 \eta^2 \eta_r |D_\eta^2(\mathscr{D}_\eta^2 f)|^2  \,\mathrm{d}r.
\end{equation}

Next, notice that
\begin{equation}\label{idenen}  
D_\eta\big(\frac{D_\eta^2 f}{\eta}\big)=D_\eta^3\big(\frac{f}{\eta}\big)+2D_\eta\Big(\frac{1}{\eta}D_\eta\big(\frac{f}{\eta}\big)\Big)=\eta D_\eta^2\Big(\frac{1}{\eta}D_\eta\big(\frac{f}{\eta}\big)\Big)+4D_\eta\Big(\frac{1}{\eta}D_\eta\big(\frac{f}{\eta}\big)\Big).
\end{equation}
Hence, this, together with a direct calculation, implies
\begin{align}
&\,\Big|\zeta (\eta^2\eta_r)^\frac{1}{2} D_\eta\Big(\frac{1}{\eta}D_\eta\big(D_\eta f+ \frac{2 f}{\eta}\big)\Big)\Big|_2^2\nonumber\\
&=\Big|\zeta(\eta^2\eta_r)^\frac{1}{2}\Big(\eta D_\eta^2\big(\frac{1}{\eta}D_\eta\big(\frac{f}{\eta}\big)\big)+6D_\eta\big(\frac{1}{\eta}D_\eta\big(\frac{f}{\eta}\big)\big)\Big)\Big|_2^2\label{9.6}\\
&=\int_0^1 \zeta^2 \eta^{2} \eta_r \Big(\Big|\eta D_\eta^2\Big(\frac{1}{\eta}D_\eta\big(\frac{f}{\eta}\big)\Big)\Big|^2 +36 \Big|D_\eta\Big(\frac{1}{\eta}D_\eta\big(\frac{f}{\eta}\big)\Big)\Big|^2\Big)\,\mathrm{d}r\nonumber\\
&\quad +\underline{12\int_0^1 \zeta^2\eta^{3}\eta_r D_\eta^2\Big(\frac{1}{\eta}D_\eta\big(\frac{f}{\eta}\big)\Big)D_\eta\Big(\frac{1}{\eta}D_\eta\big(\frac{f}{\eta}\big)\Big)\,\mathrm{d}r}_{:={AL}_{2}},\nonumber
\end{align}
where $AL_{2}$ can be handled by integration by parts:
\begin{equation*}
\begin{aligned}
AL_{2}&=-18\int_0^1 \zeta^2\eta^2\eta_r\Big|D_\eta\Big(\frac{1}{\eta}D_\eta\big(\frac{f}{\eta}\big)\Big)\Big|^2\,\mathrm{d}r-12\int_0^1 \zeta D_\eta\zeta \eta^{3} \eta_r \Big|D_\eta\Big(\frac{1}{\eta}D_\eta\big(\frac{f}{\eta}\big)\Big)\Big|^2 \,\mathrm{d}r\\
&\geq -18 \int_0^1 \zeta^2\eta^2\eta_r\Big|D_\eta\Big(\frac{1}{\eta}D_\eta\big(\frac{f}{\eta}\big)\Big)\Big|^2\,\mathrm{d}r.    
\end{aligned}
\end{equation*}
Therefore, \eqref{9.6}, combined with the above, implies 
\begin{equation}\label{l6"}
\Big|\zeta (\eta^2\eta_r)^\frac{1}{2} D_\eta\Big(\frac{1}{\eta}D_\eta\big(D_\eta f+ \frac{2 f}{\eta}\big)\Big)\Big|_2^2\geq 18\int_0^1 \zeta^2 \eta^{2} \eta_r   \Big|D_\eta\Big(\frac{1}{\eta}D_\eta\big(\frac{f}{\eta}\big)\Big)\Big|^2 \,\mathrm{d}r.
\end{equation}

Finally, it follows from the above and \eqref{L6'}--\eqref{idenen} that
\begin{equation}\label{l6""}
\int_0^1 \zeta^2 \eta^{2} \eta_r \Big|D_\eta\big(\frac{D_\eta^2 f}{\eta}\big)\Big|^2 \,\mathrm{d}r\leq C\Big|\zeta (\eta^2\eta_r)^\frac{1}{2} \sD_\eta^2 \Big(D_\eta \big(D_\eta f+ \frac{2f}{\eta}\big)\Big)\Big|_2^2.
\end{equation}
Hence, by \eqref{L6'}, \eqref{l6"}--\eqref{l6""}, and the lower and upper bounds of $(\frac{\eta}{r},\eta_r)$, we obtain the desired estimates. The reverse estimate follows directly from the definitions. The same argument applies to $\zeta_a$, since $D_\eta\zeta_a\le0$, with constants independent of $a$. The first-, second-, and third-order estimates follow similarly by integration by parts.

For the Sobolev version, approximation by smooth radial vector fields is performed in the $\eta$-coordinate on a ball containing the support of $\zeta_a$. In particular, if the radial vector field is locally $H^1$ and its distributional Laplacian is locally $L^2$, interior elliptic regularity first gives local $H^2$ regularity, and the second-order estimate follows by approximation.
\end{proof}

\section{Coordinate Transformations}\label{appb}

Generally, it is desirable to consider the Lagrangian formulation, such that we can pullback  \eqref{eq:1.1-vfbp} on the moving domain $\Omega(t)$ to a problem on a fixed domain $\Omega$. To this end, denote by $\boldsymbol{x}=\boldsymbol{\eta}(t,\boldsymbol{y})$ the position of the fluid particle $\boldsymbol{x} \in \Omega(t)$ at time $t\geq 0$ such that
\begin{equation}\label{flow-map-md}
\boldsymbol{\eta}_t(t,\boldsymbol{y})=\boldsymbol{u}(t,\boldsymbol{\eta}(t,\boldsymbol{y})) \ \ \text{for $t>0$},\qquad \text{with $\boldsymbol{\eta}(0,\boldsymbol{y})=\boldsymbol{y}$},
\end{equation}
and $(t,\boldsymbol{y})$ are the 3-D Lagrangian coordinates.
This appendix is devoted to showing the conversion of some Sobolev spaces between the 3-D Lagrangian coordinates $\boldsymbol{y}$ and the corresponding spherical coordinate $r=|\boldsymbol{y}|$ for spherically symmetric functions.

Let $0\leq a<b$, $\cJ:=\{\boldsymbol{y}\in \mathbb{R}^3: \, a\leq |\boldsymbol{y}|< b\}$, and $r=|\boldsymbol{y}|\in J:=[a,b)$. Consider a coordinate transformation $\boldsymbol{\xi}=\boldsymbol{\xi}(\boldsymbol{y})\in C^\infty(\bar\cJ)$ such that
\begin{equation*}
\boldsymbol{\xi}(\boldsymbol{y})=\xi(r)\frac{\boldsymbol{y}}{r}\quad\text{with $\xi(r)\geq 0$},\qquad\boldsymbol{\xi}:\cJ\to \cG:=\boldsymbol{\xi}(\cJ),\qquad \boldsymbol{y}\mapsto \boldsymbol{x}=\boldsymbol{\xi}(\boldsymbol{y}). 
\end{equation*} 
Assume that $\nabla_{\boldsymbol{y}}\boldsymbol{\xi}$ is a non-singular matrix and $\xi_r(r)>0$ on $J$, and then define
\begin{equation*}
\begin{gathered}
D_\xi f=\frac{f_r}{\xi_r},\qquad\cB=(\cB_{ij})_{1\leq i,j\leq 3} \qquad \text{with $\cB_{ij}:=((\nabla_{\boldsymbol{y}}\boldsymbol{\xi})^{-1})_{ij}$},\\[-2pt]
\nabla_\cB f=((\nabla_\cB f)_1,(\nabla_\cB f)_2,(\nabla_\cB f)_3)^\top \qquad \text{with $(\nabla_\cB f)_i=\sum_{k=1}^3\cB_{ki}\partial_{y_k}f$},\\[-3pt]
\boldsymbol{f}=(f_1,f_2, f_3)^\top,\qquad\nabla_\cB \boldsymbol{f}=((\nabla_\cB \boldsymbol{f})_{ij})_{1\leq i,j\leq 3} \qquad \text{with $(\nabla_\cB \boldsymbol{f})_{ij}=\sum_{k=1}^3 \cB_{kj}\partial_{y_k}f_i$},    
\end{gathered}
\end{equation*}
where $f(\boldsymbol{y})=f(r)$ and $\boldsymbol{f}(\boldsymbol{y})=f(r)\frac{\boldsymbol{y}}{r}$ are sufficiently smooth functions. 

Then we have the following coordinate transformations.
\begin{Lemma}\label{lemma-initial}
Assume that $(g,\boldsymbol{g})(\boldsymbol{x})$ are spherically symmetric functions defined on $\cG$ and $(f,\boldsymbol{f})(\boldsymbol{y})$ satisfy
\begin{equation*}
f(\boldsymbol{y})=g(\boldsymbol{\xi}(\boldsymbol{y}))=g(\boldsymbol{x}),\qquad\boldsymbol{f}(\boldsymbol{y})=\boldsymbol{g}(\boldsymbol{\xi}(\boldsymbol{y}))=\boldsymbol{g}(\boldsymbol{x}).
\end{equation*}
Then, for any $q\in [1,\infty]$, the following statements hold{\rm :}
\begin{enumerate}
\item[{\rm(i)}] Transformations for $(g,f)${\rm:} for $j=1,2,3$, 
\begin{equation*}
\begin{aligned}
\|g\|_{L^q(\cG)}&\sim\|(\det \nabla_{\boldsymbol{y}}\boldsymbol{\xi})^\frac{1}{q}f\|_{L^q(\cJ)}\sim 
\|(\xi^2\xi_r)^\frac{1}{q}f\|_{L^q(J)},\\[2pt]
\|\nabla^jg\|_{L^q(\cG)}&\sim\|(\det \nabla_{\boldsymbol{y}}\boldsymbol{\xi})^\frac{1}{q}\nabla_\cB^j f\|_{L^q(\cJ)}\sim  \|(\xi^2\xi_r)^\frac{1}{q}\sD_\xi^{j-1}(D_\xi f)\|_{L^q(J)};
\end{aligned}
\end{equation*}
\item[{\rm(ii)}] Transformations for $(\boldsymbol{g},\boldsymbol{f})${\rm:} for $j=0,1,2,3,4$,
\begin{equation*}
\|\nabla^{j}\boldsymbol{g}\|_{L^q(\cG)}\sim\|(\det \nabla_{\boldsymbol{y}}\boldsymbol{\xi})^\frac{1}{q}\nabla_\cB^{j} \boldsymbol{f}\|_{L^q(\cJ)}\sim \|(\xi^2\xi_r)^\frac{1}{q}\sD_\xi^j f\|_{L^q(J)}.
\end{equation*}
\end{enumerate}
Here, $E\sim F$ denotes $C^{-1}E\leq F\leq CE$ for some universal constant $C\geq 1$.
\end{Lemma}

\begin{proof}
It suffices to prove the transformations for $(\boldsymbol{g},\boldsymbol{f})$, since $\nabla_{\boldsymbol{y}} f=f_r\frac{\boldsymbol{y}}{r}$ can be regarded as a vector  function $\boldsymbol{h}=h\frac{\boldsymbol{y}}{r}$ with $h=f_r$. Moreover, for simplicity, we only sketch the proof for the highest-order estimates. We divide the proof into two steps.

\textbf{1.} We first prove the case when $\xi(r)=r$. In this case, $\boldsymbol{\xi}(\boldsymbol{y})=\boldsymbol{y}$ and $\cB$ is the $3\times 3$ identity matrix. It follows from direct calculations that
\begin{align*}
&\begin{aligned}
(f_k)_{y_iy_jy_\ell y_p}
&=\frac{y_i y_j y_ky_\ell y_p}{r^5}f_{rrrr}\\
&\quad +\Big(\frac{\delta_{ip} y_j y_ky_\ell+\delta_{jp} y_i y_ky_\ell+\delta_{kp} y_i y_jy_\ell+\delta_{\ell p} y_i y_jy_k}{r^3}\\
&\quad\quad \ \ + \frac{\delta_{i\ell}y_jy_ky_p+\delta_{j\ell}y_iy_ky_p+\delta_{k\ell}y_iy_jy_p}{r^3}\\
&\quad\quad \ \ +\frac{\delta_{ij}y_ky_\ell y_p+\delta_{ik}y_jy_\ell y_p+\delta_{jk}y_iy_\ell y_p}{r^3}-\frac{10y_i y_j y_ky_\ell y_p}{r^5}\Big)\big(\frac{f}{r}\big)_{rrr}\\
&\quad +\Big(\frac{\delta_{i\ell}\delta_{jp}y_k+\delta_{i\ell}\delta_{kp}y_j+\delta_{j\ell}\delta_{ip}y_k+\delta_{j\ell}\delta_{kp}y_i+\delta_{k\ell}\delta_{ip}y_j+\delta_{k\ell}\delta_{jp}y_i}{r}\\
&\quad\quad \ \ +\frac{\delta_{ij}\delta_{kp}y_\ell+\delta_{ij}\delta_{\ell p}y_k+\delta_{ik}\delta_{jp}y_\ell+\delta_{ik}\delta_{\ell p} y_j+\delta_{jk}\delta_{ip}y_\ell+\delta_{jk}\delta_{\ell p}y_i}{r}\\
\end{aligned}\\
&\qquad\quad\quad \ \ \begin{aligned}
&\quad\quad \ \ +\frac{\delta_{ij}\delta_{k\ell}y_p+\delta_{ik}\delta_{j\ell}y_p+\delta_{jk}\delta_{i\ell}y_p}{r}\\
&\quad\quad \ \ -\frac{3(\delta_{ip} y_j y_ky_\ell+\delta_{jp} y_i y_ky_\ell+\delta_{kp} y_i y_jy_\ell+\delta_{\ell p}y_i y_j y_k)}{r^3}\\
&\quad\quad \ \ -\frac{3(\delta_{i\ell}y_jy_ky_p+\delta_{j\ell}y_iy_ky_p+\delta_{k\ell}y_iy_jy_p)}{r^3}\\
&\quad\quad \ \ -\frac{3(\delta_{ij}y_ky_\ell y_p+\delta_{ik}y_jy_\ell y_p+\delta_{jk}y_iy_\ell y_p)}{r^3}+\frac{15y_i y_j y_ky_\ell y_p}{r^5}\Big)\Big(\frac{1}{r}\big(\frac{f}{r}\big)_{r}\Big)_r,   
\end{aligned}
\end{align*}
where $\delta_{ij}$ denotes the Kronecker symbol with indices $(i,j)$: $\delta_{ij}= 1$ if $i=j$, $\delta_{ij}=0$ if $i\neq j$. Then the above expressions yield
\begin{equation*}
|\nabla_{\boldsymbol{y}}^4 \boldsymbol{f}|^2=|f_{rrrr}|^2+20\Big|\big(\frac{f}{r}\big)_{rrr}\Big|^2+120 \Big|\Big(\frac{1}{r}\big(\frac{f}{r}\big)_r\Big)_r\Big|^2.
\end{equation*}
Moreover, since $(\frac{f_{rr}}{r})_r=(\frac{f}{r})_{rrr}+2(\frac{1}{r}(\frac{f}{r})_r)_r$, we obtain from the above that
\begin{equation}\label{BB}
|\nabla^4_{\boldsymbol{y}} \boldsymbol{f}|\sim |\sD_r^4 f|.
\end{equation}

Finally, denote $\omega_3$ as the surface area of the $2$-sphere,
due  to the integral identity:  
\begin{equation*}
\int_\cJ f(\boldsymbol{y})\,\mathrm{d}\boldsymbol{y}= \omega_3\int_J f(r)r^2\,\mathrm{d}r,
\end{equation*}
 we thus obtain the desired conclusions  of this lemma when $\boldsymbol{\xi}(\boldsymbol{y})=\boldsymbol{y}$ from \eqref{BB}.

\textbf{2.} For general $\boldsymbol{\xi}(\boldsymbol{y})$, we can first repeat the calculations in Step 1 with the coordinate $\boldsymbol{x}=\boldsymbol{\xi}(\boldsymbol{y})$ and the function $\boldsymbol{g}(\boldsymbol{x})$. Specifically, if we let $x:=|\boldsymbol{x}|$, then $x=\xi(r)\in W$ and $W:=[\xi(a),\xi(b))$, and we can obtain from \eqref{BB} that
\begin{equation*} 
\|\nabla^4 \boldsymbol{g}\|_{L^q(\cG)}\sim \big\|x^\frac{2}{q}\sD_x^4 g\big\|_{L^q(W)}.
\end{equation*}

Next, using the coordinate transformations $\boldsymbol{x}=\boldsymbol{\xi}(\boldsymbol{y})$ and $x=\xi(r)$, we have $\nabla^k\boldsymbol{g}=\nabla_{\cB}^k\boldsymbol{f}$ and $\partial_x^k g=D_\xi^k f$ for $k=0,1,2,3,4$. Therefore, the following integral identities
\begin{equation*}
\int_\cG g(\boldsymbol{x})\,\mathrm{d}\boldsymbol{x}=\int_\cJ f(\boldsymbol{y})(\det \nabla_{\boldsymbol{y}}\boldsymbol{\xi}) \,\mathrm{d}\boldsymbol{y},\qquad \int_W g(x)\,\mathrm{d}x=  \int_J f(r) \xi_r\,\mathrm{d}r,
\end{equation*}
lead to the desired results of this lemma.
\end{proof}

\section{Remarks on the Energy Functionals}\label{AppB}
This appendix is devoted to giving some equivalent forms of the energy functionals. In what follows, we always let $(\rho_0,u_0)$ be the initial data of {\rm\bf IBVP} \eqref{eq:VFBP-La-eta},  $\rho_0$ satisfy \eqref{distance-la}, and $u_0$ satisfy \eqref{a2-lo}.

Define 
\begin{equation}\label{E-1a}
\begin{aligned}
\mathring\cE (t,f)&=\mathring\cE_{\mathrm{in}}(t,f)+\mathring\cE_{\mathrm{ex}}(t,f),\\
\mathring\cE_{\mathrm{in}}(t,f)&:=\big|\zeta r (f,\sD_rf,f_t,\sD_rf_t)(t)\big|_{2}^2+\big|\zeta r (\sD_r^2f,\sD_r^3f)(t)\big|_{2}^2,\\
\mathring\cE_{\mathrm{ex}}(t,f)&:=\big|\chi^\sharp\rho_0^\frac{1}{2}(f,f_t)(t)\big|_{2}^2+\big|\chi^\sharp\rho_0^\frac{\delta}{2}(f_r,f_{tr})(t)\big|_{2}^2+\big|\chi^\sharp\rho_0^{(\gamma-1)(\frac{3}{2}-\varepsilon_0)}(f_{rr},f_{rrr})(t)\big|_{2}^2,
\end{aligned}
\end{equation}
and
\begin{equation}\label{D-1a}
\begin{aligned}
\mathring\cD(t,f)&=\mathring\cD_{\mathrm{in}}(t,f)+\mathring\cD_{\mathrm{ex}}(t,f),\\
\mathring\cD_{\mathrm{in}}(t,f)&:=\big|\zeta r (f_{tt}, \sD_r^2f_{t},\sD_r^4f)(t)\big|_{2}^2,\\
\mathring\cD_{\mathrm{ex}}(t,f)&:=\big|\chi^\sharp\rho_0^\frac{1}{2}f_{tt}(t)\big|_{2}^2+\big|\chi^\sharp\rho_0^{(\gamma-1)(\frac{3}{2}-\varepsilon_0)}(f_{trr},f_{rrrr})(t)\big|_{2}^2.
\end{aligned}
\end{equation}

% Clearly, we have
% \begin{equation*}
% \mathring\cX (0,f)=\cX(0,f) \qquad\text{for $\cX=\cE,\cE_{\mathrm{in}},\cE_{\mathrm{ex}},\cD,\cD_{\mathrm{in}},\cD_{\mathrm{ex}}$}.
% \end{equation*}

\begin{Lemma}\label{lemma-gaowei}
Let $T>0$, and let $\eta$ be given by
\begin{equation*}
\eta(t,r)=r+\int_0^t U(s,r)\,\mathrm{d}s,\qquad
\eta(t,0)=0,\qquad
\delta_*\leq\eta_r,\frac{\eta}{r}\leq\delta^* \quad \text{on } [0,T]\times I,
\end{equation*}
for some positive constants $(\delta_*,\delta^*)$. 
Assume that, for some $\mathrm{E}_1>1$, either 
\begin{equation}\label{C333}
\begin{aligned}
\sup_{t\in[0,T]}(\cE (t,U)+t\cD (t,U))+\int_0^T \cD(s,U)\,\ds&\leq\mathrm{E}_1<\infty, \qquad \text{or }\\
\sup_{t\in[0,T]}(\mathring\cE (t,U)+t\mathring\cD (t,U))+\int_0^T \mathring\cD(s,U)\,\ds&\leq\mathrm{E}_1<\infty.\qquad\quad 
\end{aligned}
\end{equation}
Assume that $f=f(t,r)$ is a function defined on $[0,T]\times I$ such that
\begin{equation}\label{revise c.4}
\sup_{t\in[0,T]}(\cE (t,f)+t\cD (t,f))+\int_0^T \cD(s,f)\,\ds<\infty,
\end{equation}
or the same finiteness condition holds with $(\mathring{\mathcal E},\mathring{\mathcal D})$ in place of $(\mathcal E,\mathcal D)$. Then, for any $0\leq t\leq T$,
\begin{equation*}
\mathcal E(t,f)\sim\mathring{\mathcal E}(t,f),
\end{equation*}
and
\begin{equation*}
\begin{aligned}
&\sup_{0<s\leq t}\Big(\mathcal E(s,f)+s\mathcal D(s,f)\Big)+\int_0^t\mathcal D(s,f)\,\mathrm{d}s\\
&\qquad\sim
\sup_{0<s\leq t}\Big(\mathring{\mathcal E}(s,f)+s\mathring{\mathcal D}(s,f)\Big)+\int_0^t\mathring{\mathcal D}(s,f)\,\mathrm{d}s.
\end{aligned}
\end{equation*}
% Then, for any $t\in[0,T]$, 
% \begin{equation*}
% \begin{aligned}
% &\cE_{\mathrm{in}} (t,f)\sim \mathring\cE_{\mathrm{in}} (t,f),\qquad \cE_{\mathrm{ex}} (t,f)\sim \mathring\cE_{\mathrm{ex}} (t,f),\\
% &\cE_{\mathrm{in}} (t,f)+t\cD_{\mathrm{in}} (t,f)+\int_0^t \cD_{\mathrm{in}}(s,f)\,\ds\sim \mathring\cE_{\mathrm{in}} (t,f)+t\mathring\cD_{\mathrm{in}} (t,f)+\int_0^t \mathring\cD_{\mathrm{in}}(s,f)\,\ds,\\
% &\cE_{\mathrm{ex}} (t,f)+t\cD_{\mathrm{ex}} (t,f)+\int_0^t \cD_{\mathrm{ex}}(s,f)\,\ds\sim \mathring\cE_{\mathrm{ex}} (t,f)+t\mathring\cD_{\mathrm{ex}} (t,f)+\int_0^t \mathring\cD_{\mathrm{ex}}(s,f)\,\ds,
% \end{aligned}
% \end{equation*}
where $F_1\sim F_2$ denotes
\begin{equation*}
C^{-1}F_1\leq F_2\leq CF_1
\end{equation*}
for some positive finite constant $C$, which depends only on $(\delta_*,\delta^*,\mathrm{E}_1,T,\rho_0,\gamma,\delta,\varepsilon_0)$. Moreover, if $T\leq 1$ and $\sqrt{T}\mathfrak{p}(\mathrm{E}_1) \leq 1$, $C$ can be chosen independently of $(T,\mathrm{E}_1)$, where $\mathfrak{p}$ is a generic polynomial function defined in \S \ref{subsection9.2}.
\end{Lemma}
\begin{proof}
Note that 
\begin{equation}
f_r=\eta_rD_\eta f,\qquad f_{tr}=\eta_rD_\eta f_t,
\end{equation}
Together with the bounds for $(\frac{\eta}{r},\eta_r)$, we directly obtain
\begin{equation*}
\begin{aligned}
 \big|\zeta r (f,\sD_r f,f_t,\sD_r f_t)\big|_{2}^2&\sim\big|\zeta r (f,\sD_\eta f,f_t,\sD_\eta f_t)\big|_{2}^2\\
 \big|\chi^\sharp\rho_0^\frac{1}{2}(f,f_t)\big|_{2}^2+\big|\chi^\sharp\rho_0^\frac{\delta}{2}(f_r,f_{tr})\big|_{2}^2&\sim\big|\chi^\sharp\rho_0^\frac{1}{2}(f,f_t)\big|_{2}^2+\big|\chi^\sharp\rho_0^\frac{\delta}{2}( D_\eta f,D_\eta f_{t})\big|_{2}^2,
\end{aligned}
\end{equation*}
with comparison constants depending only on $(\delta_{*},\delta^{*})$. The terms involving $f_{tt}$ are unchanged, and thus we only need to focus on the estimates for the higher spatial derivatives of $f$ and the second spatial derivatives of $f_t$. 

We divide the proof into two steps.

\smallskip
\textbf{1. Estimates for $\eta$.} In the following estimates for $\eta$, $C$ denotes a positive constant depending only on $(\delta_*,\delta^*,\rho_0,\gamma,\delta,\varepsilon_0)$, but independent of $(t,T,\mathrm E_1)$. First, in view of $\eqref{eq:VFBP-La-eta}_2$, we have 
\begin{equation}\label{bbb1}
\begin{aligned}
\eta_{trr}&=\eta_{rr} D_{\eta}U+\eta_r^2 D_{\eta}^2U,\qquad \eta_{trrr}=\eta_{rrr} D_{\eta}U+3\eta_r\eta_{rr} D_{\eta}^2U+\eta_r^3 D_{\eta}^3U,\\
\eta_{trrrr}&=\eta_{rrrr} D_{\eta}U+(3\eta^2_{rr}+4\eta_r\eta_{rrr}) D_{\eta}^2U+6\eta_{r}^2\eta_{rr}D_{\eta}^3U+\eta_r^4 D_{\eta}^4U,\\
\big(\frac{\eta}{r}\big)_{tr}&=\big(\frac{\eta}{r}\big)_r\frac{U}{\eta}+\eta_r\frac{\eta}{r}D_{\eta}\big(\frac{U}{\eta}\big),\qquad \frac{1}{r}\big(\frac{\eta}{r}\big)_{tr} =\frac{1}{r}\big(\frac{\eta}{r}\big)_r\frac{U}{\eta}+\eta_r\big(\frac{\eta}{r}\big)^2\frac{1}{\eta}    D_{\eta}\big(\frac{U}{\eta}\big),\\
\big(\frac{\eta}{r}\big)_{trr}&=\big(\frac{\eta}{r}\big)_{rr}\frac{U}{\eta}+2\eta_r\big(\frac{\eta}{r}\big)_rD_\eta\big(\frac{U}{\eta}\big)+\eta_{rr}\frac{\eta}{r}D_{\eta}\big(\frac{U}{\eta}\big)+\eta_r^2\frac{\eta}{r} D_{\eta}^2\big(\frac{U}{\eta}\big),\\
\big(\frac{\eta}{r}\big)_{trrr}&=\big(\frac{\eta}{r}\big)_{rrr}\frac{U}{\eta}+\Big(3\eta_r\big(\frac{\eta}{r}\big)_{rr}+3\eta_{rr}\big(\frac{\eta}{r}\big)_{r}+\eta_{rrr}\frac{\eta}{r}\Big)D_\eta\big(\frac{U}{\eta}\big)\\
&\quad+\Big(3\eta_{r}^2\big(\frac{\eta}{r}\big)_r+3\eta_r\eta_{rr}\frac{\eta}{r}\Big) D_{\eta}^2\big(\frac{U}{\eta}\big)+\eta_r^3\frac{\eta}{r} D_{\eta}^3\big(\frac{U}{\eta}\big),\\
\big(\frac{1}{r}\big(\frac{\eta}{r}\big)_r\big)_{tr} &=\big(\frac{1}{r}\big(\frac{\eta}{r}\big)_r\big)_r\frac{U}{\eta}+\big( \eta_{rr}\big(\frac{\eta}{r}\big)^2+3\eta_{r}\frac{\eta}{r}\big(\frac{\eta}{r}\big)_r\big)\frac{1}{\eta}D_{\eta}\big(\frac{U}{\eta}\big)+ \eta_r^2\big(\frac{\eta}{r}\big)^2D_\eta\big(\frac{1}{\eta}    D_{\eta}\big(\frac{U}{\eta}\big)\big),\\
\end{aligned}
\end{equation}
On the other hand, it follows from \eqref{C333} and Lemmas \ref{sobolev-embedding}--\ref{hardy-inequality} that
\begin{align}
|\sD_\eta U|_\infty &\leq C\big| (\sD_\eta U,\sD_\eta^2 U)\big|_1\notag\\
&\leq C\sum_{j=1}^3\Big(\big|\chi r\sD_\eta^j U\big|_2+\big|\chi^\sharp \rho_0^{(\gamma-1)(\frac{3}{2}-\varepsilon_0)}D_\eta^j U\big|_2\Big)\leq C\cE(t,U)^{\frac{1}{2}},\label{bbb2-0}\\
|\sD_\eta^2 U|_\infty 
&\leq C\sum_{j=2}^4\Big(\big|\chi r\sD_\eta^j U\big|_2+\big|\chi^\sharp \rho_0^{(\gamma-1)(\frac{3}{2}-\varepsilon_0)}D_\eta^j U\big|_2\Big)\leq C\cE(t,U)^{\frac{1}{2}}+C\cD(t,U)^\frac{1}{2}.\label{bbb2}
\end{align}

As a consequence, \eqref{bbb1}, together with \eqref{C333}, \eqref{bbb2}, and the H\"older and Gr\"onwall inequalities, implies that
\begin{equation}\label{bbb3}
|\sD_r^2\eta|_\infty\leq Ce^{Ct\mathrm{E}_1^\frac{1}{2}}\int_0^t |\sD_\eta^2 U|_\infty\,\mathrm{d}s\leq C\sqrt{t(1+t)\mathrm{E}_1}e^{Ct\mathrm{E}_1^\frac{1}{2}},
\end{equation}
and
\begin{equation*}
|\sD_r^3\eta|\leq Ce^{Ct\mathrm{E}_1^\frac{1}{2}}\Big(\int_0^t   |\sD_r^2\eta|_\infty|\sD_\eta^2 U|\,\mathrm{d}s + \int_0^t |\sD_\eta^3 U|\,\mathrm{d}s\Big) \leq Ce^{Ct\mathrm{E}_1^\frac{1}{2}}\Big(t(1+t)\mathrm{E}_1+ \int_0^t |\sD_\eta^3 U|\,\mathrm{d}s\Big),
\end{equation*}
which, along with the Minkowski inequality, leads to
\begin{equation}\label{bbb4}
|\zeta r\sD_r^3\eta|_2+
\big|\chi^\sharp \rho_0^{(\gamma-1)(\frac{3}{2}-\varepsilon_0)} \eta_{rrr}\big|_2 \leq Cte^{Ct\mathrm{E}_1^\frac{1}{2}}\big((1+t)\mathrm{E}_1+ \mathrm{E}_1^\frac{1}{2}\big)\leq Ct(1+t)\mathrm{E}_1 e^{Ct\mathrm{E}_1^\frac{1}{2}}.
\end{equation}
Similarly, the fourth-order identities imply that
\begin{equation*}
\begin{aligned}
|\mathscr D_r^4\eta(t,r)|
&\leq Ce^{Ct\mathrm E_1^{\frac{1}{2}}}\int_0^t\Big(|\mathscr D_\eta^2U(s)|_\infty|\mathscr D_r^3\eta(s,r)|
+|\mathscr D_r^2\eta(s)|_\infty^2|\mathscr D_\eta^2U(s,r)|\\
&\qquad\qquad\qquad+|\mathscr D_r^2\eta(s)|_\infty
|\mathscr D_\eta^3U(s,r)|
+|\mathscr D_\eta^4U(s,r)|
\Big)\,\mathrm{d}s.
\end{aligned}    
\end{equation*}
Applying the Minkowski inequality, and using \eqref{bbb2}--\eqref{bbb4}, we obtain
\begin{equation}\label{revise bbb5}
\begin{aligned}
&|\zeta r\mathscr D_r^4\eta(t)|_2
+\big|\chi^\sharp
\rho_0^{(\gamma-1)(\frac32-\varepsilon_0)}
\eta_{rrrr}(t)\big|_2\\
&\quad\leq
Ce^{Ct\mathrm E_1^{\frac{1}{2}}}\Big(
t(1+t)\mathrm E_1
\int_0^t|\mathscr D_\eta^2U(s)|_\infty\,\mathrm{d}s\\
&\hspace{85pt}
+\sqrt{t(1+t)\mathrm E_1}\,t\mathrm E_1^{\frac{1}{2}}
+\int_0^t\bigl(\mathcal E(s,U)+\mathcal D(s,U)\bigr)^{\frac{1}{2}}
\,\mathrm{d}s\Big)\\
&\quad\leq
Ce^{Ct\mathrm E_1^{\frac{1}{2}}}\Big(t^{\frac{3}{2}}(1+t)^{\frac{3}{2}}\mathrm E_1^{\frac{3}{2}}
+t^{\frac{3}{2}}(1+t)^{\frac{1}{2}}\mathrm E_1+\sqrt{t(1+t)\mathrm E_1}\Big).
\end{aligned}  
\end{equation}
In particular, if $T\leq 1$ and $CT\mathrm{E}_1^{\frac{1}{2}}\leq 1$, \eqref{bbb3}--\eqref{revise bbb5} yield
\begin{equation}\label{rbbb6}
\begin{aligned}
&|\mathscr D_r^2\eta(t)|_\infty
\leq C\sqrt{t}\mathrm{E}_1^{\frac{1}{2}},\quad
|\zeta r\mathscr D_r^3\eta(t)|_2+\big|\chi^\sharp
\rho_0^{(\gamma-1)(\frac{3}{2}-\varepsilon_0)}\eta_{rrr}(t)\big|_2
\leq Ct\mathrm{E}_1,\\
&|\zeta r\mathscr D_r^4\eta(t)|_2+\big|\chi^\sharp\rho_0^{(\gamma-1)(\frac{3}{2}-\varepsilon_0)}
\eta_{rrrr}(t)\big|_2\leq C\sqrt{t}\mathrm{E}_1^{\frac{3}{2}}.
\end{aligned}
\end{equation}
Choosing $\mathfrak p$ sufficiently large, the conditions 
$T\leq1$ and $\sqrt{T}\mathfrak p(\mathrm{E}_1)\leq1$ ensure both $CT\mathrm{E}_1^{\frac{1}{2}}\leq1$ and a uniform bound for all the quantities in \eqref{rbbb6}.
Under the fixed-coordinate hypothesis, the same bounds follow directly from
\begin{equation}\label{revise formula c.13}
    \sD_r^{j} \eta(t)=\int_0^t \sD_r^{j}U\,\mathrm{d}s,\qquad j=2,3,4,
\end{equation}
and the fixed-coordinate versions of  \eqref{bbb2-0} and \eqref{bbb2}. 

\smallskip
\textbf{2. Estimates for $f$.} In this step, the constants depend on the bounds for the derivatives of \(\eta\) established in Step 1. Under the small-time condition of the lemma, these bounds are independent of   
$(T,\mathrm E_1)$, and so are the comparison constants.

Now, a direct calculation gives 
\begin{equation}\label{C-f-spatial}
\begin{aligned}
f_{rr}&=\eta_{rr}D_\eta f+\eta_r^2 D_\eta^2 f,\qquad f_{rrr} =\eta_{rrr}D_\eta f+3\eta_r\eta_{rr}D_\eta^2 f+\eta_r^3 D_\eta^3 f,\\
\big(\frac{f}{r}\big)_r&=\eta_r\frac{\eta}{r} D_{\eta}\big(\frac{f}{\eta}\big) +\frac{f}{\eta}\big(\frac{\eta}{r}\big)_r,\qquad \frac{1}{r}\big(\frac{f}{r}\big)_r=\eta_r\big(\frac{\eta}{r}\big)^2 \frac{1}{\eta}D_{\eta}\big(\frac{f}{\eta}\big) +\frac{f}{\eta}\frac{1}{r}\big(\frac{\eta}{r}\big)_r,\\
\big(\frac{f}{r}\big)_{rr}&=\eta_{rr}\frac{\eta}{r} D_{\eta}\big(\frac{f}{\eta}\big)+2\eta_r\big(\frac{\eta}{r}\big)_r D_{\eta}\big(\frac{f}{\eta}\big)+\eta_r^2\frac{\eta}{r} D_{\eta}^2\big(\frac{f}{\eta}\big)+\frac{f}{\eta}\big(\frac{\eta}{r}\big)_{rr}.
\end{aligned}
\end{equation}
The remaining fourth-order radial combinations are obtained from the corresponding identities in Step 1 by replacing $U$ with $f$ and the time derivatives of $\eta$ with $f$. Moreover,
\begin{equation*}
\Big(\frac{f_{rr}}{r}\Big)_r
=\Big(\frac{f}{r}\Big)_{rrr}
+2\Big(\frac{1}{r}\Big(\frac{f}{r}\Big)_r\Big)_r.
\end{equation*}
Consequently,
\begin{equation}\label{C-f-radial}
\begin{aligned}
|\mathscr D_r^2f|
&\leq C\Big(|\mathscr D_\eta^2f|+|\mathscr D_r^2\eta||\mathscr D_\eta f|\Big),\\
|\mathscr D_r^3f|
&\leq C\Big(|\mathscr D_\eta^3f|+|\mathscr D_r^2\eta||\mathscr D_\eta^2f|
+|\mathscr D_r^3\eta||\mathscr D_\eta f|\Big),\\
|\mathscr D_r^4f|
&\leq C\Big(|\mathscr D_\eta^4f|+|\mathscr D_r^2\eta||\mathscr D_\eta^3f|
+\big(|\mathscr D_r^3\eta|+|\mathscr D_r^2\eta|^2\big)|\mathscr D_\eta^2f|+|\mathscr D_r^4\eta||\mathscr D_\eta f|\Big).
\end{aligned}
\end{equation}

Applying \eqref{bbb2-0}--\eqref{bbb2} to $f$, we have
\begin{equation*}
|\mathscr D_\eta f|_\infty
\leq C\cE(t,f)^{\frac{1}{2}},\qquad
|\mathscr D_\eta^2f|_\infty
\leq C\Big(\cE(t,f)+\cD(t,f)\Big)^{\frac{1}{2}}.
\end{equation*}
These estimates, together with Step 1 and
\eqref{C-f-spatial}--\eqref{C-f-radial}, imply
\begin{equation}\label{C-forward-energy}
\mathring\cE(t,f)\leq C\cE(t,f),
\end{equation}
and
\begin{equation*}
\begin{aligned}
&|\zeta r\mathscr D_r^4f|_2^2
+\big|\chi^\sharp
\rho_0^{(\gamma-1)(\frac{3}{2}-\varepsilon_0)}
f_{rrrr}\big|_2^2\\
&\qquad\leq C\Big(\cD(t,f)+\cE(t,f)\Big).
\end{aligned}
\end{equation*}
Here the third and fourth derivatives of $\eta$ are estimated
in weighted $L^2$, while the second and first derivatives of $f$
are estimated in $L^\infty$, respectively.

For the mixed derivatives,
\begin{equation*}
f_{trr}=\eta_r^2D_\eta^2f_t+\eta_{rr}D_\eta f_t.
\end{equation*}
Since
$(\gamma-1)(\frac{3}{2}-\varepsilon_0)>\frac{\delta}{2}$,
we obtain
\begin{equation*}
\big|\chi^\sharp
\rho_0^{(\gamma-1)(\frac{3}{2}-\varepsilon_0)}
\eta_{rr}D_\eta f_t\big|_2
\leq C\big|\chi^\sharp
\rho_0^{\frac{\delta}{2}}D_\eta f_t\big|_2.
\end{equation*}
The interior estimate follows from the second-order inequality
in \eqref{C-f-radial}, applied to $f_t$.
Since the terms involving $f_{tt}$ are unchanged, we conclude that
\begin{equation}\label{C-forward-dissipation}
\mathring\cD(t,f)
\leq C\Big(\cD(t,f)+\cE(t,f)\Big).
\end{equation}

For the reverse estimates, a direct differentiation gives
\begin{equation}\label{C-inverse-spatial}
\begin{aligned}
D_\eta^2f
&=\frac{f_{rr}}{\eta_r^2}
-\frac{\eta_{rr}}{\eta_r^3}f_r,\\
D_\eta^3f
&=\frac{f_{rrr}}{\eta_r^3}
-\frac{3\eta_{rr}}{\eta_r^4}f_{rr}
+\Big(\frac{3\eta_{rr}^2}{\eta_r^5}
-\frac{\eta_{rrr}}{\eta_r^4}\Big)f_r,\\
D_\eta^4f
&=\frac{f_{rrrr}}{\eta_r^4}
-\frac{6\eta_{rr}}{\eta_r^5}f_{rrr}
+\Big(\frac{15\eta_{rr}^2}{\eta_r^6}
-\frac{4\eta_{rrr}}{\eta_r^5}\Big)f_{rr}\\
&\quad+\Big(
\frac{10\eta_{rr}\eta_{rrr}}{\eta_r^6}
-\frac{15\eta_{rr}^3}{\eta_r^7}
-\frac{\eta_{rrrr}}{\eta_r^5}
\Big)f_r.
\end{aligned}
\end{equation}
Using the positive lower bounds for $\eta_r$ and $\frac{\eta}{r}$,
we similarly invert the remaining radial identities.
Together with Step 1 and the fixed-coordinate versions of 
\eqref{bbb2-0}--\eqref{bbb2}, these identities, applied also to $f_t$ at
second order, yield
\begin{equation}\label{C-reverse-comparison}
\mathcal E(t,f)\leq C\mathring{\mathcal E}(t,f),\qquad
\mathcal D(t,f)\leq C\Big(
\mathring{\mathcal D}(t,f)+\mathring{\mathcal E}(t,f)
\Big).
\end{equation}

Finally, multiplying the dissipation estimates by $s$ and
taking the supremum over $0<s\leq t$, and then integrating
the same estimates over $[0,t]$, we control the additional
energy terms by
\begin{equation*}
\sup_{0<s\leq t}s\mathcal E(s,f)
+\int_0^t\mathcal E(s,f)\,\mathrm{d}s
\leq 2t\sup_{0<s\leq t}\mathcal E(s,f),
\end{equation*}
and the analogous inequality for $\mathring{\mathcal E}$.
Combining this with the energy estimates in both directions
gives
\begin{equation*}
\begin{aligned}
&\sup_{0<s\leq t}\Big(
\mathcal E(s,f)+s\mathcal D(s,f)
\Big)
+\int_0^t\mathcal D(s,f)\,\mathrm{d}s\\
&\qquad\sim
\sup_{0<s\leq t}\Big(
\mathring{\mathcal E}(s,f)
+s\mathring{\mathcal D}(s,f)
\Big)
+\int_0^t\mathring{\mathcal D}(s,f)\,\mathrm{d}s.
\end{aligned}
\end{equation*}
The resulting comparison constant is bounded by $C(1+T)$,
where $C$ depends on the geometric bounds from Step 1.
Under the small-time condition of the lemma, these bounds
are uniform and $T\leq1$. Hence the comparison constant
can be chosen independently of $T$ and $\mathrm{E}_1$.

This completes the proof of Lemma \ref{lemma-gaowei}.
\end{proof}

\begin{Remark}\label{remk C.1}
The comparisons in {\rm Lemma \ref{lemma-gaowei}} also hold when
\begin{equation}
\eta(t,r)=\eta_0(r)+\int_0^t U(s,r)\,\mathrm{d}s,
\end{equation}
provided that $\eta_0$ satisfies the geometric and weighted regularity assumptions of {\rm Theorem \ref{remk31}}. The comparison constants may additionally depend on $\|\mathcal F\big(\eta_0\big)\|_{L^2}$. With the initial-data bounds included in $c_0$, the comparison constants remain independent of $T$ and $\mathrm{E_1}$ provided that $T\leq1$ and $\sqrt{T}\mathfrak{p}(c_0+\mathrm{E_1})\leq1$, after enlarging $\mathfrak{p}$ if necessary.

Indeed, the identities \eqref{bbb1} are unchanged. Their integrated estimates retain the initial spatial derivatives of $\eta_0$, which are controlled by {\rm Lemmas \ref{sobolev-embedding}--\ref{hardy-inequality}}. Under the fixed-coordinate alternative, use

\begin{equation}\label{revise App C Djcobi}
\sD_r^j\eta(t)=\sD_r^j\eta_0+\int_0^t\sD_r^jU(s)\,\mathrm{d}s\qquad (j=2,3,4).
\end{equation}
The geometric derivatives are thus bounded by their initial-data contributions and time-integral contributions. The latter are controlled by the small-time condition. The estimates \eqref{C-f-spatial}--\eqref{C-reverse-comparison} and the remaining argument of {\rm Lemma \ref{lemma-gaowei}} then give the same energy and integrated dissipation comparisons.
\end{Remark}

% Then, combined with \eqref{bbb3}, yields
% \begin{equation}\label{bbb5}
% |\sD_r^2 f|\leq C |\sD_\eta^2 f|,\qquad 
% |\sD_r^3 f|\leq  Ce^{Ct} |\sD_\eta^3 f|+ |\sD_r^3 \eta| |\sD_\eta f|.
% \end{equation}

% Finally, repeating the same calculation \eqref{bbb2-0}, we have $|\sD_\eta f|_\infty\leq C(\cE(t,f))^\frac{1}{2}$. Hence, this, together with \eqref{bbb4}--\eqref{bbb5}, implies
% \begin{equation*}
% |\zeta r\sD_r^2 f|_2^2\leq Ce^{Ct}\cE_{\mathrm{in}} (t,f),\qquad
% \big|\chi^\sharp \rho_0^{(\gamma-1)(\frac{3}{2}-\varepsilon_0)}(f_{rr},f_{rrr})\big|_2^2\leq Ce^{Ct}\cE_{\mathrm{ex}} (t,f).
% \end{equation*}
% This completes the proof of Lemma \ref{lemma-gaowei}. 

\section{Cross-Derivatives Embedding}\label{subsection2.2}

The following embedding theorem is used to obtain the higher-order elliptic estimates.
\begin{Lemma}\label{prop2.1}
Assume that $\varphi=\varphi(r)$ is a function defined on $I$ and  satisfies
\begin{equation}\label{con2.9-pre}
\varphi\in C^1(\bar I)\cap C^2((0,1]),\qquad \frac{1}{K}(1-r)\leq \varphi \leq K(1-r) \ \ \text{for some $K>1$}.
\end{equation}
Let $(b,c)$ be two parameters such that
\begin{equation}\label{con2.9}
\frac{1}{2}<b\leq \frac{c+1}{2}, 
\end{equation}
and let $f=f(r)\in L^1_{\mathrm{loc}}$ satisfy both $f\in H^1_{\varphi^{2q}}(\frac{1}{2},1)$ for some $q\in [b, \frac{c+1}{2}]$ and 
\begin{equation}\label{con2.10}
\big|\zeta^\sharp(\varphi^{b}f_r+ c\varphi^{b-1}\varphi_rf)\big|_2+|\chi^\sharp\varphi^{b} f|_2<\infty.
\end{equation}
Then, for any $\varphi$ satisfying \eqref{con2.9-pre} and $(b,c)$ satisfying \eqref{con2.9},  there exists a constant $C>0$, which depends only on $(\varphi,b,c)$, such that, for all $f$ satisfying \eqref{con2.10},
\begin{equation}\label{con2.11}
|\zeta^\sharp\varphi^{b}f_r|_2\leq C\big(\big|\zeta^\sharp(\varphi^{b}f_r+c\varphi^{b-1}\varphi_rf)\big|_2+|\chi^\sharp\varphi^{b} f|_2\big).
\end{equation}
\end{Lemma}

\begin{proof}
We divide the proof into three steps.

\smallskip
\textbf{1. Case $f\in C^\infty([\frac{1}{2},1])$.}
It follows from integration by parts, \eqref{con2.9-pre}--\eqref{con2.10}, Lemma \ref{GNinequality}, and the Young inequality that
\begin{equation}\label{2004}
\!\!\!\!\begin{aligned}
|\zeta^\sharp\varphi^b f_r|_2^2&=\big|\zeta^\sharp(\varphi^b f_r+c\varphi^{b-1}\varphi_r f)\big|_2^2-c^2|\zeta^\sharp\varphi^{b-1}\varphi_rf|_2^2 - 2c\int_0^1 (\zeta^\sharp)^2\varphi^{2b-1}\varphi_r  f f_r\,\mathrm{d}r\\
&=\big|\zeta^\sharp(\varphi^b f_r+c\varphi^{b-1}\varphi_r f)\big|_2^2+ c\int_0^1 \big(2\zeta^\sharp(\zeta^\sharp)_r\varphi_{r}+(\zeta^\sharp)^2\varphi_{rr}\big) \varphi^{2b-1} f^2\,\mathrm{d}r\\
& \quad +\underline{ (2b-1-c)c\int_0^1 (\zeta^\sharp)^2\varphi^{2b-2} (\varphi_r)^2 f^2\,\mathrm{d}r}_{\,\leq 0}- \underline{c(\zeta^\sharp)^2\varphi^{2b-1}\varphi_r f^2\Big|_{r=0}^{r=1}}_{\, =0}\\
&\leq \big|\zeta^\sharp(\varphi^b f_r+c\varphi^{b-1}\varphi_r f)\big|_2^2+C\big(|(\zeta^\sharp)_r\varphi^{-1} \varphi_{r}|_\infty|\chi^\sharp\varphi^{b}f|_2^2+ |\chi^\sharp\varphi_{rr}|_\infty\big|\zeta^\sharp\varphi^{b-\frac{1}{2}}f\big|_2^2\big)\\
&\leq  \big|\zeta^\sharp(\varphi^b f_r+c\varphi^{b-1}\varphi_r f)\big|_2^2+C|\chi^\sharp\varphi^{b}f|_2^2+\frac{1}{2}|\zeta^\sharp\varphi^b f_r|_2^2,
\end{aligned}
\end{equation}
which yields \eqref{con2.11}.

\smallskip
\textbf{2. Case  $f\in H^1_{\varphi^{2b}}(\frac{1}{2},1)$.} In this case, we can repeat the calculation in \eqref{2004} to derive \eqref{con2.11}, except for justifying the following integral equality: 
\begin{equation}\label{fenbujifen}
\begin{aligned}
-2\int_0^1 (\zeta^\sharp)^2\varphi^{2b-1}\varphi_r  f f_r\,\mathrm{d}r
&=  \int_0^1 \big(2\zeta^\sharp(\zeta^\sharp)_r\varphi_{r}+(\zeta^\sharp)^2\varphi_{rr}\big) \varphi^{2b-1} f^2\,\mathrm{d}r\\
&\quad\,\, +(2b-1)\int_0^1 (\zeta^\sharp)^2\varphi^{2b-2}(\varphi_r)^2 f^2\,\mathrm{d}r.
\end{aligned}
\end{equation}
Indeed, thanks to Lemma \ref{W-space}, there exists a sequence $\{f^\varepsilon\}_{\varepsilon>0}\subset C^\infty([\frac{1}{2},1])$ such that
\begin{equation}\label{b6b}
|\chi^\sharp\varphi^b (f^\varepsilon- f)|_2+|\chi^\sharp\varphi^b(f^\varepsilon_r- f_r)|_2\to 0 \qquad \text{as $\varepsilon\to 0$},
\end{equation}
which, along with Lemma \ref{hardy-inequality}, yields
\begin{equation}\label{b7b}
|\chi^\sharp\varphi^{b-1}(f^\varepsilon-f)|_2+\big|\chi^\sharp\varphi^{b-\frac{1}{2}}(f^\varepsilon-f)\big|_\infty \to 0 \qquad \text{as $\varepsilon\to 0$}.
\end{equation}
Hence, according to \eqref{b6b}--\eqref{b7b}  and integration by parts for $f^\varepsilon$, we have
\begin{equation*}
\begin{aligned}
-2\int_0^1 (\zeta^\sharp)^2\varphi^{2b-1}\varphi_r f^\varepsilon f^\varepsilon_r\,\mathrm{d}r
&=  \int_0^1 \big(2\zeta^\sharp(\zeta^\sharp)_r\varphi_{r}+(\zeta^\sharp)^2\varphi_{rr}\big) \varphi^{2b-1} (f^\varepsilon)^2\,\mathrm{d}r\\
&\quad +(2b-1)\int_0^1 (\zeta^\sharp)^2\varphi^{2b-2}(\varphi_r)^2 (f^\varepsilon)^2\,\mathrm{d}r.
\end{aligned}
\end{equation*}
Letting $\varepsilon\to 0$ implies that \eqref{fenbujifen} holds for all $f\in H^1_{\varphi^{2b}}(\frac{1}{2},1)$. 

This completes the proof of \eqref{con2.11} when $q=b$.

\smallskip
\textbf{3. General case.}
It suffices to establish \eqref{con2.11} when \eqref{con2.10} holds and
\begin{equation*}
b<\frac{c+1}{2}, \qquad  q=\frac{c+1}{2},\qquad \ f\in H^1_{\varphi^{c+1}}\big(\frac{1}{2},1\big),
\end{equation*}
due to the fact that $H^1_{\varphi^{2q}}(\frac{1}{2},1)\subset H^1_{\varphi^{c+1}}(\frac{1}{2},1)$ if $q\leq \frac{c+1}{2}$. Note that, in this case, integration by parts in \eqref{fenbujifen} fails owing to $(\zeta^\sharp)^2\varphi^{2b-1}\varphi_r ff_r\notin L^1$. 

To overcome this difficulty, set
\begin{equation*}
\vartheta:= \frac{c+1}{2}-b,\qquad \varphi_j:=\varphi+\frac{1}{j} \quad \text{for $j\in \NN^*$}.
\end{equation*}
We first show a variant of \eqref{G-N1} in Lemma \ref{GNinequality}, that is, for all $f\in H^1_{\varphi^{c+1}}(\frac{1}{2},1)$,
\begin{equation}\label{2..8}
\big|\zeta^\sharp\varphi^\frac{c}{2}\varphi_j^{-\vartheta}f\big|_2^2\leq C\big(\big|\chi^\sharp\varphi^\frac{c+1}{2}\varphi_j^{-\vartheta}f\big|_2^2+\big|\zeta^\sharp\varphi^\frac{c+1}{2}\varphi_j^{-\vartheta}f\big|_2\big|\zeta^\sharp\varphi^\frac{c+1}{2}\varphi_j^{-\vartheta}f_r\big|_2\big).
\end{equation}

Based on Lemma \ref{W-space} and the proof in Lemma \ref{GNinequality}, it suffices to show that \eqref{2..8} holds for $f\in C^\infty(\frac{1}{2},1)$. Clearly, we can further let $(\varphi,\varphi_j)=(d,d_j)$ with $d=d(r):=1-r$ and $d_j:=d+\frac{1}{j}$, due to
\begin{equation*}
\frac{d}{K}\leq \varphi \leq K d,\qquad \frac{d_j}{K}\leq \varphi_j \leq K d_j.
\end{equation*} 
It follows from the above reductions, integration by parts, and $\mathrm{d}(d^{c+1})=-(c+1)d^c\mathrm{d}r$ that
\begin{equation*}
\begin{aligned}
\int_0^1 (\zeta^\sharp)^2 d^c d_j^{-2\vartheta} f^2\,\mathrm{d}r&= \frac{2}{c+1} \Big( \int_0^1 \zeta^\sharp(\zeta^\sharp)_r d^{c+1} d_j^{-2\vartheta} f^2\,\mathrm{d}r+\int_0^1 (\zeta^\sharp)^2 d^{c+1} d_j^{-2\vartheta} f f_r\,\mathrm{d}r\Big)\\
&\quad +\frac{2\vartheta}{c+1} \int_0^1 (\zeta^\sharp)^2 d^{c+1} d_j^{-2\vartheta-1} f^2 \,\mathrm{d}r\\
&\leq  \frac{2}{c+1} \Big( \int_0^1 \zeta^\sharp(\zeta^\sharp)_r d^{c+1} d_j^{-2\vartheta} f^2\,\mathrm{d}r+\int_0^1 (\zeta^\sharp)^2 d^{c+1} d_j^{-2\vartheta} f f_r\,\mathrm{d}r\Big)\\
&\quad +\frac{2\vartheta}{c+1} \int_0^1 (\zeta^\sharp)^2 d^{c} d_j^{-2\vartheta} f^2 \,\mathrm{d}r,
\end{aligned}
\end{equation*}
which implies
\begin{equation}\label{2..9}
\begin{aligned}
\int_0^1 (\zeta^\sharp)^2 d^c d_j^{-2\vartheta} f^2\,\mathrm{d}r& \leq  \frac{1}{b} \Big( \int_0^1 \zeta^\sharp(\zeta^\sharp)_r d^{c+1} d_j^{-2\vartheta} f^2\,\mathrm{d}r+\int_0^1 (\zeta^\sharp)^2 d^{c+1} d_j^{-2\vartheta} f f_r\,\mathrm{d}r\Big)\\
&\leq C\big(\big|\chi^\sharp d^\frac{c+1}{2} d_j^{-\vartheta}f\big|_2^2+\big|\zeta^\sharp d^\frac{c+1}{2} d_j^{-\vartheta}f\big|_2\big|\zeta^\sharp d^\frac{c+1}{2} d_j^{-\vartheta}f_r\big|_2\big).
\end{aligned}
\end{equation}

This completes the proof of \eqref{2..8}.

\smallskip
Now, we continue to prove \eqref{con2.11}. It follows from \eqref{con2.10} and $0\leq \varphi/\varphi_j\leq 1$ that
\begin{equation}\label{con2.15}
Q_j:=\zeta^\sharp\varphi_j^{-\vartheta}(\varphi^\frac{c+1}{2}f_r+c\varphi^\frac{c-1}{2}\varphi_{r}f)\in L^2 \qquad \text{for any $j\in \NN^*$}.
\end{equation}
Using a density argument similar to that in Step 2, we can show that the following integral equality still holds, {\it i.e.}, for all $f\in H^1_{\varphi^{c+1}}(\frac{1}{2},1)$ and $j\in \NN^*$,
\begin{equation*}
\begin{aligned}
-2\int_0^1 (\zeta^\sharp)^2\varphi^{c}\varphi_j^{-2\vartheta}\varphi_r f f_r\,\mathrm{d}r&=  \int_0^1 \big(2\zeta^\sharp(\zeta^\sharp)_r\varphi_{r}+(\zeta^\sharp)^2\varphi_{rr}\big) \varphi^{c}\varphi_j^{-2\vartheta} f^2\,\mathrm{d}r\\
&\quad +\int_0^1 (\zeta^\sharp)^2\big(c\varphi^{c-1}\varphi_j^{-2\vartheta} -2\vartheta \varphi^{c}\varphi_j^{-2\vartheta-1}\big)(\varphi_r)^2 f^2\,\mathrm{d}r.
\end{aligned}
\end{equation*}
Hence, based on the above equality and the calculation similar to \eqref{2004}, we deduce from \eqref{con2.15} and the Young inequality that
\begin{equation*}
\begin{aligned}
\big|\zeta^\sharp\varphi^\frac{c+1}{2}\varphi_j^{-\vartheta} f_r\big|_2^2&=|Q_j|_2^2-c^2\big|\zeta^\sharp\varphi^{\frac{c-1}{2}}\varphi_j^{-\vartheta}\varphi_rf\big|_2^2 - 2c\int_0^1 (\zeta^\sharp)^2\varphi^{c}\varphi_j^{-2\vartheta}\varphi_r  f f_r\,\mathrm{d}r\\
&=|Q_j|_2^2+ c\int_0^1 \big(2\zeta^\sharp(\zeta^\sharp)_r\varphi_{r}+(\zeta^\sharp)^2\varphi_{rr}\big) \varphi^{c}\varphi_j^{-2\vartheta} f^2\,\mathrm{d}r\notag\\
&\quad \underline{-2\vartheta c\int_0^1 (\zeta^\sharp)^2 \varphi^{c}\varphi_j^{-2\vartheta-1} (\varphi_r)^2 f^2\,\mathrm{d}r}_{\,\leq 0}\\
&\leq |Q_j|_2^2+C\big(|(\zeta^\sharp)_r\varphi^{-1} \varphi_{r}|_\infty\big|\chi^\sharp\varphi^{\frac{c+1}{2}}\varphi_j^{-\vartheta}f\big|_2^2+ |\chi^\sharp\varphi_{rr}|_\infty\big|\zeta^\sharp\varphi^{\frac{c}{2}}\varphi_j^{-\vartheta}f\big|_2^2\big),
\end{aligned} 
\end{equation*}
which, along with \eqref{con2.9-pre} and \eqref{2..8}, gives 
\begin{equation*}
\begin{aligned}
\big|\zeta^\sharp\varphi^{\frac{c+1}{2}}\varphi_j^{-\vartheta}f_r\big|_2^2&\leq C\big(|Q_j|_2^2+\big|\chi^\sharp\varphi^{\frac{c+1}{2}}\varphi_j^{-\vartheta}f\big|_2^2\big)\\
&\leq C\big(\big|\zeta^\sharp(\varphi^{b}f_r+ c\varphi^{b-1}\varphi_rf)\big|^2_2+|\chi^\sharp\varphi^{b} f|^2_2\big).
\end{aligned}
\end{equation*}
Since $C$ is independent of $j$, we can extract a subsequence (still denoted by $j$)  such that
\begin{equation}\label{equ213}
\zeta^\sharp\varphi^{\frac{c+1}{2}}\varphi_j^{-\vartheta}f_r\to g\qquad \text{weakly in }L^2 \quad \text{as }j\to \infty,
\end{equation}
for some limit function $g\in L^2$, and 
\begin{equation*}
|g|_2\leq \liminf_{j\to \infty}\big|\zeta^\sharp\varphi^{\frac{c+1}{2}}\varphi_j^{-\vartheta}f_r\big|_2\leq C\big(\big|\zeta^\sharp(\varphi^{b}f_r+ c\varphi^{b-1}\varphi_rf)\big|_2+|\chi^\sharp\varphi^{b} f|_2\big). 
\end{equation*}
 
Note that $f_r\in L^1_{\mathrm{loc}}$, the Lebesgue dominated convergence theorem also gives
\begin{equation}\label{equ214}
\zeta^\sharp\varphi^{\frac{c+1}{2}}\varphi_j^{-\vartheta}f_r\to \zeta^\sharp\varphi^b f_r\qquad \text{in }L^1_{\mathrm{loc}} \ \ \text{as }j\to\infty.
\end{equation}
Hence, by \eqref{equ213}--\eqref{equ214} and the uniqueness of the limits, we have $g=\zeta^\sharp\varphi^b f_r$. 

This completes the proof of Lemma \ref{prop2.1}.
\end{proof}

\bigskip
\noindent{\bf Acknowledgments:}   This research is partially supported by National Key R$\&$D Program of China (No. 2022YFA1007300),  National Natural Science Foundation of China under the Grant 12471212, and The Royal Society (UK)-Newton International Fellowships NF170015.

\bigskip
\noindent{\bf Conflict of Interest:} The authors declare that they have no conflict of interest. 

\bigskip
\noindent{\bf Data availability:} Data sharing is not applicable to this article as no datasets were generated or analyzed during the current study.

\end{document}